\documentclass[leqno,a4paper,12pt]{book}
\usepackage{amssymb,amsmath,amsthm}
\usepackage[all]{xy}
\usepackage{color}
\usepackage{enumerate}
\usepackage{mathrsfs}

\newcommand{\symplecto}{\omega}
\newcommand{\nc}{\newcommand}
\nc{\on}{\operatorname}

\newtheorem{theorem}{Theorem}[section]
\newtheorem{proposition}[theorem]{Proposition}
\newtheorem{lemma}[theorem]{Lemma}
\newtheorem{corollary}[theorem]{Corollary}

\theoremstyle{definition}

\newtheorem{definition}[theorem]{Definition}
\newtheorem{notation}[theorem]{Notation}
\newtheorem{example}[theorem]{Example}

\newtheorem{remark}[theorem]{Remark}

\newtheorem{conjecture}[theorem]{Conjecture}
\newtheorem{convention}[theorem]{Convention}

\nc{\Prop}{\begin{proposition}}
\nc{\enprop}{\end{proposition}}
\nc{\Lemma}{\begin{lemma}}
\nc{\enlemma}{\end{lemma}}
\nc{\Cor}{\begin{corollary}}
\nc{\encor}{\end{corollary}}
\nc{\Def}{\begin{definition}}
\nc{\enDef}{\end{definition}}
\nc{\Th}{\begin{theorem}}
\nc{\entheorem}{\end{theorem}}

\newcommand{\C}{{\mathbb{C}}}
\newcommand{\N}{{\mathbb{N}}}
\newcommand{\R}{{\mathbb{R}}}

\newcommand{\Z}{{\mathbb{Z}}}

\newcommand{\cor}{\C^{\hbar,{\rm loc}}}
\newcommand{\coro}{\C^\hbar}
\newcommand{\cora}{{\BBK}}

\def\BBK{\mathbb{K}}

\def\phi{{\varphi}}
\def\epsilon{\varepsilon}

\newcommand{\BB}{\mathfrak B}
\newcommand{\M}{\mathscr{M}}

\def\sha{\mathscr{A}}
\def\shao{{\mathscr{A}}_0}
\def\shro{{\mathscr{R}}_0}
\def\shb{\mathscr{B}}
\def\shc{\mathscr{C}}
\def\shd{\mathscr{D}}
\def\she{\mathscr{E}}
\def\shf{\mathscr{F}}
\def\shg{\mathscr{G}}

\def\shi{\mathscr{I}}
\def\shj{\mathscr{J}}
\def\shk{\mathscr{K}}
\def\shl{\mathscr{L}}
\def\shm{\mathscr{M}}
\def\shn{\mathscr{N}}
\def\sho{\mathscr{O}}
\def\shp{\mathscr{P}}
\def\shq{\mathscr{Q}}
\def\shr{\mathscr{R}}
\def\shs{\mathscr{S}}

\def\shu{\mathscr{U}}

\def\shw{\mathscr{W}}

\newcommand{\ol}{\overline}
\newcommand{\bl}{\bigl(}
\newcommand{\br}{\bigr)}
\newcommand{\ro}{{\rm(}}
\newcommand{\rf}{\,{\rm)}}

\newcommand{\lp}{{\rm(}}
\newcommand{\rp}{{\rm)}}

\nc{\RR}{\mathrm{R}}
\nc{\LL}{\mathrm{L}}

\newcommand{\db}[1]{\raisebox{-.5ex}[2ex][1.8ex]{$#1$}}

\newcommand{\bN}{\bar{\shn}}
\newcommand{\tN}{\widetilde{\shn}}

\newcommand{\Ad}{\mathrm{Ad}}

\newcommand{\D}[1][]{\mathscr{D}_{#1}}
\newcommand{\Dl}[1][]{\mathscr{D}_{#1}^\loc}

\nc{\CH}{\on{H}}

\newcommand{\HE}[1][X]{\widehat{\mathscr{E}}_{#1}}
\newcommand{\HEo}[1][X]{\widehat{\mathscr{E}}_{#1}(0)}
\newcommand{\HW}[1][X]{\widehat{\mathscr{W}}_{#1}}
\newcommand{\HWo}[1][X]{\widehat{\mathscr{W}}_{#1}(0)}

\newcommand{\A}[1][X]{\mathscr{A}_{{#1}}}
\newcommand{\Ah}[1][X]{\mathscr{A}_{{#1}}^{\rm loc}}
\nc{\DA}[1][X]{\mathscr{D}^{\mathscr{A}}_{#1}}
\nc{\Dh}[1][X]{\mathscr{D}_{#1}[[\hbar]]}
\nc{\Dhh}[1][X]{\mathscr{D}_{#1}[\hbar]}
\nc{\Dhhh}[1][X]{\mathscr{D}_{#1}[\hbar,\opb{\hbar}]}

\nc{\Dhl}[1][X]{\shd_{#1}\Ls}
\nc{\Dho}[1][X]{\mathscr{D}_{#1}^\hbar(0)}
\nc{\Oh}[1][X]{\sho_{#1}^\hbar}
\nc{\OOh}[1][X]{\sho_{#1}[[\hbar]]}
\nc{\oh}[1][X]{\Omega_{#1}[[\hbar]]}

\nc{\Ohl}[1][X]{\sho_{#1}^{\hbar,\loc}}
\nc{\ohh}[1][X]{\Omega_{#1}^\hbar}

\newcommand{\OO}[1][X]{\mathscr{O}_{#1}}

\newcommand{\W}[1][]{\mathscr{W}_{#1}}

\newcommand{\gre}{\operatorname{\mathrm{gr}_{{\she}}}}
\newcommand{\gr}{\mathrm{gr}_\hbar}

\newcommand{\Lie}[1][]{\operatorname{\mathsf{L}}\def\temp{#1}
\ifx\temp\empty\else^{(#1)}\fi}

\newcommand{\rmo}{\mathscr{O}}
\newcommand{\rma}{\mathscr{A}}
\newcommand{\rmW}{\mathrm{W}}

\newcommand{\DQ}{\ensuremath{\mathrm{DQ}}}
\newcommand{\sts}{{\mathfrak{S}}}

\newcommand{\sta}{{\mathfrak{A}}}
\newcommand{\stx}{X}

\newcommand{\stkHom}[1][]{\mathfrak{Hom}_{\raise1.5ex\hbox to.1em{}#1}}
\newcommand{\stkMod}{\mathfrak{Mod}}

\renewcommand{\ker}{\operatorname{Ker}}
\DeclareMathOperator{\coker}{Coker}
\DeclareMathOperator{\im}{Im}

\newcommand{\rmpt}{{\rm pt}}
\newcommand{\rmptt}{{\{\rm pt\}}}
\newcommand{\eu}{{\rm eu}}

\newcommand{\hh}{{\rm hh}}
\newcommand{\thh}{{\rm thh}}

\newcommand{\ch}{{\rm ch}}
\newcommand{\hhg}{\rm hh^{gr}}
\newcommand{\hhhg}{\widehat{\rm hh}^{\rm gr}}
\newcommand{\hhl}{\rm hh}

\newcommand{\td}{{\rm td}}
\newcommand{\htd}{{\rm td}}
\newcommand{\Int}{{\rm Int}}
\newcommand{\loc}{{\rm loc}}
\newcommand{\rmK}{{\mathrm K}}
\newcommand{\rmhK}{\widehat{\mathrm K}}
\newcommand{\lc}{{\rm lc}}
\newcommand{\rmb}{{\mathrm b}}

\def\monoto{{\rightarrowtail}}
\def\epito{{\twoheadrightarrow}}
\renewcommand{\to}[1][]{\xrightarrow[]{#1}}

\newcommand{\isoto}[1][]{\xrightarrow[#1]%
{{\raisebox{-.6ex}[0ex][-.6ex]{$\mspace{1mu}\sim\mspace{2mu}$}}}}
\newcommand{\isofrom}[1][]{\xleftarrow[#1]{\sim}}

\newcommand{\To}[1][]{\xrightarrow[]{\mspace{10mu}{#1}\mspace{10mu}}}

\DeclareMathOperator{\id}{id}

\newcommand{\Hom}[1][]{\mathrm{Hom}_{\raise1.5ex\hbox to.1em{}#1}}
\newcommand{\RHom}[1][]{\RR\mathrm{Hom}_{\raise1.5ex\hbox to.1em{}#1}}
\newcommand{\Ext}[2][]{\mathrm{Ext}_{\raise1.5ex\hbox to.1em{}#1}^{#2}}
\renewcommand{\hom}[1][]{{\mathscr{H}\mspace{-4mu}om}_{\raise1.5ex\hbox to.1em{}#1}}
\newcommand{\rhom}[1][]{{\RR\mathscr{H}\mspace{-3mu}om}_{\raise1.5ex\hbox to.1em{}#1}}
\newcommand{\ext}[2][]{{\mathscr{E}\mspace{-2mu}xt}_{%
\raise1.5ex\hbox to.1em{}#1}^{#2}}
\newcommand{\BHom}[1][]{\mathrm{Bhom}_{\raise1.5ex\hbox to.1em{}#1}}

\newcommand{\Tens}[1][]{\mathbin{\otimes_{\raise1.5ex\hbox to-.1em{}{#1}}}}
\newcommand{\LTens}[1][]{\mathbin{\otimes_{\raise1.5ex\hbox to-.1em{}#1}^{L}}}
\newcommand{\Tor}[2][]{\mathrm{Tor}^{\raise1.5ex\hbox to.1em{}#1}_{#2}}
\newcommand{\tens}[1][]{\mathbin{\otimes_{\raise1.5ex\hbox to-.1em{}{#1}}}}
\newcommand{\detens}{\underline{\etens}}
\newcommand{\ldetens}{\overset{\mathrm{L}}{\underline{\etens}}}
\newcommand{\dtens}[1][]%
{{\overset{\mathrm{L}}{\underline{\otimes}}}_{#1}}
\newcommand{\tor}{\mathscr{T}\mspace{-3mu}or}

\newcommand{\etens}{\mathbin{\boxtimes}}
\newcommand{\letens}{\overset{\mathrm{L}}{\etens}}

\newcommand{\lltens}[1][]{{\mathop{\tens}\limits^{\rm L}}_{#1}}
\newcommand{\ltens}[1][]{{\mathop{\tens}\limits^{\rm L}}_{#1}}

\nc{\aut}{{\sha\mspace{-1mu}\mit{ut}}\,}
\newcommand{\End}{{\operatorname{End}}}
\newcommand{\shend}{\operatorname{{\she\mspace{-2mu}\mathit{nd}}}}
\newcommand{\Endo}[1][]{\mathrm{End}_{\raise1.5ex\hbox to.1em{}#1}}

\newcommand{\Aut}[1][]{\mathrm{Aut}_{\raise1.5ex\hbox to.1em{}#1}}
\newcommand{\sect}{\Gamma}
\newcommand{\rsect}{\mathrm{R}\Gamma}

\newcommand{\RC}{{\rm C}}

\newcommand{\Rb}{{\rm b}}
\newcommand{\coh}{{\rm coh}}
\newcommand{\grad}{{\rm grad}}
\newcommand{\gd}{{\rm gd}}

\newcommand{\Fl}{{\rm F}}
\newcommand{\SSi}{\on{SS}}

\newcommand{\oim}[1]{{#1}_*}
\newcommand{\eim}[1]{{#1}_!}
\newcommand{\roim}[1]{\RR{#1}_*}
\newcommand{\reim}[1]{\RR{#1}_!}
\newcommand{\opb}[1]{#1^{-1}}
\newcommand{\epb}[1]{#1^{!}}
\newcommand{\spb}[1]{#1^{*}}
\newcommand{\lspb}[1]{\LL{#1}^{*}}

\newcommand{\tw}[1]{\widetilde{#1}}
\newcommand{\twh}[1]{\widehat{#1}}

\DeclareMathOperator{\Ob}{Ob}
\DeclareMathOperator{\supp}{supp}

\DeclareMathOperator{\chv}{char}

\newcommand{\Fct}{\operatorname{Fct}}
\newcommand{\for}{\mathit{for}}

\newcommand{\Set}{{\bf Set}}

\def\rop{{\rm op}}
\def\op{{\rm op}}

\newcommand{\Xan}{{X_{\rm an}}}

\newcommand{\hol}{{\rm hol}}
\newcommand{\reg}{{\rm reg}}

\newcommand{\Cc}{{\C\rm c}}
\newcommand{\wCc}{{{\rm w}\C\rm c}}
\newcommand{\Rc}{{\R\rm c}}

\newcommand{\md[1]}{{\rm Mod}({#1})}

\newcommand{\mdcoh[1]}{{\rm Mod_{{\rm coh}}}({#1})}
\newcommand{\mdgd[1]}{{\rm Mod_{{\rm gd}}}({#1})}

\newcommand{\Mod}{\mathrm{Mod}}

\newcommand{\mdff[1]}{{\rm Mod_{ff}}({#1})}

\newcommand{\indlim}[1][]{\mathop{\varinjlim}\limits_{#1}}

\newcommand{\Pro}{\mathrm{Pro}}

\newcommand{\prolim}[1][]{\mathop{\varprojlim}\limits_{#1}}
\newcommand{\sprolim}[1][]{\smash{\mathop{\varprojlim}\limits_{#1}}\,}

\newcommand{\eqdot}{\mathbin{:=}}
\newcommand{\seteq}{\mathbin{:=}}

\newcommand{\cl}{\colon}
\newcommand{\scbul}{{\,\raise.4ex\hbox{$\scriptscriptstyle\bullet$}\,}}

\def\suba#1{{\raise-.2em\hbox{${\raise .2em\hbox{$\alpha$}}_{#1}$}}}

\newcommand{\ba}{\begin{array}}
\newcommand{\ea}{\end{array}}

\newcommand{\bnum}{\begin{enumerate}[{\rm(i)}]}
\newcommand{\enum}{\end{enumerate}}
\newcommand{\banum}{\begin{enumerate}[{\rm(a)}]}
\newcommand{\eanum}{\end{enumerate}}

\newcommand{\eq}{\begin{eqnarray}}
\newcommand{\eneq}{\end{eqnarray}}
\newcommand{\eqn}{\begin{eqnarray*}}
\newcommand{\eneqn}{\end{eqnarray*}}

\newcommand{\set}[2]{\left\{#1 \mathbin{;} #2 \right\}}

\nc{\Der}{\on{D}}

\nc{\Proof}{\begin{proof}}
\nc{\QED}{\end{proof}}
\nc{\bA}[1][X]{\gr({\sha}_{#1})}
\nc{\OA}[1][X]{\Omega_{#1}^{\,\mathscr{A}}}
\nc{\oA}[1][X]{\omega_{#1}^{\,\mathscr{A}}}
\nc{\omA}[1][X]{\omega_{#1}}
\nc{\oAI}[1][X]{\omega_{#1}^{\,\mathscr{A}{\otimes-1}}}
\nc{\omAI}[1][X]{\omega_{#1}^{{\otimes-1}}}
\nc{\oAh}[1][X]{\omega_{#1}^{\,\Ah[]}}
\nc{\oo}[1][X]{\omega_{#1}}
\nc{\Derb}{\mathrm{D}^{\mathrm{b}}}
\nc{\hs}{\hspace*}
\nc{\Supp}{\on{Supp}}
\nc{\tr}{\on{tr}}

\newcommand{\HHA}[1][X]{\mathcal{HH}({\sha}_{#1})}

\newcommand{\HHH}[1][]{\mathcal{HH}({#1})}
\newcommand{\RHH}[2][]{\mathrm{HH}_{#1}({\sha}_{#2})}
\newcommand{\RHHl}[2][]{\mathrm{HH}_{#1}(\sha^{\rm loc}_{#2})}
\newcommand{\RHHg}[2][]{\mathrm{HH}_{#1}(\gr\sha_{#2})}
\newcommand{\RHHhg}[2][]{\widehat{\mathrm{HH}}_{#1}(\gr\sha_{#2})}
\newcommand{\RHHO}[2][]{\mathrm{HH}_{#1}(\sho_{#2})}

\newcommand{\HHAh}[1][X]{\mathcal{HH}(\sha_{#1}^{\rm loc})}
\newcommand{\HHGA}[1][X]{\mathcal{HH}({\gr(\sha}_{#1}))}
\newcommand{\HHO}[1][X]{\mathcal{HH}({\sho}_{#1})}
\newcommand{\HHE}[1][X]{\mathcal{HH}(\widehat{\she}_{#1})}

\newcommand{\HHW}[1][X]{\mathcal{HH}(\widehat{\shw}_{#1})}
\newcommand{\HHWo}[1][X]{\mathcal{HH}(\HWo[])}

\newcommand{\HOD}[1][X]{\mathcal{HD}({\sho}_{#1})}

\newcommand{\RD}{{\rm D}}
\nc{\RDAA}[1][X]{\mathrm{D}^\prime_{\mathscr{A}_{#1}}}
\nc{\RDA}{\mathrm{D}^\prime_{\mathscr{A}}}
\nc{\RDAh}{\mathrm{D}^\prime_{\mathscr{A}^{\rm loc}}}
\nc{\RDO}{\mathrm{D}^\prime_{\mathscr{O}}}
\nc{\RDOl}{\mathrm{D}^\prime_{\mathscr{O}^{\hbar,\rm loc}}}
\nc{\RDDO}{\mathrm{D}_{\mathscr{O}}}
\nc{\RDDA}{\mathrm{D}_{\mathscr{A}}}
\nc{\RDD}{\mathrm{D}^\prime}
\nc{\conv}[1][]{\mathop{\circ}\limits_{#1}}
\nc{\sconv}[1][]{\mathop{\ast}\limits_{#1}}

\nc{\ssum}{\mathop{\mbox{\normalsize$\sum$}}}
\nc{\de}[1][X]{\delta_{#1}} 
\nc{\vs}{\vspace}
\nc{\dA}[1][X]{\mathscr{C}_{{#1}}}
\nc{\dO}[1][X]{{{\delta_{{#1}}}_*\mspace{1mu}\mathscr{O}_{{#1}}}}
\nc{\dGA}[1][X]{\gr(\mathscr{C}_{{#1}})}
\newcommand{\mdaf}[1][\sha]{{\rm Mod^{af}}({#1})}
\newcommand{\mdafd}[1][\sha]{{\rm Mod_{af}}({#1})}
\nc{\OS}[1][X]{\sho_{#1}}
\nc{\soplus}{\mathop{\text{\scriptsize\raisebox{.5ex}{$\displaystyle\bigoplus$}}}}
\nc{\Inv}{\on{Inv}}
\nc{\stkInv}{\mathfrak{Inv}}
\nc{\bwr}{{\mbox{\large$\wr$}}}
\nc{\forl}{[\mspace{-.3mu}[\hbar]\mspace{-.3mu}]}
\nc{\Ls}{(\mspace{-.3mu}(\hbar)\mspace{-.3mu})}

\nc{\be}{\begin{enumerate}}
\nc{\ee}{\end{enumerate}}
\nc{\stan}{\mathrm{stan}}
\nc{\Db}{\RD^\Rb}
\nc{\pt}{\mathrm{pt}}
\nc{\BBD}{\mathbb{D}}
\nc{\rC}{\mathrm{C}}
\nc{\scup}{\mathop{\text{\scriptsize\raisebox{.5ex}{$\displaystyle\bigcup$}}}}
\nc{\tX}{{\widetilde{X}}}
\nc{\AL}{\A[\Lambda/X]}
\nc{\ALa}{\A[\Lambda^a]}
\nc{\Gr}{\on{Gr}}
\nc{\CL}{\on{\mathrm{C}_\Lambda}}
\nc{\codim}{\on{codim}}
\nc{\Chl}{\on{char_{\Lambda}}}
\nc{\Chlo}{\on{char_{\Lambda_0}}}
\nc{\ChM}{\on{char_{M}}}
\nc{\DL}{\shd_\shl}
\nc{\DLl}{\shd_\shl^\loc}
\nc{\rmC}{\rm C}

\nc{\Zh}{\Z\forl}
\nc{\PZ}{\Pro(\Zh)}
\nc{\coco}{{cohomologically complete}}
\nc{\rpi}{{{\rm R}\pi}}
\nc{\shal}{{\sha^\loc}}

\newcommand{\proolim}[1][]{\mathop{``{\varprojlim}"}\limits_{#1}}
\newcommand{\sproolim}[1][]{\smash{\mathop{``{\varprojlim}"}\limits_{#1}}}

\newcommand{\inddlim}[1][]{\mathop{``{\varinjlim}"}\limits_{#1}}

\newcommand{\bbigsqcup}{\mathop{``\bigsqcup"}\limits}

\nc{\cc}{cohomologically complete}
\nc{\shrl}{{\shr^\loc}}

\usepackage{enumerate}
\usepackage{mathrsfs}

\numberwithin{equation}{section}

\begin{document}
\author{Masaki Kashiwara and Pierre Schapira}
\title{Deformation quantization modules}
\date{}
\maketitle
\tableofcontents

\chapter*{Introduction}\label{chapter:In}
In a few words these Notes could be considered both as an introduction to 
non commutative complex analytic geometry and to the study of
microdifferential systems. Indeed, on a complex
manifold $X$, we  replace the structure sheaf $\sho_X$ with a formal
deformation of it, that is, a $\DQ$-algebra, or better, a
$\DQ$-algebroid, and study modules over this ring, extending to this
framework classical results of Cartan-Serre and Grauert, and also  
classical results on Hochschild classes and the index theorem.  
Here, $\DQ$ stands for ``deformation quantization''. But 
 the theory of modules over  $\DQ$-algebroids is also 
a natural generalization of that of  $\shd$-modules. Indeed, when 
the Poisson structure underlying the deformation is
symplectic, the study of $\DQ$-modules naturally
generalizes that of microdifferential modules, 
and sometimes makes it easier (see Theorem~\ref{th:hol1}).

The notion of a star product is now a classical subject studied by
many authors and naturally appearing in various contexts. Two
cornerstones of its history are 
the paper \cite{BFFLS} (see also \cite{Be1, Be2})
who defines $\star$-products and the fundamental result of
\cite{Ko1} which, roughly speaking, asserts that any real Poisson
manifold may be ``quantized'', that is, endowed with a star algebra to 
which the Poisson structure is associated. 
It is now a well-known fact (see \cite{Ka1,Ko2}) 
that, in order to quantize complex Poisson
manifolds, sheaves of algebras are not well-suited and have to be
replaced by algebroid stacks. 
We refer to \cite{C-H,Ye2} for further developments.

In this paper, we consider complex manifolds endowed with
$\DQ$-algebroids, that is, algebroid stacks locally associated to
sheaves of star-algebras, and study modules over such algebroids. The
main results of this paper are:
\begin{itemize}
\item
a finiteness theorem, which asserts that the convolution
of two coherent kernels, satisfying a suitable properness assumption,
is coherent  (a kind of Grauert's theorem), 
\item
the construction of the dualizing complex and 
a duality theorem, which asserts that duality commutes with
convolution, 
\item
the construction of the Hochschild class of coherent
$\DQ$-modules and the theorem which asserts that Hochschild  
class commutes with convolution,
\item
the link between Hochschild classes and Chern
classes and also with Euler classes, in the 
 commutative case,
\item
the constructibility of the complex of solutions of an holonomic 
module into another one in the  the symplectic case.
\end{itemize}

Let us describe this paper with some details.

\medskip
In {\bf Chapter~\ref{chapter:FD}}, we systematically study rings ({\em i.e.,}
sheaves of rings) which are formal deformations of rings, and modules
over such deformed rings.
More precisely, consider a topological space $X$, a commutative unital ring  $\cora$  and
a sheaf  $\sha$ of $\cora\forl$-algebras on $X$ which is $\hbar$-complete and
without $\hbar$-torsion.
We also assume that there exists a
base of open subsets of $X$,  acyclic for coherent modules
over $\shao\eqdot\sha/\hbar\sha$.
 
We first show how to deduce various properties of the ring $\sha$ 
from the corresponding properties on $\shao$. 
For example, $\sha$ is a Noetherian ring as soon as $\shao$ is a 
Noetherian ring, and 
an $\sha$-module $\shm$ is coherent as soon as it is locally
finitely generated and 
$\hbar^n\shm/\hbar^{n+1}\shm$ is
$\shao$-coherent for all $n\geq 0$.
Then,  we introduce the property of being cohomologically complete 
for an object of the derived category $\Der(\sha)$. We prove that this notion is
local, stable by direct images and an object $\shm$ with bounded coherent
cohomology is cohomologically complete. Conversely, if $\shm$ is 
cohomologically complete, it has coherent cohomology objects as soon
as its graded
module $\shao\lltens[\sha]\shm$  has coherent cohomology over $\shao$
(see Theorem~\ref{th:formalfini2}). We also give a similar criterion
which ensures that an $\sha$-module is flat.

\medskip
In {\bf Chapter~\ref{chapter:DQ}} we consider the case where $X$ is a 
complex manifold, $\cora=\C$, $\shao=\sho_X$
 and $\sha$ is  locally isomorphic to an algebra 
$(\OO[X]\forl,\star)$ where $\star$ is a star-product. 
It is an algebra over $\coro\eqdot\C\forl$.
We call such an algebra
$\sha$ a $\DQ$-algebra.
We also consider $\DQ$-algebroids, that is, $\coro$-algebroids 
(in the sense of stacks)
locally equivalent to the 
algebroid associated with a $\DQ$-algebra. 
Remark that a 
$\DQ$-algebroid on a manifold $X$ defines a Poisson structure on it.
Conversely, a famous theorem of Kontsevich \cite{Ko1} asserts that 
on a real Poisson manifold there exists a $\DQ$-algebra to
which this Poisson structure is associated. In the complex case, there
is a similar result using  $\DQ$-algebroids.
This is a theorem of \cite{Ko2}
after a related result of \cite{Ka1} in the contact case. 

If $(X,\A[X])$ is a complex manifold $X$
endowed with a $\DQ$-algebroid $\A[X]$, we
denote by $X^a$ the manifold $X$ endowed with the $\DQ$-algebroid
$\A[X]^\rop$ opposite to $\A[X]$. 

We define the external product 
$\A[X_1\times X_2]$ of two 
$\DQ$-algebroids $\A[X_1]$ and $\A[X_2]$ on manifolds $X_1$ and $X_2$.
There exists a canonical $\A[X\times X^a]$-module
$\dA$ on $X\times X^a$ supported by the diagonal,
which corresponds to the $\A$-bimodule $\A$. 

On a complex manifold $X$ endowed with a $\DQ$-algebroid, we
construct the $\coro$-algebroid $\DA$, a deformation quantization of the ring
$\shd_X$ of differential operators. It is a 
$\coro$-subalgebroid of $\shend_{\coro}(\A)$. It turns out that
$\DA$ is equivalent to $\shd_X\forl$. 
This new algebroid allows us to construct the 
dualizing complex $\oA$ associated to a $\DQ$-algebroid $\A[X]$. 
This complex is the dual over $\DA$ of
$\A[X]$, similarly to the case of $\OO[X]$-modules. Note that 
the dualizing complex for $\DQ$-algebras has already been considered
in a more particular situation by~\cite{Do,Do-Ru}.

We also  adapt to algebroids a results of~\cite{K-S2} which allows us to
replace a coherent 
$\A[X]$-module by a complex of ``almost free'' modules, such an object being a
locally finite sum $\oplus_{i\in I}(L_i)_{U_i}$, the $L_i$'s 
being free $\A[X]$-modules of finite rank defined on a neighborhood of 
$\overline{U_i}$. We give a similar result for algebraic manifolds.

\medskip
{\bf Chapter~\ref{chapter:Kern}}. Consider three complex manifolds $X_i$  
endowed with $\DQ$-algebroids $\A[X_i]$ ($i=1,2,3$).
Let $\shk_i\in\RD^\Rb_{\coh}(\A[X_i\times X_{i+1}^a])$ ($i=1,2$)
be two coherent kernels and define their convolution by
setting
\eqn
&&\shk_1\circ\shk_2\eqdot 
\reim{p_{14}}
\bigl((\shk_1\detens\shk_2)\lltens[{\A[X_2\times X^a_2]}]\dA[X_2]\bigr).
\eneqn
Here $p_{14}$ denotes the projection of the product 
$X_1\times X_2^a\times X_2\times X_3^a$ to $X_1\times X_3^a$.

We prove in Theorem~\ref{th:kernel1} that, 
under a natural properness hypothesis, the convolution $\shk_1\circ\shk_2$
belongs to $\RD^\Rb_{\coh}(\A[X_1\times X^a_3])$ and 
 in Theorem~\ref{th:duality}
that the convolution of kernels commutes with duality.

For further applications, it is also interesting to consider the
localized algebroid $\Ah=\cor\tens[\coro]\A$, where $\cor=\C\Ls$. 
An $\Ah$-module $\shm$ is good if for
any relatively compact open subset $U$ of $X$, there exists a coherent
$\A[U]$-module which generates $\shm\vert_U$. Then we prove
that there is a natural map of the Grothendieck groups 
 $\rmK_{\gd}(\Ah)\to \rmK_{\coh}(\gr\A[U])$ and that
 this map is compatible to the composition of kernels. 

Note that these theorems extend classical results of
Cartan, Serre and Grauert on finiteness and duality for coherent
$\sho$-modules on complex manifolds to $\DQ$-algebroids. 

For  papers related to $\DQ$-algebras and 
$\DQ$-algebroids on complex Poisson manifolds, and particularly to
their classification, we refer to 
\cite{B-K,Bo,B-G-N-T,C-H,Po,P-S,VdB,Ye}.

\medskip
{\bf Chapter~\ref{chapter:HH}.} We introduce the 
Hochschild homology $\HHA[X]$ of the algebroid
$\A[X]$: 
\eqn
&&\HHA[X]\eqdot\dA[X^a]\lltens[{\A[X\times X^a]}]\dA,
\text{ an object of }\RD^\Rb(\coro_X),
\eneqn
and, using  the dualizing complex, we construct a natural convolution morphism
\eqn
&&
\conv[X_2]:\reim{p_{13}}(\opb{p_{12}}\HHA[X_1\times X^a_2]
\lltens\opb{p_{23}}\HHA[X_2\times X^a_3])\to \HHA[X_1\times X^a_3].
\eneqn
To an object $\shm$ of
$\RD^\Rb_{\coh}(\A[X])$,
we naturally associate its Hochschild  class $\hh_X(\shm)$, an element of 
$H^{0}_{\Supp(\shm)}(X;\HHA[X])$. The main result of this chapter is 
Theorem~\ref{th:HH1} which asserts that taking 
the Hochschild  class commutes with the convolution:
\eq\label{eq:hh0}
&&\hh_{X_1\times X_3^a}(\shk_1\circ\shk_2)
=\hh_{X_{1}\times X_2^a}(\shk_1)\conv[X_2]\hh_{X_{2}\times X_{3}^a}(\shk_2).
\eneq

In {\bf Chapter~\ref{chapter:Com}}, we consider the case
where the deformation is trivial. In this case, there is no need of
the parameter $\hbar$ and we are in the well-known field of complex analytic
geometry. Although the results of this chapter are considered
as well-known (see in particular \cite{Hu}), 
at least from the specialists, we have decided to 
include this chapter. Indeed, to our opinion, there is no satisfactory
proof in the literature of the fact that the 
Hochschild  class of coherent $\OO[X]$-modules is functorial with
respect to convolution. We recall in particular the formula, in which the Todd class appears,
which makes the  link between Hochschild  class and
Chern classes. This formula was conjecturally stated by the first named author
around 1991 and has only been proved very recently by Ramadoss~\cite{Ra1} in the algebraic setting and 
by Grivaux~\cite{Gri} in the general case. For other papers closely related to this
chapter, see \cite{Ca2,Ca-W,Hu,Ma,Sh}. 

\medskip
In {\bf Chapter~\ref{chapter:Symp}}
we study Hochschild homology and Hochschild classes in 
the case where the Poisson structure associated to the
deformation is symplectic.  We prove then that the 
dualizing complex $\oA$ is isomorphic to $\dA$ shifted by $d_X$, 
the complex dimension of $X$, 
and we construct canonical morphisms 
\eq\label{eq:hhatok}
&& \hbar^{d_X/2}\coro_X\,[d_X]\to \HHA[X]\to\hbar^{-d_X/2}\coro_X\,[d_X] 
\eneq
whose composition is the canonical inclusion.
The morphisms in \eqref{eq:hhatok} induce an isomorphism
\eq\label{eq:hhatokloc}
&& \cor_X\,[d_X]\simeq\HHAh[X].
\eneq
The first morphism in \eqref{eq:hhatok} 
gives an intrinsic construction of the canonical class 
in $H^{-d_X}(X;\HHA[X])$ studied and used by several authors 
(see \cite{B-G,B-N-T,F-T}). 
The isomorphism \eqref{eq:hhatokloc} allows us to associate an Euler class 
$\eu_X(\shm)\in H^{d_X}_\Lambda(X;\cor_X)$ to any coherent $\A$-module
$\shm$ supported by a closed set $\Lambda$.

Then we show how our results apply to $\shd$-modules. We recover in
particular the Riemann-Roch theorem for $\shd$-modules 
of \cite{La} as well as the functoriality of the Euler class
of $\shd$-modules of \cite{S-Sn}.

\medskip
Finally, in {\bf Chapter~\ref{chapter:Hol}}, 
we study holonomic $\Ah$-modules on complex symplectic manifolds.
We prove that if $\shl$ and $\shm$ are two  holonomic $\Ah$-modules,
then the complex $\rhom[{\Ah}](\shm,\shl)$ is 
perverse (hence, in particular, $\C$-constructible) over the field
$\cor$. 

If the intersection of the supports of the holonomic modules 
$\shl$ and $\shm$ is compact, formula ~\eqref{eq:hh0} gives in
particular
\eqn
&&\chi(\RHom[{\Ah}](\shm,\shl))=\int_X(\eu_X(\shm)\cdot\eu_X(\shl)).
\eneqn
The Euler class of a holonomic module may be
interpreted as a Lagrangian cycle, which makes its calculation quite easy.

If  the modules $\shl$ and $\shm$ are simple  along smooth
Lagrangian submanifolds, then one can estimate the microsupport of
this complex. This particular case had been already treated in
 \cite{KS08} in the analytic framework,
that is, using analytic deformations (in the sense of \cite{S-K-K}), 
not formal deformations, and the proof given here is much simpler. 

We also prove (Theorem~\ref{th:hol2}) that if $\shl_a$ is family 
of holonomic modules indexed by a holomorphic parameter $a$, then,
under suitable geometrical hypotheses, the complex of global sections 
$\RHom[{\Ah}](\shm,\shl_a)$, which  belongs to  $\Derb_f(\cor)$, does
not depend on $a$. This is a kind of
 invariance by Hamiltonian symplectomorphism of this complex. 

\vspace{0.5cm}
\noindent
We have developed the theory in the framework of complex
analytic manifolds. 
However, all along the manuscript, we explain how the results extend
(and sometimes simplify) 
in the algebraic setting, that is 
on  quasi-compact and separated  smooth varieties 
over $\C$. 

\vspace{0.5cm}
\noindent
The main results of this paper, with the exception of Chapter 7,
 have been announced in \cite{K-S4,K-S4b}.

\medskip\noindent
{\bf Acknowledgments} We would like to thank Andrea D'Agnolo,
Pietro Polesello, St{\'e}phane Guillermou, Jean-Pierre Schneiders and
Boris Tsygan  for useful comments and remarks.

\chapter{Modules over formal deformations}\label{chapter:FD}  

\section{Preliminary}
\subsubsection*{Some notations}
Throughout this paper, $\cora$
\index{k0@$\cora$}%
 denotes a commutative unital ring.

We shall mainly follow the notations of \cite{K-S3}.
In particular, if $\shc$ is  a category, 
we denote by $\shc^\rop$ the opposite category.
If $\shc$ is an additive category, we denote by 
 $\RC(\shc)$ the category of complexes of $\shc$
\index{Cminusa@$\RC(\shc)$}%
and by $\RC^*(\shc)$ ($*=+,-,\rmb$) the full subcategory consisting of
complexes bounded from below (resp.\ bounded from above, resp.\ bounded).
\index{Cminusb@$\RC^+(\shc)$}%
\index{Cminusc@$\RC^-(\shc)$}%
\index{Cminusd@$\RC^\rmb(\shc)$}%
If $\shc$ is an abelian category, we denote by 
$\Der(\shc)$ the derived category of $\shc$ 
\index{DerCa@$\Der(\shc)$}%
and by $\Der^*(\shc)$ ($*=+,-,\rmb$) the full triangulated subcategory consisting of
objects with  bounded from below (resp.\ bounded from above, resp.\ bounded)
cohomology.
\index{DerCb@$\Der^+(\shc)$}%
\index{DerCc@$\Der^-(\shc)$}%
\index{DerCd@$\Der^\rmb(\shc)$}%
We denote as usual by $\tau^{\ge n}$, $\tau^{\le n}$ etc.\ the truncation
functors in $\Der(\shc)$.
\index{tau@$\tau^\ge n$}%

If $A$ is a ring (or a sheaf of rings on a topological space $X$), 
an $A$-module means a 
left $A$-module. We denote by $A^\rop$ the opposite ring of $A$. Hence
an $A^\rop$-module is nothing but a right $A$-module.
We denote by $\md[A]$ 
\index{Mod@$\md[A]$}%
the category of $A$-modules. We set for short $\Der(A)\eqdot\Der(\md[A])$
\index{D0@$\RD(A)$}%
and we write similarly $\Der^*(A)$ instead of  $\Der^*(\md[A])$.
\index{D1@$\Derb(A)$}%
We denote by $\Derb_\coh(A)$
\index{D1@$\Derb_\coh(\sha)$}%
the full triangulated subcategory of $\Derb(A)$
consisting of objects with coherent cohomology. If $\cora$ is  Noetherian,
one denotes simply by 
$\Derb_f(\cora)$ the full subcategory of $\Derb(\cora)$ consisting of objects
with finitely generated cohomology.

We denote by $\RDD_X$ the duality functor for $\cora_X$-modules:
\eq\label{eq:rdddual}
&&\RDD_X(\scbul)\eqdot\rhom[\cora_X](\scbul,\cora_X).
\eneq
and we simply denote by $(\scbul)^\star$ the duality functor on $\RD^\Rb(\cora)$: 
\eq\label{eq:star=dual}
&&(\scbul)^\star=\RHom[\cora](\scbul,\cora).
\eneq
If $\cora$ is  Noetherian and with finite global dimension, 
$(\scbul)^\star$  sends 
$(\Derb_{f}(\cora))^\rop$ to $\Derb_f(\cora)$. 

We denote by $\{\pt\}$ 
\index{point@$\{\pt\}$}%
the set with a single element.

\subsubsection*{Finiteness conditions}
Let $X$ be a topological space and let $\sha$ be a $\cora_X$-algebra
({\em i.e.,} a sheaf of $\cora$-algebras) on
$X$. Let us recall a few classical definitions. 
\begin{itemize}
\item
An $\sha$-module $\shm$ is
locally finitely generated 
\glossary{locally finitely generated}\glossary{module!locally finitely generated}%
if there locally exists an exact sequence
\eq\label{eq:seqfp1}
&&\shl_0\to\shm\to0
\eneq
such that $\shl_0$ is locally free of finite rank over
$\sha$.
\item
An $\sha$-module $\shm$ is
locally of finite presentation
\glossary{locally of finite presentation}%
\glossary{module!locally of finite presentation}%
if there locally exists an exact sequence
\eq
&&\shl_1\to\shl_0\to\shm\to0\label{eq:fpr}
\eneq
such that $\shl_1$ and $\shl_0$ are locally free of finite rank over
$\sha$. This is equivalent to saying that 
there  locally exists an exact sequence
\eq\label{eq:seqfp2}
&&0\to\shk\to[u]\shn\to\shm\to0
\eneq
where $\shn$ is  locally free of finite rank and $\shk$ is 
locally finitely generated.
This is also equivalent to saying that 
there  locally exists an exact sequence
\eq\label{eq:seqfp3}
&&\shk\to\shn\to\shm\to0
\eneq
where $\shn$ is locally of finite presentation and $\shk$ is 
locally finitely generated.
\item
An $\sha$-module $\shm$
is pseudo-coherent
\glossary{pseudo-coherent}%
\glossary{module!pseudo-coherent}%
if for any locally defined 
morphism $u\cl\shn\to\shm$ with $\shn$
of finite presentation, $\ker u$ is locally  finitely generated. 
This is also equivalent to saying that 
any locally defined $\sha$-submodule of $\shm$ is locally
of finite presentation as soon as it is locally finitely generated.
\item
An $\sha$-module $\shm$ is coherent 
\glossary{coherent}%
\glossary{module!coherent}%
if it is locally finitely
generated and pseudo-coherent. A ring
is a coherent ring if it is so as a module over itself. 
One denotes by $\mdcoh[\sha]$ the
full additive subcategory of $\md[\sha]$ consisting of coherent
modules. Note that $\mdcoh[\sha]$ is a full  abelian subcategory of 
 $\md[\sha]$, stable by extension, and the natural functor 
$\mdcoh[\sha]\to\md[\sha]$ is exact (see \cite[Exe.~8.23]{K-S3}).
\item
An $\sha$-module $\shm$
 is Noetherian (see \cite[Def.~A.7]{Ka2}) 
\glossary{Noetherian}\glossary{module!Noetherian}\glossary{ring!Noetherian}%
if it is coherent, $\shm_x$ is a Noetherian $\sha_x$-module for any
$x\in X$, and for any open subset $U\subset X$, any filtrant family
of coherent submodules of $\shm\vert_U$ is locally stationary. 
(This means that given a family $\{\shm_i\}_{i\in I}$ of coherent submodules of 
$\shm\vert_U$ indexed by a filtrant ordered set $I$, with 
$\shm_i\subset\shm_j$ for $i\leq j$, 
there locally exists $i_0\in I$ such that $\shm_{i_0}\isoto\shm_j$ for any $j\geq i_0$.)
A ring
is a Noetherian ring if it is so as a left module over itself. 
\end{itemize}

\subsubsection*{Mittag-Leffler condition and pro-objects}
We refer to \cite{SGA} for the notions of ind-object and pro-object as well as to 
 \cite{K-S3} for an exposition.
To an abelian category $\shc$, 
one associates the abelian category $\Pro(\shc)$ of its pro-objects.
Then $\shc$ is a full abelian subcategory of $\Pro(\shc)$ stable by
kernel, cokernel and extension, the natural functor 
$\shc\hookrightarrow\Pro(\shc)$ is exact, and 
the functor $\sproolim\cl \Fct(I^\rop,\shc)\to\Pro(\shc)$ is exact
for any small filtrant  category $I$. In the sequel, we identify $\shc$ 
with a full subcategory of $\Pro(\shc)$.
If $\shc$ admits small projective
limits, we denote by  $\pi$ the left exact functor
\eqn
&&\pi\cl \Pro(\shc)\to\shc,\quad \sproolim[i]X_i\mapsto\sprolim[i]X_i.
\eneqn
If $\shc$ has enough injectives, then $\pi$ admits a right derived
functor (\loc.\ cit.):
\eqn
&&\rpi\cl \Der^+\bl\Pro(\shc)\br\to\Der^+(\shc).
\eneqn

\begin{definition}\label{def:ML}
We say that an object $M\in\Pro(\shc)$
satisfies the Mittag-Leffler condition 
\glossary{Mittag-Leffler condition}%
if, for any $N\in\shc$ and any morphism $M\to N$ in $\Pro(\shc)$,
$\im(M\to N)$ is representable by an object of $\shc$.
\end{definition}

By the definition, any quotient of 
an object which satisfies the Mittag-Leffler condition 
also satisfies the Mittag-Leffler condition.

\begin{lemma}\label{lem:ML}
Let $\{M_n\}_{n\in\Z_{\ge1}}$ be a projective system 
in an abelian category $\shc$, and set $M=\sproolim[n]M_n\in\Pro(\shc)$.
Then the following conditions are equivalent:
\bnum
\item 
$M$ satisfies the Mittag-Leffler condition,
\item
$\{M_n\}_{n\in\Z_{\ge1}}$ satisfies the Mittag-Leffler condition \ro that is, 
for any $p\ge1$, the sequence $\{\im(M_n\to M_p)\}_{n\ge p}$ is
stationary\rf,
\item
there exists 
a projective system
$\{M'_n\}_{n\in\Z_{\ge1}}$ in $\shc$ such that
the morphism $M'_{n+1}\to M'_n$ is an epimorphism for any $n\ge1$
and we have an isomorphism $M\simeq\proolim[n]M'_n$ in $\Pro(\shc)$.
\enum
\end{lemma}
\begin{proof}
(i) $\Rightarrow$ (ii).
For any $p\ge 1$, $\im(M\to M_p)\simeq\proolim[n\ge p]\im(M_n\to M_p)$
is representable by an object of $\shc$.
Hence, the sequence $\{\im(M_n\to M_p)\}_{n\ge p}$ is stationary.

\vspace{0.4em}
\noindent
(ii) $\Rightarrow$ (iii).
Set $M'_n=\im(M_k\to M_n)$ for $k\gg n$.
Then the morphisms $M'_n\to M_n$ induce a morphism 
$f\cl \proolim[n] M'_n\to\proolim[n] M_n$.
On the other hand, for each $n$, 
$M\to M_n$ decomposes as $M\to M'_n\monoto M_n$,
since taking $k\gg n$ such that
$M'_n=\im(M_k\to M_n)$, we have a morphism
$M\to M_k\to M'_n$. These morphisms  induce a morphism
$g\cl \proolim[n] M_n=M\to\proolim[n] M'_n$.
It is easy to see that $f$ and $g$ are inverse to each other.

\vspace{0.4em}
\noindent
(iii) $\Rightarrow$ (i).
For any $N\in\shc$ and any morphism $f\cl M\to N$ in $\Pro(\shc)$,
there exists $p$ such that $f$ decomposes into
$M\to M'_p\to N$. Then
$\im(M\to N)\simeq\proolim[n\ge p]\im (M'_n\to N)\simeq \im(M'_p\to N)$.
\end{proof}

Note that the following lemma is well known.
\Lemma\label{lem:MLcoh}
Let $\{M_n\}_{n\ge1}$ be a projective system of $\Z$-modules.
Then 
$\RR^i\pi(\proolim[n]M_n)\simeq 0$ for $i\not=0,1$.
If $\{M_n\}_{n\ge1}$ satisfies the Mittag-Leffler condition,
then 
$H^1\bl\RR\pi\proolim[n]M_n\br\simeq 0$.
\enlemma
Here and in the sequel, we make the following convention.
\begin{convention}
When we have a left exact functor $\shc\to[F]\shc'$ 
of abelian categories and $X\in\Der(\shc)$, the notation  
$\RR^i F(X)$ stands for $H^i\bl\RR F(X)\br$. 
For example, $\RR^i\pi\rsect(U;\shm)$ means $H^i\bl\rpi\rsect(U;\shm)\br$.
\end{convention}
\begin{lemma}\label{le:cohandprolim} 
Let $\shr$ be an algebra over a topological space $X$, and let
$\{\shm_n\}_{n\ge0}$ be a projective system of $\shr$-modules.
Set $\shm=\proolim[n]\shm_n\in\Pro(\Mod(\shr))$. Let $U$ be an open 
subset of $X$ and let $i\in\Z$. Then we have an exact sequence
\eqn
&&0\To\RR^1\pi\bl\proolim[n]H^{i-1}(U;\shm_n)\br
\To H^i(U;\rpi\shm)\To\prolim[n]H^i(U;\shm_n)\to0.
\eneqn
\end{lemma}
\begin{proof}
We have $\rsect(U;\rpi\shm)\simeq\rpi\rsect(U;\shm)$ and we also have 
$H^i(U;\shm)\simeq \proolim[n]H^i(U;\shm_n)$.
Consider the distinguished triangle
\eqn
&&\rpi\tau^{<i}\rsect(U;\shm)\To\rpi\rsect(U;\shm)
\To\rpi\tau^{\ge i}\rsect(U;\shm)\To[+1].
\eneqn
It gives rise to the exact sequence
 \eqn
&&0\to\RR^i\pi\tau^{<i}\rsect(U;\shm)\to\RR^i\pi\rsect(U;\shm)
\to\RR^i\pi\tau^{\ge i}\rsect(U;\shm)\\
&&\hs{45ex}\to\RR^{i+1}\pi\tau^{<i}\rsect(U;\shm).
\eneqn
Since $\RR^k\pi\proolim[n] M_n=0$ for $k\not=0,1$ and any projective
system $\{M_n\}_n$, we obtain
$\RR^{i+1}\pi\tau^{<i}\rsect(U;\shm)=0$.

Consider the distinguished triangle
\eqn
&&\tau^{<i-1}\rsect(U;\shm)\to\tau^{<i}\rsect(U;\shm)
\to H^{i-1}(U;\shm)[1-i]\To[+1].
\eneqn
Using the isomorphism $H^{i-1}(U;\shm)\simeq \proolim[n]H^{i-1}(U;\shm_n)$
and applying the functor $\rpi$, we get the distinguished triangle
\eqn
&&\rpi\tau^{<i-1}\rsect(U;\shm)\to\rpi\tau^{<i}\rsect(U;\shm)\\
&&\hs{25ex}\to\rpi\bl\proolim[n]H^{i-1}(U;\shm_n)[1-i]\br\To[+1].
\eneqn
We obtain
$\RR^i\pi\tau^{<i}\rsect(U;\shm)\simeq\RR^1\pi\proolim[n]H^{i-1}(U;\shm_n)$.
Finally, we have
$\RR^i\pi\tau^{\ge i}\rsect(U;\shm)\simeq\prolim[n]H^i(U;\shm_n)$.
\end{proof}

As a corollary of this lemma, we obtain the following
lemma, a slight modified version of \cite[Pr{\'e}liminaires,
Prop. (13.3.1)]{EGA}.

\begin{lemma} \label{lem:cohpro}
Let $X$ be a topological space,
$\{\shf_n\}_{n\in\Z_{>0}}$ a projective system of abelian sheaves on
$X$ and $\shf\seteq\prolim[n]\shf_n$. Assume the following conditions:
\banum
\item
for any $x\in X$ and any integer $i$, we have 
$$\indlim[x\in
U]\RR^1\pi\proolim[n]H^i(U;\shf_n)
\simeq0,$$
where $U$ ranges over an open neighborhood system of $x$,
\item for any $x\in X$ and $i>0$,
$\indlim[x\in U]\bl\prolim[n]H^i(U;\shf_n)\br=0$,
where $U$ ranges over an open neighborhood system of $x$,
\eanum
Then for any $i$,
the morphism
$$h_i\cl H^i(X;\shf)\to\prolim[n]H^i(X;\shf_n)$$
is surjective. If moreover
$\{H^{i-1}(X;\shf_n)\}_n$ satisfies the Mittag-Leffler condition, then
$h_i$ is an isomorphism.
\end{lemma}
\Proof
Set $\shm=\proolim[n]\shf_n$. By the preceding lemma,
we have an exact sequence
\eqn
&&0\To\RR^1\pi\bl\proolim[n]H^{i-1}(U;\shf_n)\br
\To H^i(U;\rpi\shm)\To\prolim[n]H^i(U;\shf_n)\to0.
\eneqn
For any $x$, taking the inductive limit with respect to 
$U$ in an open neighborhood system of $x$,
we obtain $(\RR^i\pi\shm)_x=0$ for $i\not=0$.
Hence we conclude $\rpi\shm\simeq\shf$.
Then the exact sequence above reads as
\eqn
&&0\To\RR^1\pi\bl\proolim[n]H^{i-1}(X;\shf_n)\br
\To H^i(X;\shf)\To\prolim[n]H^i(X;\shf_n)\to0.
\eneqn
Hence we have the desired result.
\QED

\section{Formal deformations of a sheaf of rings}
Now we consider the following situation:
$X$ is a topological space,  $\sha$ is a $\cora$-algebra on $X$ and
$\hbar$ is a section of $\sha$ contained in the center of $\sha$.
We set  
\eqn
&&\shao\eqdot\sha/\hbar\sha
\eneqn
\index{A0@$\sha$}%
\index{A1@$\shao$}%

Let $\shm$ be an $\sha$-module. 
We set 
\eq \widehat{\shm}\seteq\prolim[n]\shm/\hbar^n\shm,
\eneq
and call it the $\hbar$-completion of $\shm$.
\glossary{hcompletion@$\hbar$-completion}%
We say that 
\begin{itemize}
\item
$\shm$ has no $\hbar$-torsion if $\hbar\cl\shm\to\shm$ is injective, 
\glossary{hbartorsion@$\hbar$-torsion}%
\glossary{nohbartorsion@no $\hbar$-torsion}%
\item
$\shm$ is $\hbar$-separated if
\glossary{hbarseparated@$\hbar$-separated}%
$\shm\to\widehat{\shm}$ is a monomorphism, 
{\em i.e.}, $\smash{\bigcap\limits_{n\ge0}}\,\hbar^n\shm=0$,
\item
$\shm$ is $\hbar$-complete if 
\glossary{hbarcomplete@$\hbar$-complete}%
$\shm\to\widehat{\shm}$ is an isomorphism.
\end{itemize}

\begin{lemma}\label{le:complete1}
Let $\shm\in\md[\sha]$ and assume that $\shm$ has no $\hbar$-torsion. Then
\bnum
\item
$\twh{\shm}$ has no $\hbar$-torsion,
\item
$\shm/\hbar^n\shm\isoto\twh\shm/\hbar^n\twh\shm$,

\item
$\twh{\shm}\isoto\shm\widehat{\phantom{a}}\widehat{\phantom{a}}$,
i.e., $\twh\shm$ is $\hbar$-complete.
\enum
\end{lemma}
\begin{proof}
(i) Consider the exact sequence
\eqn
&& 0\To\shm/\hbar^n\shm\To[\hbar^a]\shm/\hbar^{n+a}\shm\To\shm/\hbar^{a}\shm\To0.
\eneqn
Applying the left exact functor $\prolim[n]$ we get the exact sequence
\eqn
&& 0\To\twh\shm\To[\hbar^a]\twh\shm\To\shm/\hbar^{a}\shm,
\eneqn
which gives the result.

\noindent
(ii) Consider the commutative diagram with exact rows:
\eqn
&&\xymatrix{
0\ar[r]&\shm\ar[d]\ar[r]^{\hbar^n}&\shm\ar[d]\ar[r]&\shm/\hbar^n\shm\ar[d]^{\bwr}\ar[r]&0\\
0\ar[r]&\twh\shm\ar[r]^{\hbar^n}  &\twh\shm\ar[r]  &\shm/\hbar^n\shm.
}\eneqn

\noindent
(iii) Apply the functor $\sprolim$ to the isomorphism in (ii).
\end{proof}

In this paper, with the exception of \S~\ref{section:variant},
we assume the following conditions:
\eq\label{eq:FDringa}
&&\left\{\parbox{60ex}{\bnum
\item $\sha$ has no $\hbar$-torsion,\label{FDa}
\item $\sha$ is $\hbar$-complete,\label{FDb}
\item $\shao$ is a left Noetherian ring,\label{FDc}
\enum
}\right.  
\eneq
and\vs{-2ex}
\eq\label{eq:FDringb}
&&\left\{\parbox{60ex}{\bnum
\item[{\rm (iv)}] there exists a base $\BB$ of open subsets of $X$ such that
for any $U\in\BB$ and any coherent $(\shao\vert_U)$-module $\shf$,
we have $H^n(U;\shf)=0$ for any $n>0$.\label{FDd}
\enum
}\right.  
\eneq
It follows from \eqref{eq:FDringa} that, 
for an open set $U$ and $a_n\in\sha(U)$ ($n\ge0$),
$\sum_{n\ge0}\hbar^na_n$ is a well-defined element of $\sha(U)$.

By \eqref{eq:FDringa}~\eqref{FDb}, $\hbar\sha_x$ is contained in the Jacobson radical
of $\sha_x$ for any $x\in X$. Indeed, 
for any $a\in \hbar\sha_x$, $1-a$ is invertible in $\sha_x$
since $a$ is defined on an open neighborhood $U$ of $x$,
and $1-a$ is invertible on $U$.

Hence Nakayama's lemma implies the following lemma
that we frequently use. 

\begin{lemma} \label{le:nakayamacor}
Let $\shm$ be a locally finitely generated $\sha$-module.
\bnum 
\item
If $\shm$ satisfies $\shm=\hbar\shm$, then $\shm=0$.
\item 
Let $f\cl \shn\to\shm$ be a morphism of
$\sha$-modules. If the composition $\shn\to\shm\to\shm/\hbar\shm$ 
is an epimorphism, then $f$ is an epimorphism.
\enum
\end{lemma}

For $n\in\Z_{\ge0}$, set $\sha_n=\sha/\hbar^{n+1}\sha$. 
Note that there is an equivalence of categories between the category
$\md[\sha_n]$ and the full subcategory of $\md[\sha]$ consisting of modules 
$\shm$ satisfying  $\hbar^{n+1}\shm\simeq 0$.

\begin{lemma}\label{le:cohcoh}
Let $n\in\Z_{\ge0}$.
\bnum
\item 
An $\sha_n$-module $\shn$
is locally finitely generated as an $\sha_n$-module if and only if
it is so as an $\sha$-module.
\item 
An $\sha_n$-module $\shn$
is locally of finite presentation as an $\sha_n$-module if and only if
it is so as an $\sha$-module.
\item 
An $\sha_n$-module $\shn$
is coherent as an $\sha_n$-module if and only if
it is so as an $\sha$-module.
\item 
$\sha_n$ is a left Noetherian ring.
\enum
\end{lemma}
\begin{proof}
Note that since we have $\sha_n\simeq\sha/\sha\hbar^{n+1}$,
$\sha_n$ is an $\sha$-module locally of finite presentation.

\noindent
(i) is obvious.

\noindent
(ii)-(a) Let $\shm$ be an $\sha_n$-module locally of finite presentation
and consider an exact sequence 
of  $\sha_n$-modules as in \eqref{eq:seqfp2}.
Then $\shk$ is locally finitely generated as an $\sha$-module, 
$\shn$ is locally of finite presentation as an $\sha$-module and $u$
is $\sha$-linear. Hence,
$\shm$ is locally of finite presentation as an $\sha$-module.

\noindent
(ii)-(b) Conversely assume that $\shm$ is an $\sha_n$-module which is 
locally of finite presentation as an $\sha$-module.
Consider an exact sequence of $\sha$-modules as in \eqref{eq:fpr}.
Applying the functor 
$\sha_n\tens[\sha]\scbul$,
we find and exact sequence of $\sha_n$-modules as in \eqref{eq:fpr},
which proves that $\shm$ is locally 
of finite presentation as an $\sha_n$-module.

\noindent
(iii) follows from (i) and (ii) since a module is coherent if it is 
locally finitely generated 
and any submodule locally  finitely generated is
locally of finite presentation.

\noindent
(iv) Let us prove that $\sha_n$ is a coherent ring.
Since $\sha_0$ is a coherent ring by the assumption,
$\sha_0$ is a coherent $\sha$-module.
Using the exact sequences of $\sha$-modules
\eqn
&&0\to\sha_{n-1}\to[\hbar]\sha_{n}\to \sha_0\to 0,
\eneqn
we get by induction on $n$ that $\sha_n$ is a coherent
$\sha$-module. Hence (iii) implies that $\sha_n$ is a coherent ring.

One proves similarly by induction on $n$ that 
 $(\sha_n)_x$ is a Noetherian ring for all $x\in X$ and that 
any filtrant family of coherent $\sha_n$-submodules
of a coherent $\sha_n$-module is locally stationary.
\end{proof}
\begin{lemma}\label{lem:sur}
Let $U\in\BB$, and $n\geq 0$.
\bnum
\item 
For any coherent $\sha_n$-module $\shn$, we have $H^k(U;\shn)=0$ for $k\neq 0$. 
\item 
For any epimorphism $\shn\to\shn'$ of coherent $\sha_n$-modules,
$\shn(U)\to\shn'(U)$ is surjective,
\item 
$\sha(U)\to\sha_n(U)$ is surjective.
\enum
\end{lemma}
\begin{proof}
(i) is proved by induction on $n$, using the exact
sequence
\eq\label{eq:exseq0}
0\to\hbar\shn\to\shn\to\shn/\hbar\shn\to 0.
\eneq

\vspace{0.4em}
\noindent
(ii) follows immediately from (i) and the fact that $\sha_n$ is a
coherent ring.

\vspace{0.4em}
\noindent
(iii)\quad By (ii), $\sha_{n+1}(U)\to\sha_{n}(U)$ is surjective for any $n\ge0$.
Hence, the morphism $\prolim[m]\sha_m(U)\to \sha_n(U)$ is surjective. Since the functor
$\sprolim$ commutes with the functor $\sect(U;\scbul)$, 
$\sha(U)\isoto\prolim[m]\sha_m(U)$ and the result follows.
\end{proof}

\subsubsection*{Properties of $\sha$}
Recall that $\sha$ satisfies \eqref{eq:FDringa} and
\eqref{eq:FDringb}.
 
\begin{theorem}\label{th:formalfini1}
\bnum
\item
$\sha$ is a left Noetherian ring.
\item 
Let $\shm$ be a locally finitely generated $\sha$-module.
Then $\shm$ is coherent if and only if 
$\hbar^n\shm/\hbar^{n+1}\shm$ 
is a coherent $\sha_0$-module for any $n\ge0$. \label{th:coh:coh}
\item 
Any coherent $\sha$-module
$\shm$ is $\hbar$-complete, {\em i.e.},
$\shm\isoto\widehat{\shm}$.
\label{th:coh:int}
\item 
Conversely, an $\sha$-module $\shm$ is coherent
if and only if it is $\hbar$-complete and
$\hbar^n\shm/\hbar^{n+1}\shm$ is a coherent $\sha_0$-module
for any $n\ge0$.
\label{th:crcoh}
\item
For any coherent $\sha$-module $\shm$ and any $U\in\BB$, 
we have $H^j(U;\shm)=0$ for any $j>0$.
\enum
\end{theorem}
The proof of Theorem~\ref{th:formalfini1} 
decomposes into several lemmas.

\begin{lemma}\label{lem:fg0}
Let $\shl$ be a locally free $\sha$-module of finite rank
and let $\shn$ be an $\sha$-submodule of $\shl$.
Assume that
\banum 
\item
$(\shn+\hbar\shl)/\hbar\shl$ is a coherent $\shao$-module,
\item
$\shn\cap \hbar^n\shl\subset \hbar\shn+\hbar^{1+n}\shl$ for any $n\ge1$.
\eanum 
Then we have
\bnum
\item
$\shn$ is a locally finitely generated $\sha$-module,
\item 
$\shn\cap \hbar^n\shl=\hbar^n\shn$ for any $n\ge0$,
\item
$\bigcap\limits_{n\ge0}(\shn+\hbar^n\shl)=\shn$.
\enum
\end{lemma}
\begin{proof}
First, let us show that
\eq\label{eq:NcaptnL}
&&\shn\cap \hbar\shl\subset \hbar\shn+\hbar^{n}\shl\quad\text{for any $n\ge 0$.}
\eneq
Indeed, \eqref{eq:NcaptnL} is trivial for $n\le1$. 
Let us argue by
induction, and let $n\ge2$, 
assuming the assertion for $n-1$. We have
$\shn\cap \hbar\shl\subset\shn\cap(\hbar\shn+\hbar^{n-1}\shl)
=\hbar\shn+(\shn\cap \hbar^{n-1}\shl)
\subset \hbar\shn+(\hbar\shn+\hbar^n\shl)$
by the assumption (b).
This proves \eqref{eq:NcaptnL}.

Set 
\eqn
&&\tN=\bigcap_{n\ge0}(\shn+\hbar^n\shl).
\eneqn
Then $\shn\subset\tN$ and 
\eq\label{eq:tNcapL}
&&\tN\cap \hbar\shl\subset \hbar\tN.
\eneq
Indeed we have
$\tN\cap \hbar\shl\subset 
(\shn+\hbar^{n+1}\shl)\cap \hbar\shl
\subset\shn\cap \hbar\shl+\hbar^{n+1}\shl
\subset \hbar\shn+\hbar^{n+1}\shl=\hbar(\shn+\hbar^n\shl)$ for any $n$.

Set 
\eqn
&&\bN=(\shn+\hbar\shl)/\hbar\shl
=(\tN+\hbar\shl)/\hbar\shl.
\eneqn
By the hypothesis~(a), $\bN$ is $\shao$-coherent.
Hence we may assume that
there exist finitely many sections $s_i$ of $\shn$ such that
$\bN=\sum_i\shao\ol{s}_i$, where $\ol{s}_i$ is the image of $s_i$
in $\shl/\hbar\shl$.

By hypothesis~(a) and Lemma~\ref{lem:sur}~(ii), we have for any $U\in\BB$,
$\bN(U)=\sum_i\shao(U)\ol{s}_i$.
Since $\sha(U)\to\shao(U)$ is surjective by Lemma~\ref{lem:sur}~(iii),
we have $\tN(U)\subset\sum_i\sha(U)s_i+\hbar\shl(U)$.
Since $\tN\cap \hbar\shl=\hbar\tN$, we have
$$\tN(U)\subset\sum_i\sha(U)s_i+\hbar\tN(U).$$
For $v\in\tN(U)$, we shall define
a sequence $\{v_n\}_{n\ge0}$ in $\tN(U)$ and 
sequences $\{a_{i,n}\}_{n\ge0}$
in $\sha(U)$,
inductively on $n$:
set $v_0=v$, and write
$$v_n=\sum_ia_{i,n}s_i+\hbar v_{n+1}.$$
Hence we have $\hbar^nv_n=\sum_i\hbar^na_{i,n}s_i+\hbar^{n+1}v_{n+1}$ and we obtain
$$v=v_0=\sum_i(\sum_{n\ge0}\hbar^na_{i,n})s_i.$$
Thus we have
$\tN=\sum_i\sha s_i$. Hence $\shn=\tN$ which proves (i) and (iii).

Since
$\tN\cap \hbar\shl=\hbar\tN$ by \eqref{eq:tNcapL}, we obtain (ii) for $n=1$.
For $n\ge1$ we have by induction
$\shn\cap \hbar^n\shl\subset \hbar\shn\cap \hbar^n\shl
=\hbar(\shn\cap \hbar^{n-1}\shl)\subset \hbar\cdot \hbar^{n-1}\shn$.
\end{proof}

\begin{lemma}\label{lem:fg}
Let $\shl$ be a locally free $\sha$-module of finite rank, and let
$\shn$ be an $\sha$-submodule of $\shl$.
Assume that
$(\shn+\hbar^{n+1}\shl)/\hbar^{n+1}\shl$ is a coherent $\sha$-module for any $n\ge0$.
Then we have
\bnum
\item
$\shn$ is a locally finitely generated $\sha$-module,
\item
$\bigcap_{n\ge0}(\shn+\hbar^n\shl)=\shn$,
\item
locally, 
$\hbar^n\shl\cap\shn\subset \hbar(\hbar^{n-1}\shl\cap\shn)$ for $n\gg0$,
\item
$\shn/\hbar^n\shn$ is a coherent $\sha$-module for any $n\ge0$.
\enum
\end{lemma}
\begin{proof}
We embed $\shl$ into the $\sha[\opb{\hbar}]$-module 
$\cora[\hbar,\opb{\hbar}]\tens[{\cora[\hbar]}]\shl
=\bigcup_{n\in\Z}\hbar^n\shl$.
Note that $\hbar^n$ induces an isomorphism
\eqn
&&\hbar^n\cl(\shl\cap\hbar^{-n}\shn+\hbar\shl)/\hbar\shl\isoto
(\shn\cap \hbar^n\shl+\hbar^{n+1}\shl)/\hbar^{n+1}\shl.
\eneqn
Since 
\eqn
&&\hspace{-2em}
(\shn\cap \hbar^n\shl+\hbar^{n+1}\shl)/\hbar^{n+1}\shl\simeq
\bigl((\shn+\hbar^{n+1}\shl)/\hbar^{n+1}\shl\bigr)\bigcap 
\bigl(\hbar^n\shl/\hbar^{n+1}\shl\bigr)
\eneqn
 is $\sha$-coherent, 
$\{(\shl\cap \hbar^{-n}\shn+\hbar\shl)/\hbar\shl\}_{n\ge0}$ 
is an increasing sequence of
coherent $\shao$-submodules of
$\shl/\hbar\shl$. Hence it is locally stationary:
locally there exists $n_0\ge0$ such that
$\shl\cap \hbar^{-n}\shn+\hbar\shl=\shl\cap \hbar^{-n_0}\shn+\hbar\shl$
for any $n\ge n_0$.
Set 
\eq\label{eq:shn0}
&&\shn_0\eqdot\shl\cap \hbar^{-n_0}\shn.
\eneq
Then $(\shn_0+\hbar\shl)/\hbar\shl$ is a coherent $\shao$-module and
\eqn
&&\shn_0\cap \hbar^{n}\shl\subset \hbar^n(\hbar^{-n-n_0}\shn\cap \shl)
\subset \hbar^{n}(\shn_0+\hbar\shl)\subset \hbar\shn_0+\hbar^{n+1}\shl
\eneqn
for any $n>0$.
Hence by Lemma~\ref{lem:fg0}:
\begin{itemize}
\item
$\shn_0$ is locally finitely generated over $\sha$, 
\item
$\bigcap\limits_{n\ge0}(\shn_0+\hbar^n\shl)=\shn_0$,
\item
$\shn_0\cap \hbar^n\shl=\hbar^n\shn_0$ for any $n\ge0$.
\end{itemize}

\vspace{0.4em}
\noindent
(i) Since 
$\shn\cap\hbar^{n_0}\shl=\hbar^{n_0}\shn_0$ by \eqref{eq:shn0},
the module
$\shn/\hbar^{n_0}\shn_0\simeq\shn/(\shn\cap\hbar^{n_0}\shl)\simeq
(\shn+\hbar^{n_0}\shl)/\hbar^{n_0}\shl$ is 
$\sha$-coherent. Since 
$\hbar^{n_0}\shn_0$ is locally finitely generated over $\sha$,
$\shn$ is also locally finitely generated over $\sha$.

\vspace{0.4em}
\noindent
(ii) We have
\eqn
\bigcap_{n\ge n_0}(\shn+\hbar^n\shl)
&\subset& (\shn+\hbar^{n_0}\shl)\bigcap\cap_{n\ge n_0}(\shn+\hbar^n\shl)\\
&\subset&\shn+\hbar^{n_0}\shl\bigcap\cap_{n\ge n_0}(\shn+\hbar^n\shl)\\
&\subset&\shn+\cap_{n\ge n_0}(\hbar^{n_0}\shl\cap\shn+\hbar^n\shl)\\
&\subset&\shn+\cap_{n\ge n_0}(\hbar^{n_0}\shn_0+\hbar^n\shl)\\
&\subset&\shn+\hbar^{n_0}\shn_0=\shn.
\eneqn

\vspace{0.4em}
\noindent
(iii) For $n>n_0$, we have
\eqn
\hbar^{n}\shl\cap\shn&\subset& \hbar^{n_0}(\shl\cap \hbar^{-n_0}\shn)\cap \hbar^n\shl\\
&\subset& \hbar^{n_0}(\shn_0\cap \hbar^{n-n_0}\shl)\\
&\subset& \hbar^{n_0}\hbar^{n-n_0}\shn_0=\hbar^n\shn_0\\
&\subset& \hbar(\shn\cap \hbar^{n-1}\shl).
\eneqn
\vspace{0.4em}
\noindent
(iv) 
Since $\shn$ has no $\hbar$-torsion, we have  the exact sequence 
\eqn
&&0\to\shn/\hbar^n\shn\to[\hbar]\shn/\hbar^{n+1}\shn\to \shn/\hbar\shn\to 0.
\eneqn
Hence, it is enough to show that $\shn/\hbar\shn$ is
coherent.
By (i), the images of $\shn$ and $\hbar\shn$ in $\shl/\hbar^n\shl$ are coherent.
Since $\shn\cap \hbar^n\shl\subset \hbar\shn$ for some $n$,
by (ii), we have the exact sequence
$$\dfrac{\hbar\shn}{\hbar\shn\cap \hbar^n\shl}\to\dfrac{\shn}{\shn\cap \hbar^n\shl}
\to \dfrac{\shn}{\hbar\shn}\to 0,$$
which implies that $\shn/\hbar\shn$ is coherent.
\end{proof}

\begin{corollary}\label{cor:coh}
Assume that $\shm$ is a locally finitely generated $\sha$-module.
If $\shm/\hbar^n\shm$ is a coherent $\sha$-module for all $n>0$, then
$\shm$ is an $\sha$-module locally of finite presentation
and $\smash{\bigcap\limits_{n\ge0}}\hbar^n\shm=0$.
\end{corollary}
\begin{proof}
We may assume that $\shm=\shl/\shn$ for a
locally free $\sha$-module $\shl$ of finite rank and $\shn\subset\shl$.
{}From the exact sequence
\eqn
0\to (\shn+\hbar^n\shl)/\hbar^n\shl\to\shl/\hbar^n\shl\to \shm/\hbar^n\shm\to 0,
\eneqn
 we deduce that  $(\shn+\hbar^n\shl)/\hbar^n\shl$ is coherent for any $n$.
Hence $\shn$ is locally finitely generated by  Lemma~\ref{lem:fg},
which implies that $\shm$ is locally of finite presentation.
Since 
$\bigcap\limits_{n\ge0}(\shn+\hbar^n\shl)=\shn$ by Lemma~\ref{lem:fg},
\eqn
&&\smash{\bigcap\limits_{n\ge0}}\hbar^n\shm\simeq 
\bl\smash{\bigcap\limits_{n\ge0}}(\shn+\hbar^n\shl)\br/\shn
\eneqn
vanishes.
\end{proof}

\begin{proposition}\label{prop:coh}
$\sha$ is coherent.
\end{proposition}
\begin{proof}
Let $\shi$ be a locally finitely generated $\sha$-submodule of $\sha$.
Since 
\eqn
&&(\shi+\hbar^{n+1}\sha)/\hbar^{n+1}\sha\simeq
\shi/(\shi\cap\hbar^{n+1}\sha)\subset\sha/\hbar^{n+1}\sha,
\eneqn
the $\sha$-module $\shi/\hbar^n\shi$ is coherent by 
Lemma \ref{lem:fg}~(iv). 
Hence Corollary~\ref{cor:coh} implies that
$\shi$ is locally of finite presentation.
\end{proof}

\begin{lemma}
Any filtrant family of coherent $\sha$-submodules of $\sha$ is 
locally stationary.
\end{lemma}
\begin{proof}
Let $\{\shi_i\}_{i\in I}$ be a family of 
coherent $\sha$-submodules of $\sha$ indexed by a filtrant ordered set $I$, 
with $\shi_i\subset\shi_j$ for any $i\leq j$.
Then $\{(\hbar^{-k}\shi_i\cap \sha+\hbar\sha)/\hbar\sha\}_{i\in I,\ k\ge 0}$
is increasing with respect to $k$ and $i\in I$.
Hence locally there exist $i_0$ and $k_0$ such that
$\hbar^{-k}\shi_i\cap \sha+\hbar\sha=\hbar^{-k_0}\shi_{i_0}\cap \sha+\hbar\sha$
for any $i\ge i_0$ and $k\ge k_0$.
Then, for $i\ge i_0$, the ideal $\shj_i\eqdot \sha\cap \hbar^{-k_0}\shi_i$ satisfies 
\eqn
\shj_i\cap \hbar^m\sha\subset \hbar^m(\hbar^{-m-k_0}\shi_i\cap\sha)\subset 
\hbar^m(\hbar^{-k_0}\shi_i\cap\sha+\hbar\sha)
\subset \hbar\shj_i+\hbar^{m+1}\sha
\eneqn
for any $m>0$.
Hence Lemma \ref{lem:fg0} implies that $\shj_i\cap \hbar\sha=\hbar\shj_i$.
Since we have $\shj_i\subset\shj_{i_0}+\hbar\sha$, we have
$\shj_i\subset \shj_{i_0}+(\shj_i\cap\hbar\sha)\subset \shj_{i_0}+\hbar\shj_i$.
Then Nakayama's lemma implies $\shj_i=\shj_{i_0}$, or equivalently,
$\hbar^{-k_0}\shi_i\cap \sha= \hbar^{-k_0}\shi_{i_0}\cap \sha$ for $i\ge i_0$.
Thus $\{\shi_i\cap \hbar^{k_0}\sha\}_{i}$ is locally stationary.
Since $\{\shi_i/(\shi_i\cap \hbar^{k_0}\sha)\}_i$ is
a filtrant family of coherent submodules of $\sha_{k_0-1}$,
it is also locally stationary and it follows that 
$\{\shi_i\}_i$ is locally stationary.
\end{proof}

\begin{lemma} \label{lem:Acoh}
For any $x\in X$, $\sha_x$ is a coherent ring.
\end{lemma}
\begin{proof}
Any morphism $f\cl\sha_x^{\oplus n}\to\sha_x$
extends to a morphism $\tilde{f}\cl\sha^{\oplus n}\vert_U\to\sha\vert_U$
for some open neighborhood $U$ of $x$.
Since $\shn\eqdot\ker\tilde{f}$ is coherent,
$\shn_x\simeq \ker f$ is a finitely generated $\sha_x$-module.
\end{proof}

\begin{lemma}\label{lem:seppt}
For any $x\in X$ and a finitely generated left ideal $I$ of $\sha_x$,
$I\cap \hbar^{n+1}\sha_x=\hbar(I\cap \hbar^n\sha_x)$ for $n\gg0$.
\end{lemma}
\begin{proof}
Let us take a coherent ideal $\shi$ of $\sha$ defined on a neighborhood of
$x$ such that $I=\shi_x$. Then Lemma \ref{lem:fg}
implies that 
$\shi\cap \hbar^{n+1}\sha=\hbar(\shi\cap \hbar^n\sha)$ for $n\gg0$.
\end{proof}

\begin{lemma}
For any $x\in X$, $\sha_x$ is a Noetherian ring.
\end{lemma}
\begin{proof}
Set $A=\sha_x$.
Let us show that an increasing sequence $\{I_n\}_n$
of finitely generated left ideals of $A$
is stationary.
Since $\{(\hbar^{-k}I_n\cap A+\hbar A)/\hbar A\}_{n,k}$ is
increasing with respect $n,k$,
there exist $n_0$ and $k_0$ such that
$\hbar^{-k}I_n\cap A+\hbar A=\hbar^{-k_0}I_{n_0}\cap A+\hbar A$
for $n\ge n_0$ and $k\ge k_0$.
For any $n\ge n_0$ there exists $k\ge k_0$ such that
$\hbar^{-k}I_n\cap \hbar A=\hbar(\hbar^{-k}I_n\cap A)$
by Lemma~\ref{lem:seppt}.
Hence we have
$\hbar^{-k}I_n\cap A
\subset \hbar^{-k}I_n\cap(\hbar^{-k_0}I_{n_0}\cap A+\hbar A)
\subset \hbar^{-k_0}I_{n_0}\cap A+(\hbar^{-k}I_n\cap\hbar A)
\subset \hbar^{-k_0}I_{n_0}\cap A+\hbar(\hbar^{-k}I_n\cap A)$.
Since $\hbar^{-k}I_n\cap A$ is finitely generated by Lemma \ref{lem:Acoh},
Nakayama's lemma implies that
$\hbar^{-k}I_n\cap A=\hbar^{-k_0}I_{n_0}\cap A$.
Hence $\hbar^{-k_0}I_{n}\cap A=\hbar^{-k_0}I_{n_0}\cap A$ for any $n\ge n_0$.
Therefore $I_n\cap \hbar^{k_0}A=\hbar^{k_0}(\hbar^{-k_0}I_{n}\cap A)$
is stationary.
Since
$\{I_n/(I_n\cap \hbar^{k_0}A)\}_n$ is stationary,
$\{I_n\}_n$ is stationary.
\end{proof}

Thus, we have proved that $\sha$ is a Noetherian ring. 

\begin{lemma}\label{lem:limnseq}
Let $\{\shm_n\}_{n\ge0}$ be a projective system
of coherent $\sha$-modules.
Assume that $\hbar^{n+1}\shm_n=0$ and the induced morphism
$\shm_{n+1}/\hbar^{n+1}\shm_{n+1}\to\shm_n$ is an isomorphism for any $n\ge 0$.
Then $\shm\eqdot\prolim[n]\shm_n$ is a coherent $\sha$-module and
$\shm/\hbar^{n+1}\shm\to\shm_n$ is an isomorphism for any $n\ge0$.
\end{lemma}
\begin{proof}
Since the question is local, 
we may assume that $X\in\BB$ and there exist 
a free $\cora$-module $V$ of finite rank and a morphism $V\to \shm_0(X)$
which induces an epimorphism
$\shl\eqdot \sha\otimes_{\cora}V\epito \shm_0$.
Since $\shm_{n+1}(X)\to \shm_n(X)$ is surjective 
and $V$ is projective, 
we have a projective system of morphisms
$\{V\to \shm_n(X)\}_n$:
$$\xymatrix{
&V\ar@{.>}[d]\ar@{.>}[dr]\ar@{.>}@<-.1ex>[drrr]\ar@<.2ex>[drrrr]\\
\cdots\ar@{->>}[r]
&\shm_n(X)\ar@{->>}[r]&\shm_{n-1}(X)\ar@{->>}[r]&{\cdots}\ar@{->>}[r]
&\shm_1(X)\ar@{->>}[r]&\shm_0(X),}
$$
which induces a projective system of morphisms
$\{\shl\to\shm_n\}_n$. Hence we may assume that
there exists a morphism
$\shl\to \shm$ such that the composition 
$\shl\to\shm\to \shm_0$ is an epimorphism.
Since $\shl\to \shm_n/\hbar\shm_n\isoto\shm_0$ is an
epimorphism,
$\shl\to\shm_n$ is an epimorphism by 
Lemma~\ref{le:nakayamacor}.

Set $\shl_n=\shl/\hbar^{n+1}\shl$, and
let $\shn_n$ be the kernel of $\shl_n\to\shm_n$.
Set $\shn=\smash{\prolim[n]} \shn_n$.
Then we have a commutative diagram with exact rows:
$$\xymatrix@R=3ex{
0\ar[r]&\shn\ar[r]\ar[d]&\shl\ar[r]\ar[d]&\shm\ar[d]\\
0\ar[r]&\shn_n\ar[r]&\shl_n\ar[r]&\shm_n\ar[r]&0.}$$
In the commutative diagram
$$
\xymatrix@R=4ex{
&&0\ar[d]&0\ar[d]\\
&&\hbar^{n+1}\shl_{n+1}\ar[r]\ar[d]&\hbar^{n+1}\shm_{n+1}\ar[r]\ar[d]&0\\
0\ar[r]&\shn_{n+1}\ar[r]\ar[d]&\shl_{n+1}\ar[r]\ar[d]&\shm_{n+1}\ar[r]\ar[d]&0\\
0\ar[r]&\shn_n\ar[r]&\shl_n\ar[r]\ar[d]&\shm_n\ar[d]\ar[r]&0\\
&&0&0
}$$
the rows and the columns are exact.
Hence the left vertical arrow $\shn_{n+1}\to\shn_n$ is an epimorphism.
Therefore, $\shn_{n+1}(U)\to\shn_n(U)$ is surjective for any $U\in\BB$,
and $\shn(U)\isoto\prolim[m]\shn_m(U)\to\shn_n(U)$ is surjective.
Hence $\shn\to \shn_n$ is an epimorphism for any $n\ge0$,
and
$\{\shn_n(U)\}_n$ satisfies the Mittag-Leffler condition.

Thus in the following commutative diagram
$$
\xymatrix{
0\ar[r]&\shn(U)\ar[r]\ar[d]^(.5){\sim}
&\shl(U)\ar[r]\ar[d]^(.5){\sim}&\shm(U)\ar[d]^(.5){\sim}\ar[r]&0\\
0\ar[r]&\db{\prolim[n]\shn_n(U)}\ar[r]&\db{\prolim[n]\shl_n(U)}\ar[r]&
\db{\prolim[n]\shm_n(U)}\ar[r]&0,}
$$
the bottom row is exact.
Hence $0\to\shn\to\shl\to\shm\to0$ is exact.
Since $\shn\to\shn_n$ is an epimorphism,
we have $\shm/\hbar^{n+1}\shm\simeq \coker (\shn\to\shl_n)
\simeq\coker(\shn_n\to\shl_n)\simeq \shm_n$.
Since $\shm$ is locally finitely generated and
$\shm/\hbar^{n+1}\shm$ is coherent for any $n\ge0$,
$\shm$ is coherent by Corollary~\ref{cor:coh}
and Proposition~\ref{prop:coh}.
\end{proof}

\begin{proposition}\label{prop:vancoh}
Let $\shm$ be a coherent $\sha$-module. Then we have the following properties.
\bnum
\item
$\shm$ is $\hbar$-complete, i.e., $\shm\isoto\widehat{\shm}$,
\item
for any $U\in\BB$, $H^k(U;\shm)=0$ for any $k>0$.
\enum
\end{proposition}
\begin{proof}
(i) 
Since the kernel of
$\shm\to\widehat{\shm}$ is $\bigcap\limits_{n\ge0} \hbar^n\shm$,
the morphism $\shm\to\widehat{\shm}$ is a monomorphism by 
Corollary \ref{cor:coh}.

Let us show that
$\shm\to\widehat{\shm}$ is an epimorphism.
By the preceding lemma, $\widehat{\shm}$ is a coherent $\sha$-module,
and $\widehat{\shm}/\hbar\widehat{\shm}\simeq \shm/\hbar\shm$.
Hence Nakayama's lemma implies that
$\shm\to\widehat{\shm}$ is an epimorphism.

\vspace{0.4em}
\noindent
(ii) For any $U\in\BB$, the map
$\sect(U;\shm/\hbar^{n+1}\shm)\to \sect(U;\shm/\hbar^{n}\shm)$
is surjective, and $H^k(U;\shm/\hbar^n\shm)=0$ for any $k>0$.
Hence Lemma \ref{lem:cohpro} implies (ii).
\end{proof}

\Cor\label{cor:crcoh}
Let $\shm$ be an $\sha$-module.
If $\shm$ satisfies the following conditions {\rm(i)} and {\rm(ii)},
then $\shm$ is a coherent $\sha$-module.
\bnum
\item 
$\shm$ is $\hbar$-complete,
\item
$\hbar^n\shm/\hbar^{n+1}\shm$ is a coherent $\sha_0$-module for all $n\ge0$.
\enum
\encor
\Proof
Set $\shm_n=\shm/\hbar^{n+1}\shm$. Then it is a coherent $\sha$-module by
(ii),
and $\prolim[n]\shm_n$ is a coherent $\sha$-module 
by Lemma~\ref{lem:limnseq}.
\QED

This completes the proof of Theorem~\ref{th:formalfini1}.

\bigskip
\begin{lemma}\label{lem:free}
Let $\shm$ be a coherent $\sha$-module without $\hbar$-torsion.
If $\shm/\hbar\shm$ is a locally free $\sha_0$-module of rank $r\in\Z_{\ge0}$,
then $\shm$ is a locally free $\sha$-module of rank $r$.
\end{lemma}
\begin{proof}
We may assume that there exists
a morphism of $\sha$-modules
$f\cl\shl\seteq\sha^{\oplus r}\to \shm$
such that $\shl/\hbar\shl\to\shm/\hbar\shm$ is an isomorphism.
Then, Nakayama's lemma implies that $f$ is an epimorphism.
Let $\shn$ be the kernel of $f$.
Since $\shm$ has no $\hbar$-torsion,
we have an exact sequence
$0\to \shn/\hbar\shn\to \shl/\hbar\shl\to\shm/\hbar\shm\to0$.
Hence $\shn/\hbar\shn=0$ and Nakayama's lemma implies $\shn=0$.
\end{proof}

The following proposition gives a criterion for the coherence of the
projective limit
of coherent modules, generalizing Lemma~\ref{lem:limnseq}.

\begin{proposition}\label{prop:procoh}
Let $\{\shn_n\}_{n\ge1}$ be a projective system of 
coherent $\sha$-modules. Assume
\banum
\item
the pro-object $\proolim[n]\shn_n/\hbar\shn_n$ is representable by a 
coherent $\sha_0$-module,
\item
the pro-object $\proolim[n]\ker(\shn_n\to[\hbar] \shn_n)$ is representable
by a coherent $\sha_0$-module.
\eanum
Then
\bnum
\item
$\shn\eqdot\prolim[n]\shn_n$ is a coherent $\sha$-module,
\item
$\shn/\hbar^{k+1}\shn\isoto \proolim[n]\shn_n/\hbar^{k+1}\shn_n$ for any $k\ge0$,
\item
$\ker(\shn\to[\hbar] \shn)\isoto\proolim[n]\ker(\shn_n\to[\hbar] \shn_n)$. 
\item
Assume moreover that for each $n\ge 1$ there exists $k\ge0$ such 
that $\hbar^k\shn_n=0$. Then 
the projective system $\{\shn_n\}_{n}$ satisfies the Mittag-Leffler condition.
\enum
\end{proposition}
\begin{proof}
For any $k\geq 0$, set 
\eqn
&&\shs_k\eqdot \proolim[n]\shn_n/\hbar^{k+1}\shn_n.
\eneqn
Then $\shs_0$ is representable by a coherent $\sha$-module
by hypothesis~(a).
We shall  show that $\shs_k$ is representable by a coherent $\sha$-module
for all $k\geq0$ by induction on $k$. 
Consider the exact sequences
\eq
&&0\to\hbar\shn_n/\hbar^{k+1}\shn_n\to \shn_n/\hbar^{k+1}\shn_n\to \shn_n/\hbar\shn_n\to 0,
\label{eq:ex1}\\
&&\ker(\shn_n\to[\hbar]\shn_n)\to
\shn_n/\hbar^{k}\shn_n\to[\hbar]\hbar\shn_n/\hbar^{k+1}\shn_n\to 0.
\label{eq:ex2}
\eneq

Assume that $\shs_{k-1}$ is representable by a coherent $\sha$-module.
Applying the functor $\proolim[n]$ to the exact sequence~\eqref{eq:ex2}, 
we deduce that the object 
$\proolim[n]\hbar\shn_n/\hbar^{k+1}\shn_n$ is representable 
by a coherent $\sha$-module. 
Then applying the functor $\proolim[n]$ to the exact sequence~\eqref{eq:ex1}, 
we deduce that
$\shs_{k}$ is representable by a coherent $\sha$-module.

Since $\shn_n\simeq\prolim[k]\shn_n/\hbar^{k+1}\shn_n$
by Theorem~\ref{th:formalfini1} \eqref{th:coh:int},
we have
\eqn
&&
\shn\simeq\prolim[k,n]\shn_n/\hbar^{k+1}\shn_n\simeq\prolim[k]\shs_k.
\eneqn
Since $\shs_{k+1}/\hbar^{k+1}\shs_{k+1}\simeq \shs_k$, 
Lemma \ref{lem:limnseq} implies (i), (ii).
The property (iii) is obvious.

Let us prove (iv).
By the assumption,
$\shn_n\simeq\proolim[k]\shn_n/\hbar^k\shn_n$.
Hence
$$\proolim[n]\shn_n\simeq
\proolim[k,n]\shn_n/\hbar^k\shn_n\simeq
\proolim[k]\shs_k.$$
Since $\{\shs_k\}_{k}$ satisfies the Mittag-Leffler condition,
$\{\shn_n\}_n$  satisfies the Mittag-Leffler condition
by Lemma~\ref{lem:ML}.
\end{proof}

\begin{remark}
In Proposition~\ref{prop:procoh} (iv), the condition $\hbar^k\shn_n=0$
($k\gg0$) is necessary as seen by considering the projective system
$\shn_n=\hbar^n\sha$, ($n\in\N$).
\end{remark}

\section{A variant of the preceding results}\label{section:variant}

Here, we consider rings which satisfy hypotheses \eqref{eq:FDringa}, but in which
\eqref{eq:FDringb} is replaced with 
another  hypothesis. 
Indeed, 
as we shall see, the ring $\shd_X\forl$ of differential operators 
on a complex manifold $X$
has nice properties, although
$\shd_X$ does not satisfy \eqref{eq:FDringb}.
The study of modules over $\shd_X\forl$ is performed in \cite{DGS}.

We assume that $X$ is a Hausdorff locally compact space.
By a basis $\BB$ of compact subsets of $X$,
we mean a family of compact subsets such that for any $x\in X$ and any
open neighborhood $U$ of $x$, 
there exists $K\in\BB$ such that $x\in\Int(K)\subset U$.

We consider a  $\cora$-algebra  $\sha$ on $X$ and
a section $\hbar$ of $\sha$ contained in the center of $\sha$.
Set $\shao=\sha/\hbar\sha$. 
We assume the conditions \eqref{eq:FDringa} and 
\eq\label{eq:FDringc}
&&\left\{\parbox{62ex}{
\bnum
\item[{\rm (iv')}] 
there exist a base $\BB$ of compact subsets of $X$ and a prestack 
$U\mapsto\mdgd[\shao\vert_U]$ ($U$ open in $X$) such that 
\banum
\item for any $K\in \BB$ and an open subset $U$ such that $K\subset
  U$,
there exists $K'\in\BB$ such that
$K\subset\Int(K')\subset K'\subset U$,
\item $U\mapsto \mdgd[\shao\vert_U]$ is a full subprestack of 
$U\mapsto \mdcoh[\shao\vert_U]$,
\item for an open subset $U$ and $\shm\in\mdcoh[\shao\vert_U]$,
if $\shm\vert_V$ belongs to $\mdgd[\shao\vert_V]$ for any 
relatively compact open subset $V$ of $U$,
then $\shm$ belongs to $\mdgd[\shao\vert_U]$,\label{cond:exh}
\item
for any open subset $U$ of $X$, 
$\mdgd[\shao\vert_U]$ is stable by subobjects, quotients
and extension
in $\mdcoh[\shao\vert_U]$,
\item 
for any $K\in\BB$, any open set $U$ containing $K$, any  $\shm\in\mdgd[\shao\vert_U]$
and any $j>0$, one has $H^j(K;\shm)=0$,
\item for any $\shm\in\mdcoh[\shao\vert_U]$, there exists an open covering 
$U=\bigcup_iU_i$ such that
$\shm\vert_{U_i}\in\mdgd[\shao\vert_{U_i}]$,
\label{goodlocal}
\item $\shao\in\mdgd[\shao]$.\label{cond:good}
\eanum 
\enum
}\right.  
\eneq
Note that Lemmas \ref{le:nakayamacor} and \ref{le:cohcoh} still hold.

The prestack $U\mapsto \mdgd[\shao\vert_U]$ being
given, a coherent module which belongs to
$\mdgd[\shao\vert_{U}]$ will be called a good module. Note that in
view of hypothesis (iv')~\eqref{goodlocal}, hypothesis (iv')~\eqref{cond:good}
could be deleted
 since all the results of this subsection will be of local
nature. However, we keep it for simplicity.
\begin{example}\label{exa:DL}
Let $X$ be a complex manifold, $\sho_X$ the structure sheaf and let
$\shd_X$ denote the $\C$-algebra of differential operators. 
One checks easily
that, taking for $\BB$ the set of Stein compact subsets and for
$\shao$ the $\C$-algebra $\shd_X$, 
the prestack of good $\shd_X$-modules in the sense of \cite{Ka2}
satisfies the hypotheses 
\eqref{eq:FDringc}.
\end{example}
\begin{definition}
A coherent $\sha$-module $\shm$ is good 
\glossary{good!module}%
if both the kernel and the cokernel of 
$\hbar\cl\shm\to\shm$ are good $\shao$-modules. One denotes by
$\mdgd[\sha]$ the category of good $\sha$-modules. 
\end{definition}
Note that an $\shao$-module is good if and only if it is good as an
$\sha$-module.
This allows us to state:
\begin{definition}
An $\sha_n$-module $\shm$ is good if it is good as an $\sha$-module.
\end{definition}

\begin{lemma}\label{le:gdthick1}
The category $\mdgd[\sha]$ is a subcategory of $\mdcoh[\sha]$
stable by subobjects, quotients and extension.
\end{lemma}
\begin{proof}
First note that
$\hbar^n\shm/\hbar^{n+1}\shm$ is a good $\sha_0$-module
for any $\shm\in\mdgd[\sha]$ and any integer $n\ge0$.
Indeed, it is a quotient of $\shm/\hbar\shm$.

For an $\sha$-module $\shn$, set
$\shn_\hbar\eqdot\ker(\hbar\cl\shn\to\shn)$.

We shall show that any coherent $\sha$-submodule $\shn$ 
of a good $\sha$-module
$\shm$ is a good $\sha$-module. 
It is obvious that $\shn_\hbar$ is a good $\sha_0$-module,
because it is a coherent submodule of
$\shm_\hbar$.
We shall show that $\shn/(\hbar\shn+\shn\cap\hbar^{k+1}\shm)$ is 
a good $\sha_0$-module for any $k\ge0$.
We argue by induction on $k$. For $k=0$, it is a good $\sha_0$-module
since it is a coherent submodule of $\shm/\hbar\shm$.
For $k>0$, we have an exact sequence
\eq
&&\ba{c}
0\to\dfrac{\hbar\shn+\shn\cap\hbar^{k}\shm}
{\hbar\shn+\shn\cap\hbar^{k+1}\shm}\to
\dfrac{\shn}{\hbar\shn+\shn\cap\hbar^{k+1}\shm}\\[1.5ex]
\hs{40ex}\to 
\dfrac{\shn}{\hbar\shn+\shn\cap\hbar^{k}\shm}\to 0.\ea
\label{eq:nmseq}
\eneq
Since $(\shn\cap\hbar^{k}\shm)/(\shn\cap\hbar^{k+1}\shm)$
is a coherent submodule of $\hbar^{k}\shm/\hbar^{k+1}\shm$, it is a good $\sha_0$-module.
Since $(\hbar\shn+\shn\cap\hbar^{k}\shm)/
(\hbar\shn+\shn\cap\hbar^{k+1}\shm)$ is a quotient of
$(\shn\cap\hbar^{k}\shm)/(\shn\cap\hbar^{k+1}\shm)$,
the left term in 
\eqref{eq:nmseq} is a good $\sha_0$-module.
Hence the induction proceeds and we conclude that 
$\shn/(\hbar\shn+\shn\cap\hbar^{k+1}\shm)$ is 
a good $\sha_0$-module.

\noindent 
On any compact set, we have $\shn\cap\hbar^{k+1}\shm
\subset\hbar\shn$ for $k\gg0$. Hence,
$(\shn/\hbar\shn)\vert_V$ is a good $(\sha_0\vert_V)$-module
for any relatively compact subset $V$.
Hence $\shn$ belongs to $\mdgd[\sha]$ by 
(iv')~\eqref{cond:exh}.

Consider an exact sequence $0\to \shm'\to\shm\to\shm''\to 0$ of
coherent $\sha$-modules. 
It gives rise to an exact sequence of coherent $\shao$-modules
\eqn
&&0\to \shm_\hbar'\to\shm_\hbar\to\shm_\hbar''\to\shm'/\hbar\shm'
\to\shm/\hbar\shm\to\shm''/\hbar\shm''\to 0.
\eneqn
If $\shm$ is a good $\sha$-module, then so is $\shm'$.
Hence the exact sequence above implies
that $\shm_\hbar''$ and $\shm''/\hbar\shm''$ are good $\sha_0$-modules.
This shows that $\mdgd[\sha]$ is stable by quotients.

Finally, let us show that $\mdgd[\sha]$ is stable by extension.
If $\shm_\hbar'$, $\shm_\hbar''$, $\shm'/\hbar\shm'$ and
$\shm''/\hbar\shm''$ are good $\sha_0$-modules, then
so are $\shm_\hbar$ and $\shm/\hbar\shm$ by the exact sequence
above.
\end{proof}

\begin{lemma} \label{lem:surB}
Let $K\in\BB$, and $n\geq 0$.
\bnum
\item 
For any good $\sha_n$-module $\shn$, we have $H^j(K;\shn)=0$ for $j\neq 0$. 
\item 
For any epimorphism $\shn\to\shn'$ of good $\sha_n$-modules,
$\shn(K)\to\shn'(K)$ is surjective.
\item 
$\sha(K)\to\sha_n(K)$ is surjective.
\enum
\end{lemma}
\begin{proof}
(i) is proved by induction on $n$, using the exact sequence \eqref{eq:exseq0}.

\vspace{0.4em}
\noindent
(ii) follows immediately from (i) and the fact that the kernel of a
morphism of good modules is good.

\vspace{0.4em}
\noindent
(iii)\quad By (ii), $\sha_{n+1}(K)\to\sha_{n}(K)$ is surjective for any $n\ge0$.
Hence $\prolim[m]\bl\sha_m(K)\br\to \sha_n(K)$ is surjective.

For $s\in\sha_n(K)$, there exist $K'\in\BB$ and $s'\in \sha_n(K')$
such that $K\subset \Int(K')$ and $s'\vert_K=s$.
Then $s'$ is in the image of
$\prolim[m]\bl\sha_m(K')\br\to \sha_n(K')$.
Hence $s$ is in the image of $\sha(K)\to\sha_n(K)$, because
$\prolim[m]\bl\sha_m(K')\br\to \sha_n(K')\to\sha_n(K)$
decomposes into
$$\prolim[m]\bl\sha_m(K')\br\to\prolim[m]\bl\sha_m(\Int(K'))\br\simeq\sha(\Int(K'))
\to\sha(K)\to\sha_n(K).$$
\end{proof}

The proof of the following theorem is almost the same as  
the proof of Theorem~\ref{th:formalfini1},
and we do not repeat it.

\begin{theorem}\label{th:formalfini1b}
Assume \eqref{eq:FDringa} and \eqref{eq:FDringc}.
\bnum
\item
$\sha$ is a left Noetherian ring.
\item 
Let $\shm$ be a locally finitely generated $\sha$-module.
Then $\shm$ is coherent if and only if 
$\hbar^n\shm/\hbar^{n+1}\shm$ 
is a coherent $\sha_0$-module for any $n\ge0$.
\item 
For any coherent $\sha$-module $\shm$,
$\shm$ is $\hbar$-complete, i.e.,
$\shm\isoto\widehat{\shm}$.
\item 
Conversely, an $\sha$-module $\shm$ is coherent
if and only if
$\shm$ is $\hbar$-complete 
and $\hbar^n\shm/\hbar^{n+1}\shm$ is a coherent $\sha_0$-module
for any $n\ge0$. 
\item
For any good $\sha$-module $\shm$ and any $K\in\BB$, 
we have $H^j(K;\shm)=0$ for any $j>0$. 
\enum
\end{theorem}

\section{$\hbar$-graduation and $\hbar$-localization}

In this section, $\sha$ is a sheaf of algebras satisfying
hypotheses~\eqref{eq:FDringa} and either 
\eqref{eq:FDringb} or \eqref{eq:FDringc}.

\subsubsection*{Graded  modules}
Let $\shr$ be a $\Z[\hbar]$-algebra on a topological space $X$.
We assume that $\shr$ has no $\hbar$-torsion.
We set $$\shro\seteq\shr/\shr\hbar.$$

\begin{definition}\label{def:grad}
We denote by $\gr\cl \Der(\shr)\to \Der(\shro)$ the 
left derived functor of the right exact functor
$\md[\shr]\to\md[\shro]$ given by
$\shm\mapsto \shm/\hbar \shm$.
For $\shm\in\Der(\shr)$ we call $\gr(\shm)$ the graded module associated to $\shm$. 
\end{definition}
We have
\eqn
&&\gr(\shm)\simeq\shro\lltens[\shr]\shm\simeq\Z_X\lltens[{\Z_X[\hbar]}]\shm.
\eneqn
\begin{lemma}\label{lem:grHa}
Let $\shm\in\RD(\shr)$ and let $a\in\Z$. 
Then we have an exact sequence of $\shro$-modules
\eqn
&&0\to \shro\tens[\shr]H^a(\shm)
\to H^a(\gr(\shm))\to\tor^{\shr}_{1}(\shro,H^{a+1}(\shm))\to0.
\eneqn
\end{lemma}
Although this kind of results is well-known, 
we give a proof for the reader's convenience.
\begin{proof}
The exact sequence $0\to\shr\To[\hbar]\shr\to\shro\to0$ gives rise to
the distinguished triangle
\eqn
&&\shm\To[\hbar]\shm\To\gr(\shm)\To[+1].
\eneqn
It induces a long exact sequence
\eqn
&&H^a(\shm)\To[\hbar]H^a(\shm)\To H^a(\gr(\shm))
\To H^{a+1}(\shm)\To[\hbar]H^{a+1}(\shm).
\eneqn
The result then follows from
\eqn
\shro\tens[\shr]H^a(\shm)&\simeq&\coker(H^a(\shm)\To[\hbar]H^a(\shm)),\\
\tor^{\shr}_{1}(\shro,H^{a+1}(\shm))
&\simeq&\ker(H^{a+1}(\shm)\To[\hbar]H^{a+1}(\shm)).
\eneqn
\end{proof}

\begin{proposition}\label{pro:tensgr}
\bnum
\item 
Let $\shk_1\in\Der(\shr^\rop)$ and $\shk_2\in\Der(\shr)$. Then 
\eq\label{eq:circgr}
&&\gr(\shk_1\lltens[\shr]\shk_2)
\simeq \gr(\shk_1)\lltens[\shro]\gr(\shk_2).
\eneq
\item
Let $\shk_i\in\Der(\shr)$ \lp$i=1,2$\rp. Then 
\eq\label{eq:circgr1}
&&\gr(\rhom[\shr](\shk_1,\shk_2))
\simeq \rhom[\shro](\gr(\shk_1),\gr(\shk_2)).
\eneq
\enum
\end{proposition}
\begin{proof}
(i) We have 
\eqn
\gr(\shk_1\lltens[\shr]\shk_2)
&\simeq&\shk_1\lltens[\shr]\shk_2\lltens[{\Z_X[\hbar]}]\Z_X
\simeq\shk_1\lltens[\shr]\gr(\shk_2)\\
&\simeq&
\shk_1\lltens[\shr]\shro\lltens[\shro]\gr(\shk_2))\\
&\simeq&
(\shk_1\lltens[\shr]\shro)\lltens[\shro]\gr(\shk_2)\\
&\simeq&\gr(\shk_1)\lltens[\shro]\gr(\shk_2).
\eneqn
\vspace{0.4em}
\noindent
(ii) The proof is similar.
\end{proof}
\begin{proposition}\label{pro:operationgr}
Let $f\cl X\to Y$ be a morphism of topological spaces. Let $\shm\in\Der(\Z_X[\hbar])$
 and $\shn\in\Der(\Z_Y[\hbar])$. Then
\eqn
&&\gr\roim{f}\shm\simeq\roim{f}\gr\shm,\\
&&\gr\opb{f}\shn\simeq\opb{f}\gr\shn.
\eneqn
\end{proposition}
\begin{proof}
This follows immediately from the fact that for a
complex of $\Z_X[\hbar]$-modules $\shm$, 
$\gr(\shm)$ is represented by the mapping cone of
$\shm\to[\hbar]\shm$ and similarly for $\Z_Y[\hbar]$-modules.
\end{proof}

Recall that $\sha$ is a sheaf of algebras satisfying
hypotheses~\eqref{eq:FDringa} and either
\eqref{eq:FDringb} or \eqref{eq:FDringc}.
The functor $\gr$ induces a functor (we keep the same notation):
\eq\label{eq:grcohbd}
&&\gr\cl\RD^\Rb_\coh(\sha)\to\RD^\Rb_\coh(\shao).
\eneq

The following proposition is an immediate consequence of
Lemma~\ref{lem:grHa} and Nakayama's lemma.
\begin{proposition}\label{pro:grHa}
Let $\shm\in\Derb_\coh(\sha)$ and let $a\in\Z$. 
The conditions below are equivalent:
\bnum
\item $H^a(\gr(\shm))\simeq 0$,
\item $H^a(\shm)\simeq 0$ and $H^{a+1}(\shm)$ has no $\hbar$-torsion.
\enum
\end{proposition}

\begin{corollary}\label{cor:conservative1}
The functor $\gr$ in \eqref{eq:grcohbd} is conservative
\ro {\em i.e.,} a morphism in $\Derb_\coh(\sha)$ is an isomorphism
as soon as its image by $\gr$ is an isomorphism in $\Derb_\coh(\shao)$\rf.
\end{corollary}
\begin{proof}
Consider a morphism $\phi\cl\shm\to\shn$ in $\Derb_\coh(\sha)$ and assume
that  it induces an isomorphism 
$\gr(\phi)\cl\gr(\shm)\to\gr(\shn)$ in $\Derb_\coh(\shao)$. 
Let $\shm\to\shn\to\shl\To[{+1}]$ be a distinguished triangle.
Then $\gr\shl\simeq0$,
and hence all the cohomologies of $\shl$ vanishes by the proposition above,
which means that
$\shl\simeq0$.
\end{proof}

\subsubsection*{Homological dimension}

In the sequel, for  a left Noetherian $\cora$-algebra $\shr$, 
we shall say that a coherent $\shr$-module $\shp$  is locally projective 
\glossary{locally projective}%
if, for any open subset $U\subset X$, 
the functor 
\eqn
&&\hom[\shr](\shp,\scbul)\cl \mdcoh[\shr\vert_U]\to\md[\cora_U] 
\eneqn
is exact.
This is equivalent to one of the following conditions: (i) for each $x\in X$, the stalk
$\shp_x$ is projective as an $\shr_x$-module, (ii) for each $x\in X$, the stalk
$\shp_x$ is flat as an $\shr_x$-module,  
(iii) $\shp$ is locally  a direct summand of a
free $\shr$-module of finite rank.

\begin{lemma}
A coherent $\sha$-module $\shp$ is locally
projective if and only if $\shp$ has no $\hbar$-torsion and 
$\gr\shp$ is a locally projective $\shao$-module.
\end{lemma}
\begin{proof}
We set for short 
$A\eqdot\sha_x$ and $A_0\eqdot(\shao)_x$. Note that $A_0\simeq\gr A$.

Let $P$ be a finitely generated $A$-module.

\vspace{0.4em}
\noindent
(i) Assume that $P$ is projective. Then $P$ is a direct summand of a free 
$A$-module. It follows that $P$ has no $\hbar$-torsion and 
$\gr P$ is also a direct summand of a free $A_0$-module.

\vspace{0.4em}
\noindent
(ii) Assume that $P$ has no $\hbar$-torsion and $\gr P$ is projective.
Consider an exact sequence $0\to N\to[u] L\to P\to 0$ in which $L$ is
free of finite rank. Applying the functor $\gr$
we find the exact sequence $0\to \gr N\to[{\gr u}] \gr L\to \gr P\to 0$
and $\gr P$ being projective, there exists a map 
${\overline v}\cl \gr L\to \gr N$ such that 
${\overline v}\circ{\gr u}=\id_{\gr N}$. Let us choose a map 
$v\cl L\to N$ such that  $\gr(v)=\overline v$. 
Since $\gr(v\circ u)=\id_{\gr N}$, we may write 
\eqn
&&v\circ u =\id_N -\hbar\phi
\eneqn
where $\phi\cl N\to N$ is an $A$-linear map.
The map $\id_N -\hbar\phi$ is invertible and we denote by $\psi$ its
inverse. Then $\psi\circ v\circ u=\id_N$, which proves that $P$ is a
direct summand of a free $A$-module.
\end{proof}

\begin{theorem}\label{th:hddim}
Let $d\in N$. Assume that 
any coherent $\shao$-module locally admits 
a resolution of length $\leq d$ by free $\shao$-modules of finite rank.
Then 
\banum
\item
 for any coherent locally projective $\sha$-module $\shp$,  
there locally exists a free $\sha$-module of finite rank $\shf$ 
such that $\shp\oplus\shf$ is free of finite rank,
\item
any coherent $\sha$-module locally admits 
a resolution of length $\leq d+1$ by free $\sha$-modules of finite rank.
\eanum
\end{theorem}
\begin{proof}
(a) It is well-known (see {\em e.g.,} \cite[Lem.~B.2.2]{Sc}) that
the result in (a) is true when replacing $\sha$ with $\shao$.
Now, let $\shp$ be as in the statement. Then  
 $\gr\shp$ is projective and coherent. Therefore,  
there exists a locally free $\sha$-module  
$\shf$ such that $\gr\shp\oplus\gr\shf$ is free of finite rank
over $\shao$. This implies that $\shp\oplus\shf$ is free of finite rank
over $\sha$ by Lemma~\ref{lem:free}. 

\vspace{0.4em}
\noindent
(b)-(i) Let $\shm\in\mdcoh[\sha]$ and let us first assume that $\shm$ 
has no $\hbar$-torsion.
Since $\sha$ is coherent, there exists locally an exact sequence
\eqn
&&0\to\shk\to \shl_{d-1}\to\cdots\to\shl_0\to\shm\to 0,
\eneqn
the $\sha$-modules $\shl_i$ ($0\leq i\leq d-1$) being free of
finite rank.
Applying the functor $\gr$, we find an exact sequence of
$\shao$-modules and it follows that $\gr(\shk)$ is 
projective and finitely generated. Therefore $\shk$ is 
projective and finitely generated. Let $\shf$ be as
in the statement (a). Replacing $\shk$ and $\shl_{d-1}$ with $\shk\oplus\shf$
and $\shl_{d-1}\oplus\shf$ respectively, the result follows in this case.

\vspace{0.4em}
\noindent
(b)-(ii) In general, any coherent $\sha$-module $\shm$ locally admits a resolution
$0\to\shn\to\shl\to\shm\to0$, where $\shl$ 
is a free $\sha$-module of finite rank.
Since $\shn$ has no $\hbar$-torsion,
$\shn$ admits a free resolution with length $d$, and
the result follows.
\end{proof}
\begin{corollary}\label{cor:freeres}
We make the hypotheses of Theorem~\ref{th:hddim}.
Let $\shm^\scbul$ be a complex of $\sha$-modules concentrated in 
degrees $[a,b]$ and assume that $H^i(\shm)$ is coherent for all
$i$. Then, in a neighborhood of each $x\in X$, there exists a 
quasi-isomorphism 
$\shl^\scbul\to \shm^\scbul$ where $\shl^\scbul$ is a complex of free
$\sha$-modules of finite rank concentrated in degrees $[a-d-1,b]$.
\end{corollary}
\begin{proof}
The proof uses~\cite[Lem.~13.2.1]{K-S3} (or rather the dual
statement). Since we do not use this result here, details are left to
the reader. 
\end{proof}

\subsubsection*{Localization}

For a $\Z_X[\hbar]$-algebra $\shr$ with no $\hbar$-torsion, we set
\eq\label{eq:ahbar}
&&\shr^\loc\eqdot\Z_X[\hbar,\hbar^{-1}]\tens[{\Z_X[\hbar]}]\shr,
\eneq
and we call $\shr^\loc$ the {\em $\hbar$-localization} of $\shr$.
For an $\shr$-module $\shm$, we also  set 
$$\shm^\loc\seteq\shr^\loc\tens[\shr]\shm
\simeq\Z_X[\hbar,\hbar^{-1}]\tens[{\Z_X[\hbar]}]\shm.$$

\begin{lemma}\label{le:locNoeth}
The algebra $\shal$ is Noetherian.
\end{lemma}
\begin{proof}
Let $T$ be an indeterminate. One
knows by \cite[Th.~A.30]{Ka2} that $\sha{[T]}$ is Noetherian. Since
$\shal\simeq\sha{[T]}/\sha{[T]}(T\hbar-1)$, the result follows. 
\end{proof}

\section{Cohomologically complete modules}
In order to  give a criterion for the coherency of 
the cohomologies of a complex of modules over
an algebra $\sha$ satisfying
\eqref{eq:FDringa}  and either \eqref{eq:FDringb} or \eqref{eq:FDringc}.
we introduce the notion of cohomologically complete complexes.

In this section, $\shr$ is a $\Z[\hbar]$-algebra satisfying
\eq
&&\text{$\shr$ has no $\hbar$-torsion.}
\eneq

Recall that
$\shm^\loc\seteq\Z[\hbar,\hbar^{-1}]\tens[{\Z[\hbar]}]\shm$
for an $\shr$-module $\shm$.

\begin{lemma}
For $\shm,\shm'\in\Derb(\shr^\loc)$, we have
\eqn
&&\rhom[\shr^\loc](\shm,\shm')\isoto\rhom[\shr](\shm,\shm').
\eneqn
\end{lemma}
\Proof
We have $\shr^\loc\lltens[\shr]\shm\simeq\shm$. Hence,
\eqn
\rhom[\shr^\loc](\shm,\shm')
&\simeq&\rhom[\shr^\loc](\shr^\loc\ltens[\shr]\shm,\shm')\\
&\simeq&\rhom[\shr](\shm,\shm').
\eneqn
\QED
The next result is obvious.

\begin{lemma}\label{lem:locA}
The triangulated category $\Der(\shr ^\loc)$ 
is equivalent to the full subcategory of
$\Der(\shr)$ consisting of objects $\shm$ satisfying 
one of the following equivalent conditions:
\bnum
\item
$\gr(\shm)=0$,
\item
$\hbar\cl H^i(\shm)\to H^i(\shm)$ is an isomorphism for any integer $i$,
\item
$\shm\to\shr ^\loc\ltens[\shr]\shm$ is an isomorphism,
\item
$\rhom[\shr](\shr^\loc,\shm)\to\shm$ is an isomorphism,
\item
$\rhom[\shr](\shr^\loc/\shr,\shm)\simeq0$.
\enum
\end{lemma}

\Lemma\label{lem:extU}
Let $K$ be a $\Z[\hbar]$-module with projective dimension $\le1$.
Then for any $\shm\in\Der(\shr)$, any open subset $U$ and any integer
$i$, we have an exact sequence
\eqn
&&0\To\Ext[{\Z[\hbar]}]{1}\bl K,H^{i-1}(U;\shm)\br
\To H^i\bl U;\rhom[{\Z[\hbar]}](K,\shm)\br\\
&&\hs{30ex}\To\Hom[{\Z[\hbar]}]\bl K,H^{i}(U;\shm)\br
\To0.\eneqn
\enlemma
\Proof
We have a distinguished triangle 
\eqn&&\RHom[{\Z[\hbar]}]\bl K, \tau^{<i}\rsect(U;\shm)\br
\to\RHom[{\Z[\hbar]}]\bl K,\rsect(U;\shm)\br\\
&&\hs{30ex}\to
\RHom[{\Z[\hbar]}]\bl K,\tau^{\ge i}\rsect(U;\shm)\br\To[+1].\eneqn
Since $H^k\RHom[{\Z[\hbar]}](K,N)=0$ 
for any $k\not=0,1$ and any $\Z[\hbar]$-module $N$, 
we have $H^{i+1}\RHom[{\Z[\hbar]}]\bl K,\tau^{<i}\rsect(U;\shm)\br\simeq0$.
Hence we have an exact sequence
\eqn&&0\to H^i\RHom[{\Z[\hbar]}]\bl K, \tau^{<i}\rsect(U;\shm)\br
\to H^i\RHom[{\Z[\hbar]}]\bl K,\rsect(U;\shm)\br\\
&&\hs{30ex}
\to H^i\RHom[{\Z[\hbar]}]\bl K,\tau^{\ge i}\rsect(U;\shm)\br\to0.
\eneqn
Then the result follows from 
$$H^i\RHom[{\Z[\hbar]}]\bl K, \tau^{<i}\rsect(U;\shm)\br\simeq
 \Ext[{\Z[\hbar]}]{1}\bl K, H^{i-1}(U;\shm)\br$$ and
$H^i\RHom[{\Z[\hbar]}]\bl K,\tau^{\ge i}\rsect(U;\shm)\br
\simeq
\Hom[{\Z[\hbar]}]\bl K, H^i(U;\shm)\br$.
\QED

Recall that we set
\eq
\twh{\shm}\seteq\prolim[n]\shm/\hbar^{n}\shm.
\eneq

\begin{lemma}\label{lem:comp}
Let $\shm\in\md[\shr]$ and assume that  $\shm$ has no $\hbar$-torsion.
\bnum
\item
$\hom[\shr](\shr ^\loc/\shr,\shm^\loc/\shm)\simeq
\ext[\shr]{1}(\shr ^\loc/\shr,\shm)\simeq \widehat\shm$.
\item
$\ker(\shm\to\twh\shm)\simeq\hom[\shr](\shrl,\shm)$.
In particular,
$\shm$ is $\hbar$-separated if and only if $\hom[\shr](\shrl,\shm)\simeq 0$.
\item
$\coker(\shm\to\twh \shm)\simeq\ext[\shr]{1}(\shrl,\shm)$.
In particular, $\shm$ is $\hbar$-complete if and only if 
$\ext[\shr]{j}(\shrl,\shm)\simeq0$ for $j=0,1$. \label{extcomp}
\enum
\end{lemma}
\begin{proof}
We have 
\eqn
\hom[\shr](\shr ^\loc/\shr,\shm^\loc/\shm)
&\simeq&\prolim[n]\hom[\shr](\hbar^{-n}\shr/\shr,\shm^\loc/\shm)\\
&\simeq&\prolim[n]\hom[\shr](\hbar^{-n}\shr/\shr,\hbar^{-n}\shm/\shm)\\
&\simeq&\prolim[n]\shm/\hbar^n\shm\simeq\widehat\shm.
\eneqn
Since $\rhom[\shr](\shr ^\loc/\shr,\shm^\loc)\simeq0$ by Lemma~\ref{lem:locA},
applying the functor $\rhom[\shr](\shr ^\loc/\shr,\scbul)$
to $0\to\shm\to\shm^\loc\to\shm^\loc/\shm\to0$, we obtain
an isomorphism
$\hom[\shr](\shr^\loc/\shr,\shm^\loc/\shm)
\isoto\ext[\shr]{1}(\shr ^\loc/\shr,\shm)$.
Hence we obtain (i).

\medskip
By the long exact sequence associated with $0\to\shr\to\shr ^\loc\to\shr
^\loc/\shr\to0$, we obtain
\eqn
&&\hom[\shr](\shr^\loc/\shr,\shm)\to
\hom[\shr](\shr^\loc,\shm)\to\hom[\shr](\shr,\shm)
\\
&&\hs{10ex}
\to\ext[\shr]{1}(\shr^\loc/\shr,\shm)\to
\ext[\shr]{1}(\shr^\loc,\shm)\to0,
\eneqn
which reduces to
$$0\to \hom[\shr](\shr^\loc,\shm)\to\shm\to\widehat\shm\to
\ext[\shr]{1}(\shr^\loc,\shm)\to0.$$
Hence we obtain (ii) and (iii).
\end{proof}

Consider the right orthogonal \glossary{right orthogonal}%
category $\Der(\shr^\loc)^{\perp r}$ \index{DerCf@$\Der(\shr^\loc)^{\perp r}$}%
to the full
subcategory  $\Der(\shr^\loc)$ of  $\Der(\shr)$. By definition, this is the full 
triangulated subcategory consisting of objects $\shm\in\Der(\shr)$ satisfying 
 $\Hom[\Der(\shr)](\shn,\shm)\simeq0$ for any $\shn\in\Der(\shr^\loc)$
 (see \cite[Exe.~10.15]{K-S3}).

\begin{definition}\label{def:cohco}
One says that  an object $\shm$ of $\Der(\shr)$ 
is cohomologically complete \glossary{cohomologically complete}%
if it belongs to  $\Der(\shr^\loc)^{\perp r}$.
\end{definition}

\begin{proposition}\label{prop:cccr}
\bnum
\item
For $\shm\in\Der(\shr)$, the following conditions are equivalent:
\banum
\item
$\shm$  is cohomologically complete,
\item 
$\rhom[\shr](\shr^\loc,\shm)\simeq\rhom[{\Z[\hbar]}](\Z[\hbar,\hbar^{-1}],\shm)\simeq0$,
\item
$\indlim[U\ni x]\Ext[{\Z[\hbar]}]{j}\bl{\Z[\hbar,\opb{\hbar}]},H^i(U;\shm)\br\simeq0$
for any $x\in X$, $j=0,1$ and any $i\in\Z$. Here, $U$ ranges over an open
 neighborhood system of $x$.
\eanum
\item 
$\rhom[\shr](\shr^\loc/\shr,\shm)$ is \cc\ for any $\shm\in\Der(\shr)$.
\item
For any $\shm\in\Der(\shr)$, there exists a distinguished triangle
$$\shm'\to\shm\to\shm''\To[+1]$$
with $\shm'\in\Der(\shr ^\loc)$ and
$\shm''\in\Der(\shr^\loc)^{\perp r}$.
\item
Conversely, if 
$$\shm'\to\shm\to\shm''\To[+1]$$ is a distinguished triangle
with $\shm'\in\Der(\shr ^\loc)$ and
$\shm''\in\Der(\shr^\loc)^{\perp r}$, then 
$\shm'\simeq\rhom[\shr](\shr ^\loc,\shm)$ and
$\shm''\simeq\rhom[\shr](\shr ^\loc/\shr[-1],\shm)$.
\enum
\end{proposition}
\begin{proof}
(i) (a)$\Leftrightarrow$(b) \quad For any $\shn\in\Der(\shrl)$, one has
\eqn
\Hom[\shr](\shn,\shm)&\simeq&\Hom[\shr](\shrl\lltens[\shr]\shn,\shm)\\
&\simeq&\Hom[\shr](\shn,\rhom[\shr](\shrl,\shm))
\eneqn
and it vanishes for all $\shn\in\Der(\shrl)$ if and only if 
$\rhom[\shr](\shrl,\shm)\simeq0$.

\smallskip
\noindent
(i) (b)$\Leftrightarrow$(c) follows from Lemma~\ref{lem:extU}.

\smallskip
\noindent
(ii)\quad
Since $\shr ^\loc\ltens[\shr](\shr ^\loc/\shr)\simeq0$, we have
\eqn
&&\rhom[\shr]\bl\shr ^\loc,\rhom[\shr](\shr ^\loc/\shr,\shm)\br\\
&&\hs{10ex}\simeq\rhom[\shr](\shr ^\loc\lltens[\shr](\shr ^\loc/\shr),\shm)
\simeq0,
\eneqn
and hence
$\rhom[\shr](\shr ^\loc/\shr,\shm)$ is cohomologically complete.

\noindent
(iii) We have obviously 
$\rhom[\shr](\shr ^\loc,\shm)\in\Der(\shr ^\loc)$. Hence the 
distinguished triangle
$$\rhom[\shr](\shr ^\loc,\shm)\to\rhom[\shr](\shr,\shm)\to
\rhom[\shr](\shr ^\loc/\shr[-1],\shm)\To[+1]$$
gives the result.

\smallskip
\noindent
(iv)\quad
Since $\rhom[\shr](\shr ^\loc,\shm'')\simeq0$, we have
$$\shm'
\simeq\rhom[{\shr}](\shr ^\loc,\shm')
\isoto\rhom[{\shr}](\shr ^\loc,\shm),$$
and hence $\shm''\simeq \rhom[\shr](\shr ^\loc/\shr[-1],\shm)$.
\QED

Note that $\shm\mapsto\rhom[\shr](\shr ^\loc,\shm)$ is a right adjoint
functor of the inclusion functor
$\Der(\shr ^\loc)\to\Der(\shr)$, and
the quotient category $\Der(\shr)/\Der(\shr ^\loc)$ 
is equivalent to $\Der(\shr^\loc)^{\perp r}$.

Remark that $\shm\in\Der(\shr)$ is \cc\ if and only if its image in
$\Der(\Z_X[\hbar])$ is \cc.
\Cor\label{cor:cccrt}
Let $\shm$ be an $\shr$-module. Assume the following conditions:
\banum
\item
$\shm$ has no $\hbar$-torsion and is $\hbar$-complete,
\item
for any $x\in X$, denoting by $\shu_x$ the family of open neighborhoods of $x$, we have 
$\inddlim[U\in\shu_x]H^i(U;\shm)\simeq0$ for $i\not=0$.
\eanum
Then $\shm$ is \cc.
\encor
\Proof
For $U$ open, we have the maps
\eqn
\sect(U;\shm)\to[a]\prolim[n]\sect(U;\shm)/\hbar^n\sect(U;\shm)
            \to[b]\prolim[n]\sect(U;\shm/\hbar^n\shm)
            \simeq\sect(U;\shm)
\eneqn
whose composition is the identity. Since $b$ is a monomorphism, $a$ is an isomorphism 
and therefore $\sect(U;\shm)$ is $\hbar$-complete. 
Consider the assertion
\eqn
&&\inddlim[U\in\shu_x]\Ext[{\Z[\hbar]}]{j}(\Z[\hbar,\hbar^{-1}],H^i(U;\shm))\simeq0\mbox{ for }j=0,1.
\eneqn
This assertion is true for $i=0$ since  $\sect(U;\shm)$ is $\hbar$-complete and is true for $i\not=0$ by
hypothesis~(b). The same vanishing assertion remains true
after replacing  $\inddlim$ with $\indlim$. 
Applying Proposition~\ref{prop:cccr}~(i), we find that $\shm$ is \cc.
\QED

\begin{proposition}\label{pro:conserv}
Let $\shm\in\Der(\shr)$ be a cohomologically complete object and $a\in\Z$.
If $H^i(\gr(\shm))=0$ for any $i<a$, then $H^i(\shm)=0$ for any
$i<a$.
\end{proposition}
\Proof
The exact sequence
$H^{i-1}(\gr\shm)\to H^i(\shm)\To[\hbar] H^i(\shm)\to H^i(\gr\shm)$
implies that $H^i(\shm)\To[\hbar] H^i(\shm)$
is an isomorphism for $i<a$.
Hence $\tau^{<a}\shm\in\Der(\shr ^\loc)$ and we have
$\rhom[\shr](\shr ^\loc,\tau^{<a}\shm)\simeq\tau^{<a}\shm$.
By the distinguished triangle,
$$\rhom[\shr](\shr ^\loc,\tau^{<a}\shm)
\to\rhom[\shr](\shr ^\loc,\shm)\to
\rhom[\shr](\shr ^\loc,\tau^{\ge a}\shm)\to[+1],$$
we have
$\tau^{<a}\shm\simeq\rhom[\shr](\shr ^\loc,\tau^{\ge a}\shm)[-1]$
and they belong to $\Der^{<a}(\shr)\cap\Der^{\ge a+1}(\shr)\simeq 0$.
\QED
\begin{corollary}\label{cor:conserv}
Let $\shm\in\Der(\shr)$ be a cohomologically complete object. If
$\gr(\shm)\simeq0$, then $\shm\simeq0$.
\end{corollary}

\Prop\label{prop:homcc}
Assume that  $\shm\in\Der(\shr)$ is \cc. Then
$\rhom[\shr](\shn,\shm)\in\Der(\Z_X[\hbar])$ is \cc\
for any $\shn\in\Der(\shr)$.
\enprop
\Proof
It follows from
\eqn
&&\rhom[{\Z[\hbar]}]\bl\Z[\hbar,\hbar^{-1}],\rhom[\shr](\shn,\shm)\br\\
&&\hs{25ex}\simeq
\rhom[\shr]\bl\shn,\rhom[{\Z[\hbar]}](\Z[\hbar,\hbar^{-1}],\shm)\br.
\eneqn
\QED

We can 
give an alternative definition of a cohomologically
complete module.

\begin{lemma}\label{lem:ccpro}
Let $\shm\in\Der(\shr)$. Then we have
\bnum
\item
$\rpi\bl(\proolim[n]\shr\hbar^n)\lltens[\shr]\shm\br
\simeq\rhom[{\shr}](\shr^\loc,\shm)$,
\item
$\rpi\bl(\proolim[n]\shr/\shr\hbar^n)\lltens[\shr]\shm\br
\simeq\rhom[{\shr}](\shr^\loc/\shr[-1],\shm)$.
\enum
\end{lemma}
\begin{proof}
It is enough to show (i). Set $L=\proolim[n](\shr\hbar^n)$.
Note that $L$ is flat, {\em i.e.}, the functor 
$L\tens[\shr]\scbul$ from $\md[\shr]$ to $\Pro(\md[\shr])$ is exact.

\noindent
One has the isomorphisms
\eqn
\hom[\shr](\shr^\loc,\shm)&\simeq&\hom[\shr](\indlim[n]\shr\hbar^{-n},\shm)\\
&\simeq&\prolim[n]\hom[\shr](\shr\hbar^{-n},\shm)\\
&\simeq&\prolim[n]\hom[\shr](\shr\hbar^{-n},\shr)\tens[\shr]\shm\\
&\simeq&\prolim[n]\bl\shr\hbar^{n}\tens[\shr]\shm\br.
\eneqn

\noindent
It remains to show that
$\rpi(L\lltens[{\shr}]\scbul)$ is the right
derived functor of  
$\shm\mapsto\prolim[n](\shr\hbar^n\tens[{\shr}]\shm)$. 
Hence, it is enough to
check that if $\shm$ is an injective $\shr$-module, then 
 $\rpi(L\ltens[{\shr}]\shm)$ is in degree zero.
Applying Lemma~\ref{le:cohandprolim} with $\shm_n=\shr\hbar^n
\tens[{\shr}]\shm$, we find  $H^i(U;\rpi(L\ltens[{\shr}]\shm))\simeq 0$ for
$i>0$. 
Therefore,
$\RR^i\pi(L\ltens[{\shr}]\shm)\simeq0$ for $i>1$.
On the other hand, since $\{\sect(U;\shm_n)\}_n$ satisfies the
Mittag-Leffler condition, we get that 
$\RR^1\pi(L\ltens[{\shr}]\shm)\simeq0$.
\end{proof}
Hence, $\shm$ is \cc\ if and only if the morphism
$\shm\to\rpi\bl\proolim[n](\shr/\shr\hbar^n)\ltens[\shr]\shm\br$
is an isomorphism.

\begin{proposition}\label{pro:cohcodirim}
Let $f\cl X\to Y$ be a continuous map, and $\shm\in\Der(\Z_X[\hbar])$.
If $\shm$ is \cc, then so is $\roim{f}\shm$.
\end{proposition}
\Proof
It immediately follows from
$$\rhom[{\Z_Y[\hbar]}](\Z_Y[\hbar,\hbar^{-1}],\roim{f}\shm)
\simeq\roim{f}\rhom[{\Z_X[\hbar]}](\Z_X[\hbar,\hbar^{-1}],\shm).$$
\QED

\section{Cohomologically complete $\sha$-modules}

In this section, $\sha$ is a $\cora$-algebra satisfying
hypotheses~\eqref{eq:FDringa} and either
\eqref{eq:FDringb} or \eqref{eq:FDringc}.

\begin{theorem}\label{th:cohimplcohco}
Let $\shm\in\Db_\coh(\sha)$. Then $\shm$ is cohomologically complete.
\end{theorem}
\begin{proof}
Since any coherent module is an extension of a module without
$\hbar$-torsion by an $\hbar$-torsion module, it is enough to treat
each case.

\smallskip
Assume first that $\shm$ is an $\hbar$-torsion coherent $\sha$-module.
Since the question is local, we may assume that there exists $n$ such
that $\hbar^n\shm=0$.
Then the action of $\hbar$ on the cohomology groups of
$\rhom[\sha](\sha ^\loc,\shm)$ is nilpotent and invertible,
and hence the cohomology groups vanishes.

\medskip
Now assume that $\shm$ is a coherent $\sha$-module without $\hbar$-torsion.
Then Corollary~\ref{cor:cccrt} shows that $\shm$ is \cc.
\end{proof}

\begin{corollary}\label{cor:cohimplcohco}
If $\shm\in\Db_\coh(\sha)$ and $\shn\in\Der(\sha)$,
then $\rhom[\sha](\shn,\shm)$ is cohomologically complete.
\end{corollary}
\begin{proof}
It is an immediate consequence of Proposition~\ref{prop:homcc} and 
the theorem above.
\QED

In the course of the proof of Theorem~\ref{th:formalfini2} below, we shall
use the following elementary lemma that we state here without proof.
\begin{lemma}[{\rm Cross Lemma}]\label{le:crosslemma}
Let $\shc$ be an abelian category and consider an exact diagram in
$\shc$ 
\eqn
\xymatrix@C=3ex@R=3ex{
   &X_2\ar[d]&\\
X_1\ar[r]&Y\ar[d]\ar[r]&Z_1\\
&Z_2.&
}\eneqn
Then the conditions below are equivalent:
\banum
\item $\im(X_2\to Z_1)\isoto\im(Y\to Z_1)$,
\item $\im(X_1\to Z_2)\isoto\im(Y\to Z_2)$,
\item $X_1\oplus X_2 \to Y$ is an epimorphism.
\eanum
\end{lemma}

\begin{theorem}\label{th:formalfini2}
Let $\shm\in\Der^+(\sha)$ and assume:
\banum
\item[{\rm(a)}]
$\shm$ is cohomologically complete,
\item[{\rm(b)}]
$\gr(\shm)\in\Der^+_\coh(\shao)$.
\eanum
Then,  $\shm\in\Der^+_\coh(\sha)$, and we have the isomorphism
\eqn
&&H^i(\shm)\isoto\prolim[n]H^i(\sha_n\ltens[\sha]\shm)
\eneqn
for all $i\in\Z$.
\end{theorem}

\begin{proof}
We shall assume \eqref{eq:FDringb}.
The case of hypothesis~\eqref{eq:FDringc} 
could be treated with slight modifications.

Recall that $\sha_n\eqdot\sha/\hbar^{n+1}\sha$
and set $\shm_n=\sha_n\ltens[\sha]\shm$, $\shn_n^j\eqdot H^j(\shm_n)$.

\vspace{0.4em}
\noindent
(1) \quad For each $n\in \N$, the distinguished triangle
$\sha/\hbar^{n}\sha\To[\hbar]\sha/\hbar^{n+1}\sha\To\sha/\hbar\sha\To[+1]$
induces the distinguished triangle
\eq\label{eq:exse3}
\shm_{n-1}\To[\hbar]&\shm_n\To\shm_0\To[+1].
\eneq
This triangle gives rise to the long exact sequence
\eq\label{eq:lexse3}
&&\shn_0^{j-1}\To\shn_{n-1}^j\To[\hbar] \shn_{n}^j\To \shn_0^j\To
\shn_{n-1}^{j+1}
\eneq
from which we deduce by induction on $n$
that $\shn^j_n$ is a coherent $\sha$-module for any $j$ and $n\ge0$
by using the hypothesis (b).

\vspace{0.5em}
\noindent
(2) \quad 
Let us show that
\eq\label{eq:repr}
&&\parbox{58ex}{$\proolim[n]\coker(\shn_n^j\to[\hbar]\shn^j_n)$
and $\proolim[n]\ker(\shn_n^j\to[\hbar]\shn^j_n)$ are locally representable
for all $j\in\Z$.}
\eneq

\vspace{0.4em}
\noindent
 Consider the distinguished triangle:
\eq&&
\shm_0\To[\hbar^{n+1}]\shm_{n+1}\To\shm_n\To[+1].
\label{eq:exse2}
\eneq
It gives rise to the long exact sequence
\eq\label{eq:lexse2}
&&\cdots \to \shn^j_0\To[\hbar^{n+1}] \shn^j_{n+1}\To \shn^j_n\To[\phi^j_n]
\shn^{j+1}_0\To\cdots.
\eneq
Now consider the exact diagram, deduced from \eqref{eq:lexse3} and \eqref{eq:lexse2}:
\eq
\xymatrix@C=4ex@R=4ex{
   &\shn^j_{n+1}\ar[d]&\\
\shn^j_{n-1}\ar[r]^-{\hbar\cdot}\ar[rd]_-{\phi^j_{n-1}}
              &\shn^j_{n}\ar[d]^-{\phi^j_n}\ar[r]&\shn^j_{0}\\
&\shn^{j+1}_{0}.
}\label{diag:cross}\eneq
Here the commutativity of the triangle follows from
the commutative diagram
\eqn
\xymatrix@C=7ex{
\shm_0\ar[r]^{\hbar^{n}}\ar[d]^{\id}&\shm_n\ar[r]\ar[d]^\hbar
                                      &\shm_{n-1}\ar[r]^-{+1}\ar[d]^\hbar&\\
\shm_0\ar[r]^{\hbar^{n+1}}&\shm_{n+1}\ar[r]
                                       &\shm_n\ar[r]^-{+1}&
}\eneqn
Hence $\im(\phi^j_{n-1})\subset \im(\phi^j_{n})\subset\shn^{j+1}_0$.
Therefore, the sequence $\{\im \phi_n^j\}_n$ of coherent $\sha$-submodules 
of $\shn^{j+1}_0$ is increasing and thus locally stationary. 
It follows from \eqref{diag:cross} and Lemma~\ref{le:crosslemma} that 
\eq\label{eq:decreasing}
&& \parbox{62ex}{
the decreasing
sequence $\{\im(\shn^j_n\to \shn^j_0)\}_n$ is locally stationary
for any $j\in\Z$.
}
\eneq
Since $\coker(\shn^j_{n-1}\to[\hbar] \shn^j_n)\simeq \im(\shn^j_n\to \shn^j_0)$ 
by \eqref{eq:lexse3}, we deduce that 
\eqn
&&
\proolim[n]\coker(\shn^j_n\to[\hbar] \shn^j_n)
\simeq \proolim[n]\coker(\shn^j_{n-1}\to[\hbar] \shn^j_n)
\eneqn
is locally representable.

Since 
$\ker(\shn^{j}_{n-1}\to[\hbar]\shn^{j}_{n})\simeq
\shn^{j-1}_0/\im(\shn^{j-1}_n\to\shn^{j-1}_0)$ 
by \eqref{eq:lexse3}, we get that 
$\proolim[n]\ker(\shn^j_n\to[\hbar] \shn^j_n)
\simeq \proolim[n]\ker(\shn^j_{n-1}\to[\hbar] \shn^j_n)$ is 
locally representable.

Therefore, we have proved \eqref{eq:repr}.
Then by Proposition~\ref{prop:procoh}, 
$\prolim[n]\shn_n^j$ is a coherent $\sha$-module and
$\{\shn_n^{j}\}_n$ satisfies the Mittag-Leffler condition.

\vspace{0.5em}
\noindent
(3) \quad Hence it remains to prove that 
$H^j(\shm)\isoto\prolim[n]\shn^j_n$ for any $j$.
Set $\shm'=(\proolim[n]\sha_n)\ltens[\sha]\shm\in\Der^+(\Pro(\md[\sha]))$ and
$\shn^j=H^j(\shm')\simeq\proolim[n]\shn^j_n\in\Pro(\md[\sha])$.
Lemma~\ref{lem:ccpro} implies that
$$\shm\isoto\rpi\shm'.$$  
Since the $\shn^j_n$'s are coherent $\sha$-modules,
for any any $U\in\BB$, $H^i(U;\shn^j_n)=0$ ($i>0$)
and $\{\shn^j_n(U)\}_n$ satisfies the 
Mittag-Leffler condition.
Hence in the exact sequence
\eqn
&&0\to \RR^1\pi\bl\proolim[n]H^{i-1}(U;\shn^j_n)\br
\to H^i(U;\rpi\shn^j)\to\prolim[n]H^i(U;\shn^j_n)\to0,
\eneqn
the first and the last term vanish, and we obtain
$\RR^i\pi\shn^j=0$ for any $i>0$.
Let us show that
$H^j(\shm)\isoto\prolim[n]\shn^j_n$ by induction on
$j$.
Assuming $H^j(\shm)\isoto\prolim[n]\shn^j_n$ for $j<c$, 
let us show that $H^c(\shm)\isoto\prolim[n]\shn^c_n$.
By the assumption, $H^i(\shm)\isoto\rpi(\shn^i)$ for any $i<c$.
Hence $\tau^{<c}\shm\isoto\rpi(\tau^{<c}\shm')$.
Since $\shm\isoto\rpi\shm'$, we obtain
$\tau^{\ge c}\shm\isoto\rpi(\tau^{\ge c}\shm')$.
Hence taking the $c$-th cohomology, we obtain
$H^c(\shm)\isoto\RR^0\pi H^c(\shm')\simeq\prolim[n]\shn^c_n$.
\QED
The next result will be useful.
\begin{proposition}\label{pro:cocotens}
Assume that $\sha^\rop/\hbar\sha^\rop$ is a Noetherian ring
and the flabby dimension of $X$ is finite. 
If $\shm\in\Derb(\sha)$ is cohomologically complete, then for any
$\shn\in\Derb_\coh(\sha^\rop)$, the object $\shn\ltens[\sha]\shm$ 
of $\Der^{-}(\Z[\hbar]_X)$ is cohomologically
 complete.
\end{proposition}
\begin{proof}
By the assumption on the flabby dimension, there exists $a\in\Z$ such that
$H^i\rhom[{\Z[\hbar]}](\Z[\hbar,\hbar^{-1}],\shf)=0$ for any 
$\shf\in\Der^{\le 0}(\Z_X[\hbar])$ and any $i>a$.

For any $n\in\Z$ we can locally find a finite complex $L$ of 
free $\sha ^\rop$-modules of finite rank such that there exists a distinguished triangle 
$L\ltens[\sha]\shm\to\shn\ltens[\sha]\shm\to G$ where $G\in \Der^{<n}(\Z_X[\hbar])$.
Since $L\ltens[\sha]\shm$ is \coco,
$H^i\rhom[\sha](\sha ^\loc,\shn\ltens[\sha]\shm)\simeq
H^i\rhom[\sha](\sha ^\loc,G)=0$ for $i>n+a$.
Hence $\shn\ltens[\sha]\shm$ is \coco.
\end{proof}

\subsubsection*{Flatness}
\begin{theorem}\label{th:flat}
Assume that $\sha^\rop/\hbar\sha^\rop$ is a Noetherian ring
and the flabby dimension of $X$ is finite.
Let $\shm$ be an $\sha$-module.
Assume the following conditions:
\banum
\item $\shm$ has no $\hbar$-torsion,
\item $\shm$ is \coco,
\item $\shm/\hbar\shm$ is a flat $\sha_0$-module.
\eanum
Then
$\shm$ is a flat $\sha$-module.
\end{theorem}
\Proof
Let $\shn$ be a coherent  $\sha^\rop$-module.
It is enough to show that we have $H^i(\shn\ltens[\sha]\shm)=0$ for any $i<0$.
We know by Proposition~\ref{pro:cocotens} that  $\shn\ltens[\sha]\shm$ is \coco.
Since $\gr(\shn\ltens[\sha]\shm)\simeq
(\gr\shn)\ltens[\shao](\gr\shm)$ belongs to $\Der^{\ge0}(\Z_X)$,
we have $\shn\ltens[\sha]\shm\in\Der^{\ge0}(\Z[\hbar]_X)$
by Proposition~\ref{pro:conserv}.
\QED
\begin{corollary}
In the situation of Theorem~\ref{th:flat}, assume
moreover that $\shm/\hbar\shm$ is a faithfully flat $\sha_0$-module. Then 
 $\shm$ is a faithfully flat $\sha$-module.
\end{corollary}
\begin{proof}
Let $\shn$ be a coherent $\sha^\op$-module such that 
$\shn\tens[\sha]\shm \simeq 0$. We have to show that $\shn\simeq 0$. 
By Theorem~\ref{th:flat}, we know that $\shm$ is flat, so that 
$\shn\tens[\sha]\shm \simeq \shn\lltens[\sha]\shm$.
Therefore 
\eqn
&&(\gr\shn)\ltens[\sha_0](\gr\shm)\simeq\gr(\shn\tens[\sha]\shm)\simeq0
\eneqn
 and the hypothesis that
$\shm/\hbar\shm$ is faithfully flat implies that $\gr\shn\simeq 0$. 
Since $\shn$ is coherent,  Corollary~\ref{cor:conservative1}
 implies that $\shn\simeq 0$.
\end{proof}

\begin{proposition}\label{prop:cut}
Assume \eqref{eq:FDringa} and \eqref{eq:FDringb}.
Let $U$ be an open subset of $X$ satisfying:
\eq
&&\text{$U\cap V\in \BB$ for any $V\in \BB$.}
\eneq
Then for any coherent $\sha$-module $\M$,
we have
\bnum
\item
$R^n\sect_U(\M)=0$ for any $n\not=0$,
\item
$\sect_U(\sha)\tens[\sha]\M\to\sect_U(\M)$ is an isomorphism,
\item
$\sect_U(\sha)$ is a flat $\sha^\rop$-module.
\enum
\end{proposition}
\begin{proof}
(i)\quad 
Since $R^n\sect_U(\M)$ is the sheaf associated with the presheaf 
$V\mapsto H^n(U\cap V;\M)$, (i) follows from
Theorem~\ref{th:formalfini1}~(v).

\vspace{0.4em}
\noindent
(ii)\quad The question being local, we may assume that
we have an exact sequence $0\to\shn\to\shl\to\M\to0$, where
$\shl$ is a free $\sha$-module of finite rank.
Then, we have a commutative diagram with exact rows by (i):
\eqn&&\xymatrix{
&\sect_U(\sha)\tens[\sha]\shn\ar[r]\ar[d]
&\sect_U(\sha)\tens[\sha]\shl\ar[r]\ar[d]^{\txt{\large$\wr$}}
&\sect_U(\sha)\tens[\sha]\M\ar[r]\ar[d]&0\\
0\ar[r]&\sect_U(\shn)\ar[r]
&\sect_U(\shl)\ar[r]&\sect_U(\M)\ar[r]&0.}
\eneqn
Since the middle vertical arrow is an isomorphism,
$\sect_U(\sha)\tens[\sha]\M\to \sect_U(\M)$ is an epimorphism.
Applying this to $\shn$,
$\sect_U(\sha)\tens[\sha]\shn\to \sect_U(\shn)$ is an epimorphism.
Hence, $\sect_U(\sha)\tens[\sha]\M\to \sect_U(\M)$ is an isomorphism.

\vspace{0.4em}
\noindent
(iii)\quad By (i) and (ii),
$\M\mapsto\sect_U(\sha)\tens[\sha]\M$ is an exact functor on the category
of coherent $\sha$-modules.
It follows that for all $x\in X$, the functor
$\M\mapsto(\sect_U(\sha))_x\tens[{\sha_x}]\M_x$
is exact on the category $\mdcoh[\sha]$. Therefore, $(\sect_U(\sha))_x$ is a flat
$\sha_x^\rop$-module.
\end{proof}

\begin{remark}
The results of this chapter can be generalized in the following situation.
Let $\sha$ be a sheaf of rings on a topological space $X$
and let $\shi$ be a both-sided sheaf of ideals of $\sha$.
We assume that:\\
 there exists locally a section $s$ of $\shi$ such that
$\sha\ni a\mapsto as$ and $\sha\ni a\mapsto sa$ give isomorphisms
$\sha\isoto\shi$.

We set $\sha_0=\sha/\shi$, $\sha(-n)=\shi^n\subset\sha$ 
and $\sha(n)=\rhom[\sha](\sha(-n),\sha)$ for $n\ge0$.

Then we have $\sha(n)\subset\sha(n+1)$,
and $\sha(n)\otimes_\sha\sha(m)\simeq\sha(n+m)$.

We set $\sha^\loc=\indlim[n]\sha(n)$ and for an $\sha$-module $\shm$, we set
$\shm(n)=\sha(n)\otimes_{\sha}\shm$.

We say that $\shm$ is $\shi$-torsion free if
$\shm(-1)\to\shm$ is a monomorphism. Of course, $\sha$ is 
$\shi$-torsion free.

Finally, for an $\sha$-module $\shm$ we set
$\widehat{\shm}\seteq\prolim[n]\coker(\shm(-n)\to\shm)$.

Instead of \eqref{eq:FDringa}, we assume 
\eq\label{eq:FDringaa}
&&\left\{\parbox{60ex}{\bnum
\item $\sha\isoto\widehat{\sha}$,
\item $\shao$ is a left Noetherian ring.
\enum
}\right.  
\eneq
Under the assumptions \eqref{eq:FDringaa} and \eqref{eq:FDringb},
all the results of this chapter hold
with  suitable modifications.

In particular, our theory can be applied when $X=T^*M$ is the
cotangent bundle to a complex manifold $M$ and 
$\sha=\widehat{\she_X}(0)$ is the ring of formal microdifferential
operators of order $0$ 
(see Section~\ref{section:DQcotg} for more details on the ring of
formal microdifferential operators).
\end{remark}

\chapter{$\DQ$-algebroids}\label{chapter:DQ}  

\section{Algebroids}\label{section:algebroid}

In this section, $X$ denotes a topological space and 
recall that $\cora$ is a commutative unital ring.
A $\cora$-linear category means a category $\shc$ such that
$\Hom[\shc](X,Y)$ is endowed with a $\cora$-module structure
for any $X$, $Y\in\shc$, and the
composition map $\Hom[\shc](X,Y)\times\Hom[\shc](Y,Z)\to\Hom[\shc](X,Z)$
is $\cora$-bilinear for any $X$, $Y$, $Z\in\shc$. One defines
similarly the notion of a $\cora$-linear stack.

The notion of an algebroid  has been introduced in \cite{Ko2}. We refer
to \cite{D-P} for a more systematic study and to \cite{K-S3} 
for an introduction to stacks.
Recall that a $\cora$-algebroid  $\sha$ on $X$ 
\glossary{algebroid}%
is a $\cora$-linear stack locally non empty and such that 
for any open subset $U$ of $X$, any two objects of $\sha(U)$ are locally 
isomorphic.

If $A$ is a $\cora$-algebra (an algebra, not a sheaf of algebras), 
we denote by $A^+$ the $\cora$-linear category with 
one object and having $A$ as the endomorphism ring
of this object. 

Let $\sha$ be a sheaf of $\cora$-algebras on $X$ and consider the prestack 
$U\mapsto \sha(U)^+$ ($U$ open in $X$). 
We denote by $\sha^+$ the associated stack.
Then $\sha^+$ is a $\cora$-algebroid and is 
called the $\cora$-algebroid  associated with $\sha$.  
The category $\sha^+(X)$ is equivalent to the full subcategory
of $\Mod(\sha^\op)$ consisting of objects locally isomorphic to $\sha^\op$.

Conversely, if $\sha$ is an algebroid on $X$ and
$\sigma\in\sha(X)$, then $\sha$ is equivalent to the algebroid
$\shend_\sha(\sigma)^+$.

For an algebroid $\sha$ and $\sigma$, $\tau\in\sha(U)$,
the $\cora$-algebras $\shend_{\sha}(\sigma)$ and $\shend_{\sha}(\tau)$
are locally isomorphic.
Hence, any definition of local nature concerning sheaves of
$\cora$-algebras, such as
being coherent or Noetherian, extends to $\cora$-algebroids.

Recall that for an algebroid $\sha$,  the algebroid $\sha^\rop$ is defined by  
 $\sha^\rop(U)=(\sha(U))^\rop$ ($U$ open in $X$). Then, if $\sha$ is a sheaf of $\cora$-algebras,
$(\sha^\rop)^+\simeq(\sha^+)^\rop$.

\begin{convention}\label{conv1}
If $\sha$ is a sheaf of algebras and if there is no risk of confusion, 
we shall keep the same notation $\sha$ to denote the associated algebroid.
\end{convention}
Note that two algebras may not be isomorphic 
even if the associated algebroids are equivalent.
\begin{example}
Let $X$ be a complex manifold, $\shl$ a line bundle on $X$ and denote as usual 
by $\shd_X$ the ring of differential operators on $X$. The ring of $\shl$-twisted  
differential operators is given by
\eqn
&&\shd_X^\shl\eqdot \shl\tens[\sho_X]\shd_X\tens[\sho_X]\shl^{\tens-1}.
\eneqn
In general the two algebras $\shd_X$ and $\shd_X^\shl$ are not
isomorphic although the associated algebroids are
equivalent. The equivalence is obtained by using the bi-invertible 
module $\shd_X\tens[\sho_X]\shl^{\tens-1}$
(see Definition~\ref{def:biinv1} and Lemma~\ref{lem:algiso}
below). 
\end{example}

Let $\shu=\{U_i\}_{i\in I}$ be an open covering of $X$.
In the sequel we set $U_{ij}\seteq U_i\cap U_j$, 
$U_{ijk}\seteq U_{i}\cap U_j\cap U_k$, etc.

Consider the data of
\eq\label{eq:algdatum0}
&& \left\{\parbox{300 pt}{
a $\cora$-algebroid $\sha$ on $X$,\\
$\sigma_i\in \sha(U_i)$ and 
isomorphisms $\phi_{ij}\cl \sigma_j\vert_{U_{ij}}\isoto \sigma_i\vert_{U_{ij}}$.
}\right. 
\eneq
To these data, we associate:
\begin{itemize}
\item
$\sha_i=\shend_\sha(\sigma_i)$,
\item
$f_{ij}\cl \sha_j\vert_{U_{ij}}\isoto \sha_i\vert_{U_{ij}}$,
the $\cora$-algebra isomorphism 
$a\mapsto \phi_{ij}\circ a\circ\opb{\phi_{ij}}$,
\item
$a_{ijk}$, the invertible element of $\sha_i(U_{ijk})$
given by $\phi_{ij}\circ\phi_{jk}\circ\opb{\phi_{ik}}$.
\end{itemize}
Then:
\eq\label{eq:algebroid}
&& \left\{ \parbox{300 pt}{
$f_{ij}\circ f_{jk} = \Ad(a_{ijk})\circ f_{ik}$ on $U_{ijk}$, \\
$a_{ijk} a_{ikl} = f_{ij}(a_{jkl}) a_{ijl}$ on $U_{ijkl}$.
}\right. 
\eneq
(Recall that $\Ad(a)(b)=ab\opb{a}$.)

Conversely, let 
$\sha_i$ be $\cora$-algebras  on $U_i$ ($i\in I$), let
$f_{ij}\cl \sha_j|_{U_{ij}} \isoto\sha_i\vert_{U_{ij}}$ ($i,j\in I$)
be $\cora$-algebra isomorphisms,
and let $a_{ijk}$ ($i,j,k\in I$) be invertible sections of $\sha_i(U_{ijk})$ satisfying
\eqref{eq:algebroid}.
One calls 
\eq\label{eq:descentdata}
&&( \{\sha_i\}_{i\in I},\{f_{ij}\}_{i,j\in I}, \{a_{ijk}\}_{i,j,k \in I})
\eneq
a  gluing datum for $\cora$-algebroids on $\shu$.
The following result,
which easily follows from \cite[Lem~3.8.1]{Go},
is stated (in a different form) in \cite{Ka1} and 
goes back to \cite{Gi}.

\begin{proposition}\label{prop:descdata}
Assume that $X$ is paracompact. 
Consider a gluing datum \eqref{eq:descentdata} on $\shu$. Then there exist 
an algebroid  $\sha$ on $X$ and $\{\sigma_i,\phi_{ij}\}_{i,j\in I}$ as in \eqref {eq:algdatum0}
to which this gluing datum is associated.
Moreover, the data $(\sha,\sigma_i,\phi_{ij})$ are unique
up to an equivalence of stacks, this equivalence being unique up to
a unique isomorphism.
\end{proposition}

We will give another construction in Proposition~\ref{prop:descdatab},
which may be applied to non paracompact spaces such as algebraic varieties.

\bigskip
For an algebroid  $\sha$, one defines the 
$\cora$-linear abelian category $\md[\sha]$, whose objects are called
$\sha$-modules, by setting
\glossary{modules!over an algebroid}%
\eq
&&\md[\sha]\eqdot\Fct_{\cora}(\sha,\stkMod(\cora_X)).
\eneq
Here $\stkMod(\cora_X)$ is the $\cora$-linear stack of sheaves of $\cora$-modules
on $X$ and, for two $\cora$-linear stacks $\sha_1$ and $\sha_2$,
$\Fct_{\cora}(\sha_1,\sha_2)$ is the category of $\cora$-linear functors
of stacks from $\sha_1$ to $\sha_2$.
If $\sha$ is the algebroid associated with a $\cora$-algebra $A$ on $X$,
then $\md[\sha]$ is equivalent to $\md[A]$.
The category $\md[{\A[]}]$ is a Grothendieck category and
we denote by $\RD(\A[])$ its derived category 
and by $\Derb(\A[])$ its bounded derived category. 

For a $\cora$-algebroid $\sha$, the $\cora$-linear prestack
$U\mapsto\md[\sha\vert_U]$ is a stack and we denote it by
$\stkMod(\sha)$.

In the sequel, we shall write for short  ``$\sigma\in\sha$''
instead of ``$\sigma\in\sha(U)$ for some open set $U$''.

\begin{definition}\label{def:invertible}
An $\sha$-module $\shl$ is invertible 
\glossary{invertible}\glossary{module!invertible}
if it is locally isomorphic to
$\sha$, namely for any $\sigma\in\sha$, the $\shend_\sha(\sigma)$-module
$\shl(\sigma)$ is locally isomorphic to $\shend_\sha(\sigma)$.
\end{definition}
This terminology is motivated by the fact that 
for  an invertible module  $\shl$, if we set
$\shb\eqdot(\shend_{\sha}(\shl))^\rop$, then
$\hom[\sha](\shl,\sha)\tens[\sha]\shl\simeq\shb$
and $\shl\tens[\shb]\hom[\sha](\shl,\sha)\simeq \sha$.

We denote by $\Inv(\sha)$ the full subcategory of
$\md[\sha]$ consisting of invertible $\sha$-modules and by 
$\stkInv(\sha)$ the corresponding full substack of $\stkMod(\sha)$.
Then we have equivalences of $\cora$-linear stacks
$\sha\isoto \stkInv(\sha^\rop)\isoto \stkInv(\sha)^\rop$.

Recall that for two $\cora$-linear categories $\shc$ and $\shc'$, one
defines their tensor product $\shc\tens[\cora]\shc'$ by setting
$\Ob(\shc\tens[\cora]\shc')=\Ob(\shc)\times\Ob(\shc')$ and 
\eqn
&&\Hom[{\shc\tens[\cora]\shc'}]((M,M'), (N,N'))=
\Hom[\shc](M,N)\tens[\cora]\Hom[\shc'](M',N')
\eneqn
for $M,N\in\shc$ and $N,N'\in\shc'$.
Then $\shc\tens[\cora]\shc'$ is a $\cora$-linear category.

For a pair of $\cora$-algebroids $\sha$ and $\sha'$, the $\cora$-algebroid 
$\sha\tens[\cora]\sha'$ is the $\cora$-linear stack associated with
the prestack 
$U\mapsto\sha(U)\tens[\cora]\sha'(U)$ ($U$ open in $X$).
We have
$$\md[{\sha\tens[\cora]\sha'}]\simeq\Fct_{\cora}(\sha,\stkMod(\sha')).$$
For a $\cora$-algebroid $\sha$,
$\md[{\sha\tens[\cora]\sha^\rop}]$ has a canonical object
given by
\eqn
&&
\sha\tens[\cora]\sha^\rop\ni(\sigma,\sigma'^\rop)
\mapsto \hom[\sha](\sigma',\sigma)\in
\stkMod(\cora_X). 
\eneqn
We denote this object by the same letter $\sha$.
If $\sha$ is associated with a $\cora$-algebra $A$,
this object corresponds to $A$, 
regarded as an $(A\tens[\cora]A^\rop)$-module.

For $\cora$-algebroids $\sha_i$ ($i=1,2,3$),
we have the tensor product functor
\eq\label{eq:tensalgebr}
&&\hspace{0.6em}\scbul\tens[\sha_2]\scbul\cl
\md[{\sha_1\tens[\cora]\sha_2^\rop}]\times
\md[{\sha_2\tens[\cora]\sha_3^\rop}]\\
&&\hspace{6cm}\to\md[{\sha_1\tens[\cora]\sha_3^\rop}],
\nonumber\eneq
and the $\hom$ functor
\eq\label{eq:homalgebr}
&&\hom[\sha_1](\scbul,\scbul)\cl
\md[{\sha_1\tens[\cora]\sha_2^\rop}]^\rop\times
\md[{\sha_1\tens[\cora]\sha_3^\rop}]\\
&&\hspace{8cm}\to\md[{\sha_2\tens[\cora]\sha_3^\rop}].
\nonumber\eneq
In particular, we have 
\eqn&&
\ba{lcccc}
\scbul\tens[\sha]\scbul
&\cl&\md[{\sha^\rop}]\times
\md[{\sha}]&\To&\md[{\cora_X}],\\[1ex]
\hom[\sha](\scbul,\scbul)
&\cl&\md[{\sha}]^\rop\times
\md[{\sha}]&\To&\md[{\cora_X}],\\[1ex]
\hom[\sha](\scbul,\sha)
&\cl&\md[{\sha}]^\rop&\To&\md[{\sha^\rop}].
\ea
\eneqn
Since $\md[{\A[]}]$ is a Grothendieck category, any left exact functor
from $\md[{\A[]}]$ to an abelian category admits a right
derived functor.

Now consider the tensor product in \eqref{eq:tensalgebr}. It admits a
left derived functor as soon as $\sha_3$ is $\cora$-flat. 
Indeed, any $\shm\in\md[\sha_2\tens(\sha_3)^\rop]$ is a quotient of an
$\sha_2$-flat module since there is an exact sequence
\eqn
&&\bigoplus_{s\in\Hom(\shl,\shm\vert_U)}\shl\to\shm\to0,
\eneqn
where $U$ ranges over the family of open subsets of $X$ and 
$\shl\in(\sha_2\tens(\sha_3)^\rop)^\rop(U)$. (Recall that for a
$\cora$-algebroid $\sha$, $\sha^\rop(U)$ is equivalent to 
$\stkInv(\sha)(U)$.) Note that $\shl$ is $\sha_2$-flat since 
$(\sha_3)^\rop$ is $\cora$-flat.

The following lemma is obvious.
\begin{lemma}
Let $\sha$ and $\sha'$ be $\cora$-algebroids. 
To give a functor of algebroids $\phi\cl\sha'\to\sha$ is equivalent to
giving an $(\sha'\tens\sha^\rop)$-module $\shl$ which is locally
isomorphic to $\sha$ \ro i.e.\ for $\sigma\in\sha$ and $\sigma'\in\sha'$,
$\shl(\sigma'\tens\sigma^\rop)$ is locally isomorphic to
$\shend_\sha(\sigma)$ as an $\shend_\sha(\sigma)^\rop$-module\rf.
\end{lemma}
The $\sha'\tens\sha^\rop$-module $\shl$  corresponding
to $\phi$ is the module induced from
the $\sha\tens\sha^\rop$-module $\sha$ by
$\phi\tens\sha^\rop\cl\sha'\tens\sha^\rop\to \sha\tens\sha^\rop$.

The forgetful functor 
$$\Mod(\sha)\to\Mod(\sha')$$
is isomorphic to $\M\mapsto\shl\tens[\sha]\M$.

Let $f\cl X\to Y$ be a continuous map and let 
$\sha$ be a $\cora$-algebroid on $Y$.
We denote by $\opb{f}\sha$ the $\cora$-linear 
stack associated with the prestack $\sts$
given by:
\eqn
&&\sts(U)=\left\{(\sigma,V)\,;\,\text{$V$ is an open subset of $Y$ such that
$f(U)\subset V$}\right.\\
&&\left.\hs{17ex}\text{ and $\sigma\in\sha(V)$}\right\}
\quad\text{for any open subset $U$ of $X$,}\\
&&\Hom[\sts(U)]\bl(\sigma,V),(\sigma',V')\br)
=\sect(U;f^{-1}\hom[\sha](\sigma,\sigma')).
\eneqn
Then $\opb{f}\sha$ is a $\cora$-algebroid.
We have functors
\eqn
\oim{f},\,\eim{f}\cl\md[\opb{f}\sha]\To\md[\sha],\\
\opb{f}\cl\md[\sha]\To\md[\opb{f}\sha].\\
\eneqn

For two topological spaces $X_1$ and $X_2$, let $p_i\cl X_1\times X_2\to X_i$
be the projection.
Let $\sha_i$ be a $\cora$-algebroid on $X_i$ ($i=1,2$).
We define a $\cora$-algebroid on $X_1\times X_2$, 
 called the external tensor product of $\sha_1$ and $\sha_2$,
 by setting: 
\eqn
&&\sha_1\etens\sha_2\eqdot p_1^{-1}\sha_1\tens p_2^{-1}\sha_2.
\eneqn
We have a canonical bi-functor
$$\scbul\etens\scbul\cl\md[\sha_1]\times\md[\sha_2]\To\md[\sha_1\etens\sha_2].$$

\subsubsection*{Bi-invertible modules}

The following notion of bi-invertible modules 
will appear all along these Notes
since it describes equivalences of algebroids.
\begin{definition}
Let $A$ and $A'$ be two sheaves of $\cora$-algebras.
An $A\otimes A'$-module $L$ is called {\em bi-invertible}
\glossary{bi-invertible}\glossary{module!bi-invertible}
if there exists locally a section $w$ of $L$ such that
$A\ni a\mapsto (a\otimes 1)w\in L$ and $A'\ni a' \mapsto (1\otimes a')w\in L$
give isomorphisms of $A$-modules and $A'$-modules, respectively.
\end{definition}
\begin{lemma}\label{lem:biinvert1}
Let $L$ be a bi-invertible $A\otimes A'$-module
and let $u$ be a section of $L$.
If $A\ni a\mapsto (a\otimes 1)u\in L$ is an isomorphism of 
$A$-modules, then
$A'\ni a'\mapsto (1\otimes a')u\in L$ is also an isomorphism of
$A'$-modules.
\end{lemma}
\begin{proof}
Let $w$ be as above.
There exist $a\in A$ and $b\in A$ such that 
$u=(a\otimes 1)w$ and $w=(b\otimes 1)u$. Then 
we have $u=(ab\otimes 1)u$ and hence $ab=1$. Similarly
$w=(ba\otimes 1)w$ implies $ba=1$.
Hence
we have a commutative diagram
$$\xymatrix@C=10ex{
A'\ar[r]_w^{\sim}\ar[rd]_u&L\ar[d]_{\bwr}^{\;a\otimes1}\\
&L}
$$
and we obtain the desired result.
\end{proof}

\begin{remark}
Let $A$ and $B$ be two $\cora$-algebras and let
$L$ be an $(A\tens B^\rop)$-module.
Even if $L$ is isomorphic to $A$ as an $A$-module and 
isomorphic to $B^\rop$ as a
$B^\rop$-module, $L$ is not necessarily bi-invertible,
as shown by the following example.

Let $I$ be an infinite set and take $\mathrm{o}\in I$.
Set $I^*=I\setminus\{\mathrm{o}\}$.
Then there exists a bijection $v\cl I^*\to I$.
Set 
\eqn
&&X=\{a\in \Hom[\Set](I,I);a(\mathrm{o})=\mathrm{o}\},\\
&&Y=\{b\in\Hom[{\Set}](I,I); b(\mathrm{o})=\mathrm{o} \mbox{ and }b(I^*)\subset I^*\}.
\eneqn
Set $Z=X$. Then $X$ and $Y$ are semi-groups and
$X$ acts on $Z$ from the left and $Y$ acts on $Z$ from the right.
Let $v'\in Z$ be the unique element extending $v$.
Then $\id_I\in Z$ gives an isomorphism $X\isoto Z$ 
($X\ni a\mapsto a\in Z$) and $v'\in Z$
induces an isomorphism $Y\isoto Z$ ($Y\ni b\mapsto v'\circ b\in Z$).
Let $A=\cora[X]$ and $B=\cora[Y]$ be the semigroup algebras corresponding to $X$ and $Y$.
Set $L=\cora[Z]$. Then $L$ is an $(A\tens B^\rop)$-module and
$L$ is isomorphic to $A$ as an $A$-module and 
isomorphic to $B^\rop$ as an
$B^\rop$-module.
Let $u$ be the element of $L$ corresponding to $\id_I$. Then $u$ gives an
isomorphism $A\ni a\mapsto (a\otimes 1)u\in L$. 
Since the image of $B^\rop\ni b\mapsto (1\otimes b)u\in L$ is $\cora[Y]\neq L$, 
$L$ is not bi-invertible in view of Lemma~\ref{lem:biinvert1}.
\end{remark}

However the following partial result holds.
\begin{lemma}
Let $A$ and $A'$ be $\cora$-algebras and
let $L$ be an $A\otimes A'$-module.
Assume that $L$ is isomorphic to $A$ as an $A$-module and 
isomorphic to $A'$ as an $A'$-module. 
If we assume moreover that $A_x$ is a left noetherian ring for any
$x\in X$, then $L$ is 
bi-invertible.
\end{lemma}
\begin{proof}
Assume that $A\ni a\mapsto (a\otimes 1)u\in L$
and $A'\ni a'\mapsto (1\tens a')v\in L$ are isomorphisms for some
$u,v\in L$. Set $v=(a\otimes 1)u$ and $u=(1\otimes a')v$.
There exists $a''\in A$ such that $(1\otimes a')u=(a''\otimes1)u$.
Then we have $u=(1\otimes a')v=(1\otimes a')(a\otimes 1)u
=(a\otimes 1)(1\otimes a')u=(aa''\otimes 1)u$.
Hence we obtain $aa''=1$.
Therefore the $A$-linear endomorphism $f\cl A\ni z\mapsto za''$ is an
epimorphism ($f(za)=z$).
Since $A_x$ is a left noetherian ring, $f$ is an isomorphism.
Hence, $a''$, as well as $a$, is an invertible element.
Then the following commutative diagram implies the desired result:
\eqn
&&\xymatrix@C=10ex{
A'\ar[r]_v^{\sim}\ar[rd]_u&L\ar[d]_{\bwr}^{\;a\otimes1}\\
&L.}
\eneqn
\end{proof}

\begin{definition}\label{def:biinv1}
For two $\cora$-algebroids $\sha$ and $\sha'$, we say that
an $(\sha\otimes \sha')$-module $\shl$ is  bi-invertible 
\glossary{bi-invertible}\glossary{module!bi-invertible}
if for any $\sigma\in\sha$ and $\sigma'\in\sha'$,
$\shl(\sigma\otimes \sigma')$ is a bi-invertible
$\shend_{\sha}(\sigma)\otimes \shend_{\sha'}(\sigma')$-module.
\end{definition}

\begin{lemma}\label{lem:algiso}
To give an equivalence $\sha'\isoto\sha$ is equivalent to
giving a bi-invertible $(\sha'\otimes\sha^\rop)$-module.
More precisely, 
the forgetful functor
$\stkMod(\sha)\to\stkMod(\sha')$ is given by
$\shm\mapsto\shl\tens[\sha]\shm$
for a bi-invertible $(\sha'\otimes\sha^\rop)$-module
$\shl$.
\end{lemma}

Let $\shm\in\Mod(\A[])$. We shall
denote by $\shend_{\cora}(\shm)$ the stack associated with the
prestack $\sts$ whose objects are those of $\sha$ and 
$\hom[\sts](\sigma,\sigma')=\hom[\cora](\shm(\sigma),\shm(\sigma'))$
 for $\sigma$, $\sigma'\in\A[](U)$.
Then $\shend_{\cora}(\shm)$ is a $\cora$-algebroid and there exists
a natural functor of $\cora$-algebroids $\sha\to\shend_{\cora}(\shm)$.
Note that $\shm$ may be regarded as an $\shend_{\cora}(\shm)$-module.

In particular, $\shend_{\cora}(\sha)$ is a $\cora$-algebroid, 
there is 
a functor of $\cora$-algebroids $\sha\otimes\sha^\op\to\shend_{\cora}(\sha)$,
and $\sha$ may be regarded as an $\shend_{\cora}(\sha)$-module.

\begin{lemma}\label{lem:endstack}
Let $\sha$ and $\sha'$ be $\cora$-algebroids and let $\M\in\Mod(\sha)$,
$\M'\in\Mod(\sha')$.
Assume that $\M$ and $\M'$ are locally isomorphic as $\cora$-modules,
that is, for any $\sigma\in\sha$ and $\sigma'\in\sha'$,
$\M(\sigma)$ and $\M'(\sigma')$ are locally isomorphic as $\cora_X$-modules.
Then $\shend_{\cora}(\M)$ and $\shend_{\cora}(\M')$ are equivalent
as $\cora$-algebroids.
\end{lemma}
\begin{proof}
For $\sigma\in \sha$ and $\sigma'\in\sha'$,
set 
$\shl(\sigma'\tens\sigma^\rop)
=\hom[\cora](\M(\sigma),\M'(\sigma'))$. Then 
$\shl$ is an $(\shend_{\cora}(\M')\tens\shend_{\cora}(\M)^\rop)$-module.
By the assumption, $\shl$ is a bi-invertible
$\bigl(\shend_{\cora}(\M')\otimes \shend_{\cora}(\M)^\rop\bigr)$-module.
Hence we obtain the desired result.
\end{proof}

Since Proposition~\ref{prop:descdata} does not apply to algebraic varieties,
we need an alternative local description of algebroids.

Let $\shu=\{U_i\}_{i\in I}$ be an open covering of $X$.
Consider the data of
\eq\label{eq:algdatum0b}
&& \left\{\parbox{300 pt}{
a $\cora$-algebroid $\sha$ on $X$,\\
$\sigma_i\in \sha(U_i)$.
}\right. 
\eneq
To these data, we associate 
\begin{itemize}
\item
$\sha_i\eqdot\shend_\sha(\sigma_i)$,
\item
$\shl_{ij}\eqdot\hom[\sha_i\vert_{U_{ij}}](\sigma_j\vert_{U_{ij}},\sigma_i\vert_{U_{ij}})$,
(hence $\shl_{ij}$ is a bi-invertible $\sha_i\tens\sha_j^\rop$-module on $U_{ij}$),
\item
the natural isomorphisms
\eq\label{eq:aijkl}
&&a_{ijk}\cl \shl_{ij}\tens[\sha_j]\shl_{jk}\isoto \shl_{ik}
\quad\text{in $\md[{\sha_i\tens\sha_k^\rop\vert_{U_{ijk}}}]$.}
\eneq
\end{itemize}
Then the diagram below in $\md[{\sha_i\tens\sha_l^\rop\vert_{U_{ijkl}}}]$
commutes:
\eq\label{eq:algebroidb}
&& \xymatrix{
\shl_{ij}\tens \shl_{jk}\tens \shl_{kl}\ar[r]^-{a_{ijk}}\ar[d]^-{a_{jkl}}
                         &\shl_{ik}\tens \shl_{kl}\ar[d]^-{a_{ikl}} \\
\shl_{ij}\tens \shl_{jl}\ar[r]^-{a_{ijl}}&\shl_{il}\,.
}
\eneq
 Conversely, let $\sha_i$ be sheaves of $\cora$-algebras  on $U_i$ ($i\in I$), 
let $\shl_{ij}$ be a bi-invertible $\sha_i\tens\sha_j^\rop$-module on $U_{ij}$,
and let $a_{ijk}$ be isomorphisms as in \eqref{eq:aijkl} such that the
diagram \eqref{eq:algebroidb} commutes. 
One calls 
\eq\label{eq:descentdatab}
&&( \{\sha_i\}_{i\in I},\{\shl_{ij}\}_{i,j\in I}, \{a_{ijk}\}_{i,j,k \in I})
\eneq
an algebraic gluing datum for $\cora$-algebroids on $\shu$.
\begin{proposition}\label{prop:descdatab}
Consider an algebraic gluing datum \eqref{eq:descentdatab} on $\shu$. Then there exist 
an algebroid  $\sha$ on $X$ and $\{\sigma_i,\phi_{ij}\}_{i,j\in I}$ as in \eqref {eq:algdatum0}
to which this gluing datum is associated.
Moreover, the data $(\sha,\sigma_i,\phi_{ij})$ are unique
up to an equivalence of stacks, this equivalence being unique up to
a unique isomorphism.
\end{proposition}
\begin{proof}[Sketch of proof]
We define a category $\md[{\A[X]}]$ as follows. 
An object $\shm\in\md[{\A[X]}]$ is defined 
as a family $\{\shm_i,q_{ij}\}_{i,j\in I}$ with $\shm_i\in\md[\sha_i]$
and the $q_{ij}$'s are isomorphisms
\eqn
&&q_{ij}\cl \shl_{ij}\tens[\sha_j]\shm_j\isoto\shm_i
\eneqn
making the diagram below commutative:
\eqn
&&\xymatrix{
\shl_{ij}\tens \shl_{jk}\tens \shm_{k}\ar[r]^-{q_{jk}}\ar[d]^-{a_{ijk}}
                         &\shl_{ij}\tens \shm_{j}\ar[d]^-{q_{ij}} \\
\shl_{ik}\tens \shm_{k}\ar[r]^-{q_{ik}}&\shm_{i}.
}
\eneqn
A morphism $\{\shm_i,q_{ji}\}_{i,j\in I}\to\{\shm'_i,q'_{ji}\}_{i,j\in I}$ 
in $\md[{\A}]$ is a family of morphisms $u_i\cl\shm_i\to\shm'_i$ satisfying the
natural compatibility conditions. Replacing $X$ with $U$ open in $X$,
we define a prestack $U\mapsto\md[{\A[U]}]$ and one easily checks that
this prestack is a stack and moreover that $\md[{\A[U_i]}]$ is
equivalent to $\md[\sha_i]$. We denote it by $\stkMod(\sha)$.
Then we define the algebroid $\A$ as the substack of
$(\stkMod(\sha))^\rop$ consisting of objects locally isomorphic to
$\sha_i$ on $U_i$.
\end{proof}

\subsubsection*{Invertible algebroids}
In this subsection,  $(X,\shr)$ denotes a topological space endowed with a sheaf of commutative $\cora$-algebras.
Recall (see \cite[Chap.19~\S~5]{K-S3}) that an $\shr$-linear stack 
$\sts$ is a $\cora$-linear stack $\sts$ together with a morphism of
$\cora$-algebras $\shr\to\shend(\id_\sts)$. 
Here,
$\shend(\id_{\sts})$ is the sheaf of endomorphisms of the identity functor
$\id_{\sts}$ from $\sts$ to itself. 
\begin{definition}\label{def:stralg}
\bnum
\item
An $\shr$-algebroid $\shp$ 
\glossary{Ralgebroid@$\shr$-algebroid}\glossary{algebroid!$\shr$-}%
 is a $\cora$-algebroid $\shp$ on $X$
endowed with a morphism of $\cora$-algebras $\shr\to\shend(\id_\shp)$.
\item
An $\shr$-algebroid $\shp$ on $X$ is called an invertible
$\shr$-algebroid
\glossary{invertible!$\shr$-algebroid}%
\glossary{algebroid!invertible $\shr$-}%
\glossary{Ralgebroid@$\shr$-algebroid!invertible}%
if $\shr_U\to\shend_{\shp}(\sigma)$ is an isomorphism
for any open subset $U$ of $X$ and any $\sigma\in\shp(U)$.
\enum
\end{definition}

We shall state some properties of invertible $\shr$-algebroids.
Since the proofs are more or less obvious, 
we omit them.

For two  $\shr$-algebroids $\shp_1$ and $\shp_2$,
the  $\shr$-algebroid $\shp_1\otimes_{\shr}\shp_2$  is defined as
the $\shr$-linear stack associated with the prestack $\sts$ 
given by
\eqn
&&\sts(U)=\shp_1(U)\times\shp_2(U),\\
&&\hom[\sts]\bl(\sigma_1,\sigma_2),(\sigma_1',\sigma_2')\br)
=\hom[\shp_1](\sigma_1,\sigma'_1)\tens[\shr]\hom[\shp_2](\sigma_2,\sigma'_2).
\eneqn
If  $\shp_1$ and $\shp_2$ are 
invertible, then so is $\shp_1\otimes_{\shr}\shp_2$.

We have a functor of $\cora$-linear stacks
$\shp_1\otimes_{\cora_X}\shp_2\to\shp_1\otimes_{\shr}\shp_2$.

Note that
\eq\label{eq:Ostacks1}
&&\parbox{60ex}{
If $\shp_1$ and $\shp_2$ are two invertible $\shr$-algebroids and 
$F\cl \shp_1\to\shp_2$ is a functor of $\shr$-linear stacks, then $F$ is an
equivalence.
}\eneq
\eq\label{eq:Ostacks2}
&&\parbox{60ex}{
For any invertible $\shr$-algebroid $\shp$,
$\shp\otimes_{\shr}\shp^\op$ is equivalent to $\shr$ as an $\shr$-algebroid.
}\eneq
\eq\label{eq:Ostacks3}
&&\parbox{60ex}{The set of equivalence classes of invertible $\shr$-algebroids
has a structure of an additive group 
by the operation $\scbul\otimes_{\shr}\scbul$ defined above,
and this group is isomorphic to $H^2(X;\shr^\times)$ (see \cite{Br,K-S3}).
Here $\shr^\times$ denotes the abelian sheaf of invertible sections of $\shr$.
}\eneq
\eq\label{eq:Ostacks4}
&&\parbox{60ex}{
For two invertible $\shr$-algebroids $\shp_1$ and $\shp_2$,
there is a natural functor
$$
\scbul\tens[\shr]\scbul\cl\Mod(\shp_1)\times\Mod(\shp_2)\to
\Mod(\shp_1\otimes_{\shr}\shp_2),
$$
and its derived version.
}\eneq

\subsubsection*{Invertible $\OS$-algebroids}

In this subsection,  $(X,\sho_X)$ denotes a complex  manifold. As a particular case of 
Definition~\ref{def:stralg}, taking $\cora=\C$ and $\shr=\sho_X$, we get the notions of an $\OS$-algebroid
\glossary{Oalgebroid@$\OS$-algebroid}\glossary{algebroid!$\OS$-}%
as well as that of an invertible $\OS$-algebroid.
\glossary{invertible!$\OS$-algebroid}%
\glossary{algebroid!invertible $\OS$-}%
\glossary{Oalgebroid@$\OS$-algebroid!invertible}%

\begin{lemma}\label{le:endosho}
 Any $\C$-algebra endomorphism of $\sho_X$ is equal to the identity.
\end{lemma}
Although this result is elementary and well-known, we give a proof.
\begin{proof}
Let $\phi$ be a $\C$-algebra endomorphism of $\sho_X$. For $x\in X$, denote by 
$\phi_x$ the $\C$-algebra endomorphism of $\sho_{X,x}$ induced by $\phi$ and
by $\mathfrak{m}_x$ the unique maximal ideal of 
the ring $\sho_{X,x}$. Then $\phi_x$ sends
$\mathfrak{m}_x$ to $\mathfrak{m}_x$, $\phi_x$ induces an $\C$-algebra
homomorphism 
$u_x\cl\sho_{X,x}/\mathfrak{m}_x
\to \sho_{X,x}/\mathfrak{m}_x$. 
Since the composition 
$\C\isoto\sho_{X,x}/\mathfrak{m}_x\To[u_x]
\sho_{X,x}/\mathfrak{m}_x\isoto\C$ is the identity, we obtain that $u_x$
is the identity. 
Hence,
for any $f\in\sho_X$, $\phi(f)(x)=f(x)$. Therefore $\phi(f)=f$.
\end{proof}

\begin{lemma}
Let $\shp$ be a $\C$-algebroid on a complex manifold $X$.
Assume that, for any $\sigma\in\shp$,
$\shend_{\shp}(\sigma)$ is locally isomorphic to $\OO$ as a
$\C$-algebra.
Then $\shp$ is uniquely endowed with a structure of $\OO$-algebroid,
and $\shp$ is invertible.
\end{lemma}
\begin{proof}
By Lemma~\ref{le:endosho},
for an open subset $U$ and $\sigma\in\shp(U)$,
there exists
a unique $\C$-algebra isomorphism
$\OO\vert_U\isoto\shend_\shp(\sigma)$.
It gives a structure of
$\OO$-algebroid on $\shp$.
The remaining statements are obvious.
\end{proof}

Let $\shp$ be an invertible $\OS$-algebroid. For $\sigma$, $\sigma'\in\shp(U)$, the
two $\sho_X$-module  structures 
on $\hom[\shp](\sigma,\sigma')$ induced by
$\shend_{\shp}(\sigma)\simeq\sho_X$ and by $\shend_{\shp}(\sigma')\simeq\sho_X$
coincide, and $\hom[\shp](\sigma,\sigma')$ is an invertible $\sho_X$-module.

Let $f\cl X\to Y$ be a morphism of complex manifolds.
For an invertible $\OS[Y]$-algebroid $\shp_Y$, we set 
\eqn
&&\spb{f}\shp_Y\eqdot \sho_X\otimes_{\opb{f}\sho_Y}\opb{f}\shp_Y,
\eneqn
where the tensor product $\otimes_{\opb{f}\sho_Y}$ is defined
similarly as for $\cora$-algebroids. 
Then $\spb{f}\shp_Y$ is an invertible $\OS[X]$-algebroid.
We have functors
\eq
&&\quad\spb{f}\;\cl\Mod(\shp_Y)\to \Mod(\spb{f}\shp_Y),
\quad\lspb{f}\;\cl\Derb(\shp_Y)\to \Derb(\spb{f}\shp_Y),
\eneq
and
\eq&&
\ba{rl}
\eim{f},\;\oim{f}&\cl\Mod(\spb{f}\shp_Y)\to \Mod(\shp_Y),\\[1ex]
\quad\reim{f},\,\roim{f}&\cl\Derb(\spb{f}\shp_Y)\to \Derb(\shp_Y).
\phantom{aaaaaaaaaaaaaa}
\ea
\eneq
Let $f\cl X\to Y$ be a morphism of complex manifolds,
and let $\shp_X$ \ro resp.\ $\shp_Y$\rf\ be an invertible $\OS$-algebroid
\ro resp.\ an invertible $\OS[Y]$-algebroid\rf.
If $\opb{f}\shp_Y\to \shp_X$ is a functor of $\C$-linear stacks, then
it defines a functor of $\C$-linear stacks $\spb{f}\shp_Y\to \shp_X$ and this
last functor is an equivalence by the preceding results.
\begin{remark}\label{rem:oesterle}
Invertible $\OS$-algebroids are trivial in the algebraic case.
Indeed, for a smooth algebraic variety $X$, 
the group $H^2(X;\sho^\times_X)$ is zero. 
Here the cohomology is calculated with respect to the Zariski topology.
(With the {\'e}tale topology, it does not vanish in general.)
This result and its proof below have been communicated 
to us by Prof.\ Joseph Oesterl{\'e}, 
and we thank him here.

Let $K$ be the field of rational functions on $X$, $K^\times_X$, 
the constant sheaf with the abelian group $K^\times$ as stalks,  and denote by 
$X_1$ the set of closed irreducible hypersurfaces of $X$.
One has an exact sequence
\eqn
&&0\to \sho_X^\times\to K^\times_X\to \bigoplus_{S\in X_1}\Z_S\to 0.
\eneqn
 Since $K^\times_X$ is constant, it is a flabby sheaf for the 
Zariski topology. On the other hand the sheaf $\bigoplus_{S\in X_1}\Z_S$ is also flabby. 
It follows that $H^j(X;\sho^\times_X)$ is zero for $j>1$.
\end{remark}

\section{$\DQ$-algebras}\label{section:DQalg}
{}From now on, $X$ will be  a complex  manifold. 
We denote by $\delta_X\cl X\hookrightarrow X\times X$
the diagonal embedding and we set $\Delta_X=\delta_X(X)$.
We denote by $\OO$ the structure sheaf on $X$, 
by $d_X$ 
\index{d@$d_X$\ (dimension)}%
the complex dimension, by $\Omega_X$ the sheaf of holomorphic
forms of maximal degree and by 
$\Theta_X$ the sheaf of holomorphic vector fields.
As usual, we denote by  $\shd_X$ the sheaf of rings of (finite order) differential
operators on $X$ and by $F_n(\shd_X)$ the sheaf of differential operators of order
$\le n$. Recall that a bi-differential operator $P$ on $X$ 
\glossary{bi-differential operator}%
is a $\C$-bilinear morphism
$\sho_X\times\sho_X\to\sho_X$ which is obtained as 
the composition $\opb{\delta_X}\circ\tw P$ where $\tw P$ is a 
differential operator on $X\times X$ defined on a neighborhood of the
diagonal and $\opb{\delta}$ is the restriction to the diagonal:
\eq\label{eq:bidiff1}
&&P(f,g)(x)=
(\tw P(x_1,x_2;\partial_{x_1},\partial_{x_2})(f(x_1)g(x_2))\vert_{x_1=x_2=x}.
\eneq
Hence the sheaf of bi-differential operators is isomorphic to
$\shd_X\otimes_{\sho_X}\shd_X$, 
where both $\shd_X$ are regarded as $\OO$-modules
by the left multiplications. 

\subsubsection*{Star-products}
\begin{notation}
We denote by $\coro$ 
\index{Ch@$\coro$ $=\C[[\hbar]]$}%
the ring $\C\forl$
of formal power series 
in an indeterminate $\hbar$ and by $\cor$ 
\index{Chl@$\cor$ $=\C\Ls$}%
the field $\C\Ls$ 
of Laurent series in
$\hbar$. Then $\cor$ is the fraction field of $\coro$. 
\end{notation}
We set 
\eqn
&\OO[X]\forl\eqdot \prolim[n]\sho_X\otimes(\coro/\hbar^n\coro)\simeq
\prod\limits_{n\ge0}\sho_X\hbar^n.
\eneqn
Let us recall a classical definition (see \cite{BFFLS,Ko1}). 
\begin{definition}
An associative multiplication law 
$\star$ on $\OO[X]\forl$ is a
star-product
\glossary{star product}%
 if it is $\coro$-bilinear and satisfies
\eq\label{eq:star1}
&&f\star g=\sum_{i\geq 0}P_i(f,g)\hbar^i\text{ for } f,g\in\OO[X],
\eneq
where the $P_i$'s are bi-differential operators such that 
$P_0(f,g)=fg$ and $P_i(f,1)=P_i(1,f)=0$ for all $f\in\sho_X$ and $i>0$. 
We call $(\OO[X]\forl,\star)$ a star-algebra.
\end{definition}
Note that
$1\in\sho_X\subset\sho_X\forl$ is  a unit with respect to $\star$.
Note also that we have 
\eqn
&&(\ssum_{i\ge0}f_i\hbar^i)\star(\ssum_{i\ge0}g_i\hbar^i)=\ssum_{n\ge0}
\bl\ssum_{i+j+k=n}P_k(f_i,g_j)\br\hbar^n.
\eneqn
Recall that a star-product defines a Poisson structure on $(X,\sho_X)$ by setting for 
$f,g\in\sho_X$:
\eq\label{eq:Poisson0}
&&\{f,g\}=P_1(f,g)-P_1(g,f)
=\opb{\hbar}(f\star g-g\star f)\mbox{ mod }\hbar\sho_X\forl,
\eneq
and that locally, (globally in the real case), any Poisson manifold 
$(X,\sho_X)$ may be 
endowed with a star-product to which the Poisson structure is associated. 
This is a famous theorem of Kontsevich \cite{Ko1}.

\begin{proposition}\label{pro:gauge1}
Let $\star$ and $\star'$ be star-products and let
$\phi\cl (\OO[X]\forl,\star)\to (\OO[X]\forl,\star')$
be a morphism of $\coro$-algebras.
Then there exists a unique sequence of 
differential operators $\{R_i\}_{i\ge0}$ on $X$
such that $R_0=1$ and
$\phi(f)=\sum_{i\ge0}R_i(f)\hbar^i$ for any $f\in\sho_X$. 
In particular, $\phi$ is an isomorphism.
\end{proposition}
First, we need a lemma. In this lemma, we set $F_\infty(\shd_X)=\shd_X$.

\begin{lemma}\label{le:gauge1}
Let $l\in\Z_{\ge-1}\sqcup\{\infty\}$, and $\phi\in\Endo[\C_X](\sho_X)$.
If
$[\phi,g]\in\Fl_l(\shd_X)$ for all 
$g\in\sho_X$, then $\phi\in\Fl_{l+1}(\shd_X)$. 
\end{lemma}

\begin{proof}
We may assume that $X$ is an open subset of $\C^n$ 
and we denote by $(x_1,\dots,x_n)$ the
coordinates. Set $P_i=[\phi,x_i]$. Then
\eqn
&&[P_i,x_j]=[\,[\phi,x_i],x_j]=[\,[\phi,x_j],x_i]=[P_j,x_i].
\eneqn
This implies the existence of $P\in\Fl_{l+1}(\shd_X)$ such that 
$[P,x_i]=P_i$ for all $i$. Setting $\psi\eqdot\phi-P$, we have
\eqn
&&[\psi,x_i]=0 \text{ for all } i=1,\dots,n.
\eneqn
Let us show that $\psi\in\sho_X$. Replacing $\psi$ with 
$\theta\eqdot \psi-\psi(1)$, we get
by induction on the order of the polynomials 
that $\theta(Q)=0$ and $[\theta,Q]=0$
for all $Q\in\C[x_1,\dots,x_n]$. 
Let $f\in\sho_X$. We shall prove that $\theta(f)(x)=0$ 
for all $x\in X$. It is enough to prove it for $x=0$. Then, writing 
$f=f(0)+\sum_ix_if_i$, we get
\eqn
\theta(f)&=&\theta(f(0))+\sum_i\theta(x_if_i)
=\theta(f(0))+\sum_i\bigl(x_i\theta(f_i)+[\theta,x_i]f_i\bigr)\\
&=&\sum_ix_i\theta(f_i),
\eneqn
which vanishes at $x=0$.
\end{proof}

\begin{proof}[Proof of Proposition~\ref{pro:gauge1}]
Let us write
\eq\label{eq:gauge1}
&&\phi(f)=\sum_{i\ge0}\hbar^i\phi_i(f),\quad f\in\sho_X.
\eneq
By Lemma~\ref{le:endosho},
$\phi_0=1$. We shall prove by induction that the $\phi_i$'s in 
\eqref{eq:gauge1} are differential 
operators and we assume that this is so for all $i<n$ for $n\in \Z_{>0}$. 

Let $\{P_i\}$ and $\{P'_i\}$ be the sequence of bi-differential operators 
associated with the star-products $\star$ and $\star'$, respectively.
We have
\eqn
\phi(f\star g)&=&\phi(\sum_{j\ge0}\hbar^jP_j(f,g))
              =\sum_{i,j\ge0}\hbar^{i+j}\phi_i(P_j(f,g)),\\[1ex]
\phi(f)\star'\phi(g)&=&
\sum_{i\ge0}\hbar^i\phi_i(f)\star'\sum_{j\in\N}\hbar^j\phi_j(g)
               =\sum_{i,j,k\ge0}\hbar^{i+j+k}P'_k(\phi_i(f),\phi_j(g)).
\eneqn
Since $\phi(f\star g)=\phi(f)\star'\phi(g)$, we get:
\eq\label{eq:gauge2}
&&\sum_{n=i+j}\phi_i(P_j(f,g))=\sum_{n=i+j+k}P'_k(\phi_i(f),\phi_j(g)).
\eneq
By the induction hypothesis, the left hand side of \eqref{eq:gauge2} may be 
written as $\phi_n(fg)+Q_n(f,g)$ where $Q_n$ is a bi-differential
operator. Similarly, the right hand side of \eqref{eq:gauge2} may be 
written as $\phi_n(f)g+f\phi_n(g)+R_n(f,g)$ where $R_n$ is a bi-differential operator.
For any $g\in\sho_X$, considering $g$ as an endomorphism of $\sho_X$, we get
\eqn
&&[\phi_n,g](f)\eqdot\phi_n(fg)-g\phi_n(f)=f\phi_n(g)+S_n(f),
\eneqn
where $S_n$ is a differential operator.
Then, the result follows from Lemma~\ref{le:gauge1}. 
\end{proof}

\subsubsection*{$\DQ$-algebras}

\begin{definition}
A $\DQ$-algebra $\A[]$ on $X$ 
\glossary{DQalgebra@$\DQ$-algebra}%
is a $\coro$-algebra locally
isomorphic to a star-algebra $(\OO[X]\forl,\star)$
as a $\coro$-algebra.
\end{definition}
Clearly a $\DQ$-algebra $\A[]$ satisfies the conditions:
\eq&&
\left\{
\parbox{60ex}{
\bnum
\item $\hbar\cl \sha\to \sha$ is injective,\label{cond:bidiff1}
\item $\sha\to \prolim[n]\sha/\hbar^n\sha$ is an isomorphism,\label{cond:bidiff2}
\item $\sha/\hbar \sha$ is isomorphic to $\sho_X$ as a $\C$-algebra.\label{cond:bidiff3}
\enum
}\label{cond:DQ}
\right.
\eneq
For a $\coro$-algebra $\sha$ satisfying \eqref{cond:DQ}, 
the $\C$-algebra isomorphism $\sha/\hbar \sha\isoto \sho_X$
in \eqref{cond:DQ}~\eqref{cond:bidiff3} is unique by Lemma~\ref{le:endosho}.
We denote by
\eq
&&\sigma_0\cl \sha\to\sho_X 
\eneq
the $\coro$-algebra morphism $\sha\to \sha/\hbar \sha\isoto \sho_X$.
If $\phi$ is a $\C$-linear section of
$\sigma_0\cl\sha\to\sho_X$, then $\phi$ extends to an isomorphism of
$\coro$-modules
$\tw\phi\cl \OO[X]\forl\isoto \sha$,
given by $\tw\phi(\sum_if_i\hbar^i)=\sum_i\phi(f_i)\hbar^i$.

\begin{definition}\label{def:standartsection}
We say that a $\C$-linear section $\phi\cl\sho_X\to \sha$ of $\sha\to \sho_X$
is standard 
\glossary{standard section}\glossary{section!standard}%
if there exists a sequence of bi-differential operators $P_i$ such that
\eq\label{def:standartsect}
&&\phi(f)\phi(g)=\sum_{i\ge0}\phi(P_i(f,g))\hbar^i\mbox{ for any $f,g\in\sho_X$.}
\eneq
\end{definition}
Consider a standard section $\phi\cl\sho_X\to \sha$ of $\sha\to \sho_X$.
Define a star-product $\star$ on $\OO[X]\forl$ by setting
\eqn
&&f\star g=\sum_{i\ge0}P_i(f,g)\hbar^i \mbox{ for any }f,g\in\sho_X.
\eneqn
Then we get an isomorphism of  $\coro$-algebras 
\eq\label{eq:standartiso}
&&\tw\phi\cl (\OO[X]\forl,\star)\isoto \sha.
\eneq
We call $\tw\phi$ in \eqref{eq:standartiso} {\em a standard
  isomorphism}.
\glossary{standard!isomorphism}\glossary{isomorphism!standard}%

Hence, a $\DQ$-algebra is nothing but a $\coro$-algebra satisfying 
\eqref{cond:DQ} and admitting locally a standard section. 
\begin{remark}\label{rem:conj}
We conjecture that a $\coro$-algebra satisfying \eqref{cond:DQ} locally 
admits a standard section. 
\end{remark}
Let $\sha$ be a $\DQ$-algebra.
For $f$, $g\in\sho_X$, taking $a$, $b\in\sha$ such that
$\sigma_0(a)=f$ and $\sigma_0(b)=g$, we set
\eq
&&\{f,g\}=\sigma_0(\hbar^{-1}(ab-ba))\in\sho_X.
\label{eq:Poisson}
\eneq
Then this definition does not depend on the choice
of $a$, $b$ and it defines a Poisson structure on $X$.
In particular, two $\DQ$-algebras 
induce the same Poisson structure on $X$
as soon as they are locally isomorphic.

By Proposition~\ref{pro:gauge1}, if $\phi,\phi'\cl\sho_X\to \sha$ are
two  standard sections, 
then there exists a unique sequence of differential operators
$\{R_i\}_{i\ge0}$ such that
$\phi'(f)=\sum_{i\ge0}\hbar^i\phi(R_i(f))$ for any $f\in\sho_X$.

Clearly, a $\DQ$-algebra  satisfies the hypotheses
\eqref{eq:FDringa} and \eqref{eq:FDringb}. Hence, a $\DQ$-algebra is a right and left
Noetherian ring (in particular, coherent).

\begin{lemma}\label{le:starop}
Let $\sha$  be a $\DQ$-algebra. Then 
the opposite algebra $\sha^\rop$ is also a $\DQ$-algebra.
\end{lemma}
\begin{proof}
This follows from~\eqref{eq:star1}.
\end{proof}
Let $X$ and $Y$ be complex manifolds endowed with two 
star-products $\star_X$
and $\star_Y$. Denote by $\{P_i\}_{i}$ and 
$\{Q_j\}_{j}$ the bi-differential operators associated to these
star-products as in \eqref{eq:star1}.
Let $P_i\etens Q_j$ be a bi-differential operator on
$X\times Y$ defined as follows.
Let us take differential operators 
$\tw{P}_i(x_1,x_2,\partial_{x_1},\partial_{x_2})$ and
$\tw{Q}_j(y_1,y_2,\partial_{y_1},\partial_{y_2})$
corresponding to $P_i$ and $Q_j$
as in \eqref{eq:bidiff1}.
Then we set
\eqn
&&(P_i\etens Q_j)(f,g)(x,y)\\
&&\quad=\bl\tw{P}_i(x_1,x_2,\partial_{x_1},\partial_{x_2})\tw{Q}_j(y_1,y_2,\partial_{y_1},\partial_{y_2})
(f(x_1,y_1)g(x_2,y_2))\br\vert_{\substack{x_1=x_2=x\\y_1=y_2=y}}.
\eneqn
Hence, $P_i\etens Q_j$ is the unique 
bi-differential operator on $X\times Y$ such that
$(P_i\etens Q_j)(f_1(x)g_1(y),f_2(x)g_2(y))=
P_i(f_1(x),f_2(x))\cdot Q_j(g_1(y),g_2(y))$ for any 
$f_\nu(x)\in\sho_X$ and $g_\nu(y)\in\sho_Y$ ($\nu=1,2$).

One defines the external  product 
of the star-products $\star_X$ and
$\star_Y$ 
on $\OO[X\times Y]\forl$ by setting 
\eqn
&&f\star g=
\sum_{n\geq 0}\hbar^n\sum_{i+j=n}(P_i\etens Q_j)(f,g).
\eneqn
Hence:
\begin{lemma}\label{le:starprod}
Let $X$ and $Y$ be complex manifolds, and
let $\A[X]$ be a $\DQ$-algebra on $X$ and 
$\A[Y]$ a $\DQ$-algebra on $Y$.
Then
there exists a $\DQ$-algebra $\sha$ on $X\times Y$
which contains $\A[X]\boxtimes_{\coro} \A[Y]$ as a $\coro$-subalgebra.
Moreover such an $\sha$ is unique up to a unique isomorphism.
\end{lemma}

We call $\sha$ the external  product 
\glossary{external product! of $\DQ$-algebras}%
of the $\DQ$-algebra $\A[X]$ on $X$ and
the $\DQ$-algebra $\A[Y]$ on $Y$, and denote it by
$\A[X]\detens \A[Y]$.

\begin{remark}
\bnum
\item
Any commutative $\DQ$-algebra is locally isomorphic to
$(\OO{\forl},\star)$ 
where $\star$ is the trivial star-product $f\star g=fg$.
\item 
For the trivial $\DQ$-algebra $\OO{\forl}$, we have
\eqn
&&\aut_{\text{$\coro$-alg}}(\OO{\forl})\simeq
\hbar\Theta_X\forl\seteq\prod_{n\ge1}\hbar^n\Theta_X,
\eneqn
(recall that $\Theta_X$ is the sheaf of vector fields on $X$)
and we associate to $v\seteq\sum_{n\ge1}\hbar^nv_n$ the automorphism
$f\mapsto \exp(v)f$.
\enum
\end{remark}

\subsubsection*{The ring $\DA$ and another construction for $\DQ$-algebras}
We define the $\coro$-algebra 
\eqn
&&\shd_X\forl\eqdot \prolim[n]\shd_X\otimes(\coro/\hbar^n\coro)\simeq
\prod_{n\ge0}\shd_X\hbar^n.
\eneqn
\index{Dhbar@$\shd_X\forl$}%
Then $\OO{\forl}$ has a $\shd_X\forl$-module structure,
and $\shd_X\forl\subset \shend_{\coro}(\OO[X]\forl)$.

Let $\A$ be a $\DQ$-algebra. Choose (locally)  a standard
section $\phi$ 
giving rise to a standard isomorphism of $\coro$-modules
$\tw\phi\cl \OO[X]\forl\isoto\A$. 
This last isomorphism induces an isomorphism 
\eq\label{eq:isoDA1}
&&\Phi\cl \shend_{\coro}(\OO[X]\forl)\isoto\shend_{\coro}(\A).
\eneq

\begin{definition}\label{def:DA}
Let $\A$ be a $\DQ$-algebra and let $\phi$ be a standard section.
The sheaf of rings $\DA$
\index{DA@$\DA$}%
is the $\coro$-subalgebra of
$\shend_{\coro}(\A)$, the image of 
$\shd_X\forl\subset \shend_{\coro}(\OO[X]\forl)$
by the isomorphism $\Phi$ in \eqref{eq:isoDA1}.
\end{definition}

It is easy to see that
$\DA\subset \shend_{\coro}(\A)$ does not depend 
on the choice of the standard section $\phi$ in virtue of
Proposition~\ref{pro:gauge1}.
Hence $\DA$ is well-defined on $X$ although standard sections only
locally exist.

By its construction, we have $\DA\isoto\prolim[n]\DA/\hbar^n\DA$.
Moreover, the image of the algebra morphism
$\A\otimes\A^\op\to \shend_{\coro}(\A)$,
as well as the one
of $\opb{\de}\A[X\times X^a]\to \shend_{\coro}(\A)$ is contained in
$\DA$.
Hence we have algebra morphisms
\eqn
&&\A[X]\otimes\A[X^a]\to \opb{\de}\A[X\times X^a]\to \DA. 
\eneqn

We shall  show how to construct a star-algebra from the data of
sections of $\shd_X\forl$ 
satisfying suitable commutation properties.

Let $\A\eqdot(\OO\forl,\star)$ be a star-algebra.  
There are two $\coro$-linear morphisms from $\OO{\forl}$ to 
$\shd_X\forl$ given by
\eq\label{eq:OstartoDlr}
&&\Phi^l\cl f\mapsto f\,\star,\quad\Phi^r\cl f\mapsto \star\,f.
\eneq
Hence, for $f\in\OO[X]$, we have:
\eqn
&&\Phi^l(f)=\sum_{i\geq 0}P_i(f,\scbul)\hbar^i,
\quad\Phi^r(f)=\sum_{i\geq 0}P_i(\scbul,f)\hbar^i.
\eneqn

Then $\Phi^l\cl \A\to \shd_X\forl$ and $\Phi^r\cl \A^\rop\to \shd_X\forl$ are 
two $\coro$-algebra morphisms, and
induce a  $\coro$-algebra morphism $\A\tens\A^\rop\to \shd_X\forl$.

Assume to be given a local coordinate system  $x=(x_1,\dots,x_n)$ on
$X$ and for $i=1,\dots,n$, set 
$\Phi^l(x_i)=A_i$ and $\Phi^r(x_i)=B_i$. Then $\{A_i,B_j\}_{i,j=1,\dots,n}$ are sections of 
$\shd_X\forl$ which satisfy
\eq\label{eq:OstartoD1}
&&\left\{  
\parbox{58ex}{
$A_i(1)=B_i(1)=x_i$,\\
$A_i\equiv x_i$ mod $\hbar\shd_X\forl$, $B_i\equiv x_i$ mod $\hbar\shd_X\forl$,\\
$[A_i,B_j]=0$ ($i,j=1,\dots,n$).
}
\right. 
\eneq
Conversely, we have the following result. 
\begin{proposition}\label{pro:OstartoD}
Let $\{A_i,B_j\}_{i,j=1,\dots,n}$ be sections of $\shd_X\forl$ 
which satisfy \eqref{eq:OstartoD1}.
Define the subalgebra $\A\subset\shd_X\forl$ by 
\eq\label{eq:OstartoD2}
&&\A=\set{a\in \shd_X\forl}{[a,B_i]=0, i=1,\dots,n}
\eneq
and define the $\coro$-linear map $\psi\cl\A\to\OO\forl$ by setting $\psi(a)=a(1)$. 
Then 
\banum
\item
$\psi$ is a $\coro$-linear isomorphism, 
\item
the product on $\OO\forl$ given by 
$\psi(a)\star\psi(b)\eqdot\psi(a\cdot b)$ is a star-product, $\A$ is a
$\DQ$-algebra and $\opb{\psi}$ is  a standard isomorphism,
 \item
the algebra $\A^\rop$ is obtained by replacing $A_i$ with $B_i$ \lp$i=1,\dots,n$\rp\, in 
the above construction.
\eanum
\end{proposition}
\begin{proof}
(a)-(i) $\A\cap\hbar\shd_X\forl=\hbar\A$, since $[\hbar a,B_i]=0$ implies 
$[a,B_i]=0$.
Hence we have $\A/\hbar^j\A\subset\shd_X\forl/\hbar^j\shd_X\forl$
for any $j$.

\noindent
(a)-(ii) $\A\isoto\prolim[j]\A/\hbar^j\A$.  Indeed, let 
$a=\sum_{i=0}^\infty \hbar^ia_i$ and assume that  
\eqn
&&[\sum_{i=0}^k\hbar^ia_i,B_l]=0\mbox{ mod }\hbar^{k+1}\, (l=1,\dots,n)
\eneqn
for all $k\in\N$. Then $[a,B_l]=0$ for $l=1,\dots,n$.

\noindent
(a)-(iii) Let $\psi_j\cl\hbar^j\A/\hbar^{j+1}\A\to \hbar^j\OO/\hbar^{j+1}\OO$ 
be the morphisms induced by $\psi$. 
By (a)-(ii) it is enough to check that all $\psi_j$'s are isomorphisms. 
Since all $\hbar^j\A/\hbar^{j+1}\A$ are isomorphic and all 
$ \hbar^j\OO/\hbar^{j+1}\OO$ are isomorphic, we are reduced to 
prove that $\psi_0\cl\A/\hbar\A\to\OO$ is an isomorphism. 

\noindent
(a)-(iv) $\psi_0$ is injective. Let $a_0\in \A/\hbar\A\subset\shd_X$.
Since $[a_0,x_i]\in\hbar\shd_X\forl$ implies  $[a_0,x_i]=0$, 
we get $a_0\in\OO$. Therefore,  $a_0(1)=0$ implies $a_0=0$.

\noindent
(a)-(v) $\psi_0$ is surjective. Let $y=(y_1,\dots,y_n)$ be 
a local coordinate system on a copy of $X$. Notice first that the sections 
$y_i-A_i$ of $\shd_{X\times Y}\forl$ are invertible 
on the open sets $\{y_i\neq x_i\}$. 
Let $f(x_1,\dots,x_n)\in\OO$. Define
the section $G(f)$ of $\shd_X\forl$ by
\eq
&&G(f)
=\dfrac{1}{(2\pi i)^n}\oint f(y)\opb{(y_1-A_1)}\cdots\opb{(y_n-A_n)}\,
dy_1\cdots dy_n\,.
\eneq
Then $[G(f),B_i]=0$ for all $i$.
It is obvious that $G(f)-f\in\hbar\shd_X\forl$ and $\psi_0(G(f))=f$.

\vspace{0.2cm}
\noindent
(b) Clearly, the algebra  $(\OO\forl,\star)$ satisfies \eqref{cond:DQ}.
Moreover, $f\mapsto G(f)$ is  a standard section since there exist
$P_i(f)\in\shd_X\forl$ ($i\in\N$) such that $G(f)=\sum_iP_i(f)\hbar^i$
and  $P_i(f)$ is obtained as the action of a bidifferential operator $P_i$ on $f$.

\vspace{0.2cm}
\noindent
(c) follows from $\sha^\rop=\{b\in\shend_{\coro}(\A); [b,\A]=0\}$.
\end{proof}
 
\begin{example}\label{exa:DQ1}
Let $M\eqdot\{a_{ij}\}_{i,j=1,\dots,n}$ be an $n\times n$
skew-symmetric matrix with entries in $\C$.
Let $X=\C^n$ and consider the sections of $\shd_X\forl$:
\eqn
&& A_i=x_i+\dfrac{\hbar}{2}\sum_ja_{ij}\partial_j,
\quad B_i=x_i-\dfrac{\hbar}{2}\sum_ja_{ij}\partial_j.
\eneqn
Then  $\{A_i,B_j\}_{i,j=1,\dots,n}$ satisfy \eqref{eq:OstartoD1}, 
thus define a $\DQ$-algebra $\A$.
Note that the Poisson structure associated with the $\DQ$-algebra 
$\A$ is symplectic if and only if the matrix $M$ is non-degenerate. 
\end{example}

\section{$\DQ$-algebroids}

Let us introduce the notion of a deformation quantization 
algebroid, a $\DQ$-algebroid for short. 
\begin{definition}\label{def:DFDalg}
A $\DQ$-algebroid  $\A[]$ on $X$ 
\glossary{DQalgebroid@$\DQ$-algebroid}\glossary{algebroid!$\DQ$-}%
is a $\coro$-algebroid  such that
for each open set $U\subset X$ and each $\sigma\in\A[](U)$, 
the $\coro$-algebra $\shend_{\A[]}(\sigma)$ is a $\DQ$-algebra on $U$.
\end{definition}
Note that a $\DQ$-algebroid
is called a twisted associative deformation of $\OO[X]$ in \cite{Ye}. 

By \eqref{eq:Poisson},
a $\DQ$-algebroid $\A[]$ on the complex manifold $X$ defines a 
Poisson structure on $X$. 
It is proved in \cite{Ko2} that, conversely, 
any complex Poisson
manifold $X$ may be endowed with a $\DQ$-algebroid to which
this Poisson structure is associated.

According to Convention~\ref{conv1}, if $\A[]$ is a
$\DQ$-algebra, we shall  often use the same
notation $\A[]$ for the associated $\DQ$-algebroid.

Note that any  $\DQ$-algebroid  $\A[]$ on $X$ may be obtained as the 
stack associated with a gluing datum
as in~\eqref{eq:descentdata}, where
the sheaves $\sha_i$ are $\DQ$-algebras.

Let $\A[]$ be a $\DQ$-algebroid on $X$.
For an $\A[]$-module 
$\shm$,  the local notions of being coherent 
or locally free, etc.\ make sense.

The category $\md[{\A[]}]$ is a Grothendieck category.
We denote by $\RD(\A[])$ its derived category 
and by $\RD^\Rb(\A[])$ its bounded derived category. We 
still call an object of this derived category an $\A[]$-module. 
We denote by
$\RD^\Rb_{\coh}(\A[])$ the full triangulated subcategory 
 of $\RD^\Rb(\A[])$ consisting
of objects with coherent cohomologies.

\subsubsection*{Opposite structure}
If $X$ is endowed with a $\DQ$-algebroid $\A[X]$, then we denote by
$X^a$ the manifold $X$ endowed with the algebroid
$\A[X]^\rop$, that is:
\eq\label{eq:dualalst}
&&\A[X^a]=\A[X]^\rop.
\eneq
This is a $\DQ$-algebroid by Lemma~\ref{le:starop}.

\subsubsection*{External product}

Assume that complex manifolds $X$ and $Y$ are endowed with $\DQ$-algebroids 
$\A[X]$ and $\A[Y]$ respectively.
By Lemma~\ref{le:starprod}, there is a canonical 
$\DQ$-algebroid  $\A[X]\detens\A[Y]$ on $X\times Y$ 
locally equivalent to the stack associated with 
the external product $\A[X]\detens\A[Y]$ of the $\DQ$-algebras 
\glossary{external product! of $\DQ$-algebroids}%
\index{tens@$\detens$}%
and there is a faithful functor of $\coro$-algebroids
\eq\label{eq:detWst}
&& \A[X]\etens\A[Y]\to\A[X]\detens\A[Y],
\eneq
which  induces a functor
\eq
&&\for\cl \Mod(\A[X]\detens\A[Y])\to \Mod(\A[X]\etens\A[Y]).
\eneq
When there is no risk of confusion, we set
\eqn
&&\A[X\times Y]\eqdot\A[X]\detens\A[Y].
\eneqn
Then $\A[X\times Y]$ belongs to 
$\md[{\A[X\times Y]\tens(\A[X^a]\etens\A[Y^a])}]$ 
and the
functor $\for$ admits a left adjoint functor 
$\shk\mapsto \A[X\times Y]\tens[{\A[X]\etens\A[Y]}]\shk$:
\eq\label{eq:detWst2}
\xymatrix@C=6ex{
\md[{\A[X\times Y]}]\ar@<0.5ex>[r]^(.45){\for}&
                 \md[{\A[X]\etens\A[Y]}].\ar@<0.5ex>[l]
}\eneq
We denote by $\scbul\detens\scbul$ the bi-functor $\A[X\times Y]
\mathop\otimes\limits_{\A\etens\A[Y]}
(\scbul\etens\scbul)$:
\eq\label{eq:detWst2b}
\scbul\detens\scbul&\cl&\md[{\A[X]}]\times\md[{\A[Y]}]
\To \md[{\A[X\times Y]}].
\eneq

\begin{lemma} If $\shm$ is an $\A$-module without $\hbar$-torsion,
then the functor
$$\shm\detens\scbul:\md[{\A[Y]}]\to \md[{\A[X\times Y]}]$$
is an exact functor.
\end{lemma}
\begin{proof}
We may assume that $\A$ and $\A[Y]$ are $\DQ$-algebras.
Hence it is enough to show that for any $(x,y)\in X\times Y$, setting 
$\shn\seteq\A[X\times Y]\tens[{\A}]\shm$, $\shn_{(x,y)}$ is a flat module
over $\A[Y,y]^\rop$. We may assume further that $\shm$ is a coherent $\A$-module
without $\hbar$-torsion.
For any Stein open subset $U$, let $p_U\cl U\times Y\to Y$ be the projection.
Set
$\shn_U\seteq (p_U)_*\bl(\A[X\times Y]\tens[{\A}]\shm)\vert_{U\times Y}\br$.
Then it is easy to check the conditions (a)--(c) in 
Theorem~\ref{th:flat}
are satisfied ((c) follows from the $\sho$-module version of this lemma), 
and we conclude that
$\shn_U$ is a flat $\A[Y]^\rop$-module.
Hence, $\shn_{(x,y)}\simeq\indlim[{x\in U}](\shn_U)_y$ is a flat $(\A[Y]^\rop)_y$-module.
\end{proof}
Hence
the left derived functor
$$\scbul\ldetens\scbul\cl \Der(\A[X])\times\Der(\A[Y])
\to \Der(\A[X\times Y])$$
\index{tensl@$\ldetens$}%
satisfies
$\shm^\scbul\ldetens\shn^\scbul\isoto\shm^\scbul\detens\shn^\scbul$
as soon as $\shm^\scbul$ or $\shn^\scbul$ is a complex bounded from above 
of modules without $\hbar$-torsion.

\subsubsection*{Graded modules}
For a $\coro$-algebroid $\shb$ on $X$, 
one denotes by $\gr(\shb)$ the $\C$-algebroid 
associated with the prestack $\sts$ given by
\eqn
&&\Ob(\sts(U))=\Ob(\shb(U))\quad\text{for an open subset $U$ of $X$,}\\
&&\Hom[\sts(U)](\sigma,\sigma')=
\Hom[{\shb}](\sigma,\sigma')/\hbar \Hom[{\shb}](\sigma,\sigma')
\quad\mbox{for $\sigma$, $\sigma'\in\shb(U)$}.
\eneqn
Let now $\A[X]$ be a $\DQ$-algebroid  on $X$. 
Then it is easy to see that $\bA$ is an invertible $\OS$-algebroid
and that we have a natural functor
$\A[X]\to\bA$ of $\C$-algebroids. This functor induces a functor
\eq
&&\for\cl\Mod(\bA)\to \Mod(\A).
\eneq
The functor $\for$ above  is fully faithful and
$\Mod(\bA)$ is equivalent to the 
full subcategory of
$\Mod(\A)$ consisting of objects $M$ such that $\hbar\cl M\to M$ vanishes.
The functor $\for\cl\Mod(\bA)\to \Mod(\A)$
admits a left adjoint functor
$M\mapsto M/\hbar M\simeq\C\tens[\coro] M$.
The functor $\for$ is exact and it induces a functor
\eq\label{eq:derforgr}
&&\for\cl\Der(\bA)\to \Der(\A).
\eneq
\begin{remark}\label{rem:notff}
The functor in~\eqref{eq:derforgr} is not full 
in general. 
Indeed, choose $X=\rmpt$, $\A=\coro$ and
$L=\coro/\hbar\coro$ viewed as a $\gr(A)$-module.  Then
\eqn
&&\Hom[\Derb(\coro)]\bl\for(L),\for(L[1])\br\simeq \coro/\hbar\coro,\\
&&\Hom[\Derb(\C)](L,L\,[1])\simeq0.
\eneqn
It could be also shown that this functor is not faithful in general.
\end{remark}
\bigskip
One extends Definition~\ref{def:grad} to the algebroid $\A$.
As an $(\A[X]\tens\A[X^a])$-module, $\bA$ is isomorphic to
$\C\tens[\coro]\A\simeq\A/\hbar\A$. We get the functor 
\eq\label{eq:dergr}
&&\gr\cl\Der(\A)\to\Der(\bA),\,
\shm\mapsto\gr(\A)\lltens[{\A}]\shm\simeq\C\lltens[\coro]\shm.
\eneq
Note that Lemma~\ref{lem:grHa},
Propositions~\ref{pro:tensgr} and~\ref{pro:grHa} 
as well as  Corollary~\ref{cor:conservative1} still hold.
Moreover
\begin{corollary}\label{cor:grsupp}
Let $\shm\in\Derb_\coh(\A)$. Then its support, $\Supp(\shm)$, is a
closed complex analytic subset of $X$.
\end{corollary}
\begin{proof}
By Corollary~\ref{cor:conservative1}, $\Supp(\shm)=\Supp(\gr(\shm))$. Since
$\gr(\shm)\in\Derb_\coh(\gr(\A))$ and $\gr(\A)$ is locally isomorphic 
to $\OO$, the result follows.
\end{proof}

Let $d_X$ denote the complex dimension of $X$. Applying
Theorem~\ref{th:hddim}, we get
\begin{corollary}\label{cor:hddim}
Let $\A[X]$ be a $\DQ$-algebra and let $\shm\in\mdcoh[{\A[X]}]$. 
Then, locally, $\shm$ admits a resolution by free modules
of finite rank of length $\leq d_X+1$.
\end{corollary}
\begin{proposition}\label{pro:grforadj}
The functors $\gr$ in~\eqref{eq:dergr} and $\for$ in~\eqref{eq:derforgr}
define pairs of adjoint functors $(\gr,\for)$ and $(\for,\gr[-1])$.
\end{proposition}
\begin{proof}
Consider a pair $(B,C)$ in which either $B=\A$ and $C=\bA$ or 
$B=\bA$ and $C=\A$, and let $K$ be a $(B,C)$-bimodule.
We have the adjunction formula, for $M\in\Derb(B)$ and $N\in\Der(C)$:
\eq\label{eq:adjlltensrhom}
&&\Hom[\Der(B)](K\lltens[C]N,M)\simeq\Hom[\Der(C)](N,\rhom[B](K,M)).
\eneq
(i) Let us apply formula~\eqref{eq:adjlltensrhom} with $B=\bA$, $C=\A$ and 
$K=\bA$ considered as a $(\bA,\A)$-bimodule. We get
\eqn
&&\hspace{-5ex}\Hom[\Der(\bA)](\bA\lltens[\A]\shm,\shn)\\
&&\hspace{15ex}\simeq\Hom[\Der(\A)](\shm,\rhom[\bA](\bA,\shn)),
\eneqn
and when remarking that $\rhom[\bA](\bA,\shn)\simeq\for(\shn)$, we get the first 
adjunction pairing.

\noindent
(ii) Let us apply formula~\eqref{eq:adjlltensrhom} with $C=\bA$, $B=A$ and 
$K=\bA$ considered as an $(\A,\bA)$-bimodule.
We get
\eqn
&&\hspace{-5ex}\Hom[\Der(\A)](\bA\lltens[\bA]\shn,\shm)\\
&&\hspace{15ex}\simeq\Hom[\Der(\bA)](\shn,\rhom[\A](\bA,\shm)).
\eneqn
We have $\bA\lltens[\bA]\shn\simeq\for(\shn)$ and to get the second 
adjunction pairing, notice that
\eqn
\rhom[\A](\bA,\shm)&\simeq&\rhom[\A](\bA,\A)\lltens[\A]\shm,
\eneqn
and $\rhom[\A](\bA,\A)\simeq\bA\,[-1]$.
\end{proof}

\subsubsection*{Duality}
Let $\A[X]$ be a $\DQ$-algebroid on $X$.

\begin{definition}\label{def:dual1}
Let $\shm\in\RD(\A[X])$. 
Its dual  $\RDAA\shm\in\RD(\A[X^a])$ 
\index{DA@$\RDA\shm$}\glossary{dual!of $\A$-module}%
is given by
\eq\label{eq:dual1}
&&\RDAA\shm\eqdot \rhom[{\A[X]}](\shm,\A[X]).
\eneq
\end{definition}
When there is no risk of confusion, we write $\RDA$ instead of
$\RDAA$. 

By Corollary~\ref{cor:hddim}, $\RD'_{\rma}$ sends 
$\Derb_\coh(\A)$ to $\Derb_\coh(\A[X^a])$:
$$\RDA\cl\Derb_\coh(\A)\To\Derb_\coh(\A[X^a]).$$

Assume that $\shm\in\RD^\Rb_\coh(\A[X])$. 
Then there is a canonical isomorphism:
\eq
\shm\isoto\RDA\RDA\shm.
\eneq

For a $\bA$-module $\shm$, denote by $\RDO\shm$ its dual,
\eq\label{def:dualo}
&&\RDO\shm\eqdot \rhom[{\bA}](\shm,\bA).
\eneq

\begin{proposition}\label{pro:dualgr}
Let $\shm\in\Derb_\coh(\A[X])$. Then
\eqn
&&\gr(\RDA\shm)\simeq \RDO(\gr(\shm)).
\eneqn
\end{proposition}
\begin{proof}
This follows from Proposition~\ref{pro:tensgr}.
\end{proof}

\begin{corollary}\label{cor:grHa}
Let $\shl\in\Derb_\coh(\A[X])$ and $j\in\Z$. Let us assume that\\
$\ext[{\bA[X]}]{j}(\gr(\shl),\bA)\simeq 0$.
Then $\ext[{\A[X]}]{j}(\shl,\A[X])\simeq 0$.
\end{corollary}

\begin{proof}
Applying the above proposition, we get
\eqn
\ext[{\bA}]{j}(\gr(\shl),\bA[X])&=&H^j(\RD'_\rmo(\gr(\shl)))\\
&\simeq&H^j(\gr(\RD'_\rma(\shl))).
\eneqn
Then the result follows from Proposition~\ref{pro:grHa}.
\end{proof}

\subsubsection*{Simple modules}
\glossary{simple!module}%
\glossary{module!simple}%

\begin{definition}\label{def:simplemod}
Let $\Lambda$ be a smooth submanifold of $X$  and 
let $\shl$ be a coherent $\A[X]$-module supported by $\Lambda$. 
One says that $\shl$ is simple along $\Lambda$ if
$\gr(\shl)$ is concentrated in degree $0$ and $H^0(\gr(\shl))$ 
is an invertible $\sho_\Lambda\tens[{\sho_X}]\gr(\A)$-module. (In particular, 
$\shl$ has no $\hbar$-torsion.) 
\end{definition}

\begin{proposition}\label{pro:vanishsimple}
Let $\Lambda$ be a closed submanifold of $X$ of codimension $l$ and 
let $\shl$ be a coherent $\A[X]$-module simple along $\Lambda$. Then 
$H^j(\RD'_\rma(\shl))=\ext[{\A[X]}]{j}(\shl,\A[X])$ vanishes for $j\neq l$ 
and $H^l(\RD'_\rma(\shl))$
is a coherent $\A[X^a]$-module simple along $\Lambda$.
\end{proposition}
\begin{proof}
The question being local, we may assume that $\A$ is a $\DQ$-algebra
so that $\bA\simeq\OO[X]$ and $\gr(\shl)\simeq\OO[\Lambda]$.
Then, we have $\ext[{\OO[X]}]{j}(\gr(\shl),\OO[X])\simeq 0$ for $j\neq l$.
Therefore, $\ext[{\A[X]}]{j}(\shl,\A[X])=0$ for $j\neq l$
by Corollary~\ref{cor:grHa} and 
\eqn\gr(\ext[{\A[X]}]{l}(\shl,\A[X]))&\simeq&\RDO(\gr\shl)\,[l]\\
&\simeq&\ext[{\OO[X]}]{l}(\gr(\shl),\OO[X])
\eneqn
by Proposition~\ref{pro:dualgr}.
If $\gr(\shl)$ is locally isomorphic to $\sho_\Lambda$, then so is 
$\ext[{\OO[X]}]{l}(\gr(\shl),\OO[X])$.
\end{proof}

\subsubsection*{Homological dimension of $\A$-modules}
The codimension of the support of a coherent $\sho_X$-module $\shf$ is
related to the vanishing of the $\ext[\sho_X]{j}(\shf,\sho_X)$.
Similar results hold for $\A$-modules.
\begin{proposition}\label{pro:dimext1}
Let $\shm$ be  a coherent $\A$-module. Then 
\banum
\item
$\ext[{\A}]{j}(\shm,\A)\simeq 0$ for $j<\codim\Supp\shm$,
\item
$\codim\Supp\ext[{\A}]{j}(\shm,\A)\geq j$.
\eanum
\end{proposition}
\begin{proof}
(a) First, note that $\Supp(\shm)=\Supp(\gr\shm)$. 
Therefore,  
\eqn
&&\ext[{\bA}]{j}(\gr\shm,\bA)\simeq 0 \mbox{ for } j<\codim\Supp\shm 
\eneqn
and the result follows from Corollary~\ref{cor:grHa}.

\noindent
(b) By Proposition~\ref{pro:grHa}, we know that 
\eqn
&&\Supp\ext[{\A}]{j}(\shm,\A)\subset\Supp\ext[{\bA}]{j}(\gr\shm,\bA),
\eneqn
and $\codim\Supp\ext[{\bA}]{j}(\gr\shm,\bA)\geq j$
by classical results for $\sho_X$-modules.
\end{proof}

\subsubsection*{Extension of the base ring}
Recall that $\cor\seteq\C\Ls$ is the fraction field of $\coro$. 
To a $\DQ$-algebroid $\A[X]$ we associate the $\cor$-algebroid
\eq
&&\Ah[X]=\cor\tens[\coro]\A[X]
\eneq
and we call $\Ah[X]$ the {\em $\hbar$-localization} of $\A[X]$.
It follows from Lemma~\ref{le:locNoeth} that 
the algebroid $\Ah[X]$ is Noetherian.

There naturally exists a faithful functor of $\coro$-algebroid 
\eq\label{eq:WotoW}
&&\A[X]\to \Ah[X].
\eneq
This functor gives rise to 
a pair of adjoint functors $(\loc,\for)$:
\eq\label{eq:WotoW2a}
\xymatrix@C=8ex{
\md[{\Ah[X]}]\ar@<0.5ex>[r]^{\for}& 
    \md[{\A[X]}].\ar@<0.5ex>[l]^{\loc}
}\eneq
Both functors are exact and we keep the same notations for  their derived  
functors 
\eq\label{eq:WotoW2b}
\xymatrix@C=8ex{
{\Db(\Ah[X])}\ar@<0.5ex>[r]^{\for}&
    {\Db(\A)}.\ar@<0.5ex>[l]^{\loc}
}\eneq
For $\shn\in\Db(\A)$, we set
\index{loc@$\shn^\loc$}%
\eq\label{eq:defloc}
&&\shn^\loc\eqdot\loc(\shn)\simeq\cor\tens[\coro]\shn.
\eneq
We say that an $\A[X]$-module $\shm_0$ is a submodule
\glossary{submodule!of $\Ah[\stx]$-module}%
 of an $\Ah[\stx]$-module $\shm$ if there is a monomorphism $\shm_0\to\for(\shm)$
in $\md[{\A[X]}]$.

If $\shm$ is  an $\Ah[\stx]$-module, $\shm_0$ an $\A[X]$-submodule and 
$\shm_0\tens[\coro]\cor\isoto\shm$, 
then we shall say that $\shm_0$ generates $\shm$.

The following result is of constant use and follows from
\cite[Appendix~A]{Ka2}.
\begin{lemma}
Any locally finitely generated $\A[X]$-submodule of a
coherent $\Ah[X]$-module is coherent, i.e., any coherent 
$\Ah[X]$-module is pseudo-coherent as an $\A$-module.
\end{lemma}
\begin{definition}\label{def:lattice}
A coherent $\A[X]$-submodule $\shm_0$ of a coherent $\Ah[X]$-module $\shm$ is called
an $\A[X]$-lattice 
\glossary{lattice}\glossary{Alattice@$\A$-lattice}%
of $\shm$ if $\shm_0$ generates $\shm$.
\end{definition}
We extend Definition~\ref{def:dual1} to $\Ah$-modules and,
for $\shm\in\Derb(\Ah)$,  we set 
\index{DAh@$\RDAh\shm$}\glossary{dual!of $\Ah$-module}%
\eq\label{eq:dual1h}
&&\RDAh\shm\eqdot \rhom[{\Ah[X]}](\shm,\Ah[X]).
\eneq
\begin{proposition}\label{pro:dimext1l}
Let $\shm$ be  a coherent $\Ah$-module. Then 
\banum
\item
$\ext[{\Ah}]{j}(\shm,\Ah)\simeq 0$ for $j<\codim\Supp\shm$,
\item
$\codim\Supp\ext[{\Ah}]{j}(\shm,\Ah)\geq j$.
\eanum
\end{proposition}
\begin{proof}
The result is local and we may choose an $\A[X]$-lattice $\shm_0$ of
$\shm$. Then the result follows from Proposition~\ref{pro:dimext1}.
\end{proof}

\subsubsection*{Good modules}
\begin{definition}\label{def:good}
\bnum
\item
A coherent $\Ah[X]$-module $\shm$ is good 
\glossary{good!$\A$-module}%
if, for any
relatively compact open subset $U$ of $X$, there exists 
an $(\A[X]\vert_U)$-lattice of $\shm\vert_U$.
\item
One denotes by $\mdgd[{\Ah[X]}]$ the full subcategory of 
$\mdcoh[{\Ah[X]}]$ consisting of good modules.
\item
One denotes by $\RD^\Rb_{\gd}(\Ah[X])$ the full subcategory of 
$\RD^\Rb_{\coh}(\Ah[X])$ consisting of objects $\shm$ such that 
$H^j(\shm)$ is good for all $j\in\Z$. 
\enum
\end{definition}
Roughly speaking, a coherent  $\Ah[X]$-module  $\shm$ is good if it
is endowed with a good
filtration (see \cite{Ka2}) on each open relatively compact subset of $X$.

\begin{proposition}\label{pro:good}
\banum
\item
The category $\mdgd[{\Ah[X]}]$ is a thick subcategory of $\mdcoh[{\Ah[X]}]$,
\glossary{thick subcategory}%
\ro i.e., stable by kernels, cokernels and extension\rf.
\item
The full subcategory $\RD^\Rb_{\gd}(\Ah[X])$  of $\RD^\Rb_{\coh}(\Ah[X])$ 
is triangulated.
\item
An object $\shm\in\RD^\Rb_{\coh}(\Ah[X])$ is good if and only if, for any 
open relatively compact subset $U$ of $X$, there exists 
an $\A[X]\vert_U$-module $\shm_0\in \RD^\Rb_{\coh}(\A[X]\vert_U)$ 
such that $\shm_0^\loc$ is isomorphic to $\shm\vert_U$.
\eanum
\end{proposition}
Since the proof is similar to that of \cite[Prop.~4.23]{Ka2}, we shall
not repeat it.
\begin{proposition}\label{pro:invol}
Let $\shm\in\RD^\Rb_{\coh}(\Ah[X])$. Then $\Supp(\shm)$ is a closed
complex analytic subset of $X$, involutive \lp{\em i.e.,} co-isotropic\rp\, for
the Poisson bracket on $X$.
\end{proposition}
\begin{proof}
Since the problem is local, we
may assume that $\A$ is a $\DQ$-algebra.
Then the proposition follows from
Gabber's theorem \cite{Ga}.
\end{proof}
\begin{remark}
One shall be aware that the support of a coherent
$\A$-module is not involutive in general. Indeed,
for a DQ-algebra $\A$, any coherent $\OO$-module may be 
regarded as an $\A$-module.
Hence any closed analytic subset can be the support of a coherent $\A$-module.
\end{remark}

\section{$\DQ$-modules supported by the diagonal}
\label{section:diagonal}

Let $X$ be a  complex manifold endowed with a $\DQ$-algebroid
$\A[X]$. We denote by $\A[X\times X^a]$ the external product of
$\A[X]$ and $\A[X^a]$ on $X\times X^a$.
We still denote by $\de\cl X\hookrightarrow X\times X^a$ the diagonal embedding
and we denote by $\Mod_{\Delta_X}(\A[X]\etens\A[X^a])$ the category 
of $(\A[X]\etens\A[X^a])$-modules supported by the diagonal $\Delta_X$.
Then 
\eqn
&&\oim{\de}\cl\Mod(\A[X]\tens\A[X^a])\to\Mod_{\Delta_X}(\A[X]\etens\A[X^a])
\eneqn
gives an equivalence of categories, with quasi-inverse $\opb{\de}$.
We shall often  identify these two categories by this equivalence.

Recall that we have a canonical object
$\A[X]$ in $\md[{\A[X]\tens\A[X^a]}]$ (see \S~\ref{section:algebroid}). We
identify $\A[X]$ with an $(\A[X]\etens\A[X^a])$-module supported by the
diagonal $\Delta_X$ of $X\times X^a$. In fact,
it has a structure of $\A[X\times X^a]$-module.
More generally, we have:
\begin{lemma}\label{lem:inve}
Let $\shm$ be an $(\A\tens\A[X^a])$-module.
\banum
\item
The following conditions are equivalent:
\bnum
\item $\shm$ is a bi-invertible $(\A\tens\A[X^a])$-module
\ro see {\rm Definition~\ref{def:biinv1})},
\item
$\shm$ is invertible as an $\A$-module
\ro see {\rm Definition~\ref{def:invertible})}, that is,
$\shm$ is locally isomorphic to $\A$ as an $\A$-module,
\item
$\shm$ is invertible as an $\A[X^a]$-module.
\enum\label{cond:biinv}

\item
Under the equivalent conditions in {\rm(a)}, 
$\oim{\de}\shm\to\A[X\times X^a]\tens[{\A[X]\etens\A[X^a]}]\oim{\de}\shm$ 
is an isomorphism
and $\oim{\de}\shm$ has a structure of an $\A[X\times X^a]$-module.
Moreover, $\oim{\de}{\shm}$ is a simple $\A[X\times X^a]$-module along 
the diagonal of $X\times X^a$.
\item
Conversely, if $\shn$ is a simple $\A[X\times X^a]$-module along
the diagonal of $X\times X^a$, then $\opb{\de}\shn$ satisfies 
the equivalent conditions {\rm (a)~(i)--(iii)}. 
\eanum
\end{lemma}
\begin{proof}
The statement is local and we may assume that $\A[X]=(\OO[X]\forl,\star)$.

\vspace{0.4em}
\noindent
(a)\quad Assume (ii) and take a generator $u\in\shm$ as an $\A$-module.
Then for any $a\in\A$, there exists a unique $\theta(a)\in\A$
such that $ua=\theta(a)u$.
Then $\theta\cl\A\to\A$ gives a $\coro$-algebra endomorphism of $\A$.
Hence $\theta$ is an isomorphism by Proposition~\ref{pro:gauge1}.
Thus we obtain (i).
Similarly (iii) implies (i).

\vspace{0.4em}
\noindent
(b)\quad Let us choose $u\in\shm$ as in (a) and identify
$\shm$ with $\sho_X\forl$ that
we regard as a sheaf supported by the diagonal.
The action of
$\A[X]\tens\A[X]^\rop$ on $\shm$ can be expressed by differential operators.
Namely, there exist differential operators
$\{S_i(x,\partial_{x_1},\partial_{x_2},\partial_{x_3})\}_{i\in\N}$
such that
\eqn
&&f\star a\star \theta(g)=
\sum_i\bl S_i(x,\partial_{x_1},\partial_{x_2},\partial_{x_3})
f(x_1)g(x_2)a(x_3)\br\vert_{x_1=x_2=x_3=x}\hbar^i\\
&& \hs{30ex}\text{for $f$, $g\in \A[X]$ and $a\in\OO{\forl}$.}
\eneqn
Then this action extends to an action of
$\A[X\times X^a]$ by setting
\eqn
&&f(x,y)\star a(x)=
\sum_i\bl S_i(x,\partial_{x_1},\partial_{x_2},\partial_{x_3})
f(x_1,x_2)a(x_3)\br\vert_{x_1=x_2=x_3=x}\hbar^i\\
&&\hs{30ex}\text{for $f\in \A[X\times X^a]$ and $a\in\OO{\forl}$.}
\eneqn
We denote by $\widetilde\shm$ the $\A[X\times X^a]$-module thus obtained.
Then, as an $(\A\tens\A[X^a])$-module, it is isomorphic to $\shm$.
Hence $\widetilde\shm$ is a locally finitely generated $\A[X\times X^a]$-module.
Since $\hbar^n\widetilde\shm/\hbar^{n+1}\widetilde\shm$
is isomorphic to $\sho_X$,
$\widetilde\shm$ is a coherent $\A[X\times X^a]$-module
by Theorem~\ref{th:formalfini1}~\eqref{th:coh:coh}.

\vspace{0.4em}
\noindent
Let $\tilde\shi$ be the annihilator of
$u\in\shm\simeq \widetilde\shm$. Then $\tilde\shi$ is a coherent 
left ideal of $\A[X\times X^a]$.
In the exact sequence
\eqn
&&\tor^{\coro}_{1}(\widetilde\shm,\C)\to 
\tilde\shi/\hbar\tilde\shi\to \A[X\times X^a]/\hbar\A[X\times X^a]
\to\widetilde\shm/\hbar\widetilde\shm\to0,
\eneqn
$\tor^{\coro}_{1}(\widetilde\shm,\C)$ vanishes.
Therefore we obtain an exact sequence
\eqn
&&0\to \tilde\shi/\hbar\tilde\shi\to\sho_{X\times X^a}\to\sho_X\to0,
\eneqn
and $\tilde\shi/\hbar\tilde\shi$ is isomorphic to
the defining ideal $I_\Delta\subset \sho_{X\times X^a}$
of the diagonal set $\Delta\subset X\times X^a$. This shows that 
$\widetilde\shm$ is simple along the diagonal.

\vspace{0.4em}
\noindent
Denote  by $\shi'$ the left ideal of $\A[X]\tens\A[X]^\rop$
generated by the sections $\{a\tens 1-1\tens \theta(a)\}$ where 
$a$ ranges over the family of sections of $\A[X]$ and by 
$\shi$ the left ideal of $\A[X\times X^a]$ generated by 
$\shi'$. Set $\shm'\seteq\A[X\times X^a]\tens[{\A[X]\etens\A[X^a]}]\shm$.
We have:
\eqn
&&\shm\simeq (\A[X]\tens\A[X^a])/\shi',\\
&&\shm'\simeq \A[X\times X^a]/\shi.
\eneqn
There exists a surjective $\A[X\times X^a]$-linear morphism
$\shm'\epito\widetilde\shm$,
and hence $\shi\subset\tilde\shi$.
Since $\shi/\hbar\shi\to \tilde\shi/\hbar\tilde\shi\simeq I_\Delta$
is surjective,
we conclude that $\shi=\tilde\shi$.
Hence we obtain $\shm'\simeq\widetilde\shm$.

\smallskip
\noindent
(c)\ 
By the assumption, 
$\oim{p_1}\gr(\shn)\simeq\gr(\opb{\de}\shn)$ is an invertible $\sho_X$-module.
where $p_1\cl X\times X^a\to X$ is the projection. Hence Theorem 
\ref{th:formalfini1} \eqref{th:crcoh} implies that $\opb{\de}\shn$ is a coherent $\A$-module.
It is locally isomorphic to $\A$ by Lemma \ref{lem:free} 
because $\gr(\opb{\de}\shn)$
is locally isomorphic to $\sho_X$.
\end{proof}

Thus we obtain:
\begin{proposition}
The category of bi-invertible $(\A[X]\tens\A[X^a])$-modules
is 
equivalent to the category of coherent
$\A[X\times X^a]$-modules simple along the diagonal.
\end{proposition}

\begin{definition}\label{def:CD}
We regard  $\oim{\delta_X}\A$ as an $\A[X\times X^a]$-module supported by the diagonal
and denote it by $\dA$. 
\glossary{canonical module associated with the diagonal}%
\index{CX@$\dA$}%
We call it the canonical module associated with the diagonal.
\end{definition}

The next corollary immediately follows from Lemma \ref{lem:inve}.
\begin{corollary}\label{co:existCD}
The  $\A[X\times X^a]$-module $\dA$
is  coherent and simple along the diagonal.
Moreover, $\A[X\times X^a]\tens_{\A\etens\A[X^a]}\dA\to\dA$
is an isomorphism in $\md[{\A[X\times X^a]}]$, and
$\A[X]\to \de^{-1}(\dA)$ is an isomorphism in $\md[{\A[X]\tens\A[X^a]}]$.
\end{corollary}

The next result is obvious.
\begin{lemma}
Let $Y$ be another complex manifold endowed 
with a $\DQ$-algebroid  $\A[Y]$. Then, there is a natural
isomorphism
$\dA\ldetens\dA[Y]\simeq\dA[X\times Y]$.
Here, we identify $(X\times X^a)\times (Y\times Y^a)$
with $(X\times Y)\times (X\times Y)^a$.
\end{lemma}
\begin{definition}\label{def:biinv2}
We say that $\shp\in\Derb(\A[X]\tens\A[X^a])$
is bi-invertible
\glossary{bi-invertible}%
if $\shp$ is concentrated to  some degree $n$
and $H^n(\shp)$ is bi-invertible (see Definition~\ref{def:biinv1}).
\end{definition}
We sometimes consider a  bi-invertible $(\A[X]\tens\A[X^a])$-module
as an object of $\Derb_\coh(\A[X\times X^a])$ supported by the diagonal.

For a pair of bi-invertible $(\A[X]\tens\A[X^a])$-modules
$\shp_1$ and $\shp_2$,
$\shp_1\lltens[{\A}]\shp_2$ is also a bi-invertible
$(\A[X]\tens\A[X^a])$-module.
Hence the category of bi-invertible $(\A[X]\tens\A[X^a])$-modules has
a structure of a tensor category (see e.g. \cite[\S\,4.2]{K-S3}).
It is easy to see that $\dA$ is a unit object. Namely, 
for any bi-invertible $(\A[X]\tens\A[X^a])$-module $\shp$,
 we have:
\eqn
&&\dA\lltens[{\A}]\shp\simeq \shp\lltens[{\A}]\dA\simeq\shp.
\eneqn
We have
\eqn
&&\shp\lltens[{\A}]\rhom[{\A[X]}](\shp,\A)\isoto\dA,\\
&&\rhom[{\A[X^a]}](\shp,\A)\lltens[{\A}]\shp\isoto\dA.
\eneqn
Hence we have $\rhom[{\A[X]}](\shp,\A)\simeq \rhom[{\A[X^a]}](\shp,\A)$.
\begin{definition}
For a bi-invertible $(\A[X]\tens\A[X^a])$-module $\shp$, we set
\index{tensorminusone@$\shp^{\otimes-1}$}%
\eqn
&&\shp^{\otimes-1}=\rhom[{\A[X]}](\shp,\A)\simeq
\rhom[{\A[X^a]}](\shp,\A).
\eneqn
\end{definition}
Hence we have
$$\shp^{\otimes-1}\lltens[{\A}]\shp
\simeq\shp\lltens[{\A}]\shp^{\otimes-1}\simeq\dA.$$
Note that, for two bi-invertible $(\A[X]\tens\A[X^a])$-modules
$\shp_1$ and $\shp_2$, we have 
\eqn
\rhom[\A](\shp_1,\shp_2)&\simeq&\shp_1^{\otimes-1}\lltens[\A]\shp_2,\\
\rhom[{\A[X^a]}](\shp_1,\shp_2)&\simeq&\shp_2\lltens[\A]\shp_1^{\otimes-1}.
\eneqn
For a bi-invertible $(\A[X]\tens\A[X^a])$-module $\shp$ and
$\shm,\shn\in\RD(\A[X\times Y\times Z])$,
we have the isomorphism
\eq
&&\rhom[{\A[X\times Y]}](\shm,\shn)\simeq
\rhom[{\A[X\times Y]}](\shp\lltens[\A]\shm,\shp\lltens[\A]\shn)
\label{eq:invtens}
\eneq
in $\Der(\coro_{X\times Y}\etens \A[Z])$.

\begin{remark}\label{rem:opposed}
Although it is sometimes convenient to identify $(X\times Y^a)^a$ with
$Y\times X^a$, we do not take this point view in this Note.
We identify $(X\times Y^a)^a$ with $X^a\times Y$. 
Hence, for example, we have functors
\eqn
\RDAA[X\times Y^a]&\cl&\Db(\A[X\times Y^a])\To\Db(\A[X^a\times Y]),\\
\RDAA[X\times X^a]&\cl&\Db(\A[X\times X^a])\To\Db(\A[X^a\times X]).
\eneqn
\end{remark}

The next result may be useful.
\begin{lemma}
\bnum
\item
Let $X$ and $Y$ be manifolds 
endowed with $\DQ$-algebroids $\A[X]$ and $\A[Y]$, let 
$\shm$ be an $\A[X\times Y^a]$-module and let $\shq$ be a bi-invertible 
$(\A[Y]\tens\A[Y^a])$-module. Then
\eqn
&&\RDAA[X\times Y^a](\shm\tens[{\A[Y]}]\shq)
\simeq\shq^{\otimes-1}\tens[{\A[Y]}]\RDAA[X\times Y^a](\shm).
\eneqn
\item
Let $\shp$  and $\shq$ be bi-invertible
$(\A[X]\tens\A[X^a])$-modules. Then
\eqn
&&\RDAA[X\times X^a](\shp\tens[{\A}]\shq)\simeq 
\shq^{\otimes-1}\tens[{\A}]\RDAA[X\times X^a]\shp\simeq 
\RDAA[X\times X^a]\shq\tens[{\A}]\shp^{\otimes-1},\\
&&\RDAA[X\times X^a]\dA\tens[{\A}]\shp
\simeq\shp\tens[{\A}]\RDAA[X\times X^a]\dA
\simeq\RDAA[X\times X^a](\shp^{\otimes-1}),\\
&&(\RDAA[X\times X^a]\dA)^{\otimes-1}\tens[{\A}]\shp\simeq
                    \shp\tens[{\A}](\RDAA[X\times X^a]\dA)^{\otimes-1}.
\eneqn
\enum
\end{lemma}
\begin{proof}
(i) We have the isomorphism
\eqn
\RDAA[X\times Y^a](\shm\tens[{\A[Y]}]\shq)&=&
\hom[{\A[X\times Y^a]}](\shm\tens[{\A[Y]}]\shq,\A[X\times Y^a])\\
&\simeq&\hom[{\A[X\times Y^a]}](\shm,\A[X\times Y^a]\tens[{\A[Y]}]\shq^{\otimes-1})\\
&\simeq&\hom[{\A[X\times Y^a]}](\shm,\shq^{\otimes-1}\tens[{\A[Y]}]\A[X\times Y^a])\\
&\simeq&\shq^{\otimes-1}\tens[{\A[Y]}]\RDAA[X\times Y^a](\shm).
\eneqn

\noindent
(ii) The first isomorphism follows from (i) and the the second is
proved similarly.
The two last isomorphisms follow.
\end{proof}

The next result follows immediately from Corollary~\ref{co:existCD}.
\begin{lemma}\label{le:homCdelta}
Let $\shm\in\RD^\Rb(\A[X])$, $\shl\in\RD^\Rb_\coh(\A[X])$ and 
$\shn\in\RD^\Rb(\A[X^a])$. 
Identifying $\Delta_X$ and $X$, there are natural isomorphisms
\eqn
\shm\simeq \A[X]\lltens[{\A[X]}]\shm
         &\simeq& \rhom[{\A[X]}](\A[X],\shm)\quad\text{in $\Der(\A)$,}\\
\shn\lltens[{\A[X]}]\shm
   &\simeq& (\shn\ldetens\shm)\lltens[{\A[X\times X^a]}]\dA
\quad\text{in $\Der(\coro_X)$,}\\
\rhom[{\A[X]}](\shl,\shm)
  &\simeq& \RD'_{\rma}\shl\lltens[{\A[X]}]\shm\quad\text{in $\Der(\coro_X)$,}\\
\rhom[{\A[X]}](\shm,\shl)
&\simeq& \rhom[{\A[X\times X^a]}](\shm\ldetens\RDA\shl,\dA)
\quad\text{in $\Der(\coro_X)$.}
\eneqn
\end{lemma}

\section{Dualizing complex for $\DQ$-algebroids}\label{section:dual}

\subsubsection*{The algebroid $\DA$}
We have seen that
the $\coro$-algebra $\DA\subset\shend_{\coro}(\A)$ is
well-defined for a $\DQ$-algebra $\A$ on $X$.

 Now let  $\A$ be a $\DQ$-algebroid.
Then we can regard $\A$ as an $(\A\tens \A^\rop)$-module.
In \S~\ref{section:algebroid}, 
we have defined the $\coro$-algebroid $\shend_{\coro}(\A)$
and introduced a functor of $\coro$-algebroids
$\A\tens\A^\rop\to\shend_{\coro}(\A)$.

\begin{definition}
The $\coro$-algebroid $\DA$ 
\index{DA@$\shend_{\coro}(\A)$}%
is the $\coro$-substack of 
$\shend_{\coro}(\A)$ associated to the prestack $\sts$ defined as
follows. The objects of $\sts$ are those of $\A\tens\A^\rop$. For
$\sigma_1,\sigma_2\in\A\tens\A^\rop$, with 
 $\sigma_1=\tau_1\tens\lambda_1^\rop$,
 $\sigma_2=\tau_2\tens\lambda_2^\rop$,
we choose isomorphisms 
$\phi_i\cl \tau_i\simeq\lambda_i$
($i=1,2$) and $\phi_3\cl \tau_1\simeq\tau_2$. 
Set $\shb=\shend_{\A}(\lambda_1)$. It is a $\DQ$-algebra.
The isomorphisms $\phi_i$ ($i=1,2,3$) induce an isomorphism
\eqn
\psi\cl\hom[\coro](\shb,\shb)&\isoto&
\hom[\coro](\hom(\lambda_1,\tau_1),\hom(\lambda_2,\tau_2))\\
&\isoto&\hom[{\coro}](\A(\sigma_1), \A(\sigma_2)).
\eneqn
We define 
$\hom[{\sts}](\sigma_1,\sigma_2)\subset\hom[\coro](\A(\sigma_1),\A(\sigma_2))$
as the image of $\shd^\shb_X$ by $\psi$. (This does not depend on the choice
of the isomorphism $\phi_i$ ($i=1,2,3$) in virtue of
Proposition~\ref{pro:gauge1}.)
\end{definition}
Then there are functors of $\coro$-algebroids 
\eqn
&&\A[X]\tens\A[X^a]\to\opb{\de}\A[X\times X^a]\to\DA\to\shend_{\coro}(\A)
\eneqn
and $\A$ may be regarded as an object of $\Mod(\DA)$.

\begin{proposition}\label{prop:shda}
\bnum
\item
The $\coro$-algebroid $\shend_{\coro}(\A)$ is equivalent to the $\coro$-algebroid
$\shend_{\coro}(\OOh)$ \lp regarding the $\coro$-algebra
$\shend_{\coro}(\OOh)$ as a $\coro$-algebroid\rp.
\item
The equivalence in {\rm (i)} induces an equivalence
of $\coro$-algebroids $\DA\simeq \shd_X\forl$.
\item
The equivalence in {\rm (ii)} induces an equivalence of $\coro$-linear stacks 
$$\stkMod(\DA)\simeq\stkMod(\shd_X\forl).$$ Moreover,
the $\DA$-module $\A$ is sent to the $\shd_X\forl$-module $\OOh$ 
by this equivalence.

\item
The equivalence in {\rm (ii)} also induces
an equivalence of $\C$-algebroids 
$$\gr(\DA)\simeq \shd_X,$$ 
and  an equivalence of $\C$-linear stacks 
$\stkMod({\gr(\DA)})\simeq \stkMod(\shd_X)$.
Moreover 
the $\gr(\DA)$-module $\gr(\A)$ is sent to the $\shd_X$-module $\OO$  
by this equivalence.
\enum
\end{proposition}
\begin{proof}
Recall first that for two $\coro$-algebroids $\shb$ and $\shb'$, 
to give an equivalence of $\coro$-algebroids $\shb\simeq \shb'$ is
equivalent to giving a bi-invertible $\shb^\rop\tens\shb'$-module
(Lemma~\ref{lem:algiso}).

\vspace{0.2em}
\noindent
(i) follows from Lemma~\ref{lem:endstack}. More precisely,  
we define an $\bigl(\shend_{\coro}(\A)\tens(\shend_{\coro}(\OOh))^\rop\bigr)$-module 
$\shl'$ as follows.
For $\sigma=(\sigma_1\tens\sigma_2^\rop)\in\A\tens\A^\rop$, set
\eqn
&&\shl'(\sigma)
\eqdot\hom[{\coro}](\OOh,\hom[{\A}](\sigma_2,\sigma_1)).
\eneqn
 Clearly, $\shl'$ is bi-invertible.

\vspace{0.2em}
\noindent
(ii) For $\sigma=(\sigma_1\tens\sigma_2^\rop)\in\A\tens\A^\rop$,
let us choose an isomorphism $\psi\cl\sigma_1\isoto\sigma_2$
and a standard isomorphism $\tw{\phi}\cl\OO{\forl}\isoto\shend_{\A}(\sigma_1)$.
Then they give an isomorphism
\eqn
&&f\cl\OOh\isoto\hom[{\A}](\sigma_2,\sigma_1).
\eneqn
We define a $(\DA\tens\Dh^\rop)$-submodule $\shl$ of $\shl'$ as follows:%
%
\quad
let
$\shl(\sigma)$ be the $\Dh^\rop$-submodule of 
$\shl'(\sigma)$ generated by $f$.
Then $\shl(\sigma)$ coincides 
with the submodule generated by $f$
over the $\coro$-algebra 
$\shend_{\DA}(\sigma)\subset\shend_{\coro}(\hom[{\A}](\sigma_2,\sigma_1))$. 
Moreover, $\shl(\sigma)$ does not depend on the choice
of $\psi$ and $\tw{\phi}$. It is easy to see that $\shl$ is a
bi-invertible $(\DA\tens\Dh^\rop)$-module. 

\vspace{0.2em}
\noindent
(iii) The $(\DA\tens\Dh^\rop)$-module $\shl$
gives an equivalence of categories
\eq
&&\shl\tens[\Dh]\scbul\cl\Mod(\Dh)\isoto\Mod(\DA),
\label{eq:DhDA}
\eneq
which is isomorphic to the functor induced by the
algebroid equivalence $\DA\isoto\Dh$.
Consider the $(\Dh\tens(\DA)^\rop)$-module
\eqn
&&\shl^*\seteq\hom[{\DA}](\shl,\DA).
\eneqn 
A quasi-inverse of the equivalence \eqref{eq:DhDA}
is given by
\eqn
&&\shl^*\tens[\DA]\scbul\simeq\hom[{\DA}](\shl,\scbul)
\cl\Mod(\DA)\isoto\Mod(\Dh).
\eneqn
The results follow.
\end{proof}

\subsubsection*{Dualizing complex}
Let $\A$ be a $\DQ$-algebroid on $X$. 
We shall construct a deformation of the sheaf of differential
forms of maximal degree and then the dualizing complex for $\A$. 
\begin{lemma}\label{le:DAalg1}
\bnum
\item
$\A$ has locally a resolution of length $d_X$
by free $\DA$-modules of finite rank.
\item $\gr(\ext[\DA]{d_X}(\A,\DA))\simeq\Omega_X$.
\lp Note that $\gr(\ext[\DA]{d_X}(\A,\DA))$ is a module over
$\bA\tens[\OO]\bA[X^a]\simeq\OO$ by \eqref{eq:Ostacks2}\rp.
\item
$\ext[\DA]{i}(\A,\DA)=0$ for $i\not=d_X$.
\enum
\end{lemma}
\begin{proof}
We have 
$\DA\simeq\shd_X\forl$ and $\A\simeq\OO{\forl}$ as  $\DA$-modules. 
Then the results 
follow from 
\eqn
&&\rhom[{\shd_X\forl}](\OO{\forl},\shd_X\forl)\simeq
\bl\Omega_X\forl\br[-d_X].
\eneqn

\vspace{0.4em}
\noindent
(ii) follows from
\eqn
\gr(\rhom[\DA](\A,\DA))
&\simeq&\rhom[{\gr(\DA)}](\gr(\A),\gr(\DA))\\
&\simeq&\rhom[{\shd_X}](\OO,\shd_X)\simeq\Omega_X[-d_X].
\eneqn
\end{proof}

We set \index{OmegaA@$\OA$}%
\eq\label{eq:defOA}
&&\OA\seteq\ext[\DA]{d_X}(\A,\DA)\in\md[{\A\tens\A[X^a]}].
\eneq

\begin{lemma}
The $(\A\otimes\A^\rop)$-module $\OA$ is bi-invertible.
\end{lemma}
\Proof
Under the equivalence $\DA\simeq\shd_X\forl$,
we have $\OA\simeq\Omega_X\forl$. Hence
we have an isomorphism
$\OA\isoto\prolim[n]\OA/\hbar^n\OA$.
Since $\gr(\OA)\simeq\Omega_X$ is a coherent $\gr(\A)$-module,
$\OA$ is a coherent $\A$-module by Theorem \ref{th:formalfini1}~(iv).
Since $\gr(\OA)$ is an invertible $\sho_X$-module and
$\OA$ has no $\hbar$-torsion, $\OA$ is locally isomorphic to $\A$
as an $\A$-module. Hence $\OA$ is a bi-invertible
$(\A^\op\otimes\A)$-module
by Lemma~\ref{lem:inve}~\eqref{cond:biinv}.
\QED

\begin{lemma}\label{le:DAalg2}
One has the isomorphisms
\eq\label{eq:isoDAalg1}
&&\OA\lltens[\DA]\A{}[-d_X]\simeq \rhom[\DA](\A,\A)\simeq\coro_X.
\eneq
\end{lemma}
\begin{proof}
The first isomorphism is obvious by Lemma~\ref{le:DAalg1}. 
Hence, it is enough to prove that
the natural morphism
$\coro_X\to \rhom[\DA](\A,\A)$ is an isomorphism.
By the equivalence $\DA\simeq\Dh$, we may assume that $\A=\sho_X\forl$ and 
$\DA=\shd_X\forl$.
Then $\rhom[\DA](\A,\A)$ is represented by 
an infinite product of the de Rham complexes:
$\prod_n\hbar^n\Omega^\scbul_X$.
Then the assertion follows from a classical result:
$\Omega^\scbul_X(U)$ is quasi-isomorphic to $\C$
when $U$ is a contractible Stein open subset.
\end{proof}

Note that $\OA$ and $\OA[X^a]$ are isomorphic as
$\A\tens\A[X^a]$-modules.

\begin{definition}
We set \index{omegaA@$\oA$}%
\eqn
&&\oA\seteq\oim{\de}\OA{}[d_X]\simeq\oim{\de}\rhom[{\DA}](\A,\DA)[2d_X]
\in\Derb(\A[X\times X^a])
\eneqn
and call $\oA$ the $\A{}$-dualizing sheaf.
\glossary{dualizing sheaf}%
\end{definition}
Note that $\oA$ is bi-invertible (see Definition~\ref{def:biinv2}).
Using \eqref{eq:isoDAalg1} and the morphism
$\oim{\delta_X}\OA\lltens[{\A[X\times X^a]}]\dA \to \OA\lltens[\DA]\A$, we get
the natural morphism
\eq\label{eq:morDAalg2}
&&\oA[X^a]\lltens[{\A[X\times X^a]}]\dA\to\oim{\de} \coro_X\,[2d_X].
\eneq
Applying the functor $\gr$ to the above morphisms, we get the morphism 
\eq\label{eq:morDOalg3}
&&\ba{lcl}
\oim{\delta_X}(\gr\oA[X^a])\lltens[{\gr\A[X\times X^a]}](\oim{\delta_X}\gr\dA)
&\to&
\oim{\delta_X}(\C_X\,[2d_X]),
\ea
\eneq
which coincides with the morphism derived from
\eq\label{eq:morDOalg2}
&&
\de^{-1}\bl\oim{\delta_X}(\gr\oA[X^a])
\lltens[{\gr\A[X\times X^a]}](\oim{\delta_X}\gr\dA)\br
\to\Omega_X\,[d_X]\to
\C_X\,[2d_X].
\eneq
Here we used the functor of algebroids $\de^{-1}(\gr\A[X\times X^a])\to\OO$.

Let $Y$ be another manifold endowed with a $\DQ$-algebroid $\A[Y]$. 
We introduce the notation:\index{omegaA/\bar@$\oA[X\times Y/Y]$}%
\eqn
&& \oA[X\times Y/Y]=
\oA[X]\ldetens\dA[Y]\in\RD^\Rb(\A[X\times X^a\times Y\times Y^a]).
\eneqn
Then $\oA[X\times Y/Y]$ also belongs to 
$\Derb\bigl(({\DA})^\rop\etens \A[Y\times Y^a]\bigr)$, 
and we have an isomorphism 
$\oA[X\times Y/Y]\lltens[{\DA}]\A\simeq \coro_X\etens\A[Y]$.
Hence we have a canonical morphism 
\eq
&&\oA[X^a\times Y/Y]\lltens[{\A[X\times X^a]}]\dA\to (\coro_X\etens\dA[Y])[2d_X]
\label{mor:residue}
\eneq
in $\Derb(\coro_X\etens\A[Y\times Y^a])$.

The proof of the following fundamental result will be given
later at the end of \S~\ref{sec:convdual}.

\begin{theorem}\label{th:oARDA}
We have the isomorphism 
\eq\label{eq:oARDA}
&&\oA\simeq(\RDAA[X^a\times X]\dA[X^a])^{\otimes -1}
\quad \mbox{in $\RD^\Rb(\A[X\times X^a])$.}
\eneq
\end{theorem}
Note that in formula \eqref{eq:oARDA}, $\RDAA[X^a\times X]$ is the dual over $\A[X^a\times X]$
and $(\scbul)^{\otimes -1}$ is the dual over $\A[X]$. 

\begin{corollary}\label{cor:dualDdualA}
For $\shm\in\RD^\Rb(\A[X\times X^a\times Y])$, we have
\eqn
\dA[X^a]\lltens[{\A[X\times X^a]}]\shm&\simeq&
\rhom[{\A[X\times X^a]}](\dA,\oA\lltens[\A]\shm)\\
&\simeq&\rhom[{\A[X\times X^a]}](\dA,\shm\lltens[\A]\oA).
\eneqn
\end{corollary}
\begin{proof}
We have 
\eqn
\dA[X^a]\lltens[{\A[X\times X^a]}]\shm
&\simeq&
\rhom[{\A[X\times X^a]}](\RDAA[X^a\times X]\dA[X^a],\shm)\\
&\simeq&
\rhom[{\A[X\times X^a]}](\oA\lltens[\A]\RDAA[X\times X^a]\dA[X^a],\;\oA\lltens[\A]\shm)\\
&\simeq&
\rhom[{\A[X\times X^a]}](\dA,\;\oA\lltens[\A]\shm).
\eneqn
The other isomorphism is similarly proved. 
\end{proof}

One shall be aware that, although $\OA$
is locally isomorphic to $\A$ as an $\A$-module, 
it is not always locally isomorphic to $\A$ as an 
$\A\tens\A[X^a]$-module. 

\begin{example}\label{exa:notisoC}
Let $X=\C^2$ with coordinates $(x_1,x_2)$ and let $\A$ be the
$\DQ$-algebra given by the relation
\eqn
&&[x_1,x_2]=\hbar\,x_1.
\eneqn
Let $(y_1,y_2)$ denotes the coordinates on $X^a$. Hence
\eqn
&&[y_1,y_2]=-\hbar\,y_1.
\eneqn
Then $\dA$ is the $\A[X\times X^a]$-module $\A[X\times X^a]\cdot u$
where the generator $u$ satisfies 
$(x_i-y_i)\cdot u=0$ ($i=1,2$). 
Therefore $\dA$ is quasi-isomorphic to the complex
\eq\label{eq:KoszulAA}
&&0\to \A[X\times X^a]\to[\alpha]\A[X\times X^a]^{\oplus 2}\to[\beta]\A[X\times X^a]\to 0,
\eneq
where  $\A[X\times X^a]$ on the right is in degree $0$, 
$\alpha(a)=(-a(x_2-y_2+\hbar),a(x_1-y_1))$ and $\beta(b,c)=b(x_1-y_1)+c(x_2-y_2)$.

It follows that $\RDA(\dA)\,[2]$ is isomorphic to $\A[X\times X^a]\cdot w$ where 
the generator $w$ satisfies $(x_1-y_1)\cdot w=0$,  $(y_2-x_2+\hbar)\cdot w=0$. 
The modules $\RDA(\dA)\,[2]$ and $\dA$ are isomorphic on $x_1\not=0$
by $u\leftrightarrow x_1w$. 
However, $\RDA(\dA)\,[2]$ and $\dA$ 
are not isomorphic on a neighborhood of $x_1=0$.
Indeed if they were isomorphic by $u\leftrightarrow a w$ for $a\in\A$,
then $x_1a=ax_1$ and $x_2a=a(x_2-\hbar)$.
Then $\{x_2,\sigma_0(a)\}=-\sigma_0(a)$.
Since $\{x_2,\scbul\}=-x_1\partial_{x_1}$, we have
$x_1\partial_{x_1}\sigma_0(a)=\sigma_0(a)$, which contradicts the fact that 
$\sigma_0(a)$ is invertible.
\end{example}

\begin{remark}
The fact that $\RDA\dA[X]$ is concentrated in a single degree and
plays the role of a dualizing complex in the sense of \cite{VdB0} was
already proved (in a more restrictive situation) in \cite{Do,Do-Ru}.
\end{remark}

\section{Almost free resolutions}\label{section:A}

We recall here and adapt to the framework of algebroids 
some results of \cite{K-S2}.

In this section, $\cora$ denotes a commutative unital ring,
$X$ a paracompact and locally compact space 
and $\sha$ a $\cora$-algebroid on $X$.

Let us take a family $\shs$ of open subsets of $X$.
We assume the following two conditions on $\shs$:
\eq\label{hyp:shscov}
&& \left\{
\parbox{60ex}{
\bnum
\item
for any $x\in X$, $\set{U\in \shs}{x\in U}$ is a neighborhood system of $x$,
\item  for $U$, $V\in \shs$, $U\cap V$ is a finite union of
open subsets belonging to $\shs$.
\enum
}
\right.
\eneq
Recall that invertible modules are defined in Definition~\ref{def:invertible}.
\begin{definition}\label{def:almostfree}
\bnum
\item
We define the additive category $\mdaf[\sha]$ of $\shs$-almost free
$\sha$-modules as follows.\glossary{almost free!$\sha$-module}%
\index{modalmostfree@$\mdaf[\sha]$}%
\banum
\item
An object of $\mdaf[\sha]$ is the data of 
$\{I,\{U_i,U'_i,L_i\}_{i\in I} \}$ where $I$ is an index set,  
$U_i$ and $U'_i$ are open subsets of $X$, $U_i\in\shs$, $\ol{U_i}\subset U'_i$,
the family  $\{U'_i\}_{i\in I}$ is locally finite and $L_i$  
is an invertible $\sha\vert_{U'_i}$-module. 
\item
Let $N=\{J,\{V_j,V'_j,K_j\}_{j\in J} \}$ and 
$M=\{I,\{U_i,U'_i,L_i\}_{i\in I} \}$ be two  objects of $\mdaf[\sha]$. 
A morphism $u\cl N\to M$  is the data  of
$u_{ij}\in\sect(\ol{V_j};\hom[{\sha}](K_j,L_i))$
for all $(i,j)\in I\times J$ such that $V_j\subset U_i$.
\item
The composition of morphisms is the natural one.
\item 
We denote by $\Phi\cl\mdaf[\sha]\to \md[\sha]$ the functor which sends
$\{I,\{U_i,U'_i,L_i\}_{i\in I} \}$ to $\bigoplus_{i\in I}(L_i)_{U_i}$
and which sends an element $u_{ij}$ of $\sect(\ol{V_j};\hom[{\sha}](K_j,L_i))$ to
its image in $\Hom[\sha]((K_j)_{V_j},(L_i)_{U_i})$ if $V_j \subset U_i$
and $0$ otherwise.
\eanum
\item
Similarly, we define the additive category $\mdafd$ as follows.
\banum
\item The set of objects of
$\mdafd$ is the same as the one of $\mdaf$.
\index{modalmostfree2@$\mdafd[\sha]$}%
\item
Let $N=\{J,\{V_j,V'_j,K_j\}_{j\in J} \}$ and 
$M=\{I,\{U_i,U'_i,L_i\}_{i\in I} \}$ be two  objects of $\mdaf[\sha]$. 
A morphism $u\cl N\to M$  is the data 
of $u_{ij}\in\sect(\ol{U_i};\hom[{\sha}](K_j,L_i))$
for all $(i,j)\in I\times J$ such that 
$U_i\subset V_j$.
\item
The composition of morphisms is the natural one.
\item 
We denote by $\Psi\cl\mdafd[\sha]\to \md[\sha]$ the functor which sends
$\{I,\{U_i,U'_i,L_i\}_{i\in I} \}$ to 
$\bigoplus_{i\in I}\sect_{U_i}(L_i)$
and which sends an element $u_{ij}$ of $\sect(\ol{U_i};\hom[{\sha}](K_j,L_i))$ to
its image in $\Hom[\sha](\sect_{V_j}(K_j),\sect_{U_i}(L_i))$ if $U_i\subset V_j$
and $0$ otherwise.
\eanum
\enum
\end{definition}
Note that
$\mdafd$ is equivalent to $\mdaf[\sha^\rop]^\rop$ by the functor
which sends $\{I,\{U_i,U'_i,L_i\}_{i\in I} \}$
to $\{I,\{U_i,U'_i,\hom[{\sha}](L_i,\sha)\}_{i\in I} \}$.

Recall that for an additive category $\shc$, we denote by $\RC^-(\shc)$ 
(resp.\ $\RC^+(\shc)$) 
\index{Cminus@$\RC^-(\shc)$}%
\index{Cplus@$\RC^+(\shc)$}%
the category of 
complexes of $\shc$ bounded from above (resp.\ from below).

The following theorem is proved similarly as  in \cite[Appendix]{K-S2}. 

\begin{theorem}\label{th:mainafr}
Let $\sha$ be a left coherent algebroid and let $\shm\in \RD^-_{\coh}(\sha)$. Then 
there exist $L^\scbul\in\RC^-(\mdaf[\sha])$ 
and an isomorphism $\Phi(L^\scbul)\simeq\shm$ in $\RD^-(\sha)$.
\end{theorem}

There is a dual version of Theorem~\ref{th:mainafr}. 
\begin{theorem}\label{th:mainafrdual}
Assume 
\banum
\item 
$\sha$ being regarded as an object of $\md[{\sha\tens\sha^\rop}]$,
$\rsect_U(\sha)$ is concentrated in degree $0$ for all $U\in\shs$,
\item $\sha$ is a right and left coherent algebroid,
\item
there exists an integer $d$ such that,
for any open subset $U$, any coherent $\sha\vert_U$-module 
admits locally a finite free resolution of length $d$.
\eanum
Let $\shm\in \RD^+_{\coh}(\sha)$. Then there exist 
$L^\scbul\in\RC^+(\mdafd[\sha])$
and an isomorphism $\shm\simeq\Psi(L^\scbul)$ in $\RD^+(\sha)$.
\end{theorem}
\begin{proof}
Denote by $\RD$ the duality functor 
$\rhom[\sha](\scbul,\sha)$ and keep the same notation 
with $\sha^\rop$ instead of $\sha$.
This functor sends $\RD^+_\coh(\sha)$ to $\RD^-_\coh(\sha^\rop)$
by (c). It also sends $\RD^-_\coh(\sha^\rop)$ to $\RD^+_\coh(\sha)$,
and the composition
$$\RD^+_\coh(\sha)\To[{\RD}]\RD^-_\coh(\sha^\rop)\To[{\RD}]\RD^+_\coh(\sha)$$
is isomorphic to the identity functor.

On the other hand,
if $L$ is an invertible
$\sha^\rop$-module, 
then $\RD(L)$ is an invertible 
$\sha$-module, and  by the hypothesis~(a), we have 
\eqn
&& \RD(L_{U})\simeq \sect_{U}(\RD(L))
\eneqn
for any $U\in\shs$.

Then we get the result by applying Theorem~\ref{th:mainafr} to 
$\RD(\shm)\in\Der^-_\coh(\sha^\rop)$ and using 
$\shm\isoto \RD(\RD(\shm))$. 
\end{proof}
\section{$\DQ$-algebroids in the algebraic case}\label{section:alg}

In this section, $X$ denotes a quasi-compact separated smooth
algebraic variety over $\C$.

Clearly, the notions of a $\DQ$-algebra and of a $\DQ$-algebroid make sense in
this settings and a detailed study of $\DQ$-algebroids on algebraic variety is
performed in \cite{Ye}.

Assume that $X$ is endowed with  a $\DQ$-algebroid $\A$ for
the Zariski topology. Then, in view of
Remark~\ref{rem:oesterle}, $\gr(\A)\simeq\OO$. However, this
equivalence is not unique in general.

Let us denote by $\Xan$ the complex analytic manifold associated with
$X$ and by $\rho\cl \Xan\to X$ the natural morphism. Then we can
naturally associate a $\DQ$-algebroid $\A[\Xan]$ to $\A$ and there is a
natural functor $\rho^{-1}\A\to\A[\Xan]$, whose construction is left to
the reader. Then it induces functors 
\eq\label{eq:algfct1}
&&\md[{\A}]\to \md[{\A[\Xan]}]
\eneq
and
\eq\label{eq:algfct2}
&&\mdcoh[{\A}]\to \mdcoh[{\A[\Xan]}].
\eneq
When $X$ is projective, the classical GAGA theorem of Serre extends to
$\DQ$-algebroids and it is proved in \cite{Ch}
that \eqref{eq:algfct2} is an equivalence.

\begin{lemma}\label{lem:alggood}
Let  $\shm\in\mdcoh[{\Ah}]$. The two conditions below are equivalent.
\banum
\item
$\shm$ is the inductive limit of its coherent sub-$\A$-modules,
\item
there exists an $\A[X]$-lattice of $\shm$ 
\lp see Definition~\ref{def:lattice}\,\rp.
\eanum
\end{lemma}
\begin{proof}
(a)$\Rightarrow$(b) Let $\shm=\indlim\shn$ 
where $\shn$ ranges over the filtrant family of coherent $\A$-submodules of
$\shm$. Since $\Ah$ is Noetherian, the family $\{\cor\tens[\coro]\shn\}$
is locally stationary. Since $X$ is quasi-compact, this family is stationary.

\vspace{0.2cm}
\noindent
(b)$\Rightarrow$(a) is obvious.
\end{proof}

\begin{definition}\label{def:alggood}
Let $\shm\in\mdcoh[{\Ah}]$. We say that $\shm$ is algebraically good
\glossary{algebraically good}%
\glossary{good!algebraically}%
if it satisfies the equivalent conditions in Lemma~\ref{lem:alggood}.
\end{definition}

We still denote by $\mdgd[{\Ah[X]}]$ the full subcategory of 
$\mdcoh[{\Ah[X]}]$ consisting of algebraically good modules.

The proof of  \cite[Prop.~4.23]{Ka2} extends to this case and 
$\mdgd[{\Ah[X]}]$ is a thick abelian subcategory of
$\mdcoh[{\Ah[X]}]$. Hence, we still denote by $\RD^\Rb_{\gd}(\Ah[X])$ 
the full triangulated subcategory of 
$\RD^\Rb_{\coh}(\Ah[X])$ consisting of objects $\shm$ such that 
$H^j(\shm)$ is algebraically good for all $j\in\Z$. 

\begin{remark}
We do not know if every coherent $\Ah$-module is algebraically good.
\end{remark}

\subsubsection*{ Almost free resolutions}
Recall that $X$ is endowed with a $\DQ$-algebroid $\A$ 
for the Zariski topology.

We denote by $\BB$ the family of affine open subsets $U$ of
$X$ on which the algebroid $\A[U]$ is a sheaf of algebras. 
Note that this family is stable by intersection. 
Moreover, hypotheses~\eqref{eq:FDringa} and~\eqref{eq:FDringb} are satisfied.

\begin{lemma}\label{le:algresol1}
Assume that $X$ is affine and $\A$ is a
$\DQ$-algebra. Then, for any $\shm\in\mdcoh[{\A}]$, there exist a
free $\A$-module $\shl$ of finite rank and an
epimorphism $u\cl\shl\epito\shm$.
\end{lemma}
\begin{proof}
Set $\shm_0=\shm/\hbar\shm$. Then $\shm_0$ is a coherent $\OO$-module
and there exist finitely many sections $(v_1,\dots,v_N)$ of $\shm_0$
on $X$ which generate $\shm_0$ over $\OO$. 

By Theorem~\ref{th:formalfini1}, 
$\sect(X;\shm)\to\sect(X;\shm_0)$ is surjective. 
Let $(u_1,\dots,u_N)$ be sections of $\shm$ whose image by
this  morphism are $(v_1,\dots,v_N)$. Let $\shl=\A^N$ and denote by
$(e_1,\dots,e_N)$ its canonical basis. It remains to define  $u$
by setting $u(e_i)=u_i$. 
\end{proof}

\begin{theorem}\label{th:mainafrAlg}
Let $\shm\in\mdcoh[{\A}]$. Then
there exists an isomorphism $\shm\simeq \shl^\scbul$ in $\Derb(\A)$
such that
$\shl^\scbul$ is a bounded complex of $\A$-modules and each 
$\shl^i$ is a finite
direct sum of modules of the form $\oim{i_{U}}\shl_U$,
where $i_U\cl U\hookrightarrow X$ is the embedding of an affine open
set $U$ such that $\A[U]$ is isomorphic to a $DQ$-algebra
and $\shl_U$ is a locally free $\A[U]$-module of finite rank.
\end{theorem}
Before proving Theorem~\ref{th:mainafrAlg}, we need some preliminary
results.

Let $\shu=\{U_i\}_{i\in I}$ be a finite covering of $X$
by affine  open sets such that  $\A\vert_{U_i}$ is a $\DQ$-algebra for all
$i$.

We denote by $\Sigma$ the category of non empty subsets of $I$ (the
morphisms are the inclusions maps). For $\sigma\in\Sigma$, we denote by
$\vert\sigma\vert$ its cardinal.
For $\sigma\in\Sigma$,
we set
\eqn
&&U_\sigma=\bigcap_{i\in\sigma}U_i,\quad \iota_\sigma\cl
U_\sigma\hookrightarrow X\mbox{ the natural embedding}.
\eneqn  
We introduce a category $\md[\sha,\shu]$ as follows.
An object $M$ of $\md[\sha,\shu]$ is the data of a family
$(\{M_\sigma\}_{\sigma\in\Sigma},\{q^M_{\sigma,\tau}\}_{\tau\subset\sigma\in\Sigma})$,
where $M_\sigma\in\md[{\A[U_\sigma]}]$
and $q^M_{\sigma,\tau}\cl M_\tau\vert_{U_\sigma}\to M_\sigma$ are
morphisms for $\emptyset\not=\tau\subset\sigma\in\Sigma$ 
satisfying $q^M_{\sigma,\sigma}=\id$ 
and for any $\sigma_1\subset\sigma_2\subset\sigma_3$, the diagram below commutes
\eq\label{eq:sigma123}
&&\xymatrix{
M_{\sigma_1}\vert_{U_{\sigma_3}}\ar[r]^-{q^M_{\sigma_2,\sigma_1}}
                                                 \ar[rd]_-{q^M_{\sigma_3,\sigma_1}}
       &M_{\sigma_2}\vert_{U_{\sigma_3}}\ar[d]^-{q^M_{\sigma_3,\sigma_2}}\\
                        &M_{\sigma_3}.
}\eneq
A morphism $M\to M'$ in $\md[\sha,\shu]$ is a family of morphisms
$M_\sigma\to M'_\sigma$ satisfying the natural compatibility conditions. 

Clearly, $\md[\sha,\shu]$ is an abelian category.

To an object  $M\in\md[\sha,\shu]$ we shall associate a Koszul complex 
$C^\scbul(M)$
using the construction of \cite[\S~12.4]{K-S3}.
To $M$ we associate a functor $F\cl\Sigma\to\md[{\A}]$ as follows:
$F(\sigma)=\oim{\iota_\sigma}M_\sigma$, and
$F(\tau\subset\sigma)\cl F(\tau)\to F(\sigma)$ is given by
the composition $\oim{\iota_\tau}M_\tau\to
\oim{\iota_\sigma}(M_\tau\vert_{U_\sigma})\to[q^M_{\sigma,\tau}]\oim{\iota_\sigma}M_\sigma$.

According to loc.\ cit., we get a Koszul complex $C^\scbul(M)$
\eq\label{eq:cech1}
&&C^\scbul(M)\eqdot \cdots\To 0\To C^1(M)\To[d^1] C^2(M)\To[d^2]\cdots
\eneq
where
\eqn
&&C^i(M)=\bigoplus_{\vert\sigma\vert=i}\oim{\iota_\sigma}M_\sigma
\eneqn
is in degree $i$.
This construction being functorial, we get a functor
\eq\label{eq:fctCscb}
&&C^\scbul\cl\md[\sha,\shu]\to \RC^\rmb(\md[{\A}]).
\eneq
It is convenient to introduce some notations.
We set
\eqn
&&\hspace{-.4cm}\mdcoh[\sha,\shu]=\set{M\in \md[\sha,\shu]}{
M_\sigma\in\mdcoh[{\A[U_\sigma]}]
\mbox{ for all }\sigma\in\Sigma},\\
&&\hspace{-.4cm}\mdff[\sha,\shu]=\{M\in \md[\sha,\shu]; M_\sigma 
\mbox{ is a locally free $\A[U_\sigma]$-module  }\\
&&\hspace{40ex}\mbox{of finite rank for all }\sigma\in\Sigma\}.
\eneqn
Clearly, $\mdcoh[\sha,\shu]$ is a full abelian subcategory of
$\md[\sha,\shu]$ and $\mdff[\sha,\shu]$ is a full additive subcategory of
$\mdcoh[\sha,\shu]$.
\begin{lemma}\label{le:cech0}
The functor 
$C^\scbul\cl\mdcoh[\sha,\shu]\to \RC^\rmb(\md[{\A}])$ induced by
\eqref{eq:fctCscb} is exact.
\end{lemma}
\begin{proof}
By Proposition~\ref{prop:cut} the functor  
$\iota_\sigma\cl\mdcoh[{\A[U_\sigma]}]\to\md[{\A}]$
is exact for each $\sigma\in\Sigma$. The result then easily follows. 
\end{proof}
Let us denote by 
\eq\label{eq:fctlamb}
&&\lambda\cl\mdcoh[{\A}]\to\mdcoh[\sha,\shu]
\eneq
the functor which, to 
$\shm\in\mdcoh[{\A}]$, associates the object $M$ where 
$M_\sigma=\shm\vert_{U_\sigma}$ and $q^M_{\sigma,\tau}\cl M_\tau\vert_{U_\sigma}\to M_\sigma$
is the restriction morphism.
\begin{lemma}\label{le:cech1}
The natural morphism $\shm\to C^\scbul(\lambda(\shm))\,[1]$ is a quasi-isomorphism.
\end{lemma}
\begin{proof}
Apply \cite[Th.~18.7.4~(ii)]{K-S3} with $A=\bbigsqcup_{i\in I}U_i$,
$u\cl A\to X$. By this result, the complex
\eqn
F^\scbul_u\eqdot 0\to \shm\to C^1(\lambda(\shm))\to[d^1] C^2(\lambda(\shm))\to[d^2]\cdots
\eneqn
is exact.
\end{proof}

\begin{lemma}\label{le:cech2}
Let $M\in\mdcoh[\sha,\shu]$. Then there exists an epimorphism
$L\epito M$ in $\md[\sha,\shu]$ with $L\in \mdff[\sha,\shu]$. 
\end{lemma}
\begin{proof}
Applying Lemma~\ref{le:algresol1}, we choose for each $\sigma\in\Sigma$ an epimorphism 
$L'_\sigma\epito M_\sigma$ with 
a locally free $\A[U_\sigma]$-module $L'_\sigma$ of finite rank. 
Set
\eqn
&& L_\sigma\eqdot\bigoplus_{\emptyset\neq\tau\subset\sigma}L'_\tau\vert_{U_\sigma}
\eneqn
and define the morphism $L_\sigma\to M_\sigma$ by the commutative
diagrams in which $\tau\subset\sigma$:
\eqn
&&\xymatrix{
L_\sigma\ar[r]&M_\sigma\\
L'_\tau\vert_{U_\sigma}\ar[r]\ar[u]&M_\tau\vert_{U_\sigma}.\ar[u]
}\eneqn
For $\tau\subset\sigma$, the morphism  $q^L_{\sigma,\tau}\cl L_{\tau}\vert_{U_\sigma}\to L_{\sigma}$ 
is defined by the morphisms ($\lambda\subset\tau$):
\eqn
&&\xymatrix@C=8ex{
L_{\tau}\vert_{U_\sigma}\ar[r]^{q^L_{\sigma,\tau}}&L_{\sigma}.\\
L'_\lambda\vert_{U_\sigma}\ar[u]\ar[ur]&
}\eneqn
Clearly, the family of morphisms $q^L_{\sigma,\tau}$ satisfies the
compatibility conditions similar to those in diagram
\eqref{eq:sigma123}. We have thus constructed an object
$L\in\md[\sha,\shu]$, and the family of morphisms $L_\sigma\to M_\sigma$
defines the epimorphism $L\epito M$ in $\md[\sha,\shu]$.
\end{proof}
\begin{proof}[Proof of Theorem~\ref{th:mainafrAlg}]
By Lemma~\ref{le:cech2}, there exists an exact sequence in $\mdcoh[\sha,\shu]$
\eq\label{eq:cech3}
&&0\To L_{d_X+1}\To\cdots\to L_1 \to L_0\to\lambda(\shm)\to 0
\eneq
with the $L_i$'s in $\mdff[\sha,\shu]$ (see Corollary~\ref{cor:hddim}).
Consider the complex
\eq\label{eq:cech4}
&&L^\scbul \eqdot \cdots\to L_1 \to L_0\to0.
\eneq
Hence, we have a quasi-isomorphism
$L^\scbul\to[qis]\lambda(\shm)$. Using Lemma~\ref{le:cech0}, we find
a quasi-isomorphism
\eq\label{eq:cech5}
&&C^\scbul(L^\scbul)\to[qis]C^\scbul(\lambda(\shm)).
\eneq
Then, the result follows from Lemma~\ref{le:cech1}.
\QED

\chapter{Kernels}\label{chapter:Kern} 

\section{Convolution of kernels: definition}

Integral transforms, also called ``correspondences'', are of constant
use in algebraic and analytic geometry
and we refer to the book \cite{Hu} for an exposition.
Here, we shall develop a similar formalism in the framework of
$\DQ$-modules ({\em i.e.,} modules over $\DQ$-algebroids).
\index{DQmodule@$\DQ$-module}

Consider complex manifolds $X_i$ endowed with 
$\DQ$-algebroids  $\A[X_i]$ ($i=1,2,\dots$).
\begin{notation}\label{not:www1}
\bnum
\item
Consider a product of manifolds $X\times Y\times Z$. We denote by
$p_i$ the $i$-th projection and by $p_{ij}$ the $(i,j)$-th projection
({\em e.g.,} $p_{13}$ is the projection from 
$X_1\times X_1^a\times X_2$ to $X_1\times X_2$). 
We use similar notations for a product of four manifolds. 
\item
We write $\A[i]$ and $\A[ij^a]$
instead of $\A[X_i]$ and $\A[X_i\times X_j^a]$  and
similarly with other products. We use the same notations 
for $\dA[X_i]$. 
\item
When there is no risk of confusion, we do note write the symbols 
$\opb{p_{i}}$ and similarly with $i$ replaced with $ij$, etc.\ 
\enum
\end{notation}

Let $\shk_i\in\RD^\Rb(\A[X_i\times X_{i+1}^a])$ ($i=1,2$). We set
\eq
&&\ba{rcl}
\shk_1\lltens[{\A[2]}]\shk_2
&\seteq& \opb{p_{12}}\shk_1\lltens[{\opb{p_2}\A[2]}]\opb{p_{23}}\shk_2\\
&\simeq&
(\shk_1\letens \shk_2)\lltens[\rma_{2}\etens\rma_{2^a}]\dA[2]
\in \Derb(\A[1]\etens\coro_{X_2}\etens\A[3^a]).
\ea
\label{not:tens}
\eneq
Similarly, for $\shk_i\in\RD^\Rb(\A[X_i\times X_{i+1}])$ ($i=1,2$),
we set
\eq
&&\rhom[{\A[2]}](\shk_1,\shk_2)\seteq
\rhom[{p_2^{-1}\A[2]}](p_{12}^{-1}\shk_1,p_{23}^{-1}\shk_2).
\label{not:hom}
\eneq
Here we identify $X_1\times X_2\times X_3^a$
with the diagonal set of $X_1\times X_2^a\times X_2\times X_3^a$.

This tensor product is not well suited to treat $\DQ$-modules.
For example, $\A[X\times Y]\not=\A\etens\A[Y]$. 
This leads us to introduce a kind of completion
of the tensor product as follows.
\begin{definition}\label{def:tens1}
Let $\shk_i\in\RD^\Rb(\A[X_i\times X_{i+1}^a])$ ($i=1,2$). 
We set \index{tensor@$\dtens[{\A}]$}%
\eq\label{eq:tensor}
\shk_1\dtens[{\A[2]}]\shk_2&=& 
\opb{\delta_2}(\shk_1\ldetens \shk_2)\lltens[\rma_{22^a}]\dA[2]\\
&=& \opb{p_{12}}\shk_1\lltens[\opb{p_{12}}\rma_{1^a2}]\A[123]
\lltens[\opb{p_{23}}\rma_{23^a}]\opb{p_{23}}\shk_2.\nonumber
\eneq
It is an object of
$\Derb(\opb{p_{13}}\A[{13^a}])$ where $p_{13}\cl X_1\times X_2\times
X_3\to
X_1\times X_3$ is the projection.
\end{definition}

We have a morphism in $\Derb(\opb{p_1}\A[X_1]\tens \opb{p_3}\A[X_3^a])$:
\eq\label{eq:tenstodtens}
&&\shk_1\lltens[{\A[2]}]\shk_2\to\shk_1\dtens[{\A[2]}]\shk_2.
\eneq
Note that \eqref{eq:tenstodtens} is an isomorphism if $X_1=\rmpt$ or $X_3=\rmpt$.

\begin{definition}\label{def:kernel1}
Let $\shk_i\in\RD^\Rb(\A[X_i\times X_{i+1}^a])$ ($i=1,2$). 
We set
\eq\label{eq:kernandDelta}
\shk_1\conv[X_2]\shk_2&=& 
\reim{p_{13}}(\shk_1\dtens[{\A[2]}]\shk_2)\in\Derb(\A[X_1\times X_3^a]),\\
\shk_1\sconv[X_2]\shk_2
&=&\roim{p_{13}}(\shk_1\dtens[{\A[2]}]\shk_2)\in\Derb(\A[X_1\times X_3^a]).
\label{eq:kernandDeltab}
\eneq
We call $\conv[X_2]$ the  convolution
\index{conv@$\conv[X]$}\glossary{convolution}%
of $\shk_1$ and $\shk_2$ (over $X_2$).
If there is no risk of confusion, we write
$\shk_1\circ\shk_2$ for $\shk_1\conv[X_2]\shk_2$ and similarly with $\sconv$.
\end{definition}
Note that in case where $X_3=\rmpt$ we get:
\eqn
\shk_1\conv\shk_2&\simeq& 
\reim{p_{1}}(\shk_1\lltens[\rma_{2}]\opb{p_{2}}\shk_2),
\eneqn
and in the general case, we have:
\eq\label{eq:kernandDelta2}&&\ba{rcl}
\shk_1\conv[X_2]\shk_2&\simeq& (\shk_1\ldetens\shk_2)\conv[X_2\times X_2^a]\dA[X_2]\\
&\simeq&
\reim{p_{14}}\bl(\shk_1\ldetens\shk_2)\lltens[{\A[22^a]}]\dA[2]\br,
\ea\eneq
where $p_{14}$ is the projection
$X_1\times X_2\times X_2^a\times X_3^a\to X_1\times X_3^a$.
There are canonical isomorphisms 
\eq
\shk_1\conv[X_2]\dA[X_2]\simeq \shk_1\quad\text{and}\quad
\dA[X_1]\conv[X_1]\shk_1\simeq \shk_1.
\eneq
One shall be aware that
$\conv$ and $\sconv$ are not associative in general.
(See Proposition~\ref{pro:HH1}~\eqref{ass:conv}.)

However, if $\shl$ is a 
bi-invertible $\A[X_2]\tens\A[X^a_2]$-module 
and the $\shk_i$'s ($i=1,2$) are as above, there are natural isomorphisms
\eqn
&&\shk_1\conv[X_2]\shl\simeq \shk_1\lltens[{\A[X_2]}]\shl,
\quad \shl\conv[X_2]\shk_2\simeq\shl \lltens[{\A[X_2]}]\shk_2,\\
&&(\shk_1\conv[X_2]\shl)\conv[X_2]\shk_2\simeq \shk_1\conv[X_2](\shl\conv[X_2]\shk_2).
\eneqn

For a closed subset $\Lambda_{i}$ of $X_{i}\times X_{i+1}$ ($i=1,2$), 
we set \index{lambda@$\Lambda_{1}\circ\Lambda_{2}$}%
\eq\label{eq:lambdacicr}
\Lambda_{1}\circ\Lambda_{2}&\eqdot& 
p_{13}(\opb{p_{12}}\Lambda_1\cap\opb{p_{23}}\Lambda_2)\\
&=&p_{14}((\Lambda_1\times \Lambda_2)\cap(X_1\times\Delta_2\times X_3))
\subset X_1\times X_3.\nonumber
\eneq
Note that if 
$\Lambda_i$  is a closed  complex analytic subvariety of
$X_i\times X_{i+1}^a$ ($i=1,2$) and $p_{13}$ is proper on 
$\opb{p_{12}}\Lambda_1\cap\opb{p_{23}}\Lambda_2$, 
then $\Lambda_1\circ\Lambda_2$ is a closed  complex analytic subvariety of 
$X_1\times X_3^a$.

Let us still denote by $\circ$ the convolution of 
$\gr(\sha)$-modules. More precisely for 
$\shl_i\in \RD^\Rb(\bA[X_i\times X_{i+1}^a])$ ($i=1,2$), we set
\eqn
\shl_1\circ\shl_2&=& 
\reim{p_{14}}\bl(\shl_1\ldetens\shl_2)\lltens[{\bA[22^a]}]\gr(\dA[2])\br.
\eneqn

\begin{proposition}\label{pro:circgr}
For $\shk_i\in\RD^\Rb(\A[X_i\times X_{i+1}^a])$ \lp$i=1,2$\rp, we have
\eq\label{eq:circgr2}
&&\gr(\shk_1\circ\shk_2)\simeq \gr(\shk_1)\circ\gr(\shk_2).
\eneq
\end{proposition}
\begin{proof}
Applying Proposition~\ref{pro:tensgr}, it remains to remark that the functor 
$\gr$ commutes with the functors of inverse images and proper direct images 
as well as with the functor $\detens$. 
\end{proof}

\section{Convolution of kernels: finiteness}\label{section:ker1}

In this section, we use  Notation~\ref{not:www1}

Consider  complex manifolds $X_i$ endowed with 
$\DQ$-algebroids  $\A[X_i]$ ($i=1,2,\dots$). 
We denote by $d_X$ the complex dimension of $X$ and we write for
short $d_i$ instead of $d_{X_i}$. 

We shall prove the following coherency theorem for $\DQ$-modules
by reducing it to the corresponding result
for $\sho$-modules due to Grauert (\cite{Gr}).
In the sequel, for a closed subset $\Lambda$ of $X$, we 
denote by $\Derb_{\coh,\Lambda}(\A)$
\index{DerCcoh@$\Derb_{\coh,\Lambda}(\A)$}%
the full triangulated subcategory of $\Derb_\coh(\A)$ 
consisting of objects supported by $\Lambda$. We define similarly 
$\Derb_{\gd,\Lambda}(\Ah)$.
\index{DerCgd@$\Derb_{\gd,\Lambda(\Ah)}$}%

\begin{theorem}\label{th:kernel1}
For $i=1,2$, let $\Lambda_i$ be a closed subset of $X_i\times X_{i+1}$
and $\shk_i\in\RD^\Rb_{\coh,\Lambda_i}(\A[X_i\times X_{i+1}^a])$.
Assume that
$\Lambda_1\times_{X_2}\Lambda_2$ is proper over $X_1\times X_3$,
and set $\Lambda=\Lambda_1\circ\Lambda_2$.
Then the object $\shk_1\circ\shk_2$ belongs to 
$\RD^\Rb_{\coh,\Lambda}(\A[X_1\times X_3^a])$.
\end{theorem}

\begin{proof}
Since the question is local in $X_1$ and $X_3$, we may assume from the
beginning that
$\A[X_1]$
and $\A[X_3]$ are \DQ-algebras.

We shall first show that
\eq
&&\text{$\shk_1\dtens[{\A[2]}]\shk_2$ is \cc.}\label{eq:cc}
\eneq
Since this statement is a local statement on $X_1\times X_2\times X_3$, 
we may assume that $\A[X_2]$ is a
\DQ-algebra.
Since $\shk_1$ and $\shk_2$ may be locally represented by finite complexes
of free modules of finite rank, in order to see \eqref{eq:cc},
we may assume $\shk_i\simeq\A[X_i\times X_{i+1}^a]$ ($i=1,2$).
Then 
$\shk_1\dtens[{\A[2]}]\shk_2\simeq\A[X_1\times X_2\times X_3^a]$ is
\cc\ by Theorem~\ref{th:cohimplcohco}.
Hence $\shk_1\conv\shk_2=\roim{p_{13}}(\shk_1\dtens[{\A[2]}]\shk_2)$
is also \cc\ by Proposition~\ref{pro:cohcodirim}.

On the other hand,
$\gr(\shk_1\conv\shk_2)
\simeq\roim{p_{13}}\bl p_{12}^*\gr\shk_1\ltens[{\OO[X_1\times X_2\times
X_3]}] p_{23}^*\gr\shk_2\br$
belongs to $\Db_\coh(\OO[X_1\times X_3])$ by 
Grauert's direct image theorem (\cite{Gr}).
Hence Theorem~\ref{th:formalfini2} implies that $\shk_1\conv\shk_2$ belongs to
$\Db_\coh(\A[X_1\times X_3^a])$.
\end{proof}

\begin{remark} 
In \cite{B-B-P}, its authors use a variant of
Theorem~\ref{th:kernel1} in the symplectic case. They assert that the proof
follows from Houzel's finiteness theorem on modules over 
sheaves  of multiplicatively convex nuclear Fr{\'e}chet algebras (see \cite{Ho}).
However, they do not give any proof, details being qualified of ``routine''. 
\end{remark}

\begin{corollary}\label{co:kernel1}
Let $\shm$ and $\shn$ be two objects of $\RD^\Rb_{\coh}(\A[X])$ and 
assume that $\Supp(\shm)\cap\Supp(\shn)$ is compact.  
Then the object $\RHom[{\A[X]}](\shm,\shn)$ belongs to 
$\RD^\Rb_{f}(\coro)$. 
\end{corollary}

\begin{proposition} \label{pro:HH1}
Let $\shk_i\in\RD^\Rb_{\coh}(\A[X_i\times X_{i+1}^a])$ \lp$i=1,2,3$\rp\, 
and let $\shl\in \RD^\Rb_{\coh}(\A[X_4])$. Set $\Lambda_i=\supp(\shk_i)$  and assume that 
$\Lambda_i\times_{X_{i+1}}\Lambda_{i+1}$ is proper over $X_i\times X_{i+2}$ 
\lp$i=1,2$\rp. 
\bnum
\item
There is a canonical isomorphism
$(\shk_1\conv[X_2]\shk_2)\ldetens \shl\isoto \shk_1\conv[X_{2}](\shk_2\ldetens \shl)$.
\item
There is a canonical isomorphism
$(\shk_1\conv[X_2]\shk_2)\conv[X_3]\shk_3\simeq 
\shk_1\conv[X_2](\shk_2\conv[X_3]\shk_3)$.
\label{ass:conv}
\enum
\end{proposition}
\begin{proof}
The morphism 
$(\shk_1\conv[X_2]\shk_2)\ldetens \shl\to\shk_1\conv[X_{2}](\shk_2\ldetens \shl)$
is deduced from the morphism (we do not write the 
functors $\opb{p_i},\opb{p_{ij}}$ for short):
\eqn
&&\hspace{-0.8cm}\A[13^a4]\tens[{\A[13^a]\etens\A[4]}]
\Bigr(\bl\bl\A[12^a23^a]\tens[{\A[12^a]\etens\A[23^a]}](\shk_1\letens\shk_2)\br
\lltens[{\A[22^a]}]\dA[2]\br\letens\shl\Bigr)\\
&&\hspace{1.cm}\simeq
\bl(\A[13^a4]\tens[{\A[13^a]\etens\A[4]}]\A[12^a23])\tens[{\A[12^a]\etens\A[23^a]\etens\A[4]}]
\shk_1\letens\shk_2\letens\shl\br\lltens[{\A[22^a]}]\dA[2]\\
&&\hspace{1.cm}\to 
\bl\A[12^a23^a4]\tens[{\A[12^a]\etens\A[23^a]\etens\A[4]}]\bl\shk_1\letens\shk_2
\letens\shl\br\br\lltens[{\A[22^a]}]\dA[2].
\eneqn

Applying the functor $\gr$ to this morphism 
in $\Db_\coh(\A[X_1\times X_3^a\times X^4])$,
we get an isomorphism. 
This proves the result in view of 
Corollary~\ref{cor:conservative1}.

\noindent
(ii) By (i), we have 
\eqn
(\shk_1\conv[X_2]\shk_2)\conv[X_3]\shk_3
&\simeq&\bl(\shk_1\conv[X_2]\shk_2)\ldetens\shk_3\br
\conv[X_3\times X_3^a]\dA[X_3]\\
&\simeq&\bl\shk_1\conv[X_2](\shk_2\ldetens\shk_3)\br
\conv[X_3\times X_3^a]\dA[X_3]\\
&\simeq&\Bigl(\dA[X_2]\conv[X_2\times X_2^a](\shk_1\ldetens\shk_2\ldetens\shk_3)\Bigr)
\conv[X_3\times X_3^a]\dA[X_3].
\eneqn
Then this object is isomorphic to
$(\shk_1\ldetens\shk_2\ldetens\shk_3)\conv[X_2\times X_2^a\times
X_3\times X_3^a](\dA[X_2]\ldetens\dA[X_3])$.
Similarly,
$\shk_1\conv[X_2](\shk_2\conv[X_3]\shk_3)$ is isomorphic to
$(\shk_1\ldetens\shk_2\ldetens\shk_3)\hs{-4ex}
\conv[{{\scriptsize\db{X_2\times X_2^a\times X_3\times X_3^a}}}]
\hs{-4ex}(\dA[X_2]\ldetens\dA[X_3])$.
\end{proof}

\section{Convolution of kernels: duality}\label{sec:convdual}

\subsubsection*{The duality morphism for kernels}

Denote as usual by $p_{13}\cl X_1\times X_2\times X_3^a\to X_1\times X_3^a$ 
the projection.

\begin{lemma}\label{le:duality}
For $\shk_i\in\Derb(\A[X_i\times X_{i+1}^a])$ \lp$i=1,2$\rp, 
we have a natural morphism in $\Derb(\A[X_1^a\times X_3])$:
\eq\label{mor:dual0}
&&(\RDAA[X_1\times X_2^a]\shk_1)\conv[X_2^a]\oA[X_2^a]\conv[X_2^a]
(\RDAA[X_2\times X_3^a]\shk_2) \to
\RDAA[X_1\times X_3^a](\shk_1\conv[X_2]\shk_2).
\eneq
\end{lemma}
\begin{proof}
We have
\eqn
\RDA\shk_1\dtens[{\A[2^a]}]\oA[2^a]\dtens[{\A[2^a]}]\RDA\shk_2
&\simeq&(\RDA\shk_1\ldetens\RDA\shk_2)\lltens[{\A[{2^a2}]}]\oA[2^a]\\
&\simeq&(\RDA\shk_1\ldetens\RDA\shk_2)\lltens[{\A[{12^a23^a}]}]
\oA[{12^a3^a}/{13^a}]\\
&\simeq&\RDA(\shk_1\ldetens\shk_2)\lltens[{\A[{12^a23^a}]}]
\oA[{12^a3^a}/{13^a}]\\
&\simeq&\rhom[{\A[{12^a23^a}]}](\shk_1\ldetens\shk_2,\oA[{12^a3^a}/{13^a}]).
\eneqn
Hence we have morphisms
\eqn
\RDA\shk_1\dtens[{\A[2^a]}]\oA[2^a]\dtens[{\A[2^a]}]\RDA\shk_2
&\simeq&
\rhom[{\A[{12^a23^a}]}](\shk_1\ldetens\shk_2,\oA[{12^a3^a}/{13^a}])\\
&&\hspace{-4cm}\to
\rhom[\opb{p_{13}}{\A[{13^a}]}]
\bl(\shk_1\ldetens\shk_2)\lltens[{\A[{22^a}]}]\dA[2],\oA[{12^a3^a}/{13^a}]
\lltens[{\A[{22^a}]}]\dA[2]\br\\
&&\hspace{-4cm}\to
\rhom[{\opb{p_{13}}\A[{13^a}]}](\shk_1\dtens[{\A[2]}]\shk_2,\opb{p_{13}}\A[{13^a}][2d_2]).
\eneqn
The last arrow is induced by \eqref{mor:residue}.
Taking $\reim{p_{13}}$, we obtain
\eqn
(\RDA\shk_1)\conv[X_2^a]\oA[X_2^a]\conv[X_2^a](\RDA\shk_2)
&\simeq&
\reim{p_{13}}\bl(\RDA\shk_1)\dtens[{\A[2^a]}]\oA[2^a]
\dtens[{\A[2^a]}](\RDA\shk_2)\br\\
&\to&
\roim{p_{13}}\rhom[{\opb{p_{13}}\A[{13^a}]}]
(\shk_1\dtens[{\A[2]}]\shk_2,\opb{p_{13}}\A[{13^a}][2d_2])\\
&\isoto&
\rhom[{\A[{13^a}]}](\shk_1\conv[X_2]\shk_2,\A[{13^a}]).
\eneqn
Here the last isomorphism is given by the Poincar{\'e} duality.
\end{proof}

\subsubsection*{Serre duality}
Let us recall the Serre duality for $\sho$-modules.
Let $X$ and $Y$ be complex manifolds.
Denote by $f\cl X\times Y\to X$ the projection, by
$\oo[Y]=\Omega_Y^{d_Y}[d_Y]$ the dualizing complex on $Y$ and by 
$\oo[X\times Y/X]\eqdot\OO\ldetens \oo[Y]$ the relative dualizing
complex.
For 
$\shg\in\Derb_{\coh}(\OO)$, we set
\eqn
&&\epb{f}\shg=\opb{f}\shg\lltens[{\opb{f}\OO}]\oo[X\times Y/X].
\eneqn
\begin{theorem}\label{th:Serre}
For $\shf\in\Derb_{\coh}(\OO[X\times Y])$ and $\shg\in\Derb_{\coh}(\OO)$,
we have a morphism
\eq
&&
\roim{f}\rhom[{\OO[X\times Y]}](\shf,\epb{f}\shg)
\to \rhom[{\OO[X]}](\reim{f}\shf,\shg).
\label{mor:dualo}
\eneq
If the support of $\shf$ is proper over $X$, 
then this morphism is an isomorphism.
\end{theorem}
This result is classical and we shall only 
recall a construction of the morphism \eqref{mor:dualo}
adapted to our study. 
Since $\Omega_Y$ has a $\shd_Y^{\rop}$-module structure,
we may regard $\oo[X\times Y/X]$ as an object of 
$\Derb(\OO\etens\shd_Y^\rop)$.
By the de Rham theorem, we have an isomorphism:
\eqn
&&\oo[X\times Y/X]\lltens[\shd_Y]\OO[Y]
\simeq \opb{f}\OO{}[2d_Y].
\eneqn
By composing  with the morphism
$\oo[X\times Y/X]\to \oo[X\times Y/X]\lltens[\shd_{Y}]\OO[Y]$,
we get a morphism in $\Derb(\opb{f}\OO)$:
\eqn
&&\oo[X\times Y/X]\to \opb{f}\OO{}[2d_Y].
\eneqn

\medskip
Now we have a chain of morphisms in $\Derb(\opb{f}\OO)$
\eqn
\rhom[{\OO[X\times Y]}](\shf,\epb{f}\shg)
&=& \rhom[{\OO[X\times Y]}](\shf,\opb{f}\shg\lltens[{\opb{f}\OO}]\oo[X\times Y/X])\\
&\to&
\rhom[{\opb{f}\OO}](\shf,\opb{f}\shg\lltens[{\opb{f}\OO}]\opb{f}\OO{}[2d_Y])\\
&\simeq& \rhom[{\opb{f}\OO}](\shf,\opb{f}\shg[2d_Y]).
\eneqn
On the other hand, the Poincar{\'e} duality gives an isomorphism
\eqn
&&\roim{f}\rhom[{\opb{f}\OO[X]}](\shf,\opb{f}\shg[2d_Y])
\simeq \rhom[{\OO}](\reim{f}\shf,\shg).
\eneqn

\subsubsection*{Duality for kernels}
Let $X_i$  be complex manifolds of dimension $d_i$
and let $\A[X_i]$ be $\DQ$-algebroids on $X_i$ ($i=1,2,3$).

As in Notation~\ref{not:www1}, we often write for short 
$X_{ij}$ instead of $X_i\times X_j$,  $X_{ij^a}$ instead of
$X_i\times X_j^a$, etc. We also write $\A[ij]$ instead of
$\A[X_{ij}]$, etc. and ${ij}/{i}$ instead of 
$X_{ij}/X_{i}$  etc.

\begin{theorem}\label{th:duality}
Let $\shk_i\in\Derb_\coh(\A[X_i\times X_{i+1}^a])$ \lp$i=1,2$\rp. We  assume that  
$\Supp(\shk_1)\times_{X_2}\Supp(\shk_2)$
is proper over $X_1\times X_3^a$. Then the natural morphism \lp see \eqref{mor:dual0}\rp
\eq\label{mor:dual1}
&&(\RDA\shk_1)\conv[X_2^a]\oA[X_2^a]\conv[X_2^a](\RDA\shk_2) \to
\RDA(\shk_1\conv[X_2]\shk_2)
\eneq
is an isomorphism in $\Db_\coh(\A[X_1^a\times X_3])$.
\end{theorem}

\begin{proof}
Since the question is local on $X_1\times X_3^a$,
we may assume that $\bA[X_1]$
and $\bA[X_3]$ are isomorphic to $\OO[X_1]$ and $\OO[X_3]$, respectively.
Applying the functor $\gr$, we get 
\eqn
\gr(\RDA(\shk_2)\circ\oA[X_2]\circ\RDA(\shk_1))&&\\
&&\hspace{-4cm}\simeq\reim{p_{13}}(\rhom[\sho_{{123}}](
p_{12}^{*}\gr(\shk_1)\lltens[{\sho_{123}}]
p_{23}^{*}\gr(\shk_2),\oo[X_{123}/X_{13}]))\\
&&\hspace{-4cm}\simeq
\rhom[\sho_{X_{13}}](\reim{p_{13}}
\bl p_{12}^{*}\gr(\shk_1)\lltens[{\sho_{123}}]
p_{23}^{*}\gr(\shk_2)\br,\OO[13])\\
&&\hspace{-4cm}\simeq
\gr(\RDA(\shk_1\circ\shk_2)).
\eneqn
Here the second isomorphism follows from Theorem~\ref{th:Serre}.
Hence \eqref{mor:dual1} is an isomorphism by Corollary~\ref{cor:conservative1}.
\end{proof}
Recall that $\RDD_X$ denotes  the duality functor for $\coro_X$-modules:
(see \eqref{eq:rdddual})
and $(\scbul)^\star$ the duality functor on $\RD^\Rb_{f}(\coro)$ (see \eqref{eq:star=dual}).

\begin{corollary}\label{co:duality}
Let $\shm$ and $\shn$ be two objects of $\RD^\Rb_{\coh}(\A[X])$.
\bnum
\item
There is a natural morphism in $\RD^\Rb(\coro)$
\eq\label{mor:dual000}
&&\RHom[{\A[X]}](\shn,\oA\lltens[{\A[X]}]\shm)
\to (\RHom[{\A[X]}](\shm,\shn))^\star.
\eneq
\item
If $\Supp(\shm)\cap\Supp(\shn)$ is compact,
then \eqref {mor:dual000} is an isomorphism in $\RD^\Rb_f(\coro)$.
\enum
\end{corollary}
\begin{proof}
(i) In Lemma~\ref{le:duality}, take $X_1=X_3=\rmpt$,
$X_2=X$, $\shk_1=\shn$ and $\shk_2=\RDA\shm$.

\noindent
(ii) follows from Theorem~\ref{th:duality}.
\end{proof}

In particular, if $X$ is compact, then
$\shm\mapsto \oA\tens[{\A}]\shm$ is a Serre functor on the
triangulated category $\Derb_\coh(\A)$.

\begin{remark}
For $\shk_i\in\RD^\Rb(\Ah[X_i\times X_{i+1}^a])$ ($i=1,2$), 
one can define their product $\shk_1\dtens[{\Ah[2]}]\shk_2$
similarly as in Definition~\ref{def:tens1} and their convolution
similarly as in Definition~\ref{def:kernel1}. (Details are left to the
reader.) 
One introduces \index{omegaA@$\oAh$}%
\eq
&&\oAh\seteq\cor\tens[\coro]\oA
\eneq
and for $\shm\in\RD^\Rb(\Ah[X])$, one defines its dual by setting
\index{DA@$\RDA\shm$}%
\eq\label{def:hdual1}
&&\RDA\shm\eqdot \rhom[{\Ah[X]}](\shm,\Ah[X])\in\RD^\Rb(\Ah[X^a]).
\eneq
Then Theorems~\ref{th:kernel1} and~\ref{th:duality} 
extend to good $\Ah[]$-modules. 
\begin{theorem}\label{th:hker}
Let $\Lambda_i$ be a closed subset of $X_i\times X_{i+1}$ \lp$i=1,2$\rp\, and assume that
$\Lambda_1\times_{X_2}\Lambda_2$ is proper over $X_1\times X_3$. 
Set $\Lambda=\Lambda_1\circ\Lambda_2$.
Let $\shk_i\in\Derb_{\gd,\Lambda_i}(\Ah[X_i\times X_{i+1}^a])$ \lp$i=1,2$\rp. 
Then the object $\shk_1\circ\shk_2$ belongs to 
$\Derb_{\gd,\Lambda}(\Ah[X_1\times X_3^a])$ and we have a natural isomorphism
\eqn\label{mor:hdual}
\RDA(\shk_1)\conv[X_2^a]\oAh[X_2^a]\conv[X_2^a]\RDA(\shk_2) \isoto
\RDA(\shk_1\conv[X_2]\shk_2).
\eneqn
\end{theorem}
\end{remark}

\subsubsection*{Proof of Theorem~\ref{th:oARDA}}
We are now ready to give a proof of Theorem~\ref{th:oARDA}.
In Theorem~\ref{th:duality}, set 
$X_1=X_2=X_3=X^a$ and $\shk_1=\shk_2=\dA[X^a]$.
Then we obtain
$$\RDA\dA[X^a]\conv[X]\oA\conv[X]\RDA\dA[X^a]
\simeq
\RDA(\dA[X^a]\conv[X^a]\dA[X^a])\simeq\RDA(\dA[X^a]).$$
By applying $\conv(\RDA\dA[X^a])^{\otimes-1}$,
we obtain
$\RDA\dA[X^a]\conv[X]\oA\simeq\dA$.

\section{Action of kernels on Grothendieck groups}
\subsubsection*{Grothendieck group}
For an abelian or a triangulated category $\shc$, we denote as usual
by $\rmK(\shc)$ \index{K1@$\rmK(\shc)$}%
its Grothendieck group. 
\glossary{Grothendieck group}%
For an object $M$ of $\shc$, we denote by $[M]$ 
\index{K2@$[M]$}%
its image in $\rmK(\shc)$. 
Recall that if $\shc$ is abelian, then $\rmK(\shc)\simeq
\rmK(\Derb(\shc))$.

If $A$ is a ring, we write $\rmK(A)$ instead of $\rmK(\md[A])$ and
write $\rmK_\coh(A)$ instead of $\rmK(\mdcoh[A])$.

In this subsection, we will adapt to $\DQ$-modules well-known
arguments concerning the Grothendieck group of filtered
objects. References are made to \cite[Ch.~2.2]{Ka2}.

For a closed subset $\Lambda$ of $X$, we shall write for short:
\eqn
&&\rmK_{\coh,\Lambda}(\A)\eqdot \rmK(\Derb_{\coh,\Lambda}(\A)),
\quad 
\rmK_{\coh,\Lambda}(\gr\A)\eqdot \rmK(\Derb_{\coh,\Lambda}(\gr\A)),\nonumber\\
 &&\rmK_{\gd,\Lambda}(\Ah)\eqdot \rmK(\Derb_{\gd,\Lambda}(\Ah)).
\eneqn

Recall that for an open subset $U$ of $X$ and $\shm\in\mdcoh[{\Ah}]$,
an $\A[U]$-submodule $\shm_0$ of $\shm\vert_U$ is called a lattice of
$\shm$ on $U$ if $\shm_0$ is coherent over $\A[U]$ and generates $\shm\vert_U$.

\begin{lemma}\label{lem:grgr4}
Let $0\to\shl\to\shm\to\shn\to 0$ be an exact sequence in $\mdcoh[\Ah]$.   
Then there  locally exist lattices $\shl_0$, $\shm_0$ and $\shn_0$ of 
$\shl$, $\shm$ and $\shn$ respectively, such that this
sequence induces an exact sequence of $\A$-modules:
$0\to\shl_0\to\shm_0\to\shn_0\to 0$.
\end{lemma}
\begin{proof}
(i) Let $\shm_0$ be a lattice of $\shm$ and let $\shn_0$ be its
image in $\shn$. We set $\shl_0\eqdot\shm_0\cap\shl$.  These
$\A$-modules give rise to the exact sequence of the statement and it
remains to check that $\shl_0$ and $\shn_0$ are lattices of 
$\shl$ and $\shn$, respectively. 

\vspace{0.2cm}
\noindent
(ii) Clearly, $\shn_0$ generates $\shn$, and being finitely generated,
it is coherent over $\A$. 

\vspace{0.2cm}
\noindent
(iii) Let us show that $\shl_0$ is a lattice of $\shl$.
Being the kernel of the morphism $\shm_0\to\shn_0$, $\shl_0$ is
coherent. Since the functor $(\scbul)^\loc$ is exact, the sequence
$0\to\shl_0^\loc\to\shm_0^\loc\to\shn_0^\loc\to 0$ is exact. Therefore, 
$\shl_0^\loc\simeq \shl$.
\end{proof}

\begin{lemma}\label{lem:grgr2}
Let $\shm\in\mdcoh[{\Ah}]$, let $U$ be a relatively compact open
subset of $X$ and assume that there exists 
a lattice $\shm_0$ of $\shm$ in a neighborhood of  the closure $\overline U$ of $U$. 
Then the image of $\shm_0$ in 
$\rmK_\coh(\gr\A[U])$ depends only on $\shm$.
\end{lemma}
\begin{proof}
(i) Recall that $[\gr\shm_0]$ denotes the image of $\gr\shm_0$ in $\rmK_\coh(\gr\A[U])$.
First, remark that for $N\in\N$, the two $\gr\A$-modules $\gr\shm_0$
and $\gr\hbar^N\shm_0$ are isomorphic, which implies
\eqn
&&[\gr\shm_0]=[\gr\hbar^N\shm_0].
\eneqn

\noindent
(ii) Now consider  another lattice $\shm_0'$ of $\shm$ on $\overline U$. 
Since $\shm$ is an $\Ah$-module of finite type and $\shm_0'$ generates $\shm$,
there exists $n>1$ such that $\shm_0\subset\hbar^{-n}\shm_0'$.
Similarly, there exists $m>1$ with $\shm_0'\subset\hbar^{-m}\shm_0$, so
that we have the inclusions
\eqn
&&\hbar^{m+n}\shm_0\subset\hbar^{m}\shm_0'\subset\shm_0 .
\eneqn
Using (i) we may replace $\shm'_0$ with $\hbar^{m}\shm_0'$. Hence,
changing our notations, we may assume
\eq\label{eq:Kgrgr1}
&&\hbar^{m}\shm_0\subset\shm_0'\subset\shm_0 .
\eneq

\noindent
(iii) Assume $m=1$ in \eqref{eq:Kgrgr1}. Using $\hbar^{m}\shm'_0\subset\hbar^{m}\shm_0$, 
we get the exact sequences
\eqn
&&0\to \shm'_0/\hbar\shm_0\to\shm_0/\hbar\shm_0\to\shm_0/\shm'_0\to0,\\
&&0\to \hbar\shm_0/\hbar\shm'_0\to\shm'_0/\hbar\shm'_0\to\shm'_0/\hbar\shm_0\to0,
\eneqn
and the result follows in this case.

\noindent
(iv) Now we argue by induction on $m$ in \eqref{eq:Kgrgr1} and we
assume the result is true for $m-1$ with $m\geq2$. Set
\eqn
&&\shm''_0\eqdot \hbar^{m-1}\shm_0+\shm'_0.
\eneqn
 Then
$\hbar\shm''_0\subset\shm'_0\subset\shm''_0$ and 
$\hbar^{m-1}\shm_0\subset\shm_0''\subset\shm_0$. Then the result
follows from (iii) and the induction hypothesis.
\end{proof}
We set \index{K4@$\rmhK_{\coh,\Lambda}(\gr\A)$}%
\eq\label{eq:rmhK}
&&\rmhK_{\coh,\Lambda}(\gr\A)\eqdot\prolim[U]\rmK_{\coh,\Lambda}(\gr\A[U]).
\eneq
where $U$ ranges over the family of relatively compact open subsets of $X$.
Using Lemma~\ref{lem:grgr2}, we get:
\begin{proposition}\label{pro:grgr1}
There is a
natural morphism of groups 
$$\gr\cl \rmK_{\gd,\Lambda}(\Ah)\to \rmhK_{\coh,\Lambda}(\gr\A).$$
\end{proposition}
Remark that when $X=\rmpt$, the morphism in
Proposition~\ref{pro:grgr1} reduces to the isomorphism
\eq\label{eq:grrmpt}
&&\rmK_f(\cor)\isoto \rmK_f(\C),
\eneq
and both are isomorphic to $\Z$ by $[M]\mapsto\dim M$.

\subsubsection*{Kernels}
Consider the situation of Theorem~\ref{th:kernel1}. 
Let $\Lambda_i$ be a closed subset of $X_i\times X_{i+1}$ ($i=1,2$) and assume that
$\Lambda_1\times_{X_2}\Lambda_2$ is proper over $X_1\times X_3$. 
Set $\Lambda=\Lambda_1\circ\Lambda_2$.
Since the convolution of kernels
commutes with distinguished triangles, it factors through the
Grothendieck groups. 
Moreover, one can define the convolution of $\gr\A$-kernels and 
a variant of Theorem~\ref{th:kernel1} with $\A$ replaced with
$\gr\A$ is well-known. Since the functor $\gr$ commutes with the
convolution of kernels, the diagram below commutes:

\eq\label{diag:kerngrgr1}
&&\xymatrix{
{\Ob\bl\Derb_{\coh,\Lambda_1}(\A[{12^a}])\br\times\Ob\bl\Derb_{\coh,\Lambda_2}(\A[{23^a}])\br}
\ar[r]^-\conv\ar[d]
                    &{\Ob\bl\Derb_{\coh,\Lambda}(\A[{13^a}])\br}\ar[d]\\
\rmK_{\coh,\Lambda_1}(\A[{12^a}])\times
\rmK_{\coh,\Lambda_2}(\A[{23^a}])\ar[r]^-\conv\ar[d]^-{\gr\times\gr}
                   &\rmK_{\coh,\Lambda}(\A[{13^a}])\ar[d]^-\gr\\
\rmK_{\coh,\Lambda_1}(\gr\A[{12^a}])\times 
\rmK_{\coh,\Lambda_2}(\gr\A[{23^a}])\ar[r]^-{\conv}
                   &\rmK_{\coh,\Lambda}(\gr\A[{13^a}]).
}\eneq
Similarly to \eqref{diag:kerngrgr1}, the diagram below commutes:
\eq\label{diag:kerngrgr3}
&&\xymatrix{
\Ob\bl\Derb_{\gd,\Lambda_1}(\Ah[{12^a}])\br\times\Ob\bl\Derb_{\gd,\Lambda_2}(\Ah[{23^a}])\br
\ar[r]^-\conv\ar[d]
                       &\Ob\bl\Derb_{\gd,\Lambda}(\Ah[{13^a}])\br\ar[d]\\
\rmK_{\gd,\Lambda_1}(\Ah[{12^a}])\times 
\rmK_{\gd,\Lambda_2}(\Ah[{23^a}])\ar[r]^-\conv\ar[d]^-{\gr\times\gr}
                     &\rmK_{\gd,\Lambda}(\Ah[{13^a}])\ar[d]^-\gr\\
\rmhK_{\coh,\Lambda_1}(\gr\A[{12^a}])\times 
\rmhK_{\coh,\Lambda_2}(\gr\A[{23^a}])\ar[r]^-{\conv}
                   &\rmhK_{\coh,\Lambda}(\gr\A[{13^a}]).}
\eneq

\chapter{Hochschild classes}\label{chapter:HH} 

\section{Hochschild homology and Hochschild classes}
Let $X$ be a complex manifold and let $\A$ be a $\DQ$-algebroid.
Recall that $\de\cl X\to X\times X^a$ is the diagonal embedding.
We define the  Hochschild homology $\HHA[X]$
\index{Hochschild homology@$\HHA[X]$}%
of $\A$ by:
\eq\label{def:HHA}
&&\HHA[X]\eqdot\de^{-1}(\dA[X^a]\lltens[{\A[X\times X^a]}]\dA),
\mbox{ an object of $\RD^\Rb(\coro_X)$}.
\eneq
Note that by Theorem~\ref{th:oARDA}, we get the isomorphisms:
\eqn
\HHA[X]&\simeq&\de^{-1}\rhom[{\A[X\times X^a]}](\RDAA[X^a\times X]\dA[X^a],\dA)\\
&\simeq&\de^{-1}\rhom[{\A[X\times X^a]}](\oAI,\dA).
\eneqn
We have also the isomorphisms
\eqn
\rhom[{\A[X\times X^a]}](\oAI,\dA)
&\simeq& \rhom[{\A[X\times X^a]}](\oA\conv[X]\oAI,\oA\conv[X]\dA)\\
&\simeq&\rhom[{\A[X\times X^a]}](\dA,\oA).
\eneqn
One shall be aware
that the composition of these isomorphisms does not coincide in general with 
the composition of
\eqn
\rhom[{\A[X\times X^a]}](\oAI,\dA)&\simeq& 
                     \rhom[{\A[X\times X^a]}](\oAI\conv[X]\oA,\dA\conv[X]\oA)\\
&\simeq& \rhom[{\A[X\times X^a]}](\dA,\oA).
\eneqn
We shall see that they differ up to
$\hh_X(\omA)\conv$ (see Lemma \ref{le:HH4} below). For that reason,
we shall not identify $\HHA[X]$ and $\rhom[{\A[X\times X^a]}](\dA,\oA)$.

\begin{lemma}\label{lem:HHmor12}
Let $\shm\in\RD^\Rb_{\coh}(\A[X])$. 
There are natural morphisms in  $\RD^\Rb_{\coh}(\A[X\times X^a])$:
\eq
&&\oAI\To \shm\ldetens\RDA\shm,\label{eq:HHmor01}\\
&&\shm\ldetens\RDA\shm\To\dA.\label{eq:HHmor02}
\eneq
\end{lemma}
\begin{proof}
(i) We have
\eqn
\rhom[{\A[X]}](\shm,\shm)
&\simeq& (\RDA\shm)\lltens[{\A}]\shm\\
&\simeq& 
\dA[{X^a}]\lltens[{\A[X\times X^a]}](\shm\ldetens\RDA\shm)\\
&\simeq& \rhom[{\A[X\times X^a]}](\oAI,\shm\ldetens\RDA\shm).
\eneqn
The identity of $\Hom[{\A[X]}](\shm,\shm)$ defines the desired morphism.

\noindent
(ii) Applying the duality functor $\RDAA[X\times X^a]$ to 
\eqref{eq:HHmor01}, we get \eqref{eq:HHmor02}.
\end{proof}

Let $\shm\in\RD^\Rb_{\coh}(\A[X])$. 
We have the chain of morphisms
\eq
&&\ba{rcl}
\rhom[{\A[X]}](\shm,\shm)&\isofrom&\RDA\shm
\lltens[{\A[X]}]\shm\\
&\simeq&\dA[X^a]\lltens[{\A[X\times X^a]}](\shm\ldetens\RDA\shm)\\
&\to&\dA[X^a]\lltens[{\A[X\times X^a]}]\dA=\HHA[X].
\ea
\eneq
We get a map
\eqn
&&\Hom[{\A[X]}](\shm,\shm)\to H^0_{\Supp(\shm)}(X;\HHA[X]).
\eneqn
For $u\in\End(\M)$, the image of $u$ gives an element 
\eq\label{eq:defHhMu}
&&\hh_X((\shm,u))\in H^{0}_{\Supp(\shm)}(X;\HHA[X]).
\eneq
\begin{notation}\label{not:RHH}
For a closed subset $\Lambda$ of $X$, we set
\eq
&&\RHH[\Lambda]{X}\eqdot H^0\rsect_\Lambda(X;\HHA[X]).
\eneq
\end{notation}
\begin{definition}\label{def:Hh}
Let $\shm\in\RD^\Rb_{\coh,\Lambda}(\A[X])$. 
We set $\hh_X(\shm)=\hh_X((\shm,\id_{\shm}))\in\RHH[\Lambda]{X}$ and call it
the Hochschild class of $\shm$.
\glossary{Hochschild class!of an $\A$-module}%
\index{Hochschild class@$\hh_X(\shm)$}%
\end{definition}

\begin{lemma}\label{le:HH2}
Let $\shm\in\RD^\Rb_{\coh}(\A)$. 
The composition of 
the two morphisms \eqref{eq:HHmor01} and \eqref{eq:HHmor02}:
\eqn\label{eq:HHmor0}
\oAI\To \shm\ldetens\RDA\shm\To\dA \eneqn
coincides with the Hochschild class $\hh_X(\shm)$ 
when identifying $\HHA$ with $\rhom[{\A[X\times X^a]}](\oAI,\dA)$.
\end{lemma}
\begin{proof}
The Hochschild class $\hh_X(\shm)$ is the image of $\id_\shm$ by the composition
\eqn
\rhom[\A](\shm,\shm)
&\simeq& \rhom[{\A[X\times X^a]}](\oAI,\shm\ldetens\RDA\shm)\\
&\to& \rhom[{\A[X\times X^a]}](\oAI,\dA)\simeq\HHA.
\eneqn
\end{proof}
\begin{theorem}\label{th:addhh}
The Hochschild class is  additive with respect to distinguished 
triangles. In other words, for a distinguished triangle
$\shm'\to\shm\to\shm''\to[+1]$ in $\RD^\Rb_{\coh}(\A[X])$, we have
\eq
&&\hh_X(\shm)=\hh_X(\shm')+\hh_X(\shm'').
\eneq 
\end{theorem}
\begin{proof}
Although the bifunctor $\scbul\lltens[\A]\scbul$ is not internal to
our category, the theorem of May~\cite{My} is easily adapted to this situation.
\end{proof}
By this result, the  Hochschild class factorizes through the
Grothendieck group. Therefore, if $\Lambda$ is a closed subset of $X$,
we have the morphisms
\eq\label{eq:hhGr}
&&\Derb_{\coh,\Lambda}(\A)\to \rmK_{\coh,\Lambda}(\A)\to\RHH[\Lambda]{X}.
\eneq

\subsubsection*{Duality}
Denote by $s\cl X\times X^a\to X^a\times X$ the map 
$(x,y)\mapsto (y,x)$ and recall that  $\de$ is the diagonal embedding.
Then $s\circ\de=\de$, 
$\opb{s}\dA[X]\simeq\dA[X^a]$, $\opb{s}\A[X\times X^a]\simeq\A[X^a\times X]$
 and we obtain the isomorphisms
\eqn
\HHA[X]&=&\opb{\de}(\dA[X^a]\lltens[{\A[X\times X^a]}]\dA)\\
&\simeq&\opb{\de}\opb{s}(\dA[X^a]\lltens[{\A[X\times X^a]}]\dA)\\
&\simeq&\opb{\de}(\opb{s}\dA[X^a]\lltens[\opb{s}{\A[X\times X^a]}]\opb{s}\dA)\\
&\simeq&\opb{\de}(\dA\lltens[{\A[X^a\times X]}]\dA[X^a])
=\HHA[X^a]. 
\eneqn
After identifying 
$\HHA[X]$ and $\HHA[X^a]$ by the isomorphism above, we have:
\eq\label{eq:dualHochschild}
&&\hh_{X^a}(\RDA\shm)=\hh_X(\shm).
\eneq
\begin{remark}\label{rem:HHAP}
Let $\sha$ be a $\DQ$-algebroid and let $\shp$ be an inverstible $\coro$-algebroid on $X$.
Then
\eq\label{eq:algebDP}
&&\sha^\shp\eqdot\sha\tens[\coro_X]\shp
\eneq
is a $\DQ$-algebroid on $X$. We have the natural equivalences
\eqn
&&(\sha^\rop)^{\shp^\rop}\simeq (\sha^\shp)^\rop,\\
&&\opb{\delta_X}(\sha^\shp\detens(\sha^\shp)^\rop)
\simeq\opb{\delta_X}(\sha\detens(\sha^\shp)).
\eneqn
We deduce the isomorphism
\eq\label{eq:HHAP}
&&\HHA[X]\simeq\HHH[\sha^\shp_X].
\eneq
\end{remark}

\section{Composition of  Hochschild classes}

Let $X_i$ be complex manifolds endowed with $\DQ$-algebroids 
$\A[X_i]$ ($i=1,2,3$) and denote as usual by $p_{ij}$ the 
projection from $X_1\times X_2\times X_3$ to $X_i\times X_j$ ($1\leq i<j\leq 3$).
\begin{proposition}\label{pro:compHH}
There is a natural morphism
\eqn
&&\circ: \reim{p_{13}}(\opb{p_{12}}\HHA[X_1\times X^a_2]
\lltens\opb{p_{23}}\HHA[X_2\times X^a_3])\to \HHA[X_1\times X^a_3].
\eneqn
\end{proposition}
\begin{proof}
(i) Set $Z_i=X_i\times X_{i}^a$.
We shall denote by the same letter $p_{ij}$
the projection from $Z_1\times Z_2\times Z_3$ to $Z_i\times Z_j$.

We have 
\eqn
\HHA[X_i\times X^a_j]\\
&&\hspace{-1.5cm}\simeq
(\dA[X_i^a]\ldetens\dA[{X_j}])\lltens[{\A[Z_i\times Z_j^a]}](\dA[{X_i}]\ldetens\dA[{X_j^a}])\\
&&\hspace{-1.5cm}\simeq
\rhom[{\A[Z_i\times Z_j^a]}](\oAI[{X_i}]\ldetens\oAI[{X_j^a}],\dA[X_i]\ldetens\dA[{X_j}^a])\\
&&\hspace{-1.5cm}\simeq
\rhom[{\A[Z_i\times Z_j^a]}]\bl
(\oAI[{X_i}]\ldetens\oAI[{X_j^a}])\dtens[{\A[X_j^a]}]\oA[X_j^a],
(\dA[X_i]\ldetens\dA[{X_j}^a])\dtens[{\A[X_j^a]}]\oA[X_j^a]\br\\
&&\hspace{-1.5cm}\simeq
\rhom[{\A[Z_i\times Z_j^a]}](\oAI[{X_i}]\ldetens\dA[{X_j^a}],\dA[X_i]\ldetens\oA[{X_j}^a]).
\eneqn
Set $S_{ij}\seteq\oAI[{X_i}]\ldetens\dA[{X_j^a}]
\in\Derb_\coh(\A[Z_i\times Z_j^a])$
and $K_{ij}\seteq\dA[{X_i}]\ldetens\oA[{X_j^a}]\in\Derb_\coh(\A[Z_i\times Z_j^a])$.
Then we get
\eqn
&&\HHA[X_i\times X^a_j]
\simeq \rhom[{\A[Z_i\times Z_j^a]}](S_{ij},K_{ij}).
\eneqn
Thus we obtain a morphism in $\Db(\coro_{Z_1\times Z_2\times Z_3})$
\eqn
&&\ba{ll}
\hs{-3ex}\opb{p_{12}}\HHA[X_1\times X^a_2]\lltens\opb{p_{23}}\HHA[X_2\times X^a_3]
\\[1ex]
\hs{9ex}\simeq
\opb{p_{12}}\rhom[{\A[Z_1\times Z_2^a]}](S_{12},K_{12})\lltens\opb{p_{23}}
\rhom[{\A[Z_2\times Z_3^a]}](S_{23},K_{23})\\[1ex]
\hs{9ex}\to\opb{p_{13}}\rhom[{\A[{Z_1\times Z_3^a}]}]
\bl S_{12}\dtens[{\A[Z_2]}]S_{23}, K_{12}\dtens[{\A[Z_2]}]K_{23}\br.
\ea
\eneqn
We get a morphism
\eq\label{eq:compHH1}
&&\ba{ll}
\hs{-3ex}\reim{p_{13}}(\opb{p_{12}}\HHA[X_1\times X^a_2]
\lltens\opb{p_{23}}\HHA[X_2\times X^a_3])\\[1ex]
\hs{5ex}\to\reim{p_{13}}\rhom[{\A[{Z_1\times Z_3^a}]}]
\bl S_{12}\dtens[{\A[Z_2]}]S_{23}, K_{12}\dtens[{\A[Z_2]}]K_{23}\br.
\ea
\eneq

\vspace{0.4em}
\noindent
(ii) We have a morphism
\eqn
&&\coro_{X_2}\to
\rhom[{\A[Z_2^a]}](\dA[X_{2}^a],\dA[X_{2}^a])
\simeq\dA[X_{2}^a]\dtens[{\A[Z_2]}]\oAI[{X_2}],
\eneqn
which induces the morphism:
\eqn
&&
\opb{p_{13}}(\oAI[{X_1}]\ldetens\dA[X_3^a])
\to (\oAI[{X_1}]\ldetens\dA[X_{2}^a])
\dtens[{\A[Z_2]}](\oAI[{X_2}]\ldetens\dA[X_3^a]),
\eneqn
that is, the morphism  in $\Derb(\A[Z_1\times Z_3^a])$:
\eq\label{eq:compHH2}
&&S_{13}\to \roim{p_{13}}(S_{12}\dtens[{\A[Z_2]}]S_{23}).
\eneq

\vspace{0.4em}
\noindent
(iii) We have a morphism (see \eqref{mor:residue}):
\eqn
&&(\dA[{X_1}]\ldetens\oA[{X_2^a}])\dtens[{\A[Z_2]}]
(\dA[{X_2}]\ldetens\oA[{X_3^a}])
\to 
\opb{p_{13}}(\dA[{X_1}]\ldetens\oA[X_3^a])[2d_2],
\eneqn
which induces the morphism in $\Derb(\A[Z_1\times Z_3^a])$:
\eq\label{eq:compHH3}
&&\reim{p_{13}}(K_{12}\dtens[{\A[Z_2]}]K_{23})\to K_{13}.
\eneq

\vspace{0.4em}
\noindent
(iv) Using \eqref{eq:compHH2} and \eqref{eq:compHH3} we obtain
\eq\label{eq:compHH4}
&&\ba{l}\reim{p_{13}}\rhom[{\A[{Z_1\times Z^a_3}]}]
\bl S_{12}\dtens[{\A[Z_2]}]S_{23}, K_{12}\dtens[{\A[Z_2]}]K_{23}\br\\
\hs{5ex}\to\,\rhom[{\A[{Z_1\times Z^a_3}]}]
\bl \roim{p_{13}}(S_{12}\dtens[{\A[Z_2]}]S_{23}),
\reim{p_{13}}( K_{12}\dtens[{\A[Z_2]}]K_{23})\br\\
\hs{5ex}\to\rhom[{\A[{Z_1\times Z^a_3}]}](S_{13},K_{13})
\simeq
\HHA[X_1\times X^a_{3}].
\ea
\eneq
Combining \eqref{eq:compHH1} and \eqref{eq:compHH4}, we get the result.
\end{proof}
Let us denote by $X_\R$ the real underlying manifold to $X$ and by 
$\omega_{X_\R}^{\rm top}$ 
\index{omegatop@$\omega_{X_\R}^{\rm top}$}%
the topological dualizing complex of the
space $X_\R$ with coefficients in $\coro$. Note that $X$ being smooth and oriented,
$\omega_{X_\R}^{\rm top}$ is isomorphic to $\coro_X\,[2d_X]$. 
\begin{corollary}\label{cor:hoschdualizcp}
There is a canonical morphism $\HHA[X]\tens\HHA[X]\to\omega_{X_\R}^{\rm top}$.
\end{corollary}
\begin{proof}
Let us apply Proposition~\ref{pro:compHH} with $X_2=X$,
$X_1=X_3=\rmpt$.  Denoting by $a_X$ the map $X\to\rmpt$, we get the
morphism $\reim{a_X}(\HHA[X]\tens\HHA[X])\to\coro_\rmpt$. By adjunction
we get the desired morphism.
\end{proof}
\section{Main theorem}
Consider five manifolds $X_i$ endowed with $\DQ$-algebroids $\A[X_i]$ ($i=1,\dots,5$). 

\begin{notation}\label{not:nota1}
In the sequel and until the end of this section, when there is no risk of confusion, 
we use the following conventions.
\bnum
\item For $i,j\in\{1,2,3,4,5\}$, we set 
$X_{ij}\eqdot X_i\times X_j$, $X_{ij^a}\eqdot X_i\times X_j^a$ and similarly 
with $X_{ijk}$, etc.
\item 
We sometimes omit the symbols $p_{ij},\oim{p_{ij}},\opb{p_{ij}}$, etc.
\item 
We write $\A[i]$ instead of $\A[X_i]$, $\A[ij^a]$ instead of $\A[X_{ij^a}]$
and similarly with $\dA[i],\oA[i]$, etc., and we write 
$\conv[i]$ instead of $\conv[X_i]$, $\sconv[i]$ instead of $\sconv[X_i]$,
$\hom[i]$ instead of $\hom[{\A[i]}]$ and $\tens[i]$ instead of
$\tens[{\A[i]}]$  and similarly with $ij^a$, $ijk$, etc. 
\item 
We write $\RDD$ instead of $\RDA$ and $\omA$ instead of $\oA$.
\item 
We often identify an invertible object of $\Derb(\A\otimes\A[X^a])$ with an object of 
$\Derb(\A[X\times X^a])$ supported by the diagonal.
\item 
We identify $(X_i\times X_j^a)^a$ with $X_i^a\times X_j$.
\enum
\end{notation}

Let $\Lambda_{ij}\subset X_{ij}$ ($i=1,2$, $j=i+1$) be a closed subset
and assume that $\Lambda_{12}\times_{X_2}\Lambda_{23}$ 
is proper over $X_1\times X_3$. 
Using Proposition~\ref{pro:compHH}, we get a map
\eq\label{eq:RHH}
&&\conv[2]\ \cl\RHH[\Lambda_{12}]{X_{12^a}}\times \RHH[\Lambda_{23}]{X_{23^a}}
\To \RHH[\Lambda_{12}\circ\Lambda_{23}]{X_{13^a}}.
\eneq
For $C_{ij}\in \RHH[\Lambda_{ij}]{X_{ij^a}}$ ($i=1,2$, $j=i+1$),
we obtain a class
\eq\label{eq:C1circC2}
&&C_{12}\conv[2] C_{23}\in  
\RHH[\Lambda_{12}\circ\Lambda_{23}]{X_{13^a}}.
\eneq

The morphism $(\dA[1^a]\dtens[11^a]\dA[1])\letens
(\dA[2^a]\dtens[22^a]\dA[2])\to(\dA[1^a2^a]\dtens[121^a2^a]\dA[12])$
induces the exterior product
\eq
&&\etens:\RHH[\Lambda_1]{X_1}\times\RHH[\Lambda_2]{X_2}
\to\RHH[\Lambda_1\times \Lambda_2]{X_1\times X_2}
\eneq
for $\Lambda_{i}\subset X_{i}$ ($i=1,2$).

\begin{lemma}\label{le:assoHH}
Let $\Lambda_{ij}\subset X_{ij}$ \ro$i=1,2,3$, $j=i+1$\rf and assume that 
$\Lambda_{ij}\times_{X_j}\Lambda_{jk}$ is proper over $X_{ik}$
\ro$i=1,2$, $j=i+1$, $k=j+1$\rf. 
Let $C_{ij}\in \RHH[\Lambda_{ij}]{X_{ij^a}}$ \lp$i=1,2,3$, $j=i+1$\rp.
\banum
\item
One has $(C_{12}\conv[2] C_{23})\conv[3] C_{34}=
C_{12}\conv[2]( C_{23}\conv[3] C_{34})$.
\item 
For $C_{245}\in\RHH{X_{245^a}}$ we have
$$(C_{12}\boxtimes C_{34})\conv[24]C_{245}=
C_{12}\conv[2](C_{34}\conv[4]C_{245}).$$
\item
Set $\rC_{\Delta_i}=\hh_{X_{ii^a}}(\dA[X_i])$. Then 
$C_{12}\conv[2]\rC_{\Delta_2}=\rC_{\Delta_1}\conv[1]C_{12}=C_{12}$.
\label{eq:unit}
\item
$(C_{12}\boxtimes\rC_{\Delta_3})\conv[23^a]C_{23}
=C_{12}\conv[2]C_{23}$.
Here $C_{12}\boxtimes\rC_{\Delta_3}
\in\RHH[\Lambda_{12}\times\Delta_3]{X_{12^a33^a}}$ 
is regarded as an element of
$\RHH[\Lambda_{12}\times\Delta_3]{X_{(13^a)(23^a)^a}}$.
\eanum
\end{lemma}
\Proof
The proof of (a) and (b) is left to the reader and
\eqref{eq:unit} follows from
Lemma \ref{le:HH4} below. Indeed, $\Phi_\shk$ in \eqref{eq:defPhiK}
is equal to the identity when $\shk=\dA$ since
the functor
$\shl\mapsto
\shk\sconv[2]\shl\conv[2]\omA[2]\sconv[2]\RDD\shk$
is isomorphic to the identity functor.

\noindent
(d) follows from (b) and (c).
\QED

In order to prove Theorem~\ref{th:HH1} below, we need some lemmas.

\begin{lemma}\label{le:HH3}
Let $\shk\in\RD^\Rb_{\coh}(\A[X_{12^a}])$. 
Then, there are natural morphisms in 
$\RD^\Rb(\A[X_{11^a}]):$
\eq
&&\omAI[1]\to \shk\sconv[2]\RDA\shk,\label{eq:HHmor1}\\
&&\shk\conv[2]\omA[2]\conv[2]\RDA\shk\to\dA[1].\label{eq:HHmor2}
\eneq
\end{lemma}
\begin{proof}
(i) \quad
By \eqref{eq:HHmor01}, we have a morphism in $\RD^\Rb(\A[X_{12^a21^a}])$
\eqn
&&\omAI[12^a]\to \shk\detens\RDA\shk.
\eneqn
Applying the functor $\scbul\lltens[22^a]\dA[2]$, we obtain 
\eqn
&&\hs{-2ex}\opb{p_{11^a}}\omAI[1]
\to\omAI[1]\etens (\RDA\dA[2]\ltens[22^a]\dA[2])
\to\omAI[12^a]\lltens[22^a]\dA[2]
\to(\shk\ldetens\RDA\shk)\tens[22^a]\dA[2].
\eneqn
By adjunction, we get \eqref{eq:HHmor1}.

\noindent
(ii) \ By \eqref{eq:HHmor02}, we have a morphism in $\RD^\Rb(\A[X_{12^a21^a}])$
\eqn
&&\shk\detens\RDD\shk\to \dA[12^a].
\eneqn
Applying the functor $\scbul\lltens[22^a]\omA[2]$, we obtain 
\eqn
&&(\shk\ldetens\RDA\shk)\tens[22^a]\omA[2]
\to \dA[12^a]\lltens[22^a]\omA[2]\to \dA[1]\dtens\coro_{X_2}[2d_2].
\eneqn
Here the last arrow is given by \eqref{mor:residue}.
By adjunction, we get \eqref{eq:HHmor2}.
\end{proof}

For the sake of brevity, 
we shall write $\Gamma_\Lambda\Hom$ instead of
$H^0(\rsect_\Lambda\rhom)$. 

Let $\Lambda_{12}$ be a closed subset of $X_1\times X_2^a$ and
$\Lambda_2$ a closed subset of $X_2$.
Let $\shk\in\RD^\Rb_{\coh}(\A[X_{12^a}])$ with support $\Lambda_{12}$.
We assume
\eq\label{hyp:lambdapproper}
&&\parbox{60ex}{  
$\Lambda_{12}\times_{X_2}\Lambda_{2}$ is proper over $X_{1}$.
}
\eneq
We define the map
\eq
&&\Phi_{\shk}
\cl\RHH[\Lambda_2]{X_{2}}\To\RHH[\Lambda_{12}\circ\Lambda_2]{X_{1}}\label{eq:defPhiK}
\eneq
as the composition of the sequence of maps 
\eqn
&&\hs{-3.5ex}\RHH[\Lambda_2]{2}
\simeq \Gamma_{\Lambda_2}\Hom[22^a](\omAI[2],\dA[2])\\
&&\hs{-2.5ex}\to
\Gamma_{\Lambda_{12}\times_{X_2}\Lambda_2}\Hom[11^a](
\shk\dtens[2](\omAI[2]\conv[2]\omA[2]\conv[2]\RDA\shk),
\shk\dtens[2](\dA[2]\conv[2]\omA[2]\conv[2]\RDA\shk))\\
&&\hs{-2.5ex}\to
\Gamma_{\Lambda_{12}\circ\Lambda_2}\Hom[11^a]\bl
\roim{p_{1}}(\shk\dtens[2](\omAI[2]\conv[2]\omA[2]\conv[2]\RDA\shk)),
\reim{p_1}(\shk\dtens[2](\dA[2]\conv[2]\omA[2]\conv[2]\RDA\shk))\br\\
&&\hs{-2.5ex}\simeq
\Gamma_{\Lambda_{12}\circ\Lambda_2}
\Hom[11^a](\shk\sconv[2]\RDA\shk,\shk\conv[2]\omA[2]\conv[2]\RDA\shk)\\
&&\hs{-2.5ex}\to
\Gamma_{\Lambda_{12}\circ\Lambda_2}\Hom[11^a]
(\omAI[1],\dA[1])\simeq
\RHH[{\Lambda_{12}\circ\Lambda_2}]{1}.
\eneqn
The first arrow is obtained by applying the functor
$\shl\mapsto
\shk\dtens[2](\shl\conv[2]\omA[2]\conv[2]\RDA\shk)$,
The last arrow is associated with the morphisms in Lemma~\ref{le:HH3}.
\begin{lemma}\label{le:HH4}
The map 
$\Phi_{\shk}\cl \RHH[\Lambda_2]{X_{2}}\To
 \RHH[\Lambda_{12}\circ\Lambda_2]{X_{1}}$ 
in \eqref{eq:defPhiK} is the map $\hh_{X_{12^a}}(\shk)\conv\;$ given
in \eqref{eq:C1circC2}. 
\end{lemma}

\begin{proof}
In the proof, 
we do not write $\Lambda_{12}$ and $\Lambda_2$.
Let $\lambda=\hh_{12^a}(\shk)\in \RHH{12^a}$ and let $\lambda_2\in\RHH{2}$.
We regard $\lambda$ as a morphism on $X_{12^a21^a}$:
\eqn
&&\lambda\cl\omAI[12^a]\to \shk\detens\RDA\shk\to \dA[12^a].
\eneqn
We regard $\lambda_2$ as a morphism  $\omAI[2]\to\dA[2]$.
Then $\Phi_{\shk}(\lambda_2)$ is obtained as the composition
\eqn
&&\omAI[1]\to\shk\sconv[2]\RDA\shk\to
\roim{p_{1}}\bl\shk\dtens[2](\omAI[2]\conv[2]\omA[2]\conv[2]\RDA\shk)\br\\
&&\quad\To[\lambda_2]
\reim{p_1}\bl\shk\dtens[2](\dA[2]\conv[2]\omA[2]\conv[2]\RDA\shk)\br
\to\shk\conv[2]\omA[2]\conv[2]\RDA\shk\to\dA[_1].
\eneqn
In the following diagram in the category $\Db(\A[11^a]\etens\coro_{X_2\times X_2^a})$, 
we write $\RDD$ instead of $\RDA$ for
sake of brevity:
{\scriptsize
\eqn
\xymatrix{
p_{11^a}^{-1}\omAI[1]\ar[d]\\ 
(\omAI[1]\detens\dA[2^a])\dtens[22^a]\omAI[2]\ar[r]^-{\lambda_2}\ar[d]^\bwr
        &(\omAI[1]\detens\dA[2^a])\dtens[22^a]\dA[2]\ar[d]^\bwr\ar[d]\\
((\omAI[1]\detens\omAI[2^a])\conv[2^a]\omA[2^a])\dtens[22^a]\omAI[2]
 \ar[r]^-{\lambda_2}\ar[d]
     &((\omAI[1]\detens\omAI[2^a])\conv[2]\omA[2^a])\dtens[22^a]\dA[2]
\ar[d]\ar[dr]\\
(\omA[2]\conv[2]\omAI[12^a])\dtens[22^a]\omAI[2]
 \ar[r]^-{\lambda_2}\ar[d]
     &(\omA[2]\conv[2]\omAI[12^a])\dtens[22^a]\dA[2]\ar[d]
&\bl(\shk\detens\RDD\shk)\conv[2^a]\omA[2^a]\br\dtens[22^a]\dA[2]
\ar[d]\ar[dl]_\sim\\
(\shk\detens\omA[2]\conv[2]\RDD\shk)\dtens[22^a]\omAI[2]
\ar[r]^-{\lambda_2}\ar[d]^-\bwr
&(\shk\detens\omA[2]\conv[2]\RDD\shk)\dtens[22^a]\dA[2]\ar[d]^-\bwr&
\bl\dA[12^a]\conv[2^a]\omA[2^a]\br\dtens[22^a]\dA[2]\ar[d]^\bwr\\
\shk\dtens[2](\omAI[2]\conv[2]\omA[2]\conv[2]\RDD\shk)
\ar[r]^-{\lambda_2}
   &\shk\dtens[2](\dA[2]\conv[2]\omA[2]\conv[2]\RDD\shk)\ar[d]
&(\dA[1]\detens\omA[2^a])\dtens[22^a]\dA[2]\ar[d]\\
&\shk\dtens[2]\RDD\shk\ar[r]&\dA[1]\etens\coro_{X_2}[2d_2].
}\eneqn
}
Here we used $\omA[2]\conv[2]L\simeq L\conv[2^a]\omA[2^a]$.
This diagram commutes, and the rows on the top and the right columns 
$p_{11^a}^{-1}\omAI[1]\to
\bl\omA[2]\conv[2](\shk\detens\RDA\shk)\br\dtens[22^a]\dA[2]
\to p_{11^a}^!\dA[1]$ induce
$\lambda\circ\lambda_2\cl \omAI[1]\to \dA[1]$ by adjunction.
Therefore, the diagram 
\eqn
\xymatrix{
\omAI[1]\ar[r]\ar[rrdd]|-{{\rule[-1ex]{0ex}{3ex}\lambda\circ\lambda_2}}
                                                     &\shk\sconv[2]\RDA\shk\ar[r]^-\sim
            &\shk\sconv[2](\omAI[2]\conv[2]\omA[2]\conv[2]\RDA\shk)\ar[d]^-{\lambda_2}\\
            && \shk\conv[2]\dA[2]\conv[2]\omA[2]\conv[2]\RDA\shk\ar[d]\\
           &&\dA[1]
}\eneqn
commutes, which gives the result since the composition of the rows on
the top and the vertical arrows is $\Phi_\shk(\lambda_2)$.
\end{proof}

\begin{theorem}\label{th:HH1}
Let $\Lambda_i$ be a closed subset of $X_i\times X_{i+1}$
\lp$i=1,2$\rp\, 
and assume that
$\Lambda_1\times_{X_2}\Lambda_2$ is proper over $X_1\times X_3$.
 Set $\Lambda=\Lambda_1\circ\Lambda_2$.
Let $\shk_i\in\RD^\Rb_{\coh,\Lambda_i}(\A[X_i\times X_{i+1}^a])$ \lp$i=1,2$\rp. Then
\eq\label{eq:RR10}
&&\hh_{X_{13^a}}(\shk_1\circ\shk_2)=\hh_{X_{12^a}}(\shk_1)\circ \hh_{X_{23^a}}(\shk_2)
\eneq
as elements of $\RHH[\Lambda]{X_1\times X_3^a}$.
\end{theorem}
\begin{proof}
For the sake of simplicity, we assume that $X_3=\rmpt$.
Consider the diagram in which we set $\lambda_2=\hh_{2}(\shk_2)
\in\RHH{X_2}\simeq\Hom(\omAI[2],\dA[2])$ and
we write $\RDD$ instead of $\RDA$:
\eqn
\xymatrix{
\omAI[1]\ar[r]\ar `d[dddr] [dddr]
   &{\db{\shk_1\conv[2]\omAI[2]\conv[2]\omA[2]\conv[2]\RDD\shk_1}}
\ar[r]^-{\lambda_2}\ar[d]
&{\db{\shk_1\conv[2]\dA[2]\conv[2]\omA[2]\conv[2]\RDD\shk_1}}\ar[r]&\dA[1]\\
&\shk_1\conv[2](\shk_2\detens\RDD\shk_2)\conv[2]\omA[2]\conv[2]\RDD\shk_1
                                               \ar[ru]\ar[d]^-\bwr&&\\
&(\shk_1\conv[2]\shk_2)\detens\RDD\shk_2\conv[2]\omA[2]\conv[2]
\RDD\shk_1\ar[d]^-\bwr&&\\
&{\db{(\shk_1\conv[2]\shk_2)\detens \RDD(\shk_1\conv[2]\shk_2)}}
\ar`r[rruuu][rruuu]
}\eneqn
Here, the left horizontal arrow on the top is the composition
of the morphisms
$\omAI[1]\to
\shk_1\conv[2]\RDA\shk_1\to
\shk_1\conv[2]\omAI[2]\conv[2]\omA[2]\conv[2]\RDA\shk_1$.
The composition of the arrows on the bottom is 
$\hh_{1}(\shk_1\circ \shk_2)$ 
by Lemma~\ref{le:HH2} and the  composition of the 
arrows on the top is $\Phi_{\shk_1}(\hh_{2}(\shk_2))$.
Hence, the assertion follows from
the commutativity of the diagram by Lemma~\ref{le:HH4}.
\end{proof}
Recall Diagram~\ref{diag:kerngrgr1}. Using \eqref{eq:hhGr}, we get the
commutative diagram
\eq\label{diag:kernhh1}
&&\xymatrix{
\rmK_{\coh,\Lambda_1}(\A[12^a])\times
\rmK_{\coh,\Lambda_2}(\A[23^a])\ar[r]^-\conv\ar[d]^-{\hh_{12^a}\times\hh_{23^a}}
    &\rmK_{\coh,\Lambda}(\A[13^a])\ar[d]^-{\hh_{13^a}}\\
\RHH[\Lambda_1]{12^a}\times \RHH[\Lambda_2]{23^a}\ar[r]^-\conv
 &\RHH[\Lambda]{13^a}.
}\eneq

\begin{remark}
(i) The fact that  Hochschild homology of $\OO[]$-modules
 is functorial seems to be well-known, although we do not know 
any paper in which it is explicitly stated 
(for closely related results, see {\em e.g.,} \cite{Sh,Ca2,Hu}). 

\noindent
(ii) In \cite{Ca-W}, its authors 
interpret Hochschild homology as a morphism of
functors and the action of kernels as a $2$-morphism in a suitable
$2$-category. Its authors claim that the the relation 
$\Phi_{\shk_1}\conv\Phi_{\shk_1}=\Phi_{\shk_1\conv\shk_2}$ follows by
general arguments on $2$-categories. Their result applies
 in a general framework 
including in particular  $\OO[]$-modules in the algebraic case and
presumably $\DQ$-modules
but the precise axioms are not specified in loc.\ cit. See also
\cite{Sh} for related results. Note that, as far as we understand, 
 these authors do not introduce the convolution of
Hochschild homologies and they
did not consider Lemma~\ref{le:HH4} nor Theorem~\ref{th:HH1}. 
\end{remark}

\subsubsection*{Index}
Let $\cora$ be a field, let
$M\in\RD^\Rb_{f}(\cora)$ and  let $u\in\End(M)$. One sets
\eqn
&&\tr(u,M)=\sum_{i\in\Z}(-1)^i\tr(H^i(u)\cl H^i(M)\to H^i(M)),\\
&&\chi(M)=\sum_{i\in\Z}(-1)^i\dim_{\cora}(H^i(M)).
\eneqn
If $X=\rmpt$, then $\HHA$ is isomorphic to $\coro$, and 
$\RD^\Rb_{\coh}(\A[X])=\RD^\Rb_{f}(\coro)$.

Recall that we have set $\shm^\loc=\cor\tens[\coro]\shm$. 
For $M\in\RD^\Rb_{f}(\coro)$ and $u\in \End(M)$, we have
\eq
&&\hh_\rmpt((M,u))=\tr(u^\loc,M^\loc).
\eneq
In particular,
\eqn
\hh_\rmpt(M)&=&\chi(M^\loc).
\eneqn
Moreover, we have
\eqn
\chi(M^\loc)&=&\chi(\gr(M))\\
&=&\sum_{i\in\Z}(-1)^i\bl\dim_{\C}(\C\tens[\coro] H^i(M))-\dim_{\C}
\Tor[\coro]{1}(\C, H^i(M))\br.
\eneqn
In the sequel, we set 
$$\chi(M)\seteq\chi(M^\loc).$$

As a particular case of Theorem~\ref{th:HH1}, 
consider two objects $\shm$ and $\shn$ in $\RD_\coh^\Rb(\A[X])$ 
and assume that $\Supp(\shm)\cap\Supp(\shn)$ is compact. 
Then 
$\RHom[{\A[X]}](\shm,\shn)$ belongs to $\Derb_f(\coro)$ and
\eqn
\chi(\RHom[{\A[X]}](\shm,\shn))&=&\hh_\rmpt(\RDA\shm\conv[X]\shn)\\
&=&\hh_{X^a}(\RDA\shm)\conv[X]\hh_X(\shn)\\
&=&\hh_{X}(\shm)\conv[X]\hh_X(\shn).
\eneqn
Note that we have
\eqn
\chi\bl\RHom[{\A[X]}](\shm,\shn)\br
&=&\chi\bl\RHom[{\Ah}](\shm^\loc,\shn^\loc)\br\\
&=&\chi\bl\RHom[{\gr(\A)}](\gr(\shm),\gr(\shn))\br.
\eneqn

\section{Graded  and localized  Hochschild classes}

\subsubsection*{Graded Hochschild classes}

Similarly to the case of $\A[X]$, one defines 
\index{Hochschild graded@$\HHGA$}%
\eqn
\HHGA&\eqdot&\dGA[X^a]\lltens[\gr({\A[X\times X^a]})]\dGA.
\eneqn
Note that $\HHGA\simeq\C\lltens[\coro]\HHA$
and there is a natural morphism
\eqn
\gr\cl\HHA\to\HHGA.
\eneqn

\begin{notation}\label{not:RHHg}
For a closed subset $\Lambda$ of $X$, we set
\eq\label{eq:HHG1}\index{HHG@$\RHHg[\Lambda]{X}$}%
&&\RHHg[\Lambda]{X}\eqdot H^0\rsect_\Lambda(X;\HHGA[X]).
\eneq
\end{notation}
We also need to introduce
\eq\label{eq:HHG2}\index{HHG@$\RHHhg[\Lambda]{X}$}%
&&\RHHhg[\Lambda]{X}\eqdot \prolim[U]\RHHg[\Lambda]{U},
\eneq
where $U$ ranges over the family of relatively compact open  subsets of $X$.

For $\shf\in\RD^\Rb_{\coh}(\gr(\A))$, one defines its 
Hochschild class $\hh_X(\shf)$  by the same
construction as for $\A$-modules. 
For $\shm\in\RD^\Rb_{\coh}(\A[X])$, we have:
\eqn
&&\gr(\hh_X(\shm))=\hh_X(\gr(\shm)).
\eneqn
Theorem~\ref{th:HH1} obviously also holds when replacing $\A$ with
$\gr(\A)$.
\begin{corollary}\label{cor:HH1gr}
Let $\Lambda_i$ be a closed subset of $X_i\times X_{i+1}$ \lp$i=1,2$\rp\, and assume that
$\Lambda_1\times_{X_2}\Lambda_2$
is proper over $X_1\times X_3$. Set $\Lambda=\Lambda_1\conv\Lambda_2$.
Let $\shk_i\in\RD^\Rb_{\coh,\Lambda_i}(\gr(\A[X_i\times X_{i+1}^a]))$ \lp$i=1,2$\rp. Then
\eq\label{eq:RR11}
&&\hh_{X_{13^a}}(\shk_1\circ\shk_2)
=\hh_{X_{12^a}}(\shk_1)\circ\hh_{X_{23^a}}(\shk_2)
\eneq
as elements of $\RHHg[\Lambda]{X_1\times X_3^a}$.
\end{corollary}
It follows that the diagram below commutes
\eq\label{diag:kernhh2}
&&\xymatrix{
\rmK_{\coh,\Lambda_1}(\gr\A[{12^a}])\times 
\rmK_{\coh,\Lambda_2}(\gr\A[{23^a}])\ar[r]^-{\conv}\ar[d]_-{\hh}
                             &\rmK_{\coh,\Lambda}(\gr\A[{13^a}])\ar[d]_-{\hh}\\
\RHHg[\Lambda_1]{{12^a}}\times\RHHg[\Lambda_2]{{23^a}}\ar[r]^-{\conv}
                             &\RHHg[\Lambda]{{23^a}}.
}\eneq

We shall study 
the Hochschild class of $\sho$-modules
with some details in Chapter~\ref{chapter:Com}.
 
\subsubsection*{Hochschild classes for $\A^\loc$}
One defines 
\index{Hochschild local@$\HHAh$}%
\eqn
\HHAh&\eqdot&\dA[X^a]^\loc\lltens[{\Ah[X\times X^a]}]\dA^\loc.
\eneqn
We have $\HHAh\simeq\cor\tens[\coro]\HHA$.
and there is a natural morphism
\eqn
(\scbul)^\loc\cl\HHA\to\HHAh.
\eneqn
For  $\shf\in\RD^\Rb_{\coh}(\Ah)$, one defines its 
Hochschild class $\hh_X(\shf)$  by the same
construction as for $\A$-modules.
For $\shm\in\RD^\Rb_{\coh}(\A[X])$, setting
 $\shm^\loc=\cor\tens[\coro]\shm$,
we have
\eqn
&&(\hh_X(\shm))^\loc=\hhl_X(\shm^\loc).
\eneqn

Recall that the notion of good modules and the category $\RD^\Rb_{\gd}(\Ah[X])$
have been given in Definition~\ref{def:good}.
One immediately deduces from Theorem~\ref{th:HH1} the following:
\begin{corollary}\label{th:HH1loc}
Let $\Lambda_i$ be a closed subset of $X_i\times X_{i+1}$ \lp$i=1,2$\rp\, and assume that
$\Lambda_1\times_{X_2}\Lambda_2$ is proper over $X_1\times X_3$. 
Set $\Lambda=\Lambda_1\circ\Lambda_2$.
Let $\shk_i\in\RD^\Rb_{\gd,\Lambda_i}(\Ah[X_i\times X_{i+1}^a])$ \lp$i=1,2$\rp. Then
\eq\label{eq:RR12}
&&\hhl_{X_{13^a}}(\shk_1\circ\shk_2)
=\hhl_{X_{12^a}}(\shk_1)\circ\hhl_{X_{23^a}}(\shk_2)
\eneq           
as elements of $\RHHl[\Lambda]{X_1\times X_3^a}$.
\end{corollary}

Using Proposition~\ref{pro:grgr1} and the additivity of the Hochschild class 
in Theorem~\ref{th:addhh}, we find that there is a natural map
\eq
&&\rmhK_{\coh,\Lambda}(\gr\A)\to\RHHhg[\Lambda]{X}.
\eneq
For $\shm\in\RD^\Rb_{\gd,\Lambda}(\Ah[X])$, we denote by $\hhhg_X(\shm)$
the image of $\shm$ by the sequence of maps
\eqn\index{hochschildh@$\hhhg_X(\shm)$}%
&&\RD^\Rb_{\gd,\Lambda}(\Ah[X])\to
\rmhK_{\coh,\Lambda}(\gr\A)\to\RHHhg[\Lambda]{X}.
\eneqn

Let $\Lambda_i$ be a closed subset of $X_i\times X_{i+1}$ \lp$i=1,2$\rp\, and assume that
$\Lambda_1\times_{X_2}\Lambda_2$ is proper over $X_1\times X_3$. 
Set $\Lambda=\Lambda_1\circ\Lambda_2$.

Using the commutativity of Diagrams~\ref{diag:kerngrgr3}, 
we get that the diagram below commutes
\eq\label{diag:kernhh3}
&&\xymatrix{
\Ob\bl\Derb_{\gd,\Lambda_1}(\Ah[{12^a}])\br\times 
            \Ob\bl\Derb_{\gd,\Lambda_2}(\Ah[{23^a}])\br\ar[r]^-{\conv}\ar[d]_-{\gr}
                            &\Ob\bl\Derb_{\gd,\Lambda}(\Ah[{13^a}])\br\ar[d]_-{\gr}\\
\rmhK_{\coh,\Lambda_1}(\gr\A[{12^a}])\times 
\rmhK_{\coh,\Lambda_2}(\gr\A[{23^a}])\ar[r]^-{\conv}\ar[d]
                             &\rmhK_{\coh,\Lambda}(\gr\A[{13^a}])\ar[d]\\
\RHHhg[\Lambda_1]{{12^a}}\times\RHHhg[\Lambda_2]{{23^a}}\ar[r]^-{\conv}
                             &\RHHhg[\Lambda]{{23^a}}.
}\eneq
In other words,
\eq\label{eq:RR13}
&&\hhhg_{{13^a}}(\shk_1\circ\shk_2)
=\hhhg_{{12^a}}(\shk_1)\circ\hhhg_{{23^a}}(\shk_2).
\eneq

\begin{corollary}\label{cor:indexhh1}
Let $\shm,\shn\in\Derb_\gd(\Ah[X])$ 
and assume that $\Supp(\shm)\cap\Supp(\shn)$ is compact.
Then $\RHom[{\A[X]}](\shm,\shn)$ belongs to $\Derb_f(\coro)$ and
\eqn
\chi(\RHom[{\Ah[X]}](\shm,\shn))&=&\hhhg_{X^a}(\RDA\shm)\conv\hhhg_X(\shn)\\
&=&\hhhg_{X}(\shm)\conv\hhhg_X(\shn).
\eneqn
\end{corollary}
\begin{proof}
One has by \eqref{eq:grrmpt}
\eqn
\chi(\RHom[{\Ah[X]}](\shm,\shn))&=&\hh_\rmpt(\RDA\shm\conv\shn)\\
&=&\hhhg_\rmpt(\RDA\shm\conv\shn)=\hhhg_{X^a}(\RDA\shm)\conv\hhhg_X(\shn)
\eneqn 
and the last equality follows from \eqref{eq:dualHochschild}.
\end{proof}

\begin{remark}\label{rem:alg3}
In the algebraic case, that is, in the situation of \S~\ref{section:alg}, 
one should replace $\rmhK_{\coh,\Lambda}$ with $\rmK_{\coh,\Lambda}$
and $\RHHhg[\Lambda]{X}$ with  $\RHHg[\Lambda]{X}$.
\end{remark}

We shall explain how to calculate $\hhhg_X$ in Chapter~\ref{chapter:Com}.

\chapter{The commutative case}\label{chapter:Com} 

We shall make the link between the Hochschild class and
the
Chern and Euler classes of coherent $\OO$-modules, following~\cite{Ka91}, an
unpublished letter from the first named author (M.K) to the second
(P.S), dated 18/11/1991.

\section{Hochschild class of $\sho$-modules}
In this section, we shall study the Hochschild class in the particular
case of a trivial deformation. In this case, the formal parameter
$\hbar$ doesn't play any role, and we may work with $\OO[]$-modules.
We shall use the same notations for $\OO$-modules as for 
$(\OO\forl,\star)$-modules where $\star$ is the usual commutative product.

Note that the results of this section are well known from the
specialists.
Let us quote in particular \cite{Ca2,Ca-W,Hu,Ma,Ra1,Sh,Ye0}.

\medskip

\medskip
Let $(X,\OO[X])$ be a complex manifold of complex dimension $d_X$. 
As usual, we denote by $\de[X]\cl X\hookrightarrow X\times X$ the
diagonal embedding. We denote by $\Omega^i_X$ the sheaf of holomorphic
$i$-forms and one sets  $\Omega_X\eqdot\Omega_X^{d_X}$.
We set
\eqn
&&\oo\eqdot \Omega_X\,[d_X].
\eneqn
We denote by $\RDO$ and $\RDDO$ the duality functors
\eqn
&&\RDO(\shf)=\rhom[{\OO}](\shf,\OO),\quad
\RDDO(\shf)=\rhom[{\OO}](\shf,\oo).
\eneqn
When there is no risk of confusion, we write $\RDD$ and $\RD$ instead
of $\RDO$ and $\RDDO$, respectively.

Let $f\cl X\to Y$ be a morphism of complex manifolds. For 
$\shg\in \RD^\Rb(\OO[Y])$, we set
\eqn
&&\spb{f}\shg\eqdot\OO[X]\lltens[{\opb{f}\OO[Y]}]\opb{f}\shg.
\eneqn
We use the notation 
$H^0(\spb{f})\cl\md[{\OO[Y]}]\to\md[{\OO}]$ for the (non derived) inverse image functor.

The  Hochschild homology \glossary{Hochschild homology!of $\OO[]$}%
of $\OO$ is given by:\index{HH5@$\HHO$}%
\eq\label{def:HHO1}
&&\HHO\eqdot\spb{\de}\oim{\de}\OO,
\mbox{ an object of $\RD^\Rb(\OO)$.}
\eneq
Note that 
$\eim{\de}\simeq\oim{\de}\simeq\roim{\de}$, and moreover
\eq
&&\oim{\de}\HHO\simeq\oim{\de}(\OO\lltens\spb{\de}(\oim{\de}\OO))
\simeq
\dO\lltens[{\OO[X\times X]}]\dO.
\eneq
By reformulating the construction of the Hochschild class for modules
over $\DQ$-algebroids, we get
\begin{definition}\label{def:hhc}
For $\shf\in\Derb_\coh(\OO)$, we define its Hochschild class 
\glossary{Hochschild class!of an $\sho$-module}\index{HHc3@$\hh_X(\shf)$}%
$\hh_X(\shf)\in H^0_{\Supp\shf}(X;\spb{\de}\oim{\de}\OO)$ as the composition
\eq
&&\OO\to\rhom[{\OO}](\shf,\shf)\isoto\spb{\de}(\shf\detens\RDD\shf)
\to \spb{\de}\oim{\de}\OO.
\eneq
Here the morphism 
$\shf\detens\RDD\shf\to \oim{\de}\OO$ is deduced from the morphism
$\spb{\de}(\shf\detens\RDD\shf)\isoto
\shf\lltens[{\OO}]\RDD\shf\to\OO$
by adjunction.
\end{definition}
Applying Theorem~\ref{th:HH1}, we get that for two complex manifolds 
$X$ and $Y$ and for   $\shf\in\RD^\Rb_\coh(\OO[X])$ and
$\shg\in\RD^\Rb_\coh(\OO[Y])$, we have
\eqn
&&\hh_{X\times Y}(\shf\detens\shg)=\hh_X(\shf)\etens\hh_Y(\shg).
\eneqn

Let $f\cl X\to Y$ be a morphism of complex manifolds and denote by
$\Gamma_f\subset X\times Y$ its graph.
We denote by $\hh_{X\times Y}(\OO[\Gamma_f])$ the Hochschild class of the 
coherent $\OO[X\times Y]$-module $\OO[\Gamma_f]$. Hence
\eqn
&&\hh_{X\times Y}(\OO[\Gamma_f])\in H^0(X\times Y;\HHO[X\times Y]).
\eneqn
Applying Theorem~\ref{th:HH1}, we get

\begin{corollary}\label{co:chernO}
\bnum
\item
Let $\shg\in\RD^\Rb_\coh(\OO[Y])$. Then
\eqn
&& \hh_X(\spb{f}\shg)=\hh_{X\times Y}(\OO[\Gamma_f])\circ\hh_Y(\shg).
\eneqn
\item
Let $\shf\in\RD^\Rb_\coh(\OO[X])$ and assume that $f$ is proper on
$\Supp(\shf)$. Then
\eqn
&& \hh_Y(\reim{f}\shf)=\hh_X(\shf)\circ\hh_{X\times Y}(\OO[\Gamma_f]).
\eneqn
\enum
\end{corollary}

In Proposition~\ref{pro:HHOspb} and ~\ref{pro:HHOeim} below, we give
a more direct description of the maps 
$\hh_{X\times Y}(\OO[\Gamma_f])\,\circ\;$ and 
$\circ\,\hh_{X\times Y}(\OO[\Gamma_f])$.

\begin{proposition}\label{pro:HHOspb}
Let $f\cl X\to Y$ be a morphism of complex manifolds.
\bnum
\item
There is a canonical morphism
\eq
&&\spb{f}\spb{\de[Y]}\oim{\de[Y]}\OO[Y]\to \spb{\de}\oim{\de}\OO\label{eq:HHOspb}.
\eneq
\item
This morphism together with the isomorphism 
$\OO[X]\isofrom \spb{f}\OO[Y]$ induces a morphism
\eq\label{eq:hhspbf}
&&\spb{f}\cl H^0(\rsect(Y;\spb{\de[Y]}\oim{{\de[Y]}}\OO[Y]))\to 
H^0(\rsect(X;\spb{\de}\oim{\de}\OO))
\eneq
and for $\shg\in\Derb_\coh(\OO[Y])$, we have
\eq\label{eq:HHspb}
&&\hh_X(\spb{f}\shg)=\spb{f}\hh_Y(\shg).
\eneq
\enum
\end{proposition}
\begin{proof}
(i) Consider the diagram 
\eq\label{diag:fXY}
&&\xymatrix@C=8ex{
X\ar[r]^-{{\de}}\ar[d]_-f&X\times X\ar[d]_-{\tw f}\\
Y\ar[r]^-{{\de[Y]}}&Y\times Y.
}
\eneq
Then we have  morphisms
\eqn
&&\spb{f}\spb{\de[Y]}\oim{{\de[Y]}}\OO[Y]
\simeq \spb{\de}\spb{\tw f}\oim{{\de[Y]}}\OO[Y]
\to\spb{\de}\oim{{\de}}\spb{f}\OO[Y]\simeq \spb{\de}\oim{{\de}}\OO.
\eneqn
Here the arrow $\spb{\tw f}\oim{{\de[Y]}}\to\oim{{\de}}\spb{f}$
is deduced by adjunction from
\eqn
\oim{{\de[Y]}}&\to&\oim{{\de[Y]}}\roim{f}\spb{f}
\simeq\roim{\tw f}\oim{{\de}}\spb{f}.
\eneqn

\noindent
(ii) The diagram
\eqn
\xymatrix{
\spb{\tw f}(\shg\etens\RDD\shg)\ar[d]^-\sim\ar[r]
                              &\spb{\tw f}\oim{\de[Y]}\OO[Y]\ar[d]\\
         \spb{f}\shg\etens\spb{f}\RDD\shg\ar[d]^-\sim
                             &\oim{\de[X]}\spb{f}\OO[Y]\ar[d]^-\sim\\
\spb{f}\shg\etens\RDD\spb{f}\shg\ar[r]
                            &\oim{\de[X]}\OO[X]
}\eneqn
commutes. 
It follows that the diagram below commutes.
\eqn
\xymatrix{
\spb{f}\OO[Y]\ar[r]\ar[ddd]^-\sim
         &\spb{f}\spb{\de[Y]}(\shg\etens\RDD\shg)\ar[d]^-\sim\ar[r]
                               &\spb{f}\spb{\de[Y]}\oim{\de[Y]}\OO[Y]\ar[d]^-\sim\\
         &\spb{\de[X]}\spb{\tw f}(\shg\etens\RDD\shg)\ar[d]^-\sim\ar[r]
                              &\spb{\de[X]} \spb{\tw f}\oim{\de[Y]}\OO[Y]\ar[d]\\
         &\spb{\de[X]}(\spb{f}\shg\etens\spb{f}\RDD\shg)\ar[d]^-\sim
                             &\spb{\de[X]}\oim{\de[X]}\spb{f}\OO[Y]\ar[d]^-\sim\\
\OO[X]\ar[r]
        &\spb{\de[X]}(\spb{f}\shg\etens\RDD\spb{f}\shg)\ar[r]
                            &\spb{\de[X]}\oim{\de[X]}\OO[X].
}\eneqn
Therefore, the image of 
$\hh_Y(\shg)\in\Hom[{\OO[Y]}](\OO[Y], \spb{\de[Y]}\oim{\de[Y]}\OO[Y])$ 
by the maps
\eqn
\Hom[{\OO[Y]}](\OO[Y], \spb{\de[Y]}\oim{\de[Y]}\OO[Y])&\to&
\Hom[{\OO[X]}](\spb{f}\OO[Y], \spb{f}\spb{\de[Y]}\oim{\de[Y]}\OO[Y])\\
&\to&
\Hom[{\OO[X]}](\OO[X], \spb{\de[X]}\oim{\de[X]}\OO[X])
\eneqn
is $\hh_X(\spb{f}\shg)$.
\end{proof}

\begin{remark}
Although we omit the proof, the map in~\eqref{eq:hhspbf}
coincides with $\hh_{X\times Y}(\OO[\Gamma_f])\,\circ\;$.
\end{remark}

\subsubsection*{Ring structure}

For an exposition on tensor categories, we refer to \cite{K-S3}.
 
\begin{proposition}\label{pro:HHOring}
\bnum
\item
The object $\spb{\de}\oim{\de}\OO$ is a ring in the tensor category
$(\Derb_\coh(\OO[X]),\lltens[{\OO}])$. More precisely, 
\banum
\item
the map $\mu$ obtained as the composition 
\eqn
\spb{\de}\oim{\de}\OO\lltens[{\OO}]\spb{\de}\oim{\de}\OO
               &\isoto&\spb{\de}(\oim{\de}\OO\lltens[{\OO[X\times X]}]\oim{\de}\OO)\\
               &\to&\spb{\de}\oim{\de}\OO.
\eneqn
is associative. Here the last arrow is induced by
$\oim{\de}\OO\tens\oim{\de}\OO\to \oim{\de}\OO$.
\item 
$\hh_X(\OO)$ is a unit of this ring. More precisely, 
the natural morphism $\epsilon$ defined as the composition
\eqn
&&\epsilon\cl \OO\isoto\spb{\de}\OO[X\times X]\to\spb{\de}\oim{\de}\OO
\eneqn
has the property that the composition
\eqn
\spb{\de}\oim{\de}\OO\simeq \spb{\de}\oim{\de}\OO\lltens[\sho_X]\OO
&\To[\epsilon]&\spb{\de}\oim{\de}\OO\lltens[{\OO[X]}]\spb{\de}\oim{\de}\OO\\
&\To[\mu]&\spb{\de}\oim{\de}\OO
\eneqn
is the identity.
\eanum
\item
The ring $(\spb{\de}\oim{\de}\OO,\mu)$ is commutative. More precisely, 
we have $\mu\circ\sigma=\mu$, where
$\sigma\in\Aut[{\Db(\OO)}](\spb{\de}\oim{\de}\OO
\lltens[\sho_X]\spb{\de}\oim{\de}\OO)$
is the morphism associated with $x\tens x'\mapsto x'\tens x$. 
\item
The object $\epb{\de}\eim{\de}\oo$ has a structure of
a $\spb{\de}\oim{\de}\OO$-module.
More precisely, the composition
\eqn
\spb{\de}\oim{\de}\OO\lltens[\sho_X]\epb{\de}\eim{\de}\oo
               &\to&\epb{\de}(\oim{\de}\OO\lltens[{\OO[X\times X]}]\eim{\de}\oo)\\
               &\to&\epb{\de}\eim{\de}\oo.
\eneqn
is associative and preserves the unit.
Here, the last arrow is induced by
$\oim{\de}\OO\lltens[{\OO[X\times X]}]\eim{\de}\oo
\simeq\oim{\de}(\spb{\de}\oim{\de}\OO\tens_{\OO}\oo)
\to\oim{\de}(\OO\tens_{\OO}\oo)\simeq\oim{\de}\oo$ by adjunction.

\label{HOassoc}
\enum
\end{proposition}
\begin{proof}
The verification of these assertions is left to the reader.
We only remark that the commutativity and associativity are
consequences of the
corresponding properties of $\oim{\de}\OO$. For example, the commutativity
is the consequence of
the commutativity of
$\oim{\de}\OO\lltens[{\OO[X\times X]}]\oim{\de}\OO\to\oim{\de}\OO$.
\end{proof}
\begin{notation}
For $\lambda_i\in H^0_{\Lambda_i}(X;\spb{\de}\oim{\de}\OO)$ ($i=1,2$), 
we define their product $\lambda_1\scbul\lambda_2$ as the
composition
\eqn
\OO\isoto\OO\lltens[\sho_X]\OO
\to[{\lambda_1\tens\lambda_2}]
\spb{\de}\oim{\de}\OO\lltens[\sho_X]\spb{\de}\oim{\de}\OO
\to[\mu]\spb{\de}\oim{\de}\OO.\eneqn
\end{notation}
\begin{proposition}
Let $\shf_i\in\Derb_\coh(\OO)$ \lp$i=1,2$\rp. Then 
\eq
&&\hh_X(\shf_1\lltens[\sho_X]\shf_2)=\hh_X(\shf_1)\scbul\hh_X(\shf_2).
\eneq
\end{proposition}
\begin{proof}
Consider  the commutative diagram below (in which $\tens$ stands for $\tens[\sho]$):
\eqn
\xymatrix{
\OO\ar[r]\ar `d[dddr][dddr]&
\OO\tens\OO\ar[d]&&\\
&\spb{\de}(\shf_1\etens\RDD\shf_1)\tens\spb{\de}(\shf_2\etens\RDD\shf_2)
                                                          \ar[r]\ar[d]^-\sim
   &\spb{\de}\oim{\de}\OO\tens\spb{\de}\oim{\de}\OO\ar[d]^-\bwr\\
&\spb{\de}\bl(\shf_1\etens\RDD\shf_1)\tens(\shf_2\etens\RDD\shf_2)\br\ar[d]\ar[r]
  &\spb{\de}(\oim{\de}\OO\tens\oim{\de}\OO)\ar[d]\\
&\spb{\de}\bl(\shf_1\tens\shf_2)\etens\RDD(\shf_1\tens\shf_2)\br\ar[r]
  &\spb{\de}(\oim{\de}\OO).
}\eneqn
The composition of the arrows on the top and the right 
gives $\hh_X(\shf_1)\scbul\hh_X(\shf_2)$
and the composition of the arrows on the left and the bottom gives
$\hh_X(\shf_1\lltens[\sho_X]\shf_2)$. 
\end{proof}
Note that
\eqn
&&\hh_X(\shf_1\lltens[\sho_X]\shf_2)=\spb{\de}(\hh_X(\shf_1)\etens\hh_X(\shf_2)).
\eneqn

\section{co-Hochschild class}
\begin{definition}\label{def:hhe}
For $\shf\in\Derb_\coh(\OO)$, we define its co-Hochschild class 
\glossary{co-Hochschild class}%
\index{thh@$\thh_X(\shf)$}%
$\thh_X(\shf)\in H^0_{\Supp\shf}(X;\epb{\de}\eim{\de}\oo)$ as the composition
\eq
&&\OO\to\rhom[{\OO}](\shf,\shf)\simeq\epb{\de}(\shf\detens\RDDO\shf)
\to \epb{\de}\eim{\de}\oo.
\eneq
Here, the morphism $(\shf\detens\RDDO\shf)\to\eim{\de}\oo$ is induced from
$\spb{\de}(\shf\detens\RDDO\shf)\simeq
\shf\lltens[{\OO}]\RDDO\shf\to\oo$
by adjunction.
\end{definition}

Consider the sequence of isomorphisms
\eqn
\spb{\de}\oim{\de}\OO
&\isoto&\OO\lltens[{\OO}]\spb{\de}\oim{\de}\OO
\isoto\epb{\de}(\OO\detens\omega_X)\lltens[{\OO}]\spb{\de}\oim{\de}\OO\\
&\isoto&\epb{\de}((\OO\detens\omega_X)\lltens[{\OO}]\oim{\de}\OO)
\isoto\epb{\de}\oim{\de}(\spb{\de}(\OO\detens\omega_X)\lltens[{\OO}]\OO)\\
&\isoto&\epb{\de}\eim{\de}\omega_X.
\eneqn
We denote by $\htd$ the isomorphism 
\eq\label{eq:td1}
&& \htd\cl \spb{\de}\oim{\de}\OO\isoto \epb{\de}\eim{\de}\omega_X
\eneq
constructed above. For  a closed subset $S\subset X$, we keep the same
notation $\htd$ to denote the isomorphism 
\eq\label{eq:td2}
&& \htd\cl H^0_S(X;\spb{\de}\oim{\de}\OO)\isoto H^0_S(X;\epb{\de}\eim{\de}\omega_X).
\eneq

\begin{proposition}\label{pro:td1}
For $\shf\in\Derb_\coh(\OO[X])$, we have
\eq
&&\thh_X(\shf)=\htd\circ\hh_X(\shf).
\eneq
\end{proposition}
\begin{proof}
The proof follows from the commutativity of the diagram below in which
we use the natural morphism
$\OO\to \epb{\de[X]}(\OO[X]\etens\oo[X])$
\eqn
\xymatrix{
\OO[X]\ar[r]\ar `d[ddddr][ddddr]
&\spb{\de[X]}(\shf\etens\RDD\shf)\ar[r]\ar[d]
 &\spb{\de[X]}\oim{\de[X]}\OO[X]\ar[d]^-\bwr\\
&\epb{\de[X]}(\OO[X]\etens\oo[X])\tens[]
                                  \spb{\de[X]}(\shf\etens\RDD\shf)\ar[r]\ar[d]
 &\epb{\de[X]}(\OO[X]\etens\oo[X])\tens[]\spb{\de[X]}\oim{\de[X]}\OO[X]
                                                                  \ar[d]^-\bwr\\
&\epb{\de[X]}((\OO[X]\etens\oo[X])\tens[](\shf\etens\RDD\shf)\ar[r]\ar[dd])
 &\epb{\de[X]}((\OO[X]\etens\oo[X])\tens[]\oim{\de[X]}\OO[X])\ar[d]^-\bwr\\
&&\epb{\de[X]}\oim{\de[X]}
\bl\epb{\de[X]}(\OO[X]\etens\oo[X])\tens[]\OO[X]\br
                   \ar[d]^-\bwr\\
&\epb{\de[X]}(\shf\etens\RD\shf)\ar[r]
 &\epb{\de[X]}\eim{\de[X]}\oo[X].
}\eneqn
\end{proof}

For a morphism $f\cl X\to Y$  of complex manifolds, we
denote by $\sect_{\rm f-pr}(X;\scbul)$ the functor of global sections
with $f$-proper supports.

\begin{proposition}\label{pro:HHOeim}
Let $f\cl X\to Y$ be a morphism of complex manifolds.
\bnum
\item
There is a canonical morphism
\eq
&&\reim{f}\epb{\de}\eim{\de}\omega_X\to\epb{\de[Y]}\eim{\de[Y]}\omega_Y.
\label{eq:HHOeim}
\eneq
\item
This morphism  together with the morphism $\OO[Y]\to\roim{f}\OO[X]$ induces a morphism
\eq\label{eq:hheimf}
&&\eim{f}\cl H^0(\rsect_{\rm f-pr}(X;\epb{\de}\eim{\de}\oo))\to 
H^0(\rsect(Y;\epb{\de[Y]}\eim{\de[Y]}\omega_Y))
\eneq
and for $\shf\in\RD^\Rb_\coh(\OO[X])$ such that $f$ is proper on $\Supp(\shf)$, we have
\eq\label{eq:HHeim}
&&\thh_Y(\reim{f}\shf)=\eim{f}\,\thh_X(\shf).
\eneq
\enum
\end{proposition}
\begin{proof}
(i) Consider the diagram \eqref{diag:fXY}.
Then we have  morphisms
\eqn
&&\reim{f}\epb{\de[X]}\eim{\de[X]}\oo
\to 
\epb{\de[Y]}\reim{\tw f}\eim{\de[X]}\oo
\simeq 
\epb{\de[Y]}\eim{\de[Y]}\reim{f}\oo
\to\epb{\de[Y]}\eim{\de[Y]}\oo[Y].
\eneqn
Here, the first morphism is deduced by adjunction from 
\eqn
&&\epb{\de[X]}\to\epb{\de[X]}\epb{\tw f}\reim{\tw f}
\simeq \epb{f}\epb{\de[Y]}\reim{\tw f}.
\eneqn

\noindent
(ii) The proof is similar to that of  Proposition~\ref{pro:HHOspb}
and follows from the commutativity of the diagram below in which we
write for short $\eim{f}$ and $\oim{f}$ instead of $\reim{f}$ and
$\roim{f}$ and similarly with $\tw f$.
\eqn
\xymatrix{
\oim{f}\OO[X]\ar[r]
         &\oim{f}\epb{\de[X]}(\shf\etens\RD\shf)\ar[d]^-\sim\ar[r]
                               &\eim{f}\epb{\de[X]}\eim{\de[X]}\oo[X]\ar[d]\\
         &\epb{\de[Y]}\eim{\tw f}(\shf\etens\RDD\shf)\ar[d]^-\sim\ar[r]
                        &\epb{\de[Y]} \eim{\tw f}\eim{\de[X]}\oo[X]\ar[d]^-\sim\\
         &\epb{\de[Y]}(\eim{f}\shf\etens\eim{f}\RD\shf)\ar[d]^-\sim
                             &\epb{\de[Y]}\eim{\de[Y]}\eim{f}\oo[X]\ar[d]\\
\OO[Y]\ar[r]\ar[uuu]&\epb{\de[Y]}(\eim{f}\shf\etens\RD\eim{f}\shf)\ar[r]
                            &\epb{\de[Y]}\eim{\de[Y]}\oo[Y].
}\eneqn
Therefore, the image of 
$\thh_X(\shf)\in\Hom[{\OO[X]}](\OO[X], \epb{\de[X]}\eim{\de[X]}\oo[X])$ 
by the maps
\eqn
\sect_{\rm f-pr}(X;\hom[{\OO[X]}](\OO[X], \epb{\de[X]}\eim{\de[X]}\oo[X]))&\to&
\Hom[{\OO[X]}](\roim{f}\OO[X], \reim{f}\epb{\de[X]}\eim{\de[X]}\oo[X])\\
&\to&
\Hom[{\OO[Y]}](\OO[Y], \epb{\de[Y]}\eim{\de[Y]}\oo[Y])
\eneqn
is $\thh_Y(\eim{f}\shf)$.
\end{proof}

\begin{remark}
Although we omit the proof, the map in \eqref{eq:hheimf}
coincides with $\circ\,\;\hh_{X\times Y}(\OO[\Gamma_f])$.
\end{remark}

\section{Chern and Euler classes of $\sho$-modules}

The  Hodge cohomology of $\OO$ is given by:
\glossary{Hodge cohomology}\index{Hodge@$\HOD[X]$}%
\eq\label{def:HOD}
&&\HOD[X]\eqdot\bigoplus_{i= 0}^{d_X}\Omega_X^i\,[i],
\text{ an object of }\RD^\Rb(\OO).
\eneq
\begin{lemma}\label{le:HHO1}
Let $f\cl X\to Y$ be a morphism of complex manifolds.
There are canonical morphisms
\eq
\detens&\cl&\HOD[X]\etens\HOD[Y]\to\HOD[X\times Y],\label{eq:etensHOD}\\
\spb{f}&\cl&\spb{f}\HOD[Y]\to \HOD[X],\label{eq:spbHOD}\\
\eim{f}&\cl&\reim{f}\HOD[X]\to \HOD[Y].\label{eq:eimHOD}
\eneq
\end{lemma}
\begin{proof}
The morphisms~\eqref{eq:etensHOD},~\eqref{eq:spbHOD} and~\eqref{eq:eimHOD}
are respectively associated with the morphisms
\eqn
&&\Omega_X^i\,[i]\etens\Omega_Y^j\,[j]\To\Omega_{X\times Y}^{i+j}\,[i+j],\\
&&\spb{f}\Omega_Y^i\,[i]\To\Omega_X^i\,[i],\\
&&\reim{f}\Omega_X^{i+d_X}\,[i+d_X]\To \Omega_Y^{i+d_Y}\,[i+d_Y].
\eneqn
\end{proof}

\begin{theorem}\label{th:RR1}
\banum
\item
There is an isomorphism 
\eqn
&&\alpha_X\cl\spb{\de}\oim{\de}\OO\isoto \HOD[X]
\eneqn
which commutes with the functors  $\etens$ and  $\spb{f}$. 
\item
There is an isomorphism 
\eqn
&&\beta_X\cl\HOD[X]\isoto \epb{\de}\eim{\de}\omega_X
\eneqn
which commutes with the functors  $\etens$ and  $\eim{f}$. 
\eanum
\end{theorem}
Setting $\tau\eqdot\opb{\beta_X}\circ\htd\circ\opb{\alpha_X}$, 
we get a commutative diagram in $\RD^\Rb(\OO)$:
\eq\label{diag:RR1}
&&\xymatrix{
\spb{\de}\oim{\de}\OO\ar[d]^-\sim_-{\alpha_X}\ar[rr]^-\sim_{\htd}
                                         &&\epb{\de}\eim{\de}\omega_X\\
\HOD[X]\ar[rr]^-\sim_{\tau}&&\HOD[X]\ar[u]^-\sim_-{\beta_X}.
}\eneq
The construction of $\alpha_X$ and $\beta_X$
and the proof are given in the next section.

\begin{definition}\label{def:cherneuOmd}
For $\shf\in\RD^\Rb_\coh(\OO[X])$, we set
\eq
&&\ch(\shf)=\alpha_X\circ\hh_X(\shf)
\in\bigoplus_{i= 0}^{d_X}H^i_{\Supp(\shf)}(X;\Omega_X^i),
\label{def:chernOmd}\\
&&\eu(\shf)=\beta_X^{-1}\circ\thh_X(\shf)
\in\bigoplus_{i= 0}^{d_X}H^i_{\Supp(\shf)}(X;\Omega_X^i).\label{def:euOmd}
\eneq
We call $\ch(\shf)$ the Chern class 
\glossary{Chern class}%
\index{Chern@$\ch(\shf)$}%
of $\shf$ and $\eu(\shf)$ the
Euler class of $\shf$.
\glossary{Euler class!of $\sho$-modules}%
\index{Euler@$\eu(\shf)$}%
\end{definition}
Of course, $\ch(\shf)$ coincides with the classical Chern character
and the morphism $\alpha_X$ is the so-called
Hochschild-Kostant-Rosenberg map.
\glossary{Hochschild-Kostant-Rosenberg map}%

The following conjecture was stated in~\cite{Ka91}. 
\begin{conjecture}
One has $\eu(\OO)=\td_X(TX)$, where $\td_X(TX)$ is the Todd
\index{Todd@$\td_X$}%
class of the tangent bundle $TX$. \glossary{Todd class}%
\end{conjecture}
This conjecture implies that $\eu(\shf)=\ch(\shf)\cup\td_X(TX)$.
Indeed, 
for $a,b\in H^*(X;\spb{\de}\oim{\de}\OO)$, we have
$\td(a\circ b)=a\circ\td(b)$
by Proposition~\ref{pro:HHOring}~\eqref{HOassoc}
and Lemma~\ref{lem:commod} below.

This conjecture has recently been proved by
 A.~Ramadoss~\cite{Ra1} in the algebraic case and by J.~Grivaux~\cite{Gri}
in the analytic case.

\subsubsection*{An index theorem}

Consider the particular case of two coherent $\OO$-modules $\shl_i$
($i=1,2$) such that $\Supp(\shl_1)\cap\Supp(\shl_2)$ is
compact. In this case we have (see \cite{Hu,Ra1}):
\eq\label{eq:ChTdRa}
\hh_\rmpt(\shl_1\conv\shl_2)&=&\chi(\rsect(X;\shl_1\lltens[\sho_X]\shl_2))\nonumber\\
&=&\int_X(\ch(\shl_1)\cup\ch(\shl_2)\cup\td_X(TX)).
\eneq
We consider the situation of Corollary~\ref{cor:indexhh1}. Hence,
$\A$ is a $\DQ$-algebroid on $X$.

\begin{corollary}\label{cor:indexAh1}
Let $\shm,\shn\in\Derb_{gd}(\Ah)$ and assume that
$K\eqdot\Supp(\shm)\cap\Supp(\shn)$ is compact. 
Let $U$ be a relatively compact open subset of $X$ containing $K$.
Then $\RHom[{\Ah}](\shm,\shn)$ belongs to $\Derb_f(\cor)$ and
its Euler-Poincar{\'e} index is given by the formula
\eqn
&&\chi(\RHom[{\Ah}](\shm,\shn))
=\int_U\ch_U((\gr^U\RDA\shm))\cup\ch_U(\gr^U(\shn))\cup\td_U(TU).
\eneqn
\end{corollary}
\begin{proof}
Applying Corollary~\ref{cor:indexhh1}, we have
\eqn
\chi(\RHom[{\Ah}](\shm,\shn))&=&\hh_{\rmpt}(\RDA\shm\conv\shn)\\
&=&\hh_{\rmpt}(\gr\RDA\shm_0\conv\gr\shn_0),
\eneqn
where $\shm_0$ (resp.\ $\shn_0$) is an object of $\Derb_{\coh}(\A[U])$
which generates $\shm$ (resp.\ $\shn$) on $U$.
Then, the result follows from \eqref{eq:ChTdRa}.
\end{proof}

\section{Proof of Theorem \ref{th:RR1}}

As usual, we denote by $p_i\cl X\times X\to X$ the $i$-th projection ($i=1,2$).
The following lemma is well-known.
\begin{lemma}\label{lem:ChEuK0}
Let  $\shf$ be an $(\OO\boxtimes\OO)$-module supported by the diagonal. Then 
the following conditions are equivalent:
\bnum
\item
$\oim{p_1}\shf$ is a coherent $\OO$-module,
\item
$\oim{p_2}\shf$ is a coherent $\OO$-module.
\enum
If these conditions are satisfied, then the map
$\shf\to \OO[X\times X]\tens_{\OO\boxtimes\OO}\shf$
is an isomorphism. In particular, the $(\OO\boxtimes\OO)$-module structure on
$\shf$ extends uniquely to an $\OO[X\times X]$-module structure.
\end{lemma}
We define the $\opb{p_1}\OO$-module
\eqn
&&P_k\seteq\oim{\de}\Omega^k_X\oplus\oim{\de}\Omega^{k+1}_X\mbox{ for }k\ge0,\quad P_k=0
\mbox{ for }k<0.
\eneqn
We endow the $P_k$'s with a structure of $\opb{p_2}\OO$-module by setting
\eqn
\spb{p_2}(a)(\omega_k\oplus\theta_{k+1})
&=&a\omega_k\oplus(a\theta_{k+1}-da\wedge\omega_k)
\eneqn
for $a\in\OO$, $\omega_k\in\Omega^k_X$, 
$\theta_{k+1}\in\Omega_X^{k+1}$.
This defines an action of $\opb{p_2}\OO$ since
\eqn
\spb{p_2}(a_1)\spb{p_2}(a_2)(\omega_k\oplus\theta_{k+1})
&=&\spb{p_2}(a_1)(a_2\omega_k\oplus(a_2\theta_{k+1}-da_2\wedge\omega_k))\\
&=&a_1a_2\omega_k\oplus(a_1a_2\theta_{k+1}-a_1da_2\wedge\omega_k
-da_1\wedge a_2\,\omega_k)\\
&=&a_1a_2\omega_k\oplus(a_1a_2\theta_{k+1}-d(a_1a_2)\wedge\omega_k)\\
&=&\spb{p_2}(a_1a_2)(\omega_k\oplus\theta_{k+1}).
\eneqn
By Lemma~\ref{lem:ChEuK0}, we get that $P_k$ has a structure of $\OO[X\times X]$-module
and we have an exact sequence:
\eq
&&0\to \oim{\de}\Omega_X^{k+1}\To[\alpha_k] P_k\To [\beta_k]\oim{\de}\Omega^k_X\to0.
\eneq
Hence 
$\oim{\de}\Omega^k[k]\isofrom(\oim{\de}\Omega_X^{k+1}\to P_k)\to
\oim{\de}\Omega_X^{k+1}[k+1]$
defines the morphism
\eqn
&&\xi_k\cl\oim{\de}\Omega^k[k]\to\oim{\de}\Omega_X^{k+1}[k+1].
\eneqn
It induces a morphism
\eq\label{mor:xi}
&&\xi\cl \soplus_k\oim{\de}\Omega^k_X[k]\To\soplus_k\oim{\de}\Omega^k_X[k].
\eneq
Let $d^\stan_k\cl P_k\to P_{k-1}$ be the composition
\eq\label{mor:stan}
&&d^\stan_k\cl P_k\To [\beta_k] \de{}_*\Omega^k_X\To [\alpha_{k-1}]
P_{k-1}.
\eneq
We define the complex $P_\scbul$
whose differential $d_P^{-k}\cl P_k\to P_{k-1}$ is given by $kd^\stan_k$.
Then $\im d^\stan_k\simeq\im\beta_k\simeq\oim{\de}\Omega_X^{k}$ and 
$\ker d^\stan_k\simeq\ker\beta_k\simeq\oim{\de}\Omega_X^{k+1}$. 
Therefore we have a quasi-isomorphism $P_\scbul\to\oim{\de}\OO$.
\begin{lemma}\label{lem:alpha}
The morphism
\eq\label{iso:ch}
&&
\alpha_X\cl\spb{\de}\oim{\de}\OO\to H^0(\spb{\de})(P_\scbul)\simeq\soplus_{k}\Omega^k_X[k]
\eneq
is an isomorphism in $\Db(\OO)$.
\end{lemma}

\begin{proof}[Proof of Lemma \ref{lem:alpha}]
Since the question is local, we may assume that $X$ is a vector space $V$.
Then we have a Koszul complex 
\eqn
&&\OO[X\times X]\otimes\bigwedge^\scbul V^* 
\simeq
\Bigl(\cdots\to  \OO[X\times X]\otimes\bigwedge^2 V^*\to
\OO[X\times X]\otimes V^*\to\OO[X\times X]\Bigr)
\eneqn
and an isomorphism $\OO[X\times X]\otimes\bigwedge^\scbul V^*\to\oim{\de}\OO$
in $\Db(\OO[X\times X])$.
Then applying $H^0(\spb{\de})$, we obtain an isomorphism in $\Db(\OO)$:
\eqn
&&\spb{\de}\oim{\de}\OO\isoto H^0(\spb{\de})(\OO[X\times X]\otimes\bigwedge^\scbul V^*).
\eneqn
The $\C$-linear maps $\bigwedge^kV^*\to\Omega^k_X(V)\to P_k(X\times X)$
induce 
a morphism of complexes
$\OO[X\times X]\otimes\bigwedge^\scbul V^*\to P_\scbul$
such that the diagram below commutes:
\eqn
&&\xymatrix@R=.5ex{
\OO[X\times X]\otimes\bigwedge^\scbul V^*\ar[dr]\ar[dd]\\
           &\oim{\de}\OO.\\
{P_\scbul}\ar[ru]
}
\eneqn
Since 
$H^0(\spb{\de})(\OO[X\times X]\otimes\bigwedge^\scbul V^*)[d_X]\to H^0(\spb{\de})(P_\scbul)$
is an isomorphism, we obtain the desired result.
\end{proof}

\begin{remark}
\bnum
\item
Let $I\subset \OO[X\times X]$ be the defining ideal of the diagonal
set $\de(X)$.
Then the morphism $\xi_0\cl\oim{\de}\OO\to\oim{\de}\Omega_X^1[1]$
is given by the exact sequence
$0\to\oim{\de}\Omega_X^1\to\OO[X\times X]/I^2\to\oim{\de}\OO\to0$.
Indeed, we have a commutative diagram
\[
\xymatrix{
0\ar[r]&I/I^2\ar[r]\ar[d]^\bwr&\OO[X\times X]/I^2\ar[r]\ar[d]^\bwr
&\oim{\de}\OO\ar[r]\ar[d]^{\id}&0\\
0\ar[r]&\oim{\de}\Omega_X^1\ar[r]^{\beta_0}&P_0\ar[r]^{\alpha_0}&\oim{\de}\OO\ar[r]&0.
}\]
Here, the left vertical isomorphism is given by 
$$I/I^2\ni p_1^*(a)-p_2^*(a)\longleftrightarrow
da\in\oim{\de}\Omega^1_X\quad\text{($a\in\OO$).}$$
\item
Moreover the morphism 
$\xi_k\cl \oim{\de}\Omega_X^k[k]\to
\oim{\de}\Omega_X^{k+1}[k+1]$ coincides with
the composition
\eqn
&&\oim{\de}\Omega_X^k[k]\simeq\oim{\de}\Omega_X^k[k]
\lltens[{\OO[X\times X]}]\OO[X\times X]\to
\oim{\de}\Omega_X^k[k]
\lltens[{\OO[X\times X]}]\oim{\de}\OO\\
&&\hs{7ex}\To[\xi_0]
\oim{\de}\Omega_X^k[k]\lltens[{\OO[X\times X]}]\oim{\de}\Omega_X^1[1]
\to \oim{\de}(\Omega_X^k[k]\lltens[{\OO}]\Omega_X^1[1])\\
&&\hs{14ex}\to\oim{\de}\Omega_X^{k+1}[k+1].
\eneqn
\item
Note that the morphism
$\alpha_X\cl\spb{\de}\oim{\de}\OO\isoto\soplus_{k}\Omega^k_X[k]$
coincides with the morphism obtained from
$\oim{\de}\OO\to
\soplus_k\oim{\de}\Omega^k_X[k]\To[\exp(\xi)]\soplus_k\oim{\de}\Omega^k_X[k]$
by adjunction.

\enum
\end{remark}
\begin{lemma}\label{lem:ChernEul1}
The morphism $\alpha_X$ in \eqref{iso:ch} interchanges the 
composition of the ring $\spb{\de}\oim{\de}\OO$ 
given in {\rm Proposition~\ref{pro:HHOring}~(a)}
with the composition 
\eqn
&&\Omega_X^i[i]\lltens[\sho_X]\Omega_X^j[j]\simeq
(\Omega_X^i\lltens[\sho_X]\Omega_X^j)[i+j]\To[\wedge]\Omega_X^{i+j}[i+j].
\eneqn
\end{lemma}
Note that the unit
$\OO\to\spb{\de}\oim{\de}\OO$ is given by
$\OO\simeq\spb{\de}\OO[X\times X]\to \spb{\de}\oim{\de}\OO$,
where the last arrow is induced by $\OO[X\times X]\to\oim{\de}\OO$.
\begin{proof}
We define
$$\mu_{ij}\cl P_i\tens[{\OO[X\times X]}] P_j\to P_{i+j}$$
by 
\eq&&
\ba{l}
\mu_{ij}\bl((\omega_i\oplus\theta_{i+1})
\otimes(\omega_j\oplus\theta_{j+1}))\br\\[1ex]
\hs{5ex}=(\omega_i\wedge\omega_j)\oplus(\theta_{i+1}\wedge\omega_j
+(-1)^i\omega_i\wedge\theta_{j+1}).\ea
\label{def:mu}
\eneq
This map is $\opb{p_2}(\OO)$-bilinear since: 
\eqn
&&\mu_{ij}\Bigl(\bl p_2^*(a)(\omega_i\oplus\theta_{i+1})\br
\otimes (\omega_j\oplus\theta_{j+1})\Bigr)\\
&&\hs{5ex}=\mu_{ij}\Bigl(\bl a\omega_i\oplus(a\theta_{i+1}-da\wedge\omega_i)\br
\otimes (\omega_j\oplus\theta_{j+1})\Bigr)\\
&&\hs{5ex}=(a\omega_i\wedge\omega_j)\oplus\bl(a\theta_{i+1}-da\wedge\omega_i)
\wedge\omega_j
+(-1)^ia\omega_i\wedge\theta_{j+1}\br\\
&&\hs{5ex}=p_2^*(a)
\bl(\omega_i\wedge\omega_j)\oplus(\theta_{i+1}\wedge\omega_j
+(-1)^i\omega_i\wedge\theta_{j+1})\br\\
&&\hs{5ex}=p_2^*(a)\mu_{ij}\bl(\omega_i\oplus\theta_{i+1})
\otimes (\omega_j\oplus\theta_{j+1})\br,
\eneqn
and
\eqn
&&\mu_{ij}\bl(\omega_i\oplus\theta_{i+1})
\otimes p_2^*(a)(\omega_j\oplus\theta_{j+1})\br\\
&&\hs{5ex}=\mu_{ij}\Bigl(\bl\omega_i\oplus\theta_{i+1}\br
\otimes \bl a\omega_j\oplus(a\theta_{j+1}-da\wedge\omega_j\br\Bigr)\\
&&\hs{5ex}=\bl a\omega_i\wedge\omega_j\br\oplus\bl\theta_{i+1}\wedge 
a\omega_j
+(-1)^i\omega_i\wedge(a\theta_{j+1}-da\wedge\omega_j)\br\\
&&\hs{5ex}=(a\omega_i\wedge\omega_j)\oplus(a\theta_{i+1}\wedge 
\omega_j
+(-1)^ia\omega_i\wedge \theta_{j+1}-da\wedge\omega_i\wedge\omega_j)\\
&&\hs{5ex}=p_2^*(a)
\bl\omega_i\wedge\omega_j\oplus(\theta_{i+1}\wedge\omega_j
+(-1)^i\omega_i\wedge\theta_{j+1})\br\\
&&\hs{5ex}=p_2^*(a)\mu_{ij}\bl(\omega_i\oplus\theta_{i+1})
\otimes (\omega_j\oplus\theta_{j+1})\br.
\eneqn
The morphism $\mu$ commutes with the differentials since:
\eqn
&&\mu d\bl(\omega_i\oplus\theta_{i+1})
\otimes(\omega_j\oplus\theta_{j+1})\br\\
&&\hs{5ex}=\mu_{i-1,j}\bl(0\oplus i\omega_{i})
\otimes (\omega_j\oplus\theta_{j+1})\br
+(-1)^i\mu_{i,j-1}
\bl(\omega_i\oplus\theta_{i+1})\otimes(0\oplus j\omega_{j})\br\\
&&\hs{5ex}=0\oplus\bl i\omega_{i}\wedge\omega_j
+(-1)^i(-1)^ij\omega_{i}\wedge\omega_j\br
=0\oplus(i+j)\omega_{i}\wedge\omega_j\\
&&\hs{5ex}=d\mu\bl(\omega_i\oplus\theta_{i+1})
\otimes(\omega_j\oplus\theta_{j+1})\br.
\eneqn
Hence we have a commutative diagram in $\Derb(\sho_{X\times X})$
\eqn
&&\xymatrix{
{\oim{\de}\OO\lltens[{\OO[X\times X]}]
  \oim{\de}\OO}\ar[r]\ar[d]&{\oim{\de}\OO}\\
P_\scbul\tens[{\OO[X\times X]}]  P_\scbul\ar[r]^-\mu&P_\scbul\ar[u]\,.
}
\eneqn
Therefore, applying $\spb{\de}$, the morphism
$\spb{\de}\oim{\de}\OO\lltens\spb{\de}\oim{\de}\OO\to\spb{\de}\oim{\de}\OO$
is represented by
\eqn
&&H^0(\spb{\de})P_\scbul\tens[\sho_X]  H^0(\spb{\de})P_\scbul\to H^0(\spb{\de})P_\scbul.
\eneqn
Thus we obtain the desired result.
\end{proof}

\begin{lemma}
Consider a morphism $f\cl X\to Y$. Then the diagram below commutes:
\eqn
&&\xymatrix{
\spb{f}\spb{\de[Y]}\oim{\de[Y]}\OO[Y]\ar[r]\ar[d]^{{\alpha_Y}}
&\spb{\de}\oim{\de}\OO\ar[d]^{{\alpha_X}}\\
\spb{f}(\soplus_k \Omega_Y^k[k])\ar[r]&\soplus_k\Omega_X^k[k]\,.
}\eneqn
\end{lemma}
\begin{proof}
Let $\tilde{f}\cl X\times X\to Y\times Y$ be the morphism associated with $f$.
Let us denote by $P^X_\scbul$ the complex on $X$ constructed above.
Then we easily construct  a commutative diagram
\eqn
&&\xymatrix{
H^0(\spb{\tilde{f}})P^Y_\scbul\ar[r]\ar[d]_\phi
                        &H^0(\spb{\tilde{f}})\oim{\de[Y]}\OO[Y]\ar[d]\\
P^X_\scbul\ar[r]        &\oim{\de}\OO
}\eneqn
such that
\eqn
&&\xymatrix@C=7ex{
H^0(\spb{\de}\spb{\tilde{f}})P^Y_\scbul\ar[r]\ar[d]_{\spb{\de}\phi}
        &H^0(\spb{f}\spb{\de[Y]})P^Y_\scbul\ar[r]^\sim
                                 &\spb{f}(\soplus_k\Omega_Y^k[k])\ar[d]^\psi\\
H^0(\spb{\de})P^X_\scbul\ar[rr]&&\soplus_k\Omega_X^k[k]
}
\eneqn
commutes 
where $\psi$ is given in \eqref{eq:spbHOD}.
\end{proof}

Now we set 
\eq
Q_k=\begin{cases}
P_{k-1}&\text{for $1\le k\le d_X$,}\\
\de{}_*\OO&\text{for $k=0$,}\\
0&\text{otherwise.}
\end{cases}
\eneq
and define  the differential $d^Q$ with $d_k^Q=(k-1-d_X)\;d_{k-1}^{\stan}$,
where $d_{k-1}^{\stan}$ is given by~\eqref{mor:stan} and $d_0^{\stan}\cl\OO\oplus\Omega_X^1\to\OO$
is the canonical morphism.
Then $Q_\scbul$ is a complex of $\OO[X\times X]$-modules and the canonical homomorphism
$\Omega_X^{d_X}\to\Omega_X^{d_X-1}\oplus\Omega_X^{d_X}$ induces a morphism of complexes
$\oim{\de}\omega_X\to Q_\scbul$, 
which is an isomorphism in $\Db(\OO[X\times X])$.

Let us denote by $H^0(\epb{\de})$ the functor
$\opb{\delta_X}\hom[\sho_{X\times X}](\oim{\delta}\sho_X,\scbul)$.
\begin{lemma} 
The morphism
\eqn
&&\beta_X\cl\soplus_k\Omega_X^k\simeq
H^0(\epb{\de})Q_\scbul\to\epb{\de}\oim{\de}\omega_X
\eneqn
is an isomorphism in $\Db(\OO)$.
\end{lemma}
Since the proof is similar to that of Lemma~\ref{lem:alpha}, we
omit it.

Note that the morphism $\beta_X$ coincides with the morphism obtained 
by adjunction from
\eqn
&&\soplus_k\eim{\de}\Omega_X^k\To[{\exp(-\xi)}]\soplus_k\eim{\de}\Omega_X^k
\to\eim{\de}\Omega_X^n[n]\simeq\eim{\de}\omega_X.
\eneqn

\begin{lemma} \label{lem:commod}
The morphism
$\spb{\de}\oim{\de}\OO\lltens[\sho_X]\epb{\de}\eim{\de}\oo\to\epb{\de}\eim{\de}\oo$
in {\rm Proposition~\ref{pro:HHOring}~(d)} coincides with
$\Omega_X^i[i]\tens[\sho_X]\Omega_X^j[j]\To[\bigwedge]\Omega_X^{i+j}[i+j]$. 
\end{lemma}
\begin{proof}
We define the morphism 
$\mu_{ij}\cl P_i\tens[{\OO[X\times X]}] Q_j\to Q_{i+j}$ by the same formula as 
in~\eqref{def:mu}.
Then it commutes with the differential.
Indeed the proof is similar to that of Lemma~\ref{lem:ChernEul1}
except when $i+j=d_X+1$. In this case, 
\eqn
&&\mu d\bl(\omega_i\oplus\theta_{i+1})
\otimes(\omega_{j-1}\oplus\theta_{j})\br
=0\oplus(i+j-d_X-1)\omega_{i}\wedge\omega_{j-1}=0.
\eneqn
With this morphism $\mu\cl P_\scbul\tens[{\OO[X\times X]}] Q_\scbul\to
Q_\scbul$,
the following diagram in the category of complexes is commutative:
\eqn
\xymatrix{
{P_\scbul\tens[{\OO[X\times X]}] Q_\scbul}\ar[r]^(.6)\mu\ar[d]
&{Q_\scbul}\ar[d]\\
{P_\scbul\tens[{\OO[X\times X]}] \de{}_!\oo}\ar[r]&{\de{}_!\oo}\,.
}
\eneqn
Thus we have a commutative diagram in $\Db(\OO)$:
\eqn
&&\xymatrix@C=2.8ex{
{H^0(\spb{\de})P_\scbul\tens[\sho_X] H^0(\epb{\de})Q_\scbul}\ar[r]\ar[d]^{\bwr}
&{H^0(\epb{\de})(P_\scbul\tens[{\OO[X\times X]}] Q_\scbul)}\ar[r]\ar[d]
&{H^0(\epb{\de})(Q_\scbul)}\ar[d]\\
{\spb{\de}\oim{\de}\OO\lltens[\sho_X]\epb{\de}\eim{\de}\oo}\ar[r]
&{\epb{\de}(\oim{\de}\OO\lltens[{\OO[X\times X]}]\eim{\de}\oo)}\ar[r]
&{\epb{\de}\eim{\de}\oo\,.}\\
}
\eneqn
\end{proof}

Recall that in Corollary~\ref{cor:hoschdualizcp}, we have constructed
a morphism $\HHA[X]\tens\HHA[X]\to\omega_{X_\R}^{\rm top}$. Let us
describe its image via the isomorphisms $\alpha_X$ and $\beta_X$.
Consider the diagram
\eq\label{diag:hoschdualizcp} 
&&\ba{c}\xymatrix{
\HHO\tens\HHO\ar[d]_-\lambda\ar[rd]^-u&\\
\HOD[X]\tens\HOD[X]\ar[r]_-v    &\omega_{X_\R}^{\rm top}\,.
}\ea\eneq
Here, $u$ is the map given by Corollary~\ref{cor:hoschdualizcp},
$\lambda$ is the isomorphism $\alpha_X\tens\opb{\beta_X}$ and $v$ is the
composition
\eqn
&&\soplus_k\Omega_X^k[k]\tens\soplus_k \Omega_Y^k[k]\To\soplus_k\Omega_Y^k[k]
\to\omega_{X_\R}^{\rm top},
\eneqn
where the first morphism is given by the wedge product and the last
one by the map $\Omega_X^{d_X}[d_X]\to \omega_{X_\R}^{\rm top}$.
Then 
diagram~\eqref{diag:hoschdualizcp} commutes.

\chapter{Symplectic case and $\shd$-modules}\label{chapter:Symp}

\section{Deformation quantization on cotangent bundles}
\label{section:DQcotg}

Consider the case where $X$ is an open subset of the cotangent bundle $T^*M$
of a complex manifold $M$. We denote by $\pi\cl T^*M\to M$ the
projection. As usual, we denote by $\shd_M$ 
\index{D@$\shd_M$}%
the  $\C$-algebra of differential operators on $M$. This is a right
and left Noetherian  sheaf of rings.
 
The space $T^*M$ is endowed with the filtered sheaf of $\C$-algebras 
$\HE[T^*M]$ 
\index{Ehat@$\HE[T^*M]$}%
of formal microdifferential operators of \cite{S-K-K}, 
and its subsheaf $\HEo[T^*M]$ of operators of order $\leq 0$. 

On $T^*M$, there is also a $\DQ$-algebra, denoted by $\HWo[T^*M]$
\index{What@$\HWo[T^*M]$}%
and constructed in~\cite{P-S} as follows. 
Consider the complex line $\C$ endowed with the coordinate $t$ and
denote by $(t;\tau)$ the associated symplectic coordinates on
$T^*\C$. Let $T^*_{\tau\neq0}(M\times\C)$ be the open subset of
$T^*(M\times\C)$ defined by $\tau\neq0$ and consider the map
\eqn
&&\rho\cl T^*_{\tau\neq0}(M\times\C)\to T^*M,
\quad(x,t;\xi,\tau)\mapsto (x; {\tau}^{-1}\xi).
\eneqn
Denote by $\HEo[T^*(M\times\C),\widehat t]$ the subalgebra of $\HEo[T^*(M\times\C)]$
consisting of operators not depending on $t$, that is, commuting with
$\partial_t$. Setting $\hbar=\opb{\partial_t}$, the  $\DQ$-algebra $\HWo[X]$
is defined as
\eqn
&&\HWo[X]=\oim{\rho}\HEo[T^*(M\times\C),\widehat t].
\eneqn
One denotes by $\HW[T^*M]$ 
\index{W@$\HW[T^*M]$}%
the localization of
$\HWo[T^*M]$, that is, $\HW[T^*M]=\cor\tens[\coro]\HWo[T^*M]$.
\begin{remark}
One shall be aware that $\HE[T^*M]$ and $\HEo[T^*M]$ are denoted by 
$\HE[M]$ and $\HEo[M]$, respectively, in \cite{S-K-K}. Similarly,  
 $\HW[T^*M]$ and  $\HWo[T^*M]$ are denoted by 
$\HW[M]$ and $\HWo[M]$, respectively, in \cite{P-S}.
\end{remark}

There are natural morphisms of algebras
\eq\label{eq:DtoW}
&&\pi_M^{-1}\shd_M\hookrightarrow \HE[T^*M]\hookrightarrow  \HW[T^*M].
\eneq
\begin{lemma}\label{lem:WfffE}
\banum
\item
The algebra $\HWo[T^*M]$ is faithfully flat over $\HEo[T^*M]$.
\item
The algebra $\HW[T^*M]$ is faithfully flat over $\HE[T^*M]$.
\item
 $\HE[T^*M]$ is flat over $\pi_M^{-1}\shd_M$.
\eanum
\end{lemma}
\begin{proof}
In the sequel, we set $X=T^*M$.
For an $\HEo[X]$-module $\shm$, we set 
\eqn
&&\shm^\rmW\eqdot \HWo[X]\tens[{\HEo[X]}]\shm, \\
&&\gre(\shm)=\bl\HEo[X]/\HE[X](-1)\br\lltens[{\HEo[X]}]\shm.
\eneqn
Note that the  analogue of Corollary~\ref{cor:conservative1} holds for
$\HEo[X]$-modules, that is, the functor $\gre$ above is conservative
on $\Derb_\coh(\HEo[X])$. We have
\eq\label{eq:grrmWgr}
&&\gr(\shm^\rmW)\simeq \OO[X]\tens[\sho_X(0)]\gre(\shm),
\eneq
where $\sho_X(0)$ denotes the subsheaf of $\OO[X]$ of sections
homogeneous of degree $0$ in the fiber variable of the vector bundle $T^*M$, 
and $\OO[X]$ is faithfully flat over $\sho_X(0)$.

\vspace{0.2cm}
\noindent
(a)~(i) Let us first prove the result outside of the zero-section, that
is, on $T^*M\setminus T_M^*M$.
Let us show that 
\eq
&&H^j(\HWo[X]\lltens[{\HEo[X]}]\shm)=0\quad\text{
for any $j<0$}
\label{eq:weplat}
\eneq
holds for any coherent $\HEo[X]$-module $\shm$.
\vspace{0.1cm}
\noindent
First assume that $\shm$ is torsion-free, i.e.,
$\HE(-1)\tens[{\HEo}]\shm\to\shm$ is a monomorphism.
Since $\OO[X]$ is flat over $\sho_X(0)$, 
$$\gr(\HWo[X]\lltens[{\HEo[X]}]\shm)\simeq
\OO[X]\lltens[\sho_X(0)]\gre(\shm)$$
has zero cohomologies in degree $<0$.
Hence Proposition~\ref{pro:grHa} implies
\eqref{eq:weplat}.

\vspace{0.1cm}
\noindent
Now assume that
$\HE(-1)\shm=0$. Then we have
\eqn
\HWo[X]\lltens[{\HEo[X]}]\shm
&\simeq&
\HWo[X]\lltens[{\HEo[X]}]\HEo[X]\lltens[{\HEo[X]}]\sho_X(0)\lltens[\sho_X(0)]
\shm\\
&\simeq&
\HWo[X]\lltens[{\HEo[X]}]\sho_X(0)\lltens[\sho_X(0)]\shm\\
&\simeq&
\OO\lltens[\sho_X(0)]\shm,
\eneqn
which implies \eqref{eq:weplat}.

Since any coherent $\HEo$-module is a successive extension of 
torsion-free $\HEo$-modules and $(\HEo/\HE(-1))$-modules,
we obtain 
\eqref{eq:weplat} for any coherent $\HEo$-module.

\vspace{0.1cm}
\noindent

Consider a coherent $\HEo$-module $\shm$ and assume that
$\shm^\rmW\simeq0$. Then $\gr(\shm^\rmW)\simeq0$ and this implies that 
$\gre(\shm)\simeq0$ in view of~\eqref{eq:grrmWgr}
since $\OO[X]$ is faithfully flat over $\sho_X(0)$. Since $\gre$ is conservative, the
result follows.

\vspace{0.2cm}
\noindent
(a)~(ii) To prove the result in a neighborhood of the zero section, we use
the classical trick of the dummy variable. Let $(t;\tau)$ denote a
homogeneous symplectic coordinate system on $T^*\C$.
Consider the functors
\eqn
\alpha\cl\mdcoh[{\OO[M]}]&\to&\mdcoh[{\HEo[{X\times T^*\C}]\vert_{\tau\not=0}}],\\
 \shm&\mapsto&\shm\detens(\HEo[\C]/\HEo[\C]\cdot t),\\
\beta\cl\mdcoh[{\HWo[X\vert_M]}]&\to&\mdcoh[{\HWo[{X\times T^*\C}]\vert_{\tau\not=0}}],\\
 \shm&\mapsto&\shm\detens(\HWo[T^*\C]/\HWo[T^*\C]\cdot t).
\eneqn
These two functors $\alpha$ and $\beta$ are exact and faithful. Then
the result follows from (a)~(i).

\vspace{0.2cm}
\noindent
(b)~(i) Here again, we prove the result first on $T^*M\setminus T^*_MM$. 
In this case, it follows from the isomorphism
\eqn
&&\HW[X]\simeq \HWo[X]\tens[{\HEo[T^*\C]}]\HE[T^*\C].
\eneqn 

\vspace{0.2cm}
\noindent
(b)~(ii) The case of the zero-section is deduced from (b)~(i)
similarly as for (a).

\vspace{0.2cm}
\noindent
(c) is proved for example in~\cite[Th.~7.25]{Ka2}. 
\end{proof}
Recall that for a coherent $\shd_M$-module $\shm$, the support of 
$\HE[T^*M]\tens[\opb{\pi_M}\shd_M]\opb{\pi_M}\shm$ is called its
characteristic variety and denoted by $\chv(\shm)$. 
\glossary{characteristic variety}\index{chv@$\chv$}%
It is a closed
$\C^\times$-conic complex analytic involutive subset of $T^*M$.

Now assume that $M$ is open in some finite-dimensional 
$\C$-vector space.
Denote by $(x)$ a linear coordinate system on $M$ and by 
$(x;u)$ the associated  symplectic coordinate system on $T^*M$.
Let $f,g\in\OO[X]\forl$. In this case, the $\DQ$-algebra $\HWo[X]$ is
isomorphic to the star algebra $(\OO[X]\forl,\star)$ where:
\eq\label{eq:wstar}
f\star g
&=&\sum_{\alpha\in\N^n} \dfrac{\hbar^{\vert\alpha\vert}}{\alpha !} 
(\partial^{\alpha}_uf)(\partial^{\alpha}_xg).
\eneq
This product is similar to the product of the total symbols of 
differential operators on $M$ and indeed, the
morphism of $\C$-algebras $\pi_M^{-1}\shd_M\To \HW[X]$ is given by
\eqn
f(x)\mapsto f(x),&&\partial_{x_i}\mapsto \hbar^{-1}u_i.
\eneqn

Note that there also exists an analytic version of $\HE[T^*M]$ and $\HW[T^*M]$,
obtained by using the $\C$-subalgebra  of $(\OO[X]\forl,\star)$ consisting of 
sections $f=\sum_{j\geq 0}f_j\hbar^j$ of $\sho_X\forl(U)$ 
($U$ open in $T^*M$) satisfying:
\eq\label{eq:defW2}
&&\left\{  \parbox{58ex}{
for any compact subset $K$ of $U$ there exists a positive 
constant $C_K$ such that 
$\sup\limits_{K}\vert f_{j}\vert \leq C_K^{j}j!$ for all $j> 0$.
}\right. 
\eneq
They are the total symbols of the analytic 
(no more formal) microdifferential operators of \cite{S-K-K}.

\begin{remark}\label{rem:canonicA}
(i) Let $X$ be a complex symplectic manifold. Then
 $X$ is locally isomorphic to an open subset of a cotangent 
bundle $T^*M$, for a complex manifold $M$ (Darboux's theorem), and it
is  a well-known fact that if $\A$ is a $\DQ$-algebra and the
associated Poisson structure is the symplectic structure of $X$, then 
$\A$ is locally isomorphic to $\HWo[T^*M]$.

\noindent
(ii) On $X$, there is a canonical $\DQ$-algebroid, still denoted
by $\HWo[X]$. 
It has been constructed in \cite{P-S}, after 
\cite{Ka1} had first treated  the contact case. Clearly, any $\DQ$-algebroid $\sha$ is equivalent to 
$\HWo[X]\tens[\coro_X]\shp$, where $\shp$ is an invertible $\coro_X$-algebroid.
It follows that the $\DQ$-algebroids on $X$ are
classified by  $H^2(X;(\coro_X)^\times)$. See~\cite{Po}  for a detailed study.

\noindent
(iii) Using~\eqref{eq:HHAP}, we get the isomorphism
\eq
&&\HHA[X]\simeq\HHH[{\HWo[X]}].
\eneq
\end{remark}

\section{Hochschild homology of $\A[]$}

Throughout this section, $X$ denotes a complex manifold 
endowed with a $\DQ$-algebroid $\A$ such that the associated Poisson structure
is symplectic. Hence, $X$ is symplectic and we
denote by $\alpha_X$ the symplectic $2$-form on $X$. 

We set $2n=d_X$, $Z=X\times X^a$ and we denote by $dv$ the volume form
on $X$ given by $dv=\alpha_X^{n}/n!$.

\begin{lemma}\label{le:simpleAmod1}
Let $\Lambda$ be a smooth Lagrangian submanifold of $X$ and let 
$\shl_i$ \lp$i=0,1$\rp\, be simple $\A$-modules along $\Lambda$. Then:
\bnum
\item
$\shl_0$ and $\shl_1$ are locally isomorphic,
\item
the natural morphism $\coro\to\hom[{\A}](\shl_0,\shl_0)$ is an isomorphism.
\enum
\end{lemma}
Note that the lemma above does not hold if one removes the hypothesis 
that $X$ is symplectic (see Example~\ref{exa:notisoC}).
\begin{proof}
(i) We may assume that $X=T^*M$ for a complex manifold $M$, $\A=\HWo[T^*M]$.
Choose a local coordinate system $(x_1,\dots,x_n)$ on $M$, 
and denote by $(x;u)$ the associated coordinates on $X$. 
We shall identify the section $u_i$ of $\A$ with the differential
operator $\hbar\partial_i$. 

We may assume that $\Lambda$ is the zero-section $T^*_MM$ and 
$\shl_0=\OO[M]\forl\simeq\A/\shi_0$, where $\shi_0$ is the left ideal generated by 
$(\hbar\partial_1,\dots,\hbar\partial_n)$.
Since $\shl_1$ is simple, it locally admits a generator, say $u$. 
Denote by $\shi_1$ the annihilator ideal of $u$ in $\A$. Since
$\shi_1/\hbar\shi_1$ is reduced, there exist 
sections $(P_1,\cdots,P_n)$ of $\A$ such that 
\eqn
&&\{\hbar\partial_1+\hbar P_1,\dots,\hbar\partial_n+\hbar P_n\}\subset \shi_1.
\eneqn
By identifying $\HWo[T^*M]$ with the sheaf of microdifferential
operators of order $\leq 0$ in the variable 
$(x_1,\dots,x_n,t)$ not depending on $t$ and $\hbar$ with
$\opb{\partial_t}$, a classical result of \cite{S-K-K}
(see also \cite[Th~6.2.1]{Sc} for an exposition) shows that
there exists an invertible section $P\in\A$ such that $\shi_0=\shi_1P$. 
Hence, $\shl_1\simeq\shl_0$.

\vspace{0.2cm}
\noindent
(ii) We may assume $\shl_0=\OO[M]\forl$. Then 
$\hom[{\A}](\OO[M]\forl,\OO[M]\forl)$ is isomorphic to the kernel of the map
\eqn
&&u\cl \OO[M]\forl\to (\OO[M]\forl)^n,\quad u=(\hbar\partial_1,\dots,\hbar\partial_n).
\eneqn
\end{proof}
Recall that the objects $\OA$ and $\oA$ are defined in \S~\ref{section:dual}.

\begin{lemma}\label{le:symp1}
There exists a  local system $L$ of rank one over
$\coro_X$ such that $\OA\simeq L\tens[\coro_X]\dA$ in $\md[{\A[X\times X^a]}]$.
\end{lemma}
\begin{proof}
Both $\OA$ and $\dA$ are simple $\A[X\times X^a]$-modules
along the diagonal $\Delta$.
By Lemma~\ref{le:simpleAmod1}, 
$L\seteq\hom[{\A[Z]}](\dA,\OA)$ is a local system of rank one over
$\coro$ and we have $\OA\simeq L\tens[\coro_X]\dA$. 
\end{proof}
Note that this implies the isomorphisms
\eq
\RDAA[X\times X^a]\dA&\simeq& L^{\otimes-1}\tens\dA\,[-d_X].\label{eq:symp0}
\eneq
Hence we obtain the chain of morphisms
\eqn
L&\to&L\tens\rhom[{\A[Z]}](\dA,\dA)
\simeq L\tens\RDAA[X\times X^a]\dA\lltens[{\A[Z]}]\dA\\
&\simeq&\dA\lltens[{\A[Z]}]\dA\,[-d_X]=\HHA\,[-d_X]
\simeq L^{\otimes-1}\tens\OA\lltens[{\A[Z]}]\dA\,[-d_X]\\
&\to&L^{\otimes-1}\tens\OA\lltens[{\DA}]\dA\,[-d_X]
\simeq L^{\otimes-1}.
\eneqn
Therefore, we get the morphism:
\eq\label{eq:symp1}
&&L\isoto H^{-d_X}(\HHA)\to L^{\otimes-1}.
\eneq

\begin{lemma}\label{le:symp2}
\bnum
\item $\gr(L)\to\hom[{\gr(\A[Z])}](\gr(\dA),\gr(\OA))\simeq\Omega_X$
gives an isomorphism 
$\gr(L)\isoto\C_X\cdot dv$.
\item
The morphism 
$L^{\otimes 2}\to\coro_X$ induced by \eqref{eq:symp1} decomposes as
$L^{\otimes 2}\to[{\;\phi\;}]\hbar^{2n}\coro_X\hookrightarrow \coro_X$ and
$\phi$ is an isomorphism.
\item
The diagram below commutes:
\eqn
&&\xymatrix@C=7ex{\gr(L^{\otimes2})\ar[r]^(.45)\sim\ar[d]^\bwr&
\gr(\hbar^{2n}\coro_X)&\gr(\coro_X)\ar[l]_(.4){\sim}^(.4){\hbar^{2n}}
\ar[d]^\bwr\\
(\gr(L))^{\otimes2}\ar[r]^\sim&\C_X^{\otimes2}\ar@{-}[r]^\sim&\C_X
}\eneqn
\enum
\end{lemma}
\begin{proof}
The question being local,
we may assume to be given a local coordinate system
 $x=(x_1,\dots,x_{2n})$  on $X$ and a scalar-valued 
non-degenerate skew-symmetric 
matrix $B=(b_{ij})_{1\leq i,j\leq 2n}$ such that the symplectic
 form $\alpha_X$ is given by 
\eqn
&&\alpha_X=\sum_{i,j}b_{ij}\,dx_i\wedge dx_j.
\eneqn
We set  
$$ A= (a_{ij})_{1\leq i,j\leq 2n}=\opb{B}.$$

We may assume that
$\A=(\OO\forl,\star)$ is a star-algebra with a star product
$$f\star g=\Bigl(\exp\bl\ssum_{ij}\dfrac{\hbar a_{ij}}{2}
\dfrac{\partial^2}{\partial x_i\partial x'_j}\br f(x)g(x')\Bigr)\rule[-1.5ex]{.1ex}{4ex}_{\;x'=x}.$$

Set
$$\delta_i= \sum_{j=1}^{2n}a_{ij}\partial_{x_j} \,\lp i=1,\dots,2n\rp.$$
Then,
the $\coro$-linear morphisms from $\OO{\forl}$ to  $\shd_X\forl$ 
\eq\label{eq:OstartoDlrb}
&&\Phi^l\cl f\mapsto f\,\star,\quad\Phi^r\cl f\mapsto \star\,f
\eneq
are given by
\eqn
&&\Phi^l(x_i)= x_i+ \dfrac{\hbar}{2}\delta_i,\quad
\Phi^r(x_i)= x_i- \dfrac{\hbar}{2}\delta_i.
\eneqn
These morphisms define the morphism 
\eq\label{eq:morPHI}
&&\Phi\cl \A\tens\A[X^a]\to\shd_X\forl\\
&&x_i\mapsto x_i+\dfrac{\hbar}{2}\delta_i,\quad y_i\mapsto
x_i-\dfrac{\hbar}{2}\delta_i.
\nonumber\eneq
where we denote by  $y=(y_1,\dots,y_{2n})$ a copy of the local coordinate system
on $X^a$.

We identify $\OA$ with the $(\shd_X\forl)^\rop$-module $\Omega_X\forl$. 
Then, regarding $\Omega_X\forl$ as an $\A[Z]$-module through 
$\A[Z]\vert_X\to\A[Z]^\rop\vert_X\to(\shd_X\forl)^\rop$, we have
\eqn
x_i(a\,dv)&=&(a\,dv)\Phi^r(x_i)=(a\,dv)(x_i-\dfrac{\hbar}{2}\delta_i)\\
          &=&((x_i+\dfrac{\hbar}{2}\delta_i)a)dv
\eneqn
and similarly
\eqn
&&y_i(adv)=((x_i-\dfrac{\hbar}{2}\delta_i)a)dv.
\eneqn
Hence, $a\mapsto a\,dv$ gives an $\A[Z]$-linear isomorphism
\eqn
&&\dA\simeq\OO[X]\forl\isoto\Omega_X\forl\simeq\OA.
\eneqn
Hence it gives an isomorphism
$L\seteq\hom[{\A[Z]}](\dA,\OA)\simeq\hom[{\A[Z]}](\dA,\dA)\simeq\coro_X$,
and the induced morphism
$\gr(L)\to\hom[{\gr(\A[Z])}](\gr(\dA),\gr(\OA))\simeq\Omega_X$
gives an isomorphism $\gr(L)\isoto\C_X\, dv$.
Hence we obtain (i).

For a sheaf of $\coro$-modules $\shf$, we set
\eqn
&&\shf^{(p)}=\bl\bigwedge^p(\coro_X)^{2n}\br\tens[{\coro_X}]\shf.
\eneqn
Let $(e_1,\dots,e_{2n})$ be the basis of $(\coro)^{2n}$. 
Consider the Koszul complex $K^\scbul(\A[Z];b)$ where
$b=(b_1,\dots,b_{2n})$, $b_i=(x_i-y_i)$ is the right
 multiplication by $(x_i-y_i)$ on $\A[Z]$:
\eqn
K^\scbul(\A[Z];b)&\eqdot&
0\to\A[Z]^{(0)}\to[b]\cdots\to[b] \A[Z]^{(2n)}\to 0,\\
b&=&\sum_i\scbul\wedge b_ie_i\cl K^p(\A[Z];b)\to K^{p+1}(\A[Z];b).
\eneqn

On the other hand, consider the Koszul complex $K^\scbul(\shd_X\forl;\delta)$ 
where $\delta=(\delta_1,\dots,\delta_{2n})$:
\eqn
K^\scbul(\shd_X\forl;\delta)&\eqdot&
0\to(\shd_X\forl)^{(0)}\to[\delta]\cdots\to[\delta](\shd_X\forl)^{(2n)}\to 0,\\
 \delta&=&(\delta_1,\dots,\delta_{2n}).
\eneqn
There is a quasi-isomorphism
$K^\scbul(\A[Z];b)\to[qis]\dA\,[-2n]$ in 
the category of complexes in $\md[{\A[Z]}]$.

Set
$\delta=(\delta^l_1-\delta^r_1,\dots,\delta^l_{2n}-\delta^r_{2n})$.
Then the morphism
$\Phi$ in \eqref{eq:morPHI} sends $(x_i-y_i)$ to
$\hbar\delta_i$. 
Consider the Koszul complex $K^\scbul(\shd_X\forl;\delta)$:
\eqn
K^\scbul(\shd_X\forl;\delta)&\eqdot&
0\to(\shd_X\forl)^{(0)}\to[\delta]\cdots\to[\delta](\shd_X\forl)^{(2n)}\to 0,\\
 \delta&=&(\delta_1,\dots,\delta_{2n}).
\eneqn
There is a quasi-isomorphism 
$K^\scbul(\shd_X\forl;\delta)\to[qis]\OO[X]\forl\,[-2n]$. Therefore we get a commutative diagram in $\md[{\A[Z]}]$:
\eqn&&
\xymatrix@C=4ex{
0\ar[r]&\A[Z]^{(0)}\ar[r]^-{b}\ar[d]^-{\hbar^{2n}\Phi}
                        &\cdots\ar[r]&\A[Z]^{(2n-1)}\ar[r]^-b\ar[d]^-{\hbar\Phi}
                            &\A[Z]^{(2n)}\ar[d]^-{\hbar^0\Phi}\ar[r]&0\\
0\ar[r]&(\shd_X\forl)^{(0)}\ar[r]^-{\delta}
     &\cdots\ar[r]&(\shd_X\forl)^{(2n-1)}\ar[r]^-{\delta}&(\shd_X\forl)^{(2n)}\ar[r]&\,0.
}\eneqn
The object $\OA\lltens[{\A[Z]}]\dA$ is obtained by applying the
functor $\OA\tens[{\A[Z]}]\scbul$ to the row on the top and the 
object $\OA\lltens[{\DA}]\dA$ is obtained by applying the
functor $\OA\tens[{\DA}]\scbul$ to the row on the bottom.
By identifying $\OA$ with $\Omega_X\forl$, the morphism 
$\OA\lltens[{\A[Z]}]\dA\,[-d_X]\to \OA\lltens[{\DA}]\dA\,[-d_X]$
is described by the morphism of complexes:
\eq\label{eq:oAtensAtooAtensD}
&&\ba{c}
\xymatrix{
0\ar[r]&\Omega^{0}_X\forl\ar[r]^-{\hbar d}\ar[d]^-{\hbar^{2n}}
                   &\cdots\ar[r]&\Omega^{2n-1}_X\forl\ar[r]^-{\hbar d}\ar[d]^-{\hbar}
                        &\Omega^{2n}_X\forl\ar[d]^-{\hbar^0}\ar[r]&0\\
0\ar[r]&\Omega^{0}_X\forl\ar[r]^-{d}
               &\cdots\ar[r]&\Omega^{2n-1}_X\forl\ar[r]^-{d}&\Omega^{2n}_X\forl\ar[r]&\,0.
}\ea\eneq
Here $d$ denotes the usual exterior derivative. 

Therefore, we find the commutative diagram with exact rows:
\eqn
\xymatrix{&L^{\otimes 2}\ar[d]^\bwr\\
0\ar[r]&\coro_X\ar[r]\ar[d]^{\hbar^{2n}}
&\Omega^0_X\forl\ar[r]^-{\hbar d}\ar[d]^-{\hbar^{2n}}
                     &\Omega^{1}_X\forl\ar[d]^-{\hbar^{2n-1}}\\
0\ar[r]&\coro_X\ar[r]&\Omega^0_X\forl\ar[r]^-{d}&\Omega^{1}_X\forl
}\eneqn
in which the morphism $L^{\otimes 2}\to\coro_X$ corresponds to the
morphism $L[d_X]\to L^{\otimes -1}\tens\OA\tens[{\A[Z]}]\dA$.

This completes the proof.
\end{proof}

\begin{theorem}\label{th:symp1}
Assume that $X$ is symplectic. 
\bnum
\item
Let $L$ be the local system given by 
Lemma~\ref{le:symp1}. Then there is a canonical $\coro$-linear isomorphism 
$ L\isoto\hbar^{d_X/2}\coro_X$,
hence, a canonical $\A[Z]$-linear isomorphism 
\eq\label{eq:OAisodA}
&&\OA\isoto\hbar^{d_X/2}\coro_X\tens[\coro_X]\dA. 
\eneq
\item
The isomorphism~\eqref{eq:OAisodA}
together with \eqref{eq:symp1} induce  canonical morphisms
\eq
&&\hbar^{d_X/2}\coro_X\,[d_X]\to[\iota_X]\HHA\to[\tau_X]
\hbar^{-d_X/2}\coro_X\,[d_X]
\eneq
and the composition $\tau_X\circ\iota_X $ is the
canonical morphism
$\hbar^{d_X/2}\coro_X\,[d_X]\to\hbar^{-d_X/2}\coro_X\,[d_X]$.
\item
$H^j(\HHA)\simeq 0$ unless $-d_X\leq j\leq 0$ and the morphism
$\iota_X$  induces an isomorphism
\eq
&&\iota_X\cl\hbar^{d_X/2}\coro_X\isoto H^{-d_X}(\HHA).\label{eq:tracedens}
\eneq
In particular, there is a canonical non-zero section in $H^{-d_X}(X;\HHA)$. 
\enum
\end{theorem}
\begin{proof}
(i) By Lemma~\ref{le:symp2}, we have an isomorphism  
$(\hbar^{-d_X/2}L)^{\tens 2}\simeq\coro_X$ together with a compatible isomorphism
$\gr(\hbar^{-d_X/2}L)\simeq\C_X$. This  implies $\hbar^{-d_X/2}L\simeq\coro_X$
since the only invertible element $a\in\coro$ satisfying
$a^2=1$, $\sigma_0(a)=1$ is $a=1$. 

\vspace{0.2cm}
\noindent
(ii)--(iii)\quad Denote by $(\Omega_X^\scbul\forl,\hbar d)$ and  
$(\Omega_X^\scbul\forl, d)$ 
the complexes given by the top row and the bottom row of 
\eqref{eq:oAtensAtooAtensD}, respectively.
The morphism $\iota_X$ is represented by 
\eqn
&&L[d_X]\to L^{\otimes -1}\tens(\Omega_X^\scbul\forl,\hbar d)[d_X]
\eneqn
and the morphism $\tau_X$ is the composition 
\eqn
&&L^{\otimes -1}\tens(\Omega_X^\scbul\forl,\hbar d)[d_X]\to
L^{\otimes-1}\tens(\Omega_X^\scbul\forl, d)[d_X]\isoto L^{\otimes-1}[d_X].
\eneqn
\end{proof}
Applying Theorem~6.2.4 together with Corollary~\ref{co:duality}, we obtain:

\begin{corollary}\label{cor:CY}
Let $X$ be a compact complex symplectic manifold. 
Then $\RD^\Rb_{\gd}(\Ah[X])$ is a Calabi-Yau triangulated category of dimension $d_X$ over $\cor$.
\end{corollary}

\begin{remark}
The statement in Theorem~9.2~(ii) of~\cite{KS08} is not correct. 
If $Y$ is a compact complex contact manifold of dimension $d_Y$, then the
 dimension of the Calabi-Yau category associated 
to it in loc.\ cit.\ is $d_Y$, not $d_Y-1$.
\end{remark}

\section{Euler classes of $\Ah[]$-modules}

\begin{theorem}\label{th:symp2}
The complex $\HHAh$ is concentrated in degree $-d_X$ and 
the morphisms $\iota_X$  and $\tau_X$ in Theorem~\ref{th:symp1}
induce isomorphisms
\eq
&&\cor_X\,[d_X]\isoto[\iota_X]\HHAh\isoto[\tau_X]\cor_X[d_X].
\label{eq:tracedenslh}
\eneq
\end{theorem}
\begin{proof}
This follows from the fact that
$(\Omega_X^\scbul\forl,\hbar d)\to(\Omega_X^\scbul\forl,d)$
becomes a quasi-isomorphism after applying the functor 
$(\scbul)^\loc=\cor\tens[\coro](\scbul)$.
\end{proof}

\begin{definition}\label{def:eulerclforAh}
Let $\shm\in\Derb_\coh(\Ah)$. We set
\eq
&&\eu_X(\shm)=\tau_X(\hh_X(\shm))\in H^{d_X}_{\Supp(\shm)}(X;\cor_X)
\eneq
and call $\eu_X(\shm)$ the Euler class of $\shm$.
\glossary{Euler class!of $\Ah$-modules}%
\index{Euler class@$\eu_X(\shm)$}%
\end{definition}
\begin{remark}\label{rem:BGFT1}
(i) The existence of  a canonical section in $H^{-d_X}(X;\HHAh)$ is well
known when $X=T^*M$ is a cotangent bundle, see in particular \cite{B-G,F-T,W}. 
It is intensively used in  \cite{B-N-T} where  these authors call it  
the ``trace density map''.

\noindent
(ii) The Hochschild and cyclic homology of an algebroid stack have
been defined in \cite{B-G-N-T2} where the Chern character of a perfect
complex is constructed in the negative cyclic homology. It gives in
particular an alternative construction of the Hochschild class of a
coherent $\DQ$-module, but it is not clear whether the two
constructions give the same class.
\end{remark}

Consider the diagram
\eq\label{diag:DR}
&&\xymatrix{
\eim{p_{13}}(\opb{p_{12}}\HHAh[X_1\times X^a_2]\tens\opb{p_{23}}\HHAh[X_2\times X^a_3])
\ar[r]^-{\star}\ar[d]^-{\tau_{12^a}\tens\tau_{23^a}}
               & \HHAh[X_1\times X^a_3]\ar[d]^-{\tau_{13^a}}\\
\eim{p_{13}}(\opb{p_{12}}\cor_{X_{12}}{}[d_{12}]\tens\opb{p_{23}}
\cor_{X_{23}}{}[d_{23}])\ar[r]^-{\int_2(\cdot\cup\cdot)}
               &\cor_{X_{13}}{}[d_{13}].
}\eneq
Here, the horizontal arrow in the bottom denoted by 
$\int_2(\scbul\cup\scbul)$ is obtained 
by taking the cup product and integrating on $X_2$ (Poincar{\'e} duality), 
using the fact that 
the manifold $X_2$ has real dimension $2\,d_{2}$ and is oriented.
The arrow in the top denoted by $\star$ 
is obtained by Proposition~\ref{pro:compHH}.
\begin{proposition}\label{pro:HHtoDRcirc}
{\rm Diagram~\ref{diag:DR}} commutes.
\end{proposition} 
\begin{proof}
Since $X_1$ and $X_3$ play the role of parameter spaces, we may assume that 
$X_1=X_3=\rmptt$. 
We set $X_2=X$ and denote by $a_X$ the projection $X\to\rmptt$.
We are reduce to prove the commutativity of the
diagram below:
\eq\label{diag:DR2}
&&\xymatrix{
\eim{a_X}(\HHAh[X]\tens\HHAh[X])\ar[d]_-{\tau\tens\tau}\ar[rrd]^-\star&&\\
\eim{a_X}(\cor_X\,[d_X]\tens\cor_X\,[d_X])\ar[rr]_-{\int_X(\cdot\cup\cdot)}&&\cor.
}
\eneq
This will follow by applying the functor $\eim{a_X}$ to Diagram~\ref{diag:DR3} below. 
\end{proof}
\begin{lemma}\label{le:HHtoDR2}
The diagram below commutes.
\eq\label{diag:DR3}
&&\xymatrix{
\HHAh[X]\tens\HHAh[X]\ar[rd]^-\star\ar[d]_-{\tau\tens\tau}&\\
\cor_X\,[d_X]\tens\cor_X\,[d_X]\ar[r]&\cor_X\,[2d_X].
}\eneq
\end{lemma}
\begin{proof}
The morphism $L\tens L[2d_X]
\simeq\cor_X\,[d_X]\tens\cor_X\,[d_X]\to\cor_X\,[2d_X]$
is given by
\eqn
L\tens L[2d_X]&\to& L[d_X]\tens\rhom[{\A[Z]}](\dA,\OA)[d_X]\\
&\simeq& L\tens
\RDA\dA{[d_X]}\tens[{\A[Z]}]\oA
\simeq\dA[X^a]\tens[{\A[Z]}]\oA\to\C^\hbar_X[2d_X].
\eneqn
On the other hand,
$L\tens L[2d_X]\to \HHA[X]\tens\HHA[X]\to\C^\hbar_X[2d_X]$
is given by
\eqn
L\tens L[2d_X]&\to& \rhom[{\A[Z^a]}](\RDA\dA,\dA[X^a])\tens\rhom[{\A[Z]}](\dA,\oA)\\
&\simeq &
\rhom[{\A[Z^a]}](\RDA\dA,\dA[X^a])\tens\bigl(\RDA(\dA)\tens[{\A[Z]}]\oA\bigr)\\
&\to&\dA[X^a]\tens[{\A[Z]}]\oA\to\C^\hbar_X[2d_X].
\eneqn
These two morphisms give the same morphism from $L\tens L[2d_X]$ to
$\C^\hbar_X[2d_X]$.
\end{proof}
\begin{corollary}\label{cor:EPindAhmod}
Let $\shk_i\in\RD^\Rb_{\coh}(\Ah[X_i\times X_{i+1}^a])$ \lp$i=1,2$\rp. 
Assume that the projection $p_{13}$ defined on
$X_1\times X_2\times X_3$ is proper on 
$\opb{p_{12}}\Supp(\shk_1)\cap\opb{p_{23}}\Supp(\shk_2)$. 
Then
\eq\label{eq:EPindAhmod}
&&
\eu_{X_{13^a}}(\shk_1\conv[2]\shk_2)
=\int_{X_2}\eu_{X_{12^a}}(\shk_1)\cup\eu_{X_{23^a}}(\shk_2).
\eneq
\end{corollary}
\begin{remark}\label{rem:BGFT2}
Consider an object $\shm\in\Derb_\coh(\Ah)$. Then, according to
Definition~\ref{def:eulerclforAh}, its Euler class is well-defined in
the de Rham cohomology of $X$ with values in $\cor$. Now assume that $\shm$ is generated
by $\shm_0\in\Derb_\coh(\A)$ and consider $\gr(\shm_0)$. Assume for
simplicity that $\gr(\A)=\OO$ (the general case can be treated with
suitable modifications). Then $\gr(\shm_0)\in\Derb_\coh(\OO)$ and we
may consider its Chern class in de Rham cohomology. 
A natural question is to compare these
two classes. A precise conjecture had been made in the case of
$\shd$-modules by one of the authors (PS) and J-P.~Schneiders in~\cite{S-Sn} 
and proved by P.~Bressler, R.~Nest and B.~Tsygan in \cite{B-N-T}. These authors,
together with A.~Gorokhovsky, recently treated the general case of
$\DQ$-algebroids in the symplectic setting in \cite{B-G-N-T3}. The
formula they obtain makes use of a cohomology class naturally
associated to the deformation $\A$.
\end{remark}
\section{Hochschild classes of $\shd$-modules}\label{section:HHD}

We shall apply the preceding result to the study of the Euler class of $\shd$-modules.

Recall after \cite{Ka2} that  a coherent $\shd_M$-module $\shm$ is 
{\em good} 
\glossary{good!$\shd$-modules}%
if, for any open relatively compact set $U\subset M$, 
there exists a coherent sub-$\sho_U$-module $\shf$ of $\shm\vert_U$ which
generates it on $U$ as a $\shd_M$-module. One denotes by 
$\RD^b_\gd(\shd_{M})$ the full sub-triangulated category of 
$\RD^b_\coh(\shd_{M})$ consisting of objects with good cohomology.
 
{}From now on, we set
\eqn
&&X=T^*M.
\eneqn
We introduce the functor
\eq\label{eq:DtoW2}
(\scbul)^\rmW\cl \md[\shd_M]&\to&\md[{\HW[X]}]\\
\shm&\mapsto&\HW[X]\tens[\opb{\pi_M}\shd_M]\opb{\pi_M}\shm.\nonumber
\eneq
\index{MW@$\shm^\rmW$}
The next result shows that one can, in some
sense,  reduce the
study of $\shd$-modules to that of $\HW[X]$-modules.

\begin{proposition}\label{pro:DtoW2}
The functor $\shm\mapsto \shm^\rmW\vert_{T^*_MM}$ is exact and faithful.
\end{proposition}
\begin{proof}
The morphism
\eqn
\shm&\to&(\HE[T^*M]\tens[\opb{\pi_M}\shd_M]\opb{\pi_M}\shm)\vert_{T^*_MM}.
\eneqn
is an isomorphism, and hence the result is a particular case of
Lemma~\ref{lem:WfffE}. 
\end{proof}
It follows that $(\scbul)^\rmW$ sends $\Derb_\coh(\shd_M)$ to
$\Derb_\coh(\HW[X])$ and $\Derb_\gd(\shd_M)$ to
$\Derb_\gd(\HW[X])$.

\begin{definition}
Let $\shm\in \Derb_\gd(\shd_M)$. We set 
\index{Hochschild class@$\hhg_X(\shm)$}%
\glossary{Hochschild class!of an $\shd$-module}%
\eq\label{eq:hhgDmod}
&&\hhg_X(\shm)=\hhg_X(\shm^\rmW)\in \RHHO[{\chv(\shm)}]{X}.
\eneq
\end{definition}

For $\Lambda$ a closed subset of $T^*M$, we denote by
$\rmK_{\gd,\Lambda}(\shd_M)$ the Grothendieck group of the full abelian
subcategory of $\mdgd[\shd_M]$ consisting of $\shd$-modules whose
characteristic is contained in $\Lambda$. 

Let $V$ be an open relatively compact subset of $M$.
By slightly modifying the proof of 
 Proposition~\ref{pro:grgr1}, we get morphisms of groups \eq\label{eq:KDmodW}
&& \rmK_{\gd,\Lambda}(\shd_M)
\To \rmK_{\coh,\Lambda}(\OO[\opb{\pi}V]).
\eneq

Let $M_i$ ($i=1,2,3$) be three complex manifolds and set $X_i=T^*M_i$. Denote by 
$q_{ij}$ the $ij$-th projection defined on $M_1\times M_2\times M_3$ and by 
$p_{ij}$ the $ij$-th projection defined on $X_1\times X_2\times X_3$ ($1\leq i<j\leq 3$).
We set, as for $\DQ$-algebras,
$\shd_{M^a}\eqdot(\shd_M)^\rop$ and we write for short 
$M_{ij}$ or $M_{ij^a}$ instead of $M_i\times M_j$ or $M_i\times M_j^a$ and
similarly with  $X_{ij}$. We also write $\shd_{ij}$ instead of $\shd_{M_{ij}}$
and similarly with $ij^a$, etc. For example,
\eqn
&&\shd_{12^a}=
\sho_{M_{12}}\tens[(\sho_{M_1}\etens\sho_{M_2})](\shd_{M_1}\etens(\shd_{M_2})^\rop).
\eneqn
Then $\shd_1$ may be regarded as a $\shd_{11^a}$-module supported on
the diagonal of $X_1\times X_{1^a}$.
Let $\shk_i\in\Derb(\shd_{ij^a})$ ($i=1,2$, $j=i+1$).
Set 
\eqn
&&\shk_1\conv[M_2]\shk_2\eqdot
\reim{q_{13^a}}
\Bigl(\shd_{2}\tens[{\shd_{2^a2}}]\shd_{12^a23^a}
\lltens[{\shd_{12^a}\etens\shd_{23^a}}]
(\shk_1\etens\shk_2)\Bigr).
\eneqn

\begin{theorem}\label{th:DtoW}
Let $\Lambda_i$ be a closed subset of $X_i\times X_{i+1}$ \lp$i=1,2$\rp\, and assume that
the projection $p_{13}$ defined on $X_1\times X_2\times X_3$ is proper on 
$\opb{p_{12}}\Lambda_1\cap\opb{p_{23}}\Lambda_2$. Set $\Lambda=\Lambda_1\circ\Lambda_2$.
Let $\shk_i\in\Derb_\gd(\shd_{ij^a})$ \lp$i=1,2$, $j=i+1$\rp\, with
$\chv(\shk_i)\subset\Lambda_i$ \lp$i=1,2$\rp.
 Then 
$\shk_1\conv[M_2]\shk_2\in\Derb_\gd(\shd_{13^a})$,
$\chv(\shk_1\conv[M_2]\shk_2)\subset\Lambda$  and
\eq\label{eq:DtoW3}
&&(\shk_1\conv[M_2]\shk_2)^\rmW\isoto \shk_1^\rmW\conv[X_2]\shk_2^\rmW.
\eneq
\end{theorem}
The proof is straightforward and is left to the reader.
By using Diagram~\ref{diag:kernhh3}, we get:
\begin{theorem}\label{th:hhgDkern}
In the situation of Theorem~\ref{th:DtoW}, let $V_{ij}$ be a
relatively compact open subset of $M_i\times M_j$ \lp$i=1,2$, $j=i+1$\rp\,
and assume that $\opb{\pi}V_{12^a}\times_{M_2}\opb{\pi}V_{23^a}$ contains 
$(\Lambda_1\times_{X_2}\Lambda_2)\cap\opb{q_{13^a}}\opb{\pi}V_{13^a}$.
Then the diagram below commutes
\eqn
&&\xymatrix{
\Derb_{\gd,\Lambda_1}(\D[12^a])\times 
            \Derb_{\gd,\Lambda_2}(\D[23^a])\ar[r]^-{\conv}\ar[d]_-{\gr}
                            &\Derb_{\gd,\Lambda}(\D[13^a])\ar[d]_-{\gr}\\
\rmK_{\coh,\Lambda_1}(\OO[\opb{\pi}V_{12^a}])\times 
\rmK_{\coh,\Lambda_2}(\OO[\opb{\pi}V_{23^a}])\ar[r]^-{\conv}\ar[d]_-{\hh\times\hh}
                          &\rmK_{\coh,\Lambda}(\OO[\opb{\pi}V_{13^a}])\ar[d]_-{\hh}\\
\RHHO[\Lambda_1]{\opb{\pi}V_{12^a}}\times\RHHO[\Lambda_2]{\opb{\pi}V_{23^a}}\ar[r]^-{\conv}
                             &\RHHO[\Lambda]{\opb{\pi}V_{13^a}}.
}\eneqn
In particular
\eq
&&\hhg_{\opb{\pi}V_{13^a}}(\shk_1\conv[2]\shk_2)
=\hhg_{\opb{\pi}V_{12^a}}(\shk_1)\circ\hhg_{\opb{\pi}V_{23^a}}(\shk_2)
\eneq
in $\RHHO[\Lambda]{\opb{\pi}V_{13^a}}$.
\end{theorem}

As a particular case, and using Corollary~\ref{cor:indexAh1}, 
we recover a theorem of Laumon \cite{La} in the analytic
framework. 

\section{Euler classes of  $\shd$-modules}\label{section:EulerDmod}

We keep the notations of \S~\ref{section:HHD} and we set $X=T^*M$.
One defines the Hochschild homology $\HHE[X]$ of $\HE[X]$ and the
Hochschild class $\hh_X(\shm)$ of a coherent $\HE[X]$-module $\shm$
similarly as for $\HHA$.

In the sequel, we identify a coherent $\shd_M$-module $\shm$ with 
$\HE\tens[\pi^{-1}\shd_M]\pi^{-1}\shm$. 
In particular, 
we define by this way the  Hochschild class $\hh_X(\shm)$ of a coherent $\shd$-module 
$\shm$.
Hence
\eq
&&\hh_X(\shm)\in H^{d_X}_{\chv(\shm)}(X;\HHE[X]).
\eneq
\begin{lemma}\label{le:euEeuW}
There is a natural isomorphism 
\eq\label{eq:HHEtoC}
&&\HHE[X]\isoto\C_X\,[d_X]
\eneq
 which makes
the diagram below commutative:
\eqn
\xymatrix{
{\HHE[X]}\ar[d]\ar[r]^-\sim&{\C_X{}[d_X]}\ar[d]\\
{\HHW[X]}\ar[r]^-\sim_-{\tau}&\cor_X\,[d_X].
}
\eneqn 
\end{lemma}
\Proof[Sketch of proof]
We take coordinates $(x_1,\ldots,x_n,u_1,\ldots u_n)$,
and set $\widetilde{\OO}\seteq\prod_{k\le 0}\hbar^{-k}\sho_X(k)$,
where $\sho_X(k)$ is the sheaf of holomorphic functions on $X$ 
homogeneous of degree $k$ with respect to the variables
$(u_1,\ldots,u_n)$.
Then $\widetilde{\OO}$ is isomorphic to $\HEo$ as a sheaf.
Moreover, $\HHE[X]$ is represented by the Koszul complex
of $\partial/\partial x_i$, $\hbar\partial/\partial
u_i\in\shend(\widetilde{\OO})$ ($i=1,\ldots,n$).
On the other hand, as we have seen,
$\HHW[X]$ 
is represented by the Koszul complex
of $\hbar\partial/\partial x_i$, $\hbar\partial/\partial
u_i\in\shend(\sho_X((\hbar)))$ ($i=1,\ldots,n$).
Hence we have a commutative diagram
$$
\xymatrix{
0\ar[r]&{\widetilde{\OO}}\ar[r]\ar[d]^-{\hbar^{-n}}
                   &\cdots\ar[r]&{\widetilde{\OO}{}^{2n}}\ar[r]\ar[d]
                        &{\widetilde{\OO}{}}\ar[d]^-{\hbar^0}\ar[r]&0\\
0\ar[r]&{\sho_X((\hbar))}\ar[r]
               &\cdots\ar[r]&{\sho_X((\hbar))^{2n}}\ar[r]&{\sho_X((\hbar))}
\ar[r]&\,0,
}
$$
in which the top row represents $\HHE[X]$
and the bottom row represents $\HHW[X]$.
\QED
\begin{definition}
Let $\shm\in\Derb_\coh(\HE[X])$. We denote by $\eu_X(\shm)$
the image of $\hh_X(\shm)$ in $H^{d_X}_{\chv(\shm)}(X;\C_X)$ by the morphism in \eqref{eq:HHEtoC}
and call it the Euler class of $\shm$. 
\glossary{Euler class!of $\shd$-modules}%
\end{definition}

The next result immediately follows from Lemma~\ref{le:euEeuW}. 
\begin{proposition}
For  $\shm\in \Derb_\coh(\shd_M)$,  $\eu_X(\shm^W)$ is the image of $\eu_X(\shm)$
by the natural map 
$H^{d_X}_{\chv(\shm)}(X;\C_X)\to H^{d_X}_{\chv(\shm)}(X;\cor_X)$.
\end{proposition}

Applying Theorem~\ref{th:HH1}, we get:
\begin{theorem}\label{th:euDkern}
In the situation of Theorem~\ref{th:DtoW}, one has:
\eq
&&\eu_{13^a}(\shk_1\conv[2]\shk_2)=\eu_{12^a}(\shk_1)\circ\eu_{23^a}(\shk_2)
\eneq
in $H^{d_1+d_3}_{\Lambda_1\circ\Lambda_2}(X_{13};\C_{X_{13}})$.
\end{theorem}

This formula is equivalent to  the results of \cite{S-Sn} on the functoriality of
the Euler class of $\shd$-modules. Note that the results of loc.\
cit.\ also deal with constructible sheaves.

\chapter{Holonomic $\DQ$-modules}\label{chapter:Hol}

The aim of this chapter is to study holonomic $\DQ$-modules  on
symplectic manifolds. 
More precisely, we will prove that, if
$\shl$ and $\shm$ are two  holonomic $\Ah$-modules on a symplectic manifold $X$,
then the complex $\rhom[{\Ah}](\shm,\shl)$ is 
perverse (hence, in particular, $\C$-constructible) over the field $\cor$.
It follows from the preceding results in Chapter~\ref{chapter:Symp} that if
 the intersection of the supports of $\shm$ and $\shl$ is compact,
 then the Euler-Poincar{\'e} index of this complex is given by the
 integral $\int_X\eu_X(\shm)\cdot\eu_X(\shl)$. We show here that 
the Euler class of a holonomic module is  a Lagrangian cycle, which
makes its calculation easy.

If moreover $\shl$ and $\shm$ are  simple holonomic modules supported
on smooth 
Lagrangian submanifolds $\Lambda_0$ and $\Lambda_1$, then the microsupport 
of the complex $\rhom[{\Ah}](\shm,\shl)$ is contained in 
the normal cone $\rmC(\Lambda_0,\Lambda_1)$.
This last result was first obtained in \cite{KS08} in the analytic framework,
that is, using $\W[X]$-modules, not $\HW$-modules, 
which made the proofs much more intricate. 

Finally we prove that, in some sense, 
the complex $\rhom[{\Ah}](\shm,\shl)$ is invariant by Hamiltonian symplectomorphism.

\section{$\A[]$-modules along a Lagrangian submanifold}
Let $X$ be a complex symplectic manifold endowed with a
$\DQ$-algebroid $\A$.

\subsubsection*{The algebra $\AL$}
Let $\Lambda$ be a smooth Lagrangian
submanifold of $X$ and let $\shl$ be  a coherent $\A$-module
simple along $\Lambda$.

Locally, $X$ is isomorphic as a symplectic manifold to  $T^*\Lambda$, the cotangent
bundle to $\Lambda$.
We set for short 
\eqn
\index{Oh@$\Oh[M]$}\index{Ohloc@$\Ohl[M]$}\index{Omegah@$\ohh[M]$}%
&&\Oh[\Lambda]\eqdot\sho_\Lambda\forl,\quad \Ohl[\Lambda]\eqdot\sho_\Lambda\Ls. 
\eneqn
There are local isomorphisms 
\eqn
&&\A\simeq\HWo,\quad \shl\simeq \Oh[\Lambda].
\eneqn
Then $\shend_{\coro}(\shl)\simeq\shend_{\coro}(\Oh[\Lambda])$ 
(see Lemma~\ref{lem:endstack}) and  
the subalgebroid 
of $\shend_{\coro}(\shl)$ corresponding to the
subring 
$\Dh[\Lambda]$ of $\shend_{\coro}(\Oh[\Lambda])$
is well-defined. We denote it by $\DL$. 

\begin{lemma}
\bnum
\item
$\DL$ is equivalent to $\Dh[\Lambda]$ as a $\coro$-algebroid.
\item
The $\coro$-algebra $\DL$ satisfies \eqref{eq:FDringa} and
\eqref{eq:FDringc}. In particular, it is right and left Noetherian.
\enum
\end{lemma}
\begin{proof}
(i) follows by similar arguments as in Proposition~\ref{prop:shda}~(ii).

\noindent
(ii) follows from Example~\ref{exa:DL}.
\end{proof}
The functor $\A\vert_\Lambda\to \shend_{\coro}(\shl)$ factorizes as
\eq\label{eq:AtoDh}
&&\A\vert_\Lambda\to\D[\shl],
\eneq
and setting $\Dl[\shl]\eqdot(\D[\shl])^\loc$, this functor
induces a functor
\eq
&&\Ah\vert_\Lambda\to\Dl[\shl].
\eneq

We denote by $I_\Lambda\subset \OO$ the defining ideal of $\Lambda$.
Let $\shi$ be the kernel of the composition
\eqn
&&\opb{\hbar}\A\To[\hbar]\A\To[\sigma]\OO\to\OO[\Lambda].
\eneqn
Then we have $\shi/\A\simeq I_\Lambda$.
\begin{definition}
We denote by $\AL$ \index{AL@$\AL$}%
the $\coro$-subalgebroid of $\Ah$ generated by $\shi$.
\end{definition}
Note that the algebra $\AL$ is the analogue in the framework of
$\DQ$-algebras of the algebra $\she_\Lambda$ constructed
in~\cite{K-O}.

The ideal $\hbar\shi$ is contained in $\A$, hence acts on $\shl$ and
one sees easily that $\hbar\shi$ sends $\shl$ to $\hbar\shl$. 
Hence, $\shi$ acts on $\shl$ and defines a 
functor $\AL\to\DL$. 
We thus have the functors of algebroids
\eqn
&&\xymatrix{
{\A\vert_\Lambda}\ar[r]\ar[dr]&{\AL\vert_\Lambda}\ar[d]\ar[r]&{\Ah\vert_\Lambda}\ar[d]\\
&{\DL}\ar[r]&{\DLl}\,.
}\eneqn
In particular, $\shl$ is naturally an $\AL$-module.
\begin{lemma}\label{le:ALpropert1}
\bnum
\item
$\shi^k=\AL\cap\hbar^{-k}\A$ for any $k\ge0$,
\item
$\shi^k/\shi^{k-1}\simeq I_\Lambda^k$ for $k>0$,
\item
$\AL$ is a right and left Noetherian algebroid,
\item
$\gr(\AL)\vert_\Lambda\isoto\gr\DL\simeq\D[\Lambda]$,
\item
$(\AL)^\loc\simeq\Ah$ and $\Ah$ is flat over $\AL$.
\enum
\end{lemma}
\begin{proof}
Since the question is local, we may assume that
$X=T^*\C^{n}$ with coordinates $(x,u)$,  $\Lambda=\{u=0\}$
 and $\A$ is the star-algebra as
in \eqref{eq:wstar}.
Set 
\eqn
&&\sha'\seteq\{\sum_kf_k(x,u)\hbar^k\in\Ah;\,
f_k(x,u)\in I_\Lambda^{-k} \mbox{ for }k<0\}.
\eneqn
Then we can check that $\sha'$ is a subalgebra of $\Ah$ and it contains $\shi$.
Hence it contains $\AL$.
It is easy to see that 
the image of $\shi^k\to\hbar^{-k}\A/\hbar^{-k+1}\A$ contains $\hbar^{-k}I_\Lambda^k$.
On the other hand, 
the image of $\sha'\cap\hbar^{-k}\A\to \hbar^{-k}\A/\hbar^{-k+1}\A$
coincides with $\hbar^{-k}I_\Lambda^k$.
Hence, $\AL\cap\hbar^{-k}\A$ and $\sha'\cap\hbar^{-k}\A$
have the same image $\hbar^{-k}I_\Lambda^k$  in $\hbar^{-k}\A/\hbar^{-k+1}\A$.
We conclude that $\AL=\sha'$ and
$\AL\cap\hbar^{-k}\A\subset \shi^{k}+\hbar^{-k+1}\A$.
Hence, an induction on $k$ shows (i).

\smallskip
\noindent
(ii) is now obvious.

\smallskip
\noindent
(iii) Considering the filtration $\{\AL\cap\hbar^{-k}\A\}_{k\ge0}$ of $\AL$,
the result follows by \cite[Theorem A.32]{Ka2}.

\smallskip
\noindent
(iv) is obvious.

\smallskip
\noindent
(v) follows from $\A\subset\AL\subset\A^\loc$.
\end{proof}

By this lemma, for a coherent $\AL$-module $\shn$, we may regard $\gr(\shn)$ 
as an object of $\Derb_\coh(\shd_\Lambda)$. Recall that
$\Db_\hol(\shd_\Lambda)$ denotes the full triangulated category of 
$\Db_\coh(\shd_\Lambda)$ consisting of objects with holonomic cohomology.

\begin{lemma}\label{le:DLflatAL}
The algebroid $\DL$ is flat over $\AL$ and $\DLl$ is flat over $\Ah$.
\end{lemma}
\begin{proof}
It is enough to prove the first statement.

Let us show that $H^j(\DL\ltens[\AL]\shm)\simeq0$ for 
any coherent $\AL$-module $\shm$ and any $j<0$.

\noindent
(i) Assume that $\shm$ has no $\hbar$-torsion. Using Lemma~\ref{le:ALpropert1}~(iv),
we have for $j<0$, $H^j\gr(\DL\ltens[\AL]\shm) \simeq H^j\gr\shm\simeq0$, 
and hence $H^j(\DL\ltens[\AL]\shm)\simeq0$ by Proposition~\ref{pro:grHa}.

\smallskip
\noindent
(ii) Assume that $\hbar\shm=0$. Then
\eqn
\DL\ltens[\AL]\shm \simeq \DL\ltens[\AL]\gr\AL\ltens[\gr\AL]\shm 
\simeq \gr\DL\ltens[\gr\AL]\shm \simeq \shm.
\eneqn

\noindent 
(iii)  In the general case, set ${}_n\shn\eqdot\ker(\hbar^n\cl\shm\to\shm)$ and 
$\shm_{tor}\eqdot\bigcup\limits_n{}_n\shn$. Note that this union is 
locally stationary. Defining $\shm_{tf}$ by the exact sequence,
\eqn
0\to\shm_{tor}\to\shm\to\shm_{tf}\to 0,
\eneqn
this module has no $\hbar$-torsion.
It is thus enough to prove the result for the ${}_n\shn$'s and this
follows from (ii) by induction on $n$, using the exact sequence
\eqn
0 \to {}_n\shn \to {}_{n+1}\shn \to {}_{n+1}\shn/{}_n\shn \to 0.
\eneqn
\end{proof}

\begin{definition}
An object $\shn$ of  $\Db_\coh(\AL)$ is holonomic 
\glossary{holonomic!ALmodule@$\AL$-module}%
if $\gr(\shn)$ is Lagrangian in $T^*\Lambda$, that is, if
$\gr(\shn)$ belongs to $\Db_\hol(\shd_\Lambda)$.
\end{definition}
Note that this condition is equivalent to saying that
$H^i(\shn)/\hbar H^i(\shn)$ and $\ker(\hbar\cl H^i(\shn)\to H^i(\shn))$ 
are holonomic $\shd_\Lambda$-modules for any $i$ (see Lemma~\ref{lem:grHa}).

\subsubsection*{Microsupport and constructible sheaves}
Let us recall some notions and results of \cite{K-S1}.

Let $M$ be a {\em real analytic} manifold and $\cora$ a Noetherian 
commutative ring of finite global dimension.
\glossary{microsupport}\index{SS@$\SSi$}%
For $F\in\RD^\Rb(\cora_M)$, we denote by $\SSi(F)$ its microsupport, a
closed $\R^+$-conic ({\em i.e.,} invariant by the $\R^+$-action on
$T^*M$)
\glossary{Rconic@$\R^+$-conic}\glossary{conic!R@$\R^+$-}%
subset of $T^*M$. Recall that this set 
is involutive (one also says {\em co-isotropic}), see~\cite[Def.~6.5.1]{K-S1}.

An object $F$ of $\Derb(\cora_M)$ is
\glossary{constructible!weakly $\R$-}%
\glossary{weakly $\R$-constructible}%
{\em weakly $\R$-constructible} if there exists a subanalytic stratification
$M=\bigsqcup_{\alpha\in A}M_\alpha$ such that
$H^j(F)\vert_{M_\alpha}$ is  locally constant for all $j\in\Z$ and all
$\alpha\in A$. The object $F$ is {\em $\R$-constructible}
\glossary{constructible!$\R$-}%
\glossary{Rconstructible@$\R$-constructible}%
if moreover $H^j(F)_x$ is finitely generated for all
$x\in M$ and all $j\in\Z$.
One denotes by  $\Derb_\Rc(\cora_M)$ 
\index{Dbw@$\Derb_\Rc(\cora_M)$}%
the full subcategory of  $\Derb(\cora_M)$ consisting of 
$\R$-constructible objects. Recall that the duality functor 
$\RDD_X(\scbul)$ (see \eqref{eq:rdddual}) is an anti-auto-equivalence of
the category $\Derb_\Rc(\cora_M)$.

If $M$ is complex analytic, one defines similarly the notions of 
(weakly) $\C$-constructible sheaf, replacing ``subanalytic''
 with ``complex analytic''.
\glossary{constructible!$\C$-}%
\glossary{Cconstructible@$\C$-constructible}%
We denote by  $\Derb_\wCc(\cora_M)$ 
\index{Dbw@$\Derb_\wCc(\cora_M)$}%
the full subcategory of  $\Derb(\cora_M)$ 
consisting of weakly-$\C$-constructible  objects 
and by $\Derb_\Cc(\cora_M)$ 
\index{Dbc@$\Derb_\Cc(\cora_M)$}%
the full subcategory consisting of
$\C$-constructible objects.
Also recall (\cite{K-S1}) that $F\in\Derb(\cora_M)$ is
weakly-$\C$-constructible if and only if its microsupport is a closed
$\C^\times$-conic ({\em i.e.,} invariant by the $\C^\times$-action on $T^*M)$)
\glossary{Cconic@$\C^\times$-conic}\glossary{conic!C@$\C^\times$-}%
complex analytic Lagrangian subset of $T^*M$ or, 
equivalently, if it is contained in a closed
$\C^\times$-conic 
complex analytic isotropic subset of $T^*M$.

\begin{proposition}\label{pro:sscohco}
Let $F\in\Derb(\Z_M[\hbar])$ and assume that $F$ is cohomologically
complete. Then 
\eq\label{eq:SSgr}
&&\SSi(F)=\SSi(\gr(F)).
\eneq
\end{proposition}
\begin{proof}
The inclusion 
\eqn
&&\SSi(\gr(F)) \subset \SSi(F)
\eneqn
follows from the distinguished triangle
$F\to[\hbar]F\to\gr(F)\to[+1]$.
Let us prove the converse inclusion.

Using the definition of the microsupport, it is enough to prove that given two
open subsets $U\subset V$ of $M$,
$\rsect(V;F)\to\rsect(U;F)$ is an isomorphism
as soon as  
$\rsect(V;\gr(F))\to \rsect(U;\gr(F))$ is an isomorphism.
Consider a distinguished triangle $\rsect(V;F)\to\rsect(U;F)\to G\to[+1]$.
Then we get a distinguished triangle 
$\rsect(V;\gr(F))\to \rsect(U;\gr(F))\to \gr(G)\to[+1]$. Therefore,
$\gr(G)\simeq0$. On the other hand, 
$G$ is cohomologically complete, thanks to
Proposition~\ref{pro:cohcodirim} and  $G\simeq0$ by Corollary~\ref{cor:conserv}.
\end{proof}

\begin{proposition}\label{pro:RCcohco}
Let $F\in \Derb_\Rc(\coro_X)$. Then $F$ is cohomologically complete.
\end{proposition}
\begin{proof} 
One has
\eqn
\inddlim[U\ni x]\Ext[{\Z[\hbar]}]{j}\bl\Z[\hbar,\hbar^{-1}],H^i(U;F)\br&\simeq&
\Ext[{\Z[\hbar]}]{j}\bl\Z[\hbar,\hbar^{-1}],\inddlim[U\ni x]H^i(U;F)\br\\
&\simeq&\Ext[{\Z[\hbar]}]{j}\bl\Z[\hbar,\hbar^{-1}],F_x\br
\simeq0
\eneqn
where the last isomorphism follows from the fact that $F_x$ is
cohomologically complete when taking $X=\rmpt$. 

Hence, the hypothesis~(i)~(c) of Proposition~\ref{prop:cccr} is satisfied.
\end{proof}

\subsubsection*{Propagation for solutions of $\AL$-modules}

\begin{proposition}\label{pro:SScar2}
Let $\shn$ be a coherent $\AL$-module.  Then 
\eq
&&\SSi(\rhom[\AL](\shn,\shl))\subset\chv(\gr\shn).
\eneq
\end{proposition}
\begin{proof}
By Lemma~\ref{le:DLflatAL}, we  have 
\eqn
&&\rhom[\AL](\shn,\shl)\simeq \rhom[\DL](\DL\tens[\AL]\shn,\shl).
\eneqn
Since $\gr(\DL\tens[\AL]\shn)=\gr(\shn)$, Proposition~\ref{pro:SScar2} will follow
from Proposition~\ref{pro:SScar2b} below, already obtained in \cite{DGS}.
\end{proof}
\begin{proposition}\label{pro:SScar2b}
Let $\shn$ be a coherent $\DL$-module.  Then 
\eq\label{eq:prop3}
&&\SSi(\rhom[\DL](\shn,\shl))=\chv(\gr\shn).
\eneq
\end{proposition}
\begin{proof}
Set $F=\rhom[\DL](\shn,\shl)$. Then $F$ is cohomologically
complete by Corollary~\ref{cor:cohimplcohco}
and $\SSi(F)=\SSi(\gr(F))$ by Proposition~\ref{pro:sscohco}. On the other hand, 
$\gr(F)\simeq \rhom[{\D[\Lambda]}](\gr\shn,\sho_\Lambda)$ by
Proposition~\ref{pro:tensgr} and the microsupport of this complex is equal
to $\chv(\gr\shn)$ by~\cite[Th~11.3.3]{K-S1}. 
\end{proof}

\subsubsection*{Constructibility of solutions}

Theorem~\ref{th:constDhh0} below has already been obtained  in \cite{DGS} in the
framework of $\Dh[M]$-modules.

Recall that $\shl$ is a coherent $\A$-module, simple along $\Lambda$.

\begin{theorem}\label{th:constDhh0}
Let $\shn$ be a holonomic $\AL$-module.
\banum
\item
The objects $\rhom[\AL](\shn,\shl)$ and $\rhom[\AL](\shl,\shn)$
belong to $\Derb_\Cc(\coro_\Lambda)$ and their microsupports 
are contained in $\chv(\gr\shn)$.
\item
There is a natural isomorphism in $\Derb_\Cc(\coro_\Lambda)$
\eq\label{eq:dualDhh}
&&\rhom[\AL](\shn,\shl)\isoto\RDD_X\bl\rhom[\AL](\shl,\shn)\br\,[d_X].
\eneq
\eanum
\end{theorem}
The morphism in (b) is similar to the morphism in
Lemma~\ref{le:duality} and is associated with
\eqn
&&\rhom[\AL](\shn,\shl)\otimes\rhom[\AL](\shl,\shn)\\
&&\hspace{5ex}\to\rhom[\AL](\shl,\shl)\to\rhom[{\D[\shl]}](\shl,\shl)
\simeq\coro_\Lambda\to \coro_X\,[d_X].
\eneqn
\begin{proof}
(a) It is enough to treat $F\eqdot\rhom[\AL](\shn,\shl)$.
In view of Proposition~\ref{pro:SScar2}, 
$F$ is weakly $\C$-constructible and it remains to show that 
for each $x\in \Lambda$, $F_x$ belongs to $\Derb_f(\coro)$. 

If $U$ is a sufficiently small open ball centered at $x$,
then $\rsect(U;F)\to F_x$ is an isomorphism (\cite{K-S1}).
The finiteness of the complex $\gr(F_x)$
follows from the classical finiteness theorem 
for  holonomic $\shd$-modules of \cite{Ka3}. Since $F$ is
cohomologically complete, Proposition~\ref{pro:cohcodirim} implies that
$\rsect(U;F)$ is \cc.
Hence the result follows from Theorem~\ref{th:formalfini2}.

\vspace{0.2cm}
\noindent
(b) follows from Corollary~\ref{cor:conservative1}, 
since we know  by~\cite{Ka3}  that
\eqref{eq:dualDhh} is an isomorphism after applying the functor $\gr$.
\end{proof}

\subsubsection*{$\AL$ modules and $\Ah$-modules}
\begin{definition}
A coherent $\AL$-submodule $\shn$ of a coherent $\Ah$-module $\shm$ is called
an $\AL$-lattice 
\glossary{lattice@$\AL$-lattice}%
of $\shm$ if $\shn$ generates $\shm$ as an $\Ah$-module.
\end{definition}

\begin{lemma}
Let $\shm$ be a coherent $\Ah$-module and let $\shn\subset\shm$ be
an $\AL$-lattice of $\shm$.
Then $\chv\bl\gr(\shn)\br\subset T^*\Lambda$ does not depend on the choice of 
 $\shn$. 
\end{lemma}

The proof is similar to the one of Lemma ~\ref{lem:grgr2}, and we shall
not repeat it.

\begin{definition}\label{def:chl}
Let $\shm$ be a coherent $\Ah$-module and let $\shn\subset\shm$ be
an $\AL$-lattice of $\shm$. We set
\eqn
&&\Chl(\shm)\eqdot\chv(\gr\shn).\index{caract@$\Chl(\shm)$}%
\eneqn
\end{definition}
\begin{example}
Let $X=\C^{2}$ endowed with the symplectic coordinates $(x;u)$ and let $\Lambda$ be the 
Lagrangian manifold given by the equation $\{u=0\}$. In this case, 
$\AL=\sha_X[u\opb{\hbar}]$. 

Now let $\alpha\in\C$  and consider the modules 
$\shm=\A^\loc/\A^\loc(xu-\alpha\hbar)$ and $\shn=\AL/\AL(xu\opb{\hbar}-\alpha)$.
Then $\shn$ is an $\AL$-lattice of $\shm$ and  
$\gr\shn\simeq\shd_\Lambda/\shd_\Lambda(x\partial_x-\alpha)$.
\end{example}
\begin{lemma}
Let $\shm$ be a coherent $\Ah$-module.
\bnum
\item
$\Chl(\shm)$ is  a closed conic complex analytic subset of
$T^*\Lambda$ and this set is involutive. 
\item 
Let $0\to\shm'\to\shm\to\shm''\to0$ be an exact sequence of $\Ah$-modules.
Then 
$\Chl(\shm)=\Chl(\shm')\cup\Chl(\shm'')$.
\enum
\end{lemma}
\begin{proof}
(i) is a well-known result of $\shd$-module theory, see \cite{Ka2}.

\noindent
(ii) Let $\shn$ be an $\AL$-lattice of $\shm$.
Set $\shn'=\shm'\cap\shn$ and $\shn''\subset \shm''$ be the image of
$\shn$. Then $\shn'$ and $\shn''$ are $\AL$-lattices of $\shm'$ and
$\shm''$, respectively.
Since we have an exact sequence
$$0\to\shn'/\hbar\shn'\to \shn/\hbar\shn\to\shn''/\hbar\shn''\to0,$$
we have
$\Chl(\shm)=\chv(\shn/\hbar\shn)
=\chv(\shn'/\hbar\shn')\cup\chv(\shn''/\hbar\shn'')
=\Chl(\shm')\cup\Chl(\shm'')$.
\end{proof}

\begin{proposition}\label{pro:codimchl}
For a coherent $\Ah$-module $\shm$, we have
\eqn
&&\codim\Chl(\shm)\ge\codim\Supp(\shm).
\eneqn
\end{proposition}
\begin{proof}
In the course of the proof, we 
shall have to consider the analogue of the algebra $\AL$ but with
$\A[X^a]$ instead of $\A$. We shall denote by $\ALa$ this algebra.
\index{ALa@$\ALa$}%
We shall show that
$\codim\Supp(\shm)\ge r$ implies $\codim\Chl(\shm)\ge r$
by descending induction on $r$. Applying
Proposition~\ref{pro:dimext1l}~(a), we have
$\rhom[\Ah](\shm,\Ah)\simeq\tau^{\ge r}\rhom[\Ah](\shm,\Ah)$, where 
$\tau^{\ge r}$ is the truncation functor.
Hence we have a distinguished triangle in $\Db_\coh(\Ah[X^a])$:
\eq
&&\ext[\Ah]{r}(\shm,\Ah)[-r]\to \rhom[\Ah](\shm,\Ah)
\to \shk\To[+1],\label{eq:78}
\eneq
where $\shk=\tau^{>r}\rhom[\Ah](\shm,\Ah)$.
Note that $\codim(\Supp(\shk))>r$ by Proposition~\ref{pro:dimext1l}~(b).
Setting $\shm'=\ext[\Ah]{r}(\shm,\Ah)$, 
the distinguished triangle \eqref{eq:78} induces a distinguished
triangle in $\Db_\coh(\Ah)$:
\eqn
&&\rhom[{\Ah[X^a]}](\shk,\Ah[X^a])\to\shm\to\rhom[{\Ah[X^a]}](\shm',\Ah[X^a])[r]\To[+1].
\eneqn
Setting $\shm_1=\ext[{\Ah[X^a]}]{r}(\shm',\Ah[X^a])$, we obtain a morphism
$\phi\cl\shm\to \shm_1$
and $\ker(\phi)$ has codimension greater than $r$.
Hence, $\codim\Chl(\ker(\phi))>r$ by the induction hypothesis.
Since
$\Chl(\shm)\subset\Chl(\shm_1)\cup\Chl(\ker(\phi))$,
it is enough to show that
$\codim\Chl(\shm_1)\ge r$.

\medskip
Hence we may assume from the beginning that 
$\shm=\ext[{\Ah[X^a]}]{r}(\shm',\Ah[X^a])$ for a coherent
$\Ah[X^a]$-module $\shm'$. 
Let us take an $\ALa$-lattice $\shn'$ of $\shm'$.
Set $\shn_0=\ext[\ALa]{r}(\shn',\ALa)$.
Then we have $\shn_0^\loc\simeq\shm$, 
and it induces a morphism
$\shn_0\to\shm$.
Let $\shn$ be the image of the morphism
$\shn_0\to\shm$.
Then $\shn$ is an $\AL$-lattice of $\shm$.
Hence we have
$\Chl(\shm)=\chv(\shn/\hbar\shn)$,
which implies 
\eq\label{eq:chlsubsupp}
&&\Chl(\shm)\subset\chv(\shn_0/\hbar\shn_0).
\eneq
On the other hand, we have an exact sequence
\eqn
&&\ext[\ALa]{r}(\shn',\ALa)\To[\hbar]
\ext[{\ALa}]{r}(\shn',\ALa)\to
\ext[{\ALa}]{r}(\shn',\gr(\ALa)).
\eneqn
Since we have
$\ext[{\ALa}]{r}(\shn',\gr(\ALa))
\simeq\ext[{\gr(\ALa)}]{r}(\gr\shn',\gr(\ALa))$, we have a
monomorphism
$$\shn_0/\hbar\shn_0\monoto\ext[{\gr(\ALa)}]{r}(\gr\shn',\gr(\ALa)).$$
Hence we obtain
$\chv(\shn_0/\hbar\shn_0)\subset\chv\Bigl(\ext[{\gr(\ALa)}]{r}(\gr\shn',\gr(\ALa))\Bigr)$.
Since $\chv\Bigl(\ext[{\gr(\ALa)}]{r}(\gr\shn',\gr(\ALa))\Bigr)$has
codimension $\ge r$ by {\em e.g.,}~\cite[Theorem 2.19]{Ka2}, we conclude that
 $\codim\chv(\shn_0/\hbar\shn_0)\ge r$. 
By \eqref{eq:chlsubsupp}, we obtain $\codim\Chl(\shm)\ge r$.
\end{proof}

\section{Holonomic $\DQ$-modules}

In a complex symplectic manifold $X$, an 
\glossary{isotropic subvariety}%
isotropic subvariety $\Lambda$ is a locally closed complex analytic
subvariety such that $\Lambda_{\reg}$ is isotropic, {\em i.e.,} the
$2$-form defining the symplectic structure vanishes on  $\Lambda_{\reg}$.
Here, $\Lambda_\reg$ \index{Lambdar@$[\Lambda_\reg]$}%
denotes the smooth part of $\Lambda$. 

A Lagrangian subvariety 
\glossary{Lagrangian subvariety}%
$\Lambda$ is an  isotropic subvariety
of pure dimension $d_X/2$. Equivalently, $\Lambda$ is a  
subvariety of pure dimension $d_X/2$ such that $\Lambda_{\reg}$ is involutive. 
\begin{definition}\label{def:simpleWmod}
\banum
\item
An $\Ah$-module $\shm$ is holonomic 
\glossary{holonomic}%
if it is coherent and its support is a Lagrangian subvariety of $X$. 
\item
An $\A$-module $\shn$ is holonomic if it is coherent, without
$\hbar$-torsion and $\shn^\loc$ is a holonomic $\Ah$-module.
\item
Let $\Lambda$ be a smooth Lagrangian submanifold of $X$. 
\glossary{simple!holonomic}%
\glossary{holonomic!simple}%
We say that an $\Ah$-module $\shm$ is simple holonomic along $\Lambda$ 
if there exists locally an $\A$-module $\shm_0$
simple along $\Lambda$ 
such that $\shm\simeq\shm_0^\loc$.
\eanum
\end{definition}

\begin{lemma}\label{lem:dualhol3}\label{P:NEW}
Let $\shm$ be a holonomic $\Ah$-module. Then
$\RDAh\shm\,[d_X/2]$ is concentrated in degree $0$ and is
holonomic.
\end{lemma}
\begin{proof}
This follows from Proposition~\ref{pro:dimext1l} and 
the involutivity theorem (Proposition~\ref{pro:invol}).
\end{proof} 
Let $X$ be a complex symplectic manifold and let 
$\shm$ and $\shl$ be two holonomic $\Ah$-modules.
Using Lemma~\ref{le:homCdelta} (more precisely, an $\Ah$-variant of
this lemma) and Theorem~\ref{th:symp1}, we have
\eq&&\ba{rcl}
\rhom[{\Ah}](\shm,\shl)&\simeq&
\rhom[{\Ah[X\times X^a]}](\shm\ldetens\RDA\shl,\dA^\loc),\\
\rhom[{\Ah}](\shl,\shm)&\simeq&
\rhom[{\Ah[X\times X^a]}](\shl\ldetens\RDA\shm,\dA^\loc)\\
&\simeq&\rhom[{\Ah[X\times X^a]}](\RDA(\dA^\loc), \shm\ldetens\RDA\shl)\\
&\simeq&\rhom[{\Ah[X\times X^a]}](\dA^\loc,
\shm\ldetens\RDA\shl)[d_X].
\ea
\label{eq:homCdelta}
\eneq

\begin{theorem}\label{th:hol1}
Let $X$ be a complex symplectic manifold and let 
$\shm$ and $\shl$ be two holonomic $\Ah$-modules. Then
\bnum
\item
the object $\rhom[{\Ah}](\shm,\shl)$ belongs to  $\Derb_\Cc(\cor_X)$,
\item
there is a canonical isomorphism:
\eq\label{eq:dualmorph1}
&&\rhom[{\Ah}](\shm,\shl)
\isoto\bl\RDD_X\rhom[{\Ah}](\shl,\shm)\br\,[d_X].
\eneq
\item
the object $\rhom[{\Ah}](\shm,\shl)[d_X/2]$ is perverse.
\enum
\end{theorem}
\begin{proof}
Using \eqref{eq:homCdelta}, we may assume from the beginning that 
$\shl$ is a simple holonomic $\Ah$-module supported on a
smooth Lagrangian submanifold $\Lambda$ of $X$. 
Let $\shl_0$ be an $\A$-module simple along $\Lambda$ 
such that $\shl\simeq\shl_0^\loc$.

\smallskip
\noindent
(i)-(ii) Let $\shn$ be an
$\AL$-lattice of $\shm$.
By Lemma~\ref{le:ALpropert1}~(v), we have
\eqn
\rhom[{\Ah}](\shm,\shl)
&\simeq&\rhom[\AL](\shn,\shl_0)^\loc.
\eneqn
Then the results follow 
from Proposition~\ref{pro:codimchl} and Theorem~\ref{th:constDhh0}.

\noindent
(iii) Since the problem is local,  we may assume that $X=T^*M$,
$\Ah=\HW$ and $\shl_0=\Oh[M]$. 

By~(ii), it is enough to check the statement:
\eq&&
\parbox{60ex}{
$H^j\Bigl(\rsect_N\bl\rhom[{\AL}](\shn,\shl_0)\br\Bigr)$ 
vanishes for  $j<l$ and for any
closed smooth submanifold $N$ of $M$ of codimension $l$.}
\eneq
Since $F\seteq\rsect_N\bl\rhom[{\AL}](\shn,\shl_0)\br$ is
$\C$-constructible, 
it is enough to show that
$H^j(\gr(F))=0$ for $j<l$.
This follows from the well-known fact that $H^j(\rsect_N(\OO[M]))=0$ for $j<l$.
\end{proof}
Assume for simplicity that $X$ is open in some cotangent bundle 
$T^*M$. We shall compare the sheaf of solutions of holonomic
$\HE$-modules and $\HW$-modules. Recall that $\HW[X]$ is faithfully flat over
$\HE[X]$ by Lemma~\ref{lem:WfffE}.
\begin{corollary}\label{cor:holEhatW}
Let $\shm$ and $\shl$ be two holonomic $\HE[X]$-modules. Then
the object $\rhom[{\HE[X]}](\shm,\shl)$ belongs to $\Derb_\Cc(\C_X)$. 
\end{corollary}
\begin{proof}
Let $t$ denote the coordinate on the complex line $\C$, 
let $E$ denote the ring $\HE[T^*\C\vert_{t=0,\tau=1}]$ and let $L$
be the $E$-module $E/E\cdot t$. Then we have the embedding
\eqn
&&\cor\hookrightarrow E,\quad \hbar\mapsto \opb{\partial_t}.
\eneqn  
Set for short $F\eqdot\rhom[{\HE[X]}](\shm,\shl)$. Then
\eqn
F&\simeq&
\rhom[E](L,\rhom[{\HE[X]}](\shm,
(\HE[X\times T^*\C]/\HE[X\times T^*\C]\cdot t)\vert_{t=0,\tau=1}\lltens[{\HE[X]}]\shl))\\
&\simeq&\rhom[E](L,\rhom[{\HW}](\HW\tens[{\HE[X]}]\shm,\HW\tens[{\HE[X]}]\shl)).
\eneqn
Set $G\eqdot \rhom[{\HW}](\HW\tens[{\HE[X]}]\shm,\HW\tens[{\HE[X]}]\shl)$.
Applying Theorem~\ref{th:hol1}, we find that
$G\in\Derb_\Cc(\cor_X)$ and it follows that  $F\in\Derb_\wCc(\C_X)$. 

Moreover, for each $x\in X$, $G_x$ is of finite type over $\cor$ and is an $E$-module. One 
easily deduces that $F_x\simeq\RHom[E](L,G_x)$ is a $\C$-vector space of finite
dimension.
\end{proof}
\section{Lagrangian cycles}

Given two
holonomic $\Ah$ modules $\shm$ and $\shl$ such that
$\Supp(\shm)\cap\Supp(\shl)$ is compact, the Euler-Poincar{\'e}
index is given by
\eq\label{eq:EPindex}
\ba{rcl}\chi(X;\shm,\shl)&=&\chi(\RHom[{\Ah}](\shm,\shl))\\
&=&\sum_i\,(-)^i\dim \Ext[{\Ah}]{i}(\shm,\shl).
\ea\eneq
Applying \eqref{eq:EPindAhmod}, we get
\eq\label{eq:EPform}
&&\chi(X;\shm,\shl)=\int_X(\eu_X(\shm)\cdot\eu_X(\shl)).
\eneq
Recall that $\eu_X(\shm)=(-1)^{d_X/2}\eu_X(\RDAh\shm)$, 
and also recall that $d_X$
being even, $\eu_X(\shm)\cdot\eu_X(\shl)=\eu_X(\shl)\cdot\eu_X(\shm)$.

We shall explain how to calculate the Euler classes 
by using the theory of Lagrangian cycles. We
refer to~\cite[Ch.~9~\S~3]{K-S1} for a detailed study of these cycles. 

Recall that $\cora$ denotes a commutative Noetherian unital ring of finite global
dimension.

Consider a closed Lagrangian subvariety  $\Lambda$  of $X$. 
We define the sheaf:
\eq\label{eq:LagrCyc1}
&&\LL^\cora_\Lambda\eqdot H^{d_X}_\Lambda(\cora_X), 
\eneq
and we simply write $\LL_\Lambda$
\index{L@$\LL_\Lambda$}%
\index{Lc@$\LL^\cora_\Lambda$}%
instead of $\LL^\Z_\Lambda$.
The next results are obvious and well-known (see loc.\ cit.).
\begin{lemma}\label{le:propcyc}
\bnum
\item
 $U\mapsto H^{d_X}_{\Lambda\cap U}(U;\cora_X)$ \lp$U$ open in $X$\rp\,
is a sheaf and this sheaf coincides with $\LL^\cora_\Lambda$, 
\item
$H^i_{\Lambda\setminus\Lambda_\reg}(\LL^\cora_\Lambda)\simeq 0$ for $i=0,1$,
\item
if $s$ is  a section of $\LL^\cora_\Lambda$, then its support is open
and closed in $\Lambda$,
\item
there is a canonical section in $\sect(\Lambda;\LL_\Lambda)$
which gives an isomorphism $\LL_\Lambda\vert_{\Lambda_\reg}\isoto\Z_{\Lambda_\reg}$.
\enum
\end{lemma}
We denote by $[\Lambda]$ \index{K3@$[\Lambda]$}%
the section given in (iv) above.
\begin{definition}
We call a section of $\LL^\cora_\Lambda$ on an open set $U$ of
$\Lambda$ a Lagrangian cycle on $U$. 
\glossary{Lagrangian cycle}
\end{definition}
Recall that $\rmK_{\coh,\Lambda}(\sho_X)$ denotes the Grothendieck
group of the category $\Derb_{\coh,\Lambda}(\sho_X)$. We denote by 
$\shk_{\coh,\Lambda}(\sho_X)$ 
\index{K3@$\shk_{\coh,\Lambda}(\sho_X)$}%
the sheaf associated with the presheaf $U\mapsto\rmK_{\coh,\Lambda\cap U}(\sho_U)$.
Then,  there is a well defined $\Z$-linear map
\eq\label{eq:LagrCyc3}
\kappa&\cl&
\shk_{\coh,\Lambda}(\sho_X)\to \LL_\Lambda.
\eneq
This map is characterized by the property that 
\eq
&&\kappa(\OO[\Lambda])=[\Lambda]\in \sect(\Lambda;\LL_\Lambda).
\eneq
Let $\shm\in \Derb_\hol(\Ah)$ and let $\Lambda$ be a closed Lagrangian 
subvariety of $X$ which contains $\Supp(\shm)$. 

Let $\shm_0$ be an $\A$-lattice of $\shm$ on an open set $U$ of
$X$. Then $\gr(\shm_0)$ defines an element 
$[\gr(\shm_0)]\in \rmK_{\coh,\Lambda}(\sho_X\vert_U)$, 
hence an element of $\sect(U;\shk_{\coh,\Lambda}(\sho_X))$. 
This element depends only on $\shm$, and we thus have a morphism
\eqn
&&\rmK_{\coh,\Lambda}(\Ah)\to
\sect(\Lambda;\shk_{\coh,\Lambda}(\sho_X)).
\eneqn
Composing with the map $\kappa$,  we obtain a map
\eq\label{eq:LagrCyc2}
&&\rmK_{\coh,\Lambda}(\Ah)\to\sect(\Lambda;\LL_\Lambda).
\eneq
\begin{definition}
We denote by $\lc_X(\shm)$ \index{lc@$\lc_X(\shm)$}
the image of
$\shm\in\Derb_{\coh,\Lambda}(\Ah)$ by the morphism
in \eqref{eq:LagrCyc2} and call it the Lagrangian cycle of $\shm$.
\end{definition}
On the other-hand, recall (see Definition~\ref{def:eulerclforAh}) that the Euler class
$\eu_X(\shm)$ of $\shm$ belongs to $H^{d_X}_\Lambda(X;\cor_X)$. Hence, the Euler
class of $\shm$ is a Lagrangian cycle supported by $\Lambda$:
\eq\label{eq:LagrCyc4}
&&\eu_X(\shm)\in \sect(\Lambda;\LL^{\cor}_\Lambda).
\eneq
The map $\Z\to\cor$ induces the morphism
\eq
&&\iota_X\cl \LL_\Lambda\to\LL^{\cor}_\Lambda.
\eneq

The next lemma is easily checked. 
\begin{lemma}\label{le:eu=lc}
Let $\Lambda$ be a smooth Lagrangian submanifold of $X$ and let $\shl$ be a
coherent $\Ah$-module, simple along $\Lambda$. Then 
 $\eu_X(\shl)=\iota_X([\Lambda])$.
\end{lemma}

\begin{theorem}\label{th:eu=lc}
One has $\eu_X(\shm)=\iota_X\circ\lc_X(\shm)$.
\end{theorem}
\begin{proof}
By Lemma~\ref{le:propcyc}, it is enough to prove the result at the
generic point of $\Lambda$. Hence, we may assume that $\Lambda$ is
smooth. Let $x\in\Lambda$ and let us choose a smooth Lagrangian 
submanifold $S_x$ of
$X$ which intersects $\Lambda$ transversally at the single point $x$.
Let us also choose a simple $\Ah$-module $\shl$ simple along $S_x$.  
Using \eqref{eq:EPform}, we find 
\eqn
&&\chi(\rhom[{\Ah}](\shl,\shm)_x)=\int_X(\eu_X(\shl)\cdot\eu_X(\shm)).
\eneqn

Let $\shl_0$ and $\shm_0$ be $\A$-lattices of $\shl$ and $\shm$,
respectively. We also have  
\eqn
\chi(\rhom[{\Ah}](\shl,\shm)_x)&=&\chi(\rhom[{\bA}](\gr(\shl_0),\gr(\shm_0))_x)\\
&=&\int_X(\kappa([\gr(\shl_0)])\cdot\kappa([\gr(\shm_0)])).
\eneqn
Clearly, we have
\eq
&&\kappa([\gr(\shl_0)])=[S_x].
\eneq
By Lemma~\ref{le:eu=lc},  $\eu(\shl_0)=[S_x]$. Therefore,
\eq\label{eq:LagrCyc5}
&&\int_X([S_x]\cdot\eu_X(\shm))=\int_X([S_x]\cdot\lc_X(\shm))
\eneq
for any  smooth Lagrangian submanifold $S_x$
which intersects $\Lambda$ transversally at $x$.
This completes the proof.
\end{proof}
\begin{remark}
The Euler class of a holonomic $\Ah$-module supported by a Lagrangian
variety $\Lambda$ is easy to calculate, since it is enough to
calculate it at generic points of $\Lambda$. 
Moreover, the integral in \eqref{eq:EPform} is invariant by smooth (real) homotopy of the
Lagrangian cycles $\lc_X(\shm)$ and $\lc_X(\shl)$ and one may deform
them in order that they intersect transversally at the smooth part of
their support. See \cite[Ch.~9,\S~3]{K-S1} for a detailed study.
\end{remark}

\section{Simple holonomic modules}

When $\shl_0$ and $\shl_1$ are simple along smooth Lagrangian manifolds,
one can give an estimate on the microsupport of
$\rhom[{\Ah}](\shl_1,\shl_0)$.
It follows from Lemma~\ref{le:simpleAmod1} that two simple holonomic modules along
$\Lambda$ are locally isomorphic.
\begin{example}
Assume $X=T^*M$ for a complex manifold $M$ and $\A=\HWo[X]$. Then 
$\Ohl[M]$ is a simple holonomic $\Ah$-module along $M$.
\end{example}
Recall that on a complex symplectic manifold $X$, the symplectic form gives 
the Hamiltonian isomorphism from the cotangent bundle
to the tangent bundle:
\eq\label{eq:sympiso1}
&& H\cl T^*X\isoto TX, \quad 
\langle\theta,v\rangle=\symplecto(v,H(\theta)), \quad
v\in TX,\,\theta\in T^*X.
\eneq
For a smooth Lagrangian submanifold $\Lambda$ of $X$
the isomorphism \eqref{eq:sympiso1} induces an
isomorphism between the normal bundle to $\Lambda$ in $X$ and its
cotangent bundle $T^*\Lambda$.

For the notion of normal cone, see {\em e.g.,} \cite[Def.~4.1.1]{K-S1}.
The next result is proved in \cite[Prop.~7.1]{KS08}.

\begin{proposition}\label{pro:normalcone}
Let $X$ be a complex symplectic manifold and  
let $\Lambda_0$ and $\Lambda_1$ be two closed complex analytic
isotropic subvarieties of $\stx$. Then,
after identifying $TX$ and  $T^*X$ by \eqref{eq:sympiso1}, the
normal cone $\rmC(\Lambda_0,\Lambda_1)$ is  a complex analytic $\C^\times$-conic 
isotropic subvariety of $T^*X$. 
\end{proposition}

\begin{theorem}\label{th:holsimple}
Let $\shl_i$ be a simple holonomic $\Ah$-module along a  smooth Lagrangian manifold
$\Lambda_i$ \lp$i=0,1$\rp.  Then
\eq\label{eq:SSsol1}
&&\SSi\bl\rhom[{\Ah}](\shl_1,\shl_0)\br\subset \rmC(\Lambda_0,\Lambda_1).
\eneq
\end{theorem}

\begin{proof}[Idea of the proof of Theorem~\ref{th:holsimple}]
(i) By identifying $\rhom[{\Ah}](\shl_1,\shl_0)$ with a sheaf supported by
$\Lambda_0$, the estimate \eqref{eq:SSsol1} is equivalent to the estimate
\eq\label{eq:SSsol2}
&&\SSi(\rhom[{\Ah}](\shl_1,\shl_0))\subset \rmC_{\Lambda_0}(\Lambda_1).
\eneq

\noindent
(ii) The problem being local, we may assume $X=T^*M$, $\A=\HWo$, $\Lambda_0=M$,
$\shl_0=\Ohl[M]$. If $\Lambda_1=\Lambda_0$, Theorem~\ref{th:holsimple} is
immediate. Hence, we assume $\Lambda_0\neq\Lambda_1$.

Then there exists a non constant holomorphic function $\phi\cl M\to\C$ such that 
\eqn
&&\Lambda_1=\set{(x;u)\in \stx}{u=\grad\,\phi(x)}
\eneqn
Consider the ideal
\eq\label{eq:idealI}
&&\shi_W=\sum_{i=1}^n\HW\cdot(\hbar\partial_{x_i}-\phi'_i).
\eneq
We may assume that $\shl_1=\HW/\shi_W$.
Let $u\in\shl_1$ be the image of $1\in\HW[X]$ and denote by $\shn$ the 
$\A[\Lambda_0/X]$-submodule of $\shl_1$ generated by $u$.

To conclude, it remains to prove the inclusion
\eq\label{eq:chvgrn}
&& \chv(\gr(\shn)) \subset \rmC(\Lambda_1,T^*_MM).
\eneq
We shall not give the proof of~\eqref{eq:chvgrn} here and refer to \cite{KS08}.
Let us simply mention that the proof uses~\cite[Th.~6.8]{Ka2}.
\end{proof}

\begin{remark}
Consider a smooth Lagrangian submanifold $\Lambda$ of $X$ and denote by 
$\ch(\Omega_\Lambda)\in H^1(\Lambda;\sho^\times_\Lambda)$ the
class corresponding to the line bundle $\Omega_\Lambda$. To the exact sequence 
\eqn
&&1\to\C^\times_\Lambda\to \sho^\times_\Lambda\to[d \log]d\sho_\Lambda\to 0
\eneqn
one associates the maps $\beta$ and $\gamma$:
\eqn
&&H^1(\Lambda;\sho^\times_\Lambda)\to[\beta]H^1(\Lambda;d\OO[\Lambda])\to[\gamma]
H^2(\Lambda;\C^\times_\Lambda).
\eneqn
We shall denote by $\C_\Lambda^{1/2}$ the invertible 
$\C_\Lambda$-algebroid associated with the cohomology class 
$\gamma(\dfrac{1}{2}\beta(\ch(\Omega_\Lambda))\in H^2(\Lambda;\C^\times_\Lambda)$
(see \eqref{eq:Ostacks3}).

Consider an invertible 
$\coro_\Lambda$-algebroid $\sta$ on $\Lambda$ and denote by 
${\rm Inv}(\sta)$ the category of invertible $\sta$-modules
(see Definition~\ref{def:invertible}).
On the other hand, denote by ${\rm Simple}(\Lambda)$ 
the category of simple $\A$-modules along $\Lambda$.
It can be easily deduced from Lemma~\ref{le:simpleAmod1} 
that, given a $\DQ$-algebroid $\A$, there exist 
an invertible $\coro_\Lambda$-algebroid $\sta$
and an equivalence of categories
\eq\label{eq:eqvLocsyst}
&&{\rm Simple}(\Lambda)\simeq {\rm Inv}(\sta).
\eneq
When $\A$ is the canonical algebroid $\HWo[X]$ (see
remark~\ref{rem:canonicA}), it is proved in~\cite{D-S} that  
one has an equivalence $\sta\simeq \coro_\Lambda\tens[\C_\Lambda]\C_\Lambda^{1/2}$.
\end{remark}

\section{Invariance by deformation}
We shall show that in the situation of Theorem~\ref{th:hol1}, 
$\rhom[{\Ah}](\shm,\shl)$ is, in some sense,
invariant by Hamiltonian symplectomorphism.

First, we need a lemma.
\begin{lemma}\label{lem:hollocfree}
Let $M$ be a complex manifold, $X=T^*M$ and let 
$\shm$ be a holonomic $\HW$-module. Assume that the
projection $\pi_M\cl X\to M$ is proper \lp hence, finite\rp\, on $\Supp(\shm)$. Then 
$\oim{\pi_M}\shm$ is a locally free $\Ohl[M]$-module of finite rank.
\end{lemma}
\begin{proof}
(i) In the sequel, we write $\A$ and $\Ah$ instead of $\HWo$ and
$\HW$, respectively. Since $\pi_M$ is finite on $\Supp(\shm)$, 
 $\roim{\pi_M}\shm$ is concentrated in degree $0$.
Let us prove that this sheaf is $\Ohl[M]$-coherent. 
Denote by $\Gamma_\pi$ the
graph of the projection $\pi_M$ and consider the diagram
\eqn
&&\xymatrix{
&M\times X\ar[ld]\ar[rd]&{\,\Gamma_\pi}\ar@{_(->}[l]_-s \ar[d]^-p\\
M&&X.
}\eneqn
Using the morphism of $\coro$-algebras $\opb{\pi_M}\Oh[M]\hookrightarrow\A$,
we may regard $\shl\eqdot\oim{s}\opb{p}\A[X^a]$ as 
a coherent $\A[M\times X^a]$-module simple along $\Gamma_\pi$. Then
\eqn
&&\roim{\pi_M}\shm\simeq \shl^\loc\conv[X]\shm.
\eneqn
We may apply Theorem~\ref{th:hker}
and we get that $\roim{\pi_M}\shm$ is  $\Ohl[M]$-coherent. 

\vspace{0.2cm}
\noindent
(ii) Let $n=d_M=\frac{1}{2}d_X$. By Lemma~\ref{lem:dualhol3},
$\RDAh(\shm)\,[n]$ is concentrated in degree $0$ and it follows from
a similar argument as in (i) that $\RDA(\shm)\conv\shl'\,[n]$ is $\Ohl[M]$-coherent and 
concentrated in degree $0$ for any coherent $\Ah[X\times M]$-module $\shl'$ 
simple along $\Gamma_\pi$. 
Denote by $\RDOl$ 
\index{DOl@$\RDOl\shm$}
the duality functor over $\Ohl[M]$.
Applying again Theorem~\ref{th:hker}, we get 
\eqn
\RDOl(\shm\conv\shl^\loc)&\simeq&\RDAh(\shl^\loc)\conv\oAh\conv\RDAh(\shm)\\
&\simeq&\roim{\pi_M}\Bigl(
\roim{p}(\RDA(\shl)\conv\oA)\ltens[\A]\RDAh(\shm)\Bigr).
\eneqn
Since $\oA\conv\RDA(\shl)\simeq\shl'\,[n]$ for an $\A[M\times X]$-module 
$\shl'$ simple along $\Gamma_\pi$ and $\RDAh(\shm)$ is concentrated in
degree $n$,
$\RDOl(\oim{\pi_M}\shm)$ is concentrated in degree zero.
Therefore, $\oim{\pi_M}\shm$ is a locally projective $\Ohl[M]$-module of
finite rank. To conclude, note that, for $x\in M$, 
any finitely generated projective $\Ohl[M,x]$-module is
free, by a result of \cite{Pop} (see \cite{Sw}).
\end{proof}

Recall 
the situation of \eqref{eq:lambdacicr}: we have three symplectic
manifolds $X_i$ ($i=1,2,3$) and closed subsets 
$\Lambda_{i}$ of $X_{i}\times X_{i+1}$ ($i=1,2$). Assume that the
$\Lambda_i$ ($i=1,2$) are closed  subvarieties and the projection 
$p_{13}$ is proper on $\opb{p_{12}}\Lambda_1\cap\opb{p_{23}}\Lambda_2$. 
Then $\Lambda_1\circ\Lambda_2$ is a closed  subvariety of 
$X_1\times X_3$. Now assume that $\Lambda_{i}$ ($i=1,2$) is isotropic in
$X_{i}\times X^a_{i+1}$. Then  $\Lambda_1\circ\Lambda_2$ is isotropic in
$X_1\times X_3^a$ by classical results (see {\em e.g.,}~\cite[Prop.~8.3.11]{K-S1}).

In the sequel, we denote by $\BBD$ the open unit disc in the complex line
$\C$, endowed with the coordinate $t$.
We set for short
\eqn
&&Y\eqdot T^*\BBD,
\eneqn
and we consider the projections
\eqn
&&\xymatrix{
&X\times Y\ar[ld]_-{p_1}\ar[d]_-{p}\ar[r]^-{p_2}\ar[rd]^-q&Y\ar[d]^-\pi\\
X&X\times \BBD\ar[r]^-s&\BBD.
}\eneqn
Assume to be given a Lagrangian subvariety $\Lambda\subset X\times Y$
satisfying
\eq\label{eq:hyplagr1}
&&\parbox{60ex}{
the restriction $p\vert_\Lambda\cl\Lambda\to X\times \BBD$ is finite.
}  \eneq
For $a\in \BBD$, writing for short $T^*_a\BBD$ instead of $T^*_{\{a\}}\BBD$, 
we set
\eqn
&&\Lambda_a\eqdot\Lambda\conv T^*_a\BBD=p_1(\Lambda\cap\opb{q}(a)),
\eneqn
and this set is a Lagrangian subvariety of $X$. 

We introduce the ``skyscraper'' $\Ah[Y]$-module 
\eq\label{eq:skyscapper}
&&\shc_a\eqdot \Ah[Y]/\Ah[Y]\cdot(t-a).
\eneq

\begin{theorem}\label{th:hol2}
Let $X$ be a complex symplectic manifold, let $\Lambda$ be a Lagrangian subvariety
of $X\times Y$ satisfying \eqref{eq:hyplagr1}, 
and let $V$ be a Lagrangian subvariety of $X$.
Let $\shl$ be a holonomic $\Ah[X\times Y]$-module 
such that $\Supp(\shl)\subset\Lambda$
and let $\shn$ be a holonomic $\Ah[X]$-module such that
$\Supp(\shn)\subset V$.
Assume that the map $q\cl\Lambda\cap(\opb{p_1}V)\to \BBD$ is proper.
For $a\in \BBD$, we set
$\shl_a\eqdot \shl\conv[Y]\shc_a$
and $\shm\eqdot\roim{p_2}\rhom[{\opb{p_1}\Ah}](\opb{p_1}\shn,\shl)$. Then 
\bnum
\item 
$\shl_a$ is concentrated in degree $0$ and is 
a holonomic $\Ah$-module supported by $\Lambda_a$,
\item
$\shm$  is a coherent $\Ah[Y]$-module supported by 
$V\conv[X]\Lambda$,
\item
$F_a\eqdot\RHom[{\Ah}](\shn,\shl_a)\simeq 
\rsect(Y;(\OO[\BBD]/\OO[\BBD](t-a))\ltens[{\OO[\BBD]}]\shm)$ is an
object of $\Derb_f(\cor)$, 
and $F_a$ and $F_b$ are isomorphic for any $a,b\in \BBD$.
\enum
\end{theorem}
\begin{proof}
(i)\quad First note that $t-a\cl\shl\to\shl$ is a monomorphism.
Indeed for any $s\in\ker(t-a\cl\shl\to\shl)$, 
$\Ah[{X\times Y}]s\subset\shl$ is a coherent $\Ah[X\times Y]$-module whose support is
involutive and of codimension $>d_{X\times Y}/2$, hence empty.
Therefore 
$\shl_a=\shl\conv[Y](\Ah[Y]/\Ah[Y]\cdot(t-a))
\simeq\roim{p_1}\bigl((\OO[\BBD]/\OO[\BBD](t-a))\tens[{\OO[\BBD]}]\shl\bigr)$,
 and (i) follows immediately from the hypothesis \eqref{eq:hyplagr1}.

\noindent
(ii) We have
\eqn
&&\roim{p_2}\rhom[{\opb{p_1}\Ah}](\opb{p_1}\shn,\shl)\simeq \RDA(\shn)\conv\shl.
\eneqn
By the hypothesis, the projection $\Lambda\cap (V\times Y)\to Y$ is proper.
It  follows from  Theorem~\ref{th:kernel1}  that 
$\shm$ belongs to $\Db_\coh(\Ah[Y])$ and is 
supported by the isotropic variety $\Lambda_Y\eqdot V\conv[X]\Lambda$. 

\noindent
(iii) By the hypothesis, the projection $\pi\cl \Lambda_Y\to \BBD$ is
proper, hence finite. It follows easily that $H^i(\shm)$ is a
holonomic $\Ah[Y]$-module and
$H^i(\roim{\pi}\shm)\simeq \oim{\pi}H^i(\shm)$ is 
a locally free $\Ohl[\BBD]$-module of finite rank by Lemma~\ref{lem:hollocfree}. Hence
$$H^i\Bigl(\rsect(Y;(\OO[\BBD]/\OO[\BBD](t-a))\ltens[{\OO[\BBD]}]\shm)\Bigr)
\simeq\sect(Y;H^i(\shm)/(t-a)H^i(\shm))$$
is a finite-dimensional
$\cor$-vector space whose dimension does not depend on $a\in \BBD$.
\end{proof}

We shall make a link between the hypotheses in Theorem~\ref{th:hol2}
and the Hamiltonian  deformations of a Lagrangian variety $\Lambda_0$. 

Assume to be given a holomorphic map
\eqn
&&\Phi(x,t)\cl  X\times\BBD\to X
\eneqn
such that $\Phi(\cdot,a)\cl X\to X$ is a symplectomorphism for each
$a\in \BBD$ and is the identity for $a=0$. Set
\eqn
&&\Gamma\eqdot \{(x,t,\Phi(x,t))\}, 
\mbox{ the graph of $\Phi$ in } X\times X^a\times \BBD.
\eneqn
Consider the differential
\eqn
&&\dfrac{\partial\Phi}{\partial t}\cl X\times \BBD\to TX\simeq T^*X.
\eneqn
We make the  hypothesis:

\eq\label{eq:hypsympH}
&&\parbox{60ex}{
there exists $f\cl X\times \BBD\to\C$ such that $\dfrac{\partial\Phi}{\partial t}=H_f$,
}\eneq
where $H_f$ denotes as usual the Hamiltonian vector field.
In this case, we can define (identifying $T^*\BBD$ with $\BBD\times \C$)
\eqn
&&\tw\Gamma\eqdot\{((x,\Phi(x,t));(t,f(x,t)))\}  
\subset X\times X^a\times T^*\BBD
\eneqn
and $\tw\Gamma$ is Lagrangian.
Let $\Lambda_0$ be a Lagrangian subvariety of $X$. 
We set: 
\eqn
&&\Lambda\eqdot \Lambda_0\circ\tw\Gamma.
\eneqn
Then $\Lambda$ will satisfy hypotheses~\eqref{eq:hyplagr1} and $\Lambda_a=\Phi(x,a)(\Lambda_0)$.

\begin{example}
Let $X=T^*M$, $V=T_M^*M$ and let $\phi\cl M\times \BBD\to\C$ be a
holomorphic function. Set $Y=T^*\BBD$ and let
\eqn
&&\Lambda=\{(x,t;u,\tau)\in X\times Y;(u,\tau)=\grad_{x,t}\,\phi(x,t)\},\\
&&\Lambda_a=\{(x;u)\in X; u=\grad_{x}\,\phi(x,a)\}.
\eneqn 
Consider the family of symplectomorphisms
\eqn
&&\Phi(x,u,t)=(x,u+\phi_x'(x,t)-\phi_x'(x,0)).
\eneqn
Then 
\eqn
&&\dfrac{\partial\Phi}{\partial t}= -H_{\partial_t\phi} 
\mbox{ and }\Lambda_a=\Phi(x,u,a)\Lambda_0.
\eneqn
Set $Z=\{(x,t)\in M\times \BBD;\grad_x\,\phi(x,t)=0\}$
and assume that
\eqn
&&\mbox{the projection $Z\to \BBD$ is proper.}
\eneqn
Consider the ideals 
\eqn
&&\shi=\sum_{i=1}^n\Ah[X\times Y]\cdot(\hbar\partial_{x_i}-\phi'_{x_i})
+\Ah[X\times Y]\cdot(\hbar\partial_{t}-\phi'_t),\\
&&\shi_a=\sum_{i=1}^n\Ah\cdot(\hbar\partial_{x_i}-\phi'_{x_i}(\cdot,a)).
\eneqn
Set $\shn=\Ah\tens[\shd_M]\OO[M]$ and $\shl=\A[X\times Y]/\shi$.
Hence we have $\shl_a=\Ah/\shi_a$
and $H^i\bl\RHom[{\Ah}](\shl_a,\shn)\br$ does not depend on $a\in \BBD$.
\end{example}

\providecommand{\bysame}{\leavevmode\hbox to3em{\hrulefill}\thinspace}

\newpage
\thispagestyle{empty}
\markboth{Index des notations}{}
\renewcommand{\indexname}{Index of notations}

\begin{theindex}

  \item $\sha$, 16
  \item $\shao$, 16
  \item $\AL$, 150
  \item $\ALa$, 156

  \indexspace

  \item $\Chl(\shm)$, 155
  \item $\coro$ $=\C[[\hbar]]$, 60
  \item $\ch(\shf)$, 125
  \item $\cor$ $=\C\Ls$, 60
  \item $\chv$, 135
  \item $\RC^-(\shc)$, 84
  \item $\RC(\shc)$, 11
  \item $\RC^+(\shc)$, 11
  \item $\RC^-(\shc)$, 11
  \item $\RC^\rmb(\shc)$, 11
  \item $\conv[X]$, 92
  \item $\RC^+(\shc)$, 84
  \item $\dA$, 76

  \indexspace

  \item $\shd_M$, 133
  \item $d_X$\ (dimension), 60
  \item $\RD(A)$, 11
  \item $\Derb(A)$, 11
  \item $\Derb_\coh(\sha)$, 11
  \item $\DA$, 65
  \item $\RDA\shm$, 70, 98
  \item $\shend_{\coro}(\A)$, 78
  \item $\RDAh\shm$, 72
  \item $\Derb_\Cc(\cora_M)$, 152
  \item $\Derb_\Rc(\cora_M)$, 152
  \item $\Derb_\wCc(\cora_M)$, 152
  \item $\Der(\shc)$, 11
  \item $\Der^+(\shc)$, 11
  \item $\Der^-(\shc)$, 11
  \item $\Derb_{\coh,\Lambda}(\A)$, 93
  \item $\Der^\rmb(\shc)$, 11
  \item $\Der(\shr^\loc)^{\perp r}$, 38
  \item $\Derb_{\gd,\Lambda(\Ah)}$, 93
  \item $\shd_X\forl$, 65
  \item $\RDOl\shm$, 164
  \item $\DQ$-module, 91

  \indexspace

  \item $\HE[T^*M]$, 133
  \item $\eu(\shf)$, 125
  \item $\eu_X(\shm)$, 142

  \indexspace

  \item $\HHO$, 118
  \item $\hh_X(\shf)$, 118
  \item $\RHHg[\Lambda]{X}$, 114
  \item $\RHHhg[\Lambda]{X}$, 114
  \item $\hh_X(\shm)$, 104
  \item $\hhg_X(\shm)$, 145
  \item $\HHGA$, 114
  \item $\HHA[X]$, 103
  \item $\HHAh$, 115
  \item $\hhhg_X(\shm)$, 115
  \item $\HOD[X]$, 124

  \indexspace

  \item $\cora$, 11
  \item $\rmK(\shc)$, 99
  \item $[M]$, 99
  \item $[\Lambda]$, 160
  \item $\shk_{\coh,\Lambda}(\sho_X)$, 160
  \item $\rmhK_{\coh,\Lambda}(\gr\A)$, 101

  \indexspace

  \item $\LL_\Lambda$, 159
  \item $\Lambda_{1}\circ\Lambda_{2}$, 93
  \item $[\Lambda_\reg]$, 157
  \item $\LL^\cora_\Lambda$, 159
  \item $\lc_X(\shm)$, 160
  \item $\shn^\loc$, 72

  \indexspace

  \item $\md[A]$, 11
  \item $\mdaf[\sha]$, 84
  \item $\mdafd[\sha]$, 84
  \item $\shm^\rmW$, 144

  \indexspace

  \item $\Oh[M]$, 150
  \item $\Ohl[M]$, 150
  \item $\OA$, 81
  \item $\oA$, 81
  \item $\oAh$, 98
  \item $\oA[X\times Y/Y]$, 82
  \item $\ohh[M]$, 150
  \item $\omega_{X_\R}^{\rm top}$, 107

  \indexspace

  \item $\{\pt\}$, 12

  \indexspace

  \item $\SSi$, 152

  \indexspace

  \item $\tau^\ge n$, 11
  \item $\detens$, 68
  \item $\ldetens$, 69
  \item $\dtens[{\A}]$, 92
  \item $\shp^{\otimes-1}$, 77
  \item $\thh_X(\shf)$, 122
  \item $\td_X$, 126

  \indexspace

  \item $\HW[T^*M]$, 133
  \item $\HWo[T^*M]$, 133

\end{theindex}

\newpage
\markboth{Index terminologique}{}
\renewcommand{\indexname}{Terminologies}
\begin{theindex}

  \item $\A$-lattice, 72
  \item algebraically good, 86
  \item algebroid, 49
    \subitem $\DQ$-, 67
    \subitem $\OS$-, 58
    \subitem $\shr$-, 57
    \subitem invertible $\OS$-, 58
    \subitem invertible $\shr$-, 57
  \item almost free
    \subitem $\sha$-module, 84

  \indexspace

  \item bi-differential operator, 60
  \item bi-invertible, 54, 55, 76

  \indexspace

  \item canonical module associated with the diagonal, 76
  \item $\C^\times$-conic, 153
  \item $\C$-constructible, 152
  \item characteristic variety, 135
  \item Chern class, 125
  \item co-Hochschild class, 122
  \item coherent, 13
  \item cohomologically complete, 38
  \item conic
    \subitem $\C^\times$-, 153
    \subitem $\R^+$-, 152
  \item constructible
    \subitem $\C$-, 152
    \subitem $\R$-, 152
    \subitem weakly $\R$-, 152
  \item convolution, 92

  \indexspace

  \item $\DQ$-algebra, 62
  \item $\DQ$-algebroid, 67
  \item dual
    \subitem of $\A$-module, 70
    \subitem of $\Ah$-module, 72
  \item dualizing sheaf, 81

  \indexspace

  \item Euler class
    \subitem of $\Ah$-modules, 142
    \subitem of $\shd$-modules, 147
    \subitem of $\sho$-modules, 125
  \item external product
    \subitem  of $\DQ$-algebras, 64
    \subitem  of $\DQ$-algebroids, 68

  \indexspace

  \item good
    \subitem $\A$-module, 73
    \subitem $\shd$-modules, 144
    \subitem algebraically, 86
    \subitem module, 30
  \item Grothendieck group, 99

  \indexspace

  \item $\hbar$-complete, 16
  \item $\hbar$-separated, 16
  \item $\hbar$-torsion, 16
  \item $\hbar$-completion, 16
  \item Hochschild class
    \subitem of an $\A$-module, 104
    \subitem of an $\shd$-module, 145
    \subitem of an $\sho$-module, 118
  \item Hochschild homology
    \subitem of $\OO[]$, 118
  \item Hochschild-Kostant-Rosenberg map, 126
  \item Hodge cohomology, 124
  \item holonomic, 157
    \subitem $\AL$-module, 152
    \subitem simple, 157

  \indexspace

  \item invertible, 51
    \subitem $\OS$-algebroid, 58
    \subitem $\shr$-algebroid, 57
  \item isomorphism
    \subitem standard, 63
  \item isotropic subvariety, 157

  \indexspace

  \item Lagrangian cycle, 160
  \item Lagrangian subvariety, 157
  \item lattice, 72
  \item $\AL$-lattice, 155
  \item locally finitely generated, 12
  \item locally of finite presentation, 12
  \item locally projective, 34

  \indexspace

  \item microsupport, 152
  \item Mittag-Leffler condition, 13
  \item module
    \subitem bi-invertible, 54, 55
    \subitem coherent, 13
    \subitem invertible, 51
    \subitem locally finitely generated, 12
    \subitem locally of finite presentation, 12
    \subitem Noetherian, 13
    \subitem pseudo-coherent, 12
    \subitem simple, 71
  \item modules
    \subitem over an algebroid, 51

  \indexspace

  \item Noetherian, 13
  \item no $\hbar$-torsion, 16

  \indexspace

  \item $\OS$-algebroid, 58
    \subitem invertible, 58

  \indexspace

  \item pseudo-coherent, 12

  \indexspace

  \item $\shr$-algebroid, 57
    \subitem invertible, 57
  \item $\R^+$-conic, 152
  \item $\R$-constructible, 152
  \item right orthogonal, 38
  \item ring
    \subitem Noetherian, 13

  \indexspace

  \item section
    \subitem standard, 63
  \item simple
    \subitem holonomic, 157
    \subitem module, 71
  \item standard
    \subitem isomorphism, 63
  \item standard section, 63
  \item star product, 60
  \item submodule
    \subitem of $\Ah[\stx]$-module, 72

  \indexspace

  \item thick subcategory, 73
  \item Todd class, 126

  \indexspace

  \item weakly $\R$-constructible, 152

\end{theindex}


%

\newpage
\markboth{Index terminologique}{}
\renewcommand{\indexname}{}
\parbox[t]{16em}
{\scriptsize{
\noindent
Masaki Kashiwara\\
Research Institute for Mathematical Sciences,\\
Kyoto University,\\
Kyoto 606-01 Japan\\
email: masaki@kurims.kyoto-u.ac.jp}
}
\hspace{0.1cm}
\parbox[t]{16em}
{\scriptsize{
Pierre Schapira\\
Institut de Math{\'e}matiques\\
Universit{\'e} Pierre et Marie Curie\\
e-mail: schapira@math.jussieu.fr\\
http://www.math.jussieu.fr/\textasciitilde schapira/}}

\end{document}